\documentclass{amsbook}

\usepackage{etex} 

\usepackage[linktocpage]{hyperref}
\usepackage[a4paper, centering, vcentering]{geometry}
\usepackage{fullpage}
\usepackage[centertableaux]{ytableau}
\usepackage{imakeidx}     

\usepackage[T1]{fontenc} 
\usepackage{lmodern} 

\usepackage{mathtools} 
\mathtoolsset{showonlyrefs,showmanualtags} 

\usepackage[english]{babel}
\usepackage[all]{xy}
\usepackage{amsthm,amssymb,amsbsy,amsmath,amsfonts,mathrsfs,amscd}
\usepackage{cite}
\usepackage{url}

\usepackage{color}

\usepackage[usenames,dvipsnames]{pstricks}
\usepackage{epsfig}
\usepackage{pst-grad} 
\usepackage{pst-plot} 

\usepackage{tikz}\usetikzlibrary{arrows}

\newrgbcolor{blue}{0 0 0.6} 

\newtheorem{maintheorem}{Theorem} 

\newtheorem{maintheoremintro}{Theorem} 

\newtheorem{theorem}{Theorem}[chapter]
\newtheorem*{theorem*}{Theorem}
\newtheorem{corollary}[theorem]{Corollary}
\newtheorem{lemma}[theorem]{Lemma}
\newtheorem*{lemma*}{Lemma}
\newtheorem{proposition}[theorem]{Proposition}
\newtheorem*{proposition*}{Proposition}

\theoremstyle{definition} 
\newtheorem{definition}[theorem]{Definition} 
\newtheorem*{definition*}{Definition} 
\newtheorem{example}[theorem]{Example}
\theoremstyle{remark}
\newtheorem{remark}[theorem]{Remark}
\newtheorem*{remark*}{Remark}

\numberwithin{section}{chapter}
\numberwithin{equation}{chapter}

\newcommand{\bt}{\begin{theorem}}
\newcommand{\et}{\end{theorem}}
\newcommand{\bl}{\begin{lemma}}
\newcommand{\el}{\end{lemma}}
\newcommand{\bp}{\begin{proposition}}
\newcommand{\ep}{\end{proposition}}
\newcommand{\bc}{\begin{corollary}}
\newcommand{\ec}{\end{corollary}}
\newcommand{\bdeff}{\begin{definition}}
\newcommand{\edeff}{\end{definition}}
\newcommand{\brem}{\begin{remark}}
\newcommand{\erem}{\end{remark}}
\newcommand{\bex}{\begin{example}}
\newcommand{\eex}{\end{example}}
\newcommand{\bcen}{\begin{center}}
\newcommand{\ecen}{\end{center}}

\newcommand{\lp}{\left(}
\newcommand{\rp}{\right)}

\newcommand{\bi}{\begin{itemize}}
\newcommand{\iii}{\item}
\newcommand{\ei}{\end{itemize}}
\newcommand{\bd}{\begin{description}}
\newcommand{\ed}{\end{description}}

\newcommand{\bqn}{\begin{eqnarray}}
\newcommand{\eqn}{\end{eqnarray}}

\newcommand{\lam}{\lambda}

\newcommand{\g}{\gamma}
\newcommand{\al}{\alpha}
\newcommand{\eps}{\varepsilon}

\newcommand{\R}{\mathbb{R}}
\newcommand{\N}{\mathbb{N}}

\newcommand{\mb}[1]{\mathbb{ #1 }}
\newcommand{\mc}[1]{\mathcal{ #1 }}

\newcommand{\all}{\forall\,}
\newcommand{\ex}{\exists\,}
\newcommand{\la}{\langle}
\newcommand{\ra}{\rangle}

\newcommand{\virg}[1]{``#1''}
\newcommand{\tx}[1]{\mathrm{#1}}

\newcommand{\til}[1]{\widetilde{#1}}
\newcommand{\wh}[1]{\widehat{#1}}
\newcommand{\wt}[1]{\widetilde{#1}}

\newcommand{\review}{}
\newcommand{\finereview}{}

\newcommand{\VecM}{\mathrm{Vec}(M)}
\newcommand{\distr}{\mathscr{D}}

\newcommand{\metr}[2]{\langle #1|#2\rangle}

\newcommand{\Pg}[1]{\left\{ #1 \right\}}

\newcommand{\EXP}{\mc{E}}
\newcommand{\im}{\tx{Im}\,}
\newcommand{\lapl}{\Delta}
\newcommand{\dive}{\tx{div}}
\newcommand{\grad}{\nabla}
\newcommand{\LL}{L}
\newcommand{\A}{A}
\newcommand{\DD}{\mathscr{F}}
\newcommand{\J}{J}
\newcommand{\HH}{\mathcal{H}}
\renewcommand{\SS}{S}
\newcommand{\cc}{c}
\newcommand{\m}{\mathfrak{m}}
\newcommand{\cb}{\mathbb{U}}
\newcommand{\fib}{\cb}

\newcommand{\U}{\mc{U}}
\newcommand{\End}{E}
\newcommand{\gt}{\gamma_1(t)}
\newcommand{\gs}{\gamma_2(s)}

\newcommand{\n}{\mc{N}}
\newcommand{\Ric}{\tx{Ric}}
\newcommand{\QQ}{\mc{Q}}
\newcommand{\Qz}{\mc{I}}
\newcommand{\RR}{\mc{R}}
\newcommand{\ve}{\mathcal{V}}
\newcommand{\hor}{\mathcal{H}}
\newcommand{\f}{\mathfrak{f}}
\newcommand{\uu}{\bar{u}}
\newcommand{\gd}{\mathcal{N}} 
\newcommand{\id}{\mathbb{I}} 
\newcommand{\Sred}{S^\flat} 
\newcommand{\DDa}{\DD} 
\newcommand{\tanf}{\mathsf{T}} 
\newcommand{\dist}{\mathsf{d}}
\newcommand{\popp}{\mu}
\newcommand{\ad}{\tx{ad}}
\newcommand{\euler}{\mathfrak{e}} 

\DeclareMathOperator{\rank}{\tx{rank}}
\DeclareMathOperator{\spn}{\tx{span}}
\DeclareMathOperator{\spec}{\tx{spec}}
\DeclareMathOperator{\trace}{\tx{tr}}
\DeclareMathOperator{\discr}{\tx{discr}}

\makeindex

\begin{document}

\frontmatter

\title{Curvature: a variational approach}

\author{A. Agrachev}
\address{SISSA, Italy, MI RAS and IM SB RAS, Russia.}
\email{agrachev@sissa.it}

\author{D. Barilari}
\address{CNRS, CMAP \'Ecole Polytechnique and \'Equipe INRIA GECO Saclay \^Ile-de-France, Paris, France.}
\curraddr{IMJ-PRG, UMR CNRS 7586, Universit\'e Paris-Diderot,
Paris, France.}
\email{davide.barilari@imj-prg.fr}

\author{L. Rizzi}
\address{SISSA, Trieste, Italy.}
\curraddr{CNRS, CMAP \'Ecole Polytechnique and \'Equipe INRIA GECO Saclay \^Ile-de-France, Paris, France.}
\email{luca.rizzi@cmap.polytechnique.fr}

\thanks{The first author has been supported by the grant of the Russian Federation for the state support of research, Agreement No 14 B25 31 0029. The second author has been supported by the European Research Council, ERC StG 2009 \virg{GeCoMethods}, contract number 239748, by the ANR Project GCM, program \virg{Blanche}, project number NT09-504490. The third author has been supported by INdAM (GDRE CONEDP) and Institut Henri Poincar\'e, Paris, where part of this research has been carried out. We warmly thank Richard Montgomery and Ludovic Rifford for their careful reading of the manuscript. We are also grateful to Igor Zelenko and Paul W.Y.\ Lee for very stimulating discussions.}

\date{\today}



\begin{abstract}
The curvature discussed in this paper is a  far reaching generalisation of the Riemannian sectional curvature.
We give a unified definition of curvature which applies to a wide class of geometric structures whose geodesics arise from optimal control problems, including Riemannian, sub-Riemannian, Finsler and sub-Finsler spaces. Special attention is paid to the sub-Riemannian
(or Carnot--Carath\'eodory) metric spaces. Our construction of curvature is direct and naive, and similar to the original approach of Riemann. In particular, we extract geometric invariants from the asymptotics of the cost of optimal control problems. Surprisingly, it works in a very general setting and, in particular, for all sub-Riemannian spaces.
\end{abstract}

\maketitle

\tableofcontents

\chapter{Introduction}

%
%

The curvature discussed in this paper is a  far reaching generalisation of the Riemannian sectional curvature.
We give a unified definition of curvature which applies to a wide class of geometric structures whose geodesics arise from optimal control problems, including Riemannian, sub-Riemannian, Finsler and sub-Finsler spaces. Special attention is paid to the sub-Riemannian
(or Carnot--Carath\'eodory) metric spaces. Our construction of curvature is direct and naive, and similar to the original approach of Riemann. Surprisingly, it works in a very general setting and, in particular, for \emph{all} sub-Riemannian spaces. 

\medskip Interesting metric spaces often appear as limits of families of Riemannian metrics. We first try
to explain  our curvature by describing it in  the case of a contact sub-Riemannian structure arising as such a limit and then we move to the general construction.

Let $M$ be an odd-dimensional Riemannian manifold endowed with a contact vector distribution $\distr\subset TM$.
Given $x_0,x_1\in M$, the contact sub-Riemannian distance $\dist(x_0,x_1)$ is the infimum of the lengths of
Legendrian curves connecting $x_0$ and $x_1$ (Legendrian curves are  integral curves of the
distribution $\distr$). The metric $\dist$ is easily realized as the limit of a family of Riemannian metrics $\dist^\varepsilon$
as $\varepsilon\to 0$. To define $\dist^\varepsilon$ we start from the original Riemannian structure on $M$,
keep fixed the length of vectors from $\distr$ and multiply by $\frac 1\varepsilon$ the length of the orthogonal
to $\distr$ tangent vectors to $M$, thus defining a Riemannian metric $g^{\eps}$, whose distance is $\dist^{\eps}$. It is easy to see that $\dist^\varepsilon\to\dist$ uniformly on compacts in $M\times M$ as $\varepsilon\to 0$.

The distance converges. What about the curvature? Let $\omega$ be a contact one-form that annihilates
$\distr$, i.e. $\distr=\omega^\perp$. Given $v_1,v_2\in T_xM,\ v_1\wedge v_2\ne 0,$ we denote by
$K^\varepsilon(v_1\wedge v_2)$ the sectional curvature of the two-plane $\tx{span}\{v_1,v_2\}$ with respect to the metric $g^\varepsilon$. It is not hard to show that $K^\varepsilon(v_1\wedge v_2)\to-\infty$ if $v_1,v_2\in \distr$
and $d\omega(v_1,v_2)\ne 0$. Moreover, $\mathrm{Ric}^\varepsilon(v)\to-\infty$ as $\varepsilon\to 0$ for any nonzero
vector $v\in \distr$, where $\mathrm{Ric}^\varepsilon$ is the Ricci curvature for the metric $\dist^\varepsilon$.
On the other hand, the distance between $x$ and the conjugate locus of $x$ tends to 0 as $\varepsilon\to 0$
so $K^\varepsilon(v_1\wedge v_2)$ tends to $+\infty$ for some $v_1,v_2\in T_xM$, as well as $\mathrm{Ric}^\eps(v)\to +\infty$ for
some $v\in T_xM$.

What about the geodesics? For any $\varepsilon>0$ and any $v\in T_xM$ there is a unique geodesic of the
 Riemannian metric $\dist^\varepsilon$ that starts from $x$ with velocity $v$. On the other hand, the velocities of all geodesics of the limit metric $\dist$ belong to $\distr$ and for any nonzero vector $v\in \distr$ there exists a
 one-parametric family of geodesics whose initial velocity is equal to $v$. So when written on the tangent bundle the convergence of the flows behave poorly. However, the family of geodesic flows converges if we rewrite it as a family of flows on the cotangent bundle.

Indeed, any Riemannian structure on $M$ induces a self-adjoint isomorphism $G:TM\to T^*M$, where $\langle Gv,v\rangle$ is the square of the length of the vector $v\in TM$, and $\langle \cdot,\cdot \rangle$ denotes the standard pairing between tangent and cotangent vectors. The geodesic flow, treated as flow on $T^*M$, is a Hamiltonian flow associated with the Hamiltonian function $H:T^*M\to\mathbb R$, where $H(\lam)=\tfrac{1}{2}\langle \lam,G^{-1}\lam\rangle,\ \lam\in T^*M$. Let $(\lam(t),\g(t))$
be a trajectory of the Hamiltonian flow, with $\lam(t)\in T^*_{\g(t)}M$. The square of the Riemannian distance from $x_0$
is a smooth function on a neighborhood of $x_0$ in $M$ and the differential of this function at $\g(t)$ is equal
to $2t\lam(t)$ for any small $t\ge 0$. Let $H^\varepsilon$ be the Hamiltonian corresponding to the metric $\dist^\varepsilon$. It is easy to see that $H^\varepsilon$ converges with all derivatives to a Hamiltonian $H^0$.
Moreover, geodesics of the limit sub-Riemannian metric are just projections to $M$ of the trajectories of the
Hamiltonian flow on $T^*M$ associated to $H^0$.

We will recover the Riemannian curvature from the asymptotic expansion of the square of the distance from $x_0$
along a geodesic: this is essentially what Riemann did. Then we can write a similar expansion for the square of the limit sub-Riemannian distance to get an idea of the curvature in this case.
Note that the metrics $\dist^\varepsilon$ converge to $\dist$ with all derivatives in any point of $M\times M$,
where $\dist$ is smooth. The metrics $\dist^\varepsilon$ are not smooth at the diagonal but their squares are
smooth. The point is that no power of $\dist$ is smooth at the diagonal! Nevertheless, the desired asymptotic
expansion can be controlled.

Fix a point $x_0\in M$ and $\lam_0\in T^*_{x_0}M$ such that $\langle \lam_0,\distr\rangle\neq 0$. Let
$(\lam^\varepsilon(t),\g^\varepsilon(t))$, for $\varepsilon\ge 0$, be the trajectory of the Hamiltonian flow associated
to the Hamiltonian $H^\varepsilon$ and initial condition $(\lam_{0},x_{0})$. We set:
\begin{equation}
c_t^\varepsilon(x)\doteq -\frac 1{2t}(\dist^\varepsilon)^2(x,\g^\varepsilon(t))\ \mathrm{if}\ \varepsilon>0,\qquad
c^0_{t}(x)\doteq -\frac 1{2t}\dist^2(x,\g^{0}(t)).
\end{equation}
There exists an interval $(0,\delta)$ such that the functions $c^\varepsilon_t$ are smooth at $x_0$ for all
$t\in(0,\delta)$ and all $\varepsilon\ge 0$. Moreover, $d_{x_0}c^\varepsilon_t=\lam_0$. Let
$\dot c_t^\varepsilon=\frac\partial{\partial t}c_t^\varepsilon$, then $d_{x_0}\dot c^\varepsilon_t=0$.
In other words, $x_0$ is a critical point of the function $\dot c^\varepsilon_t$ so its Hessian
$d^2_{x_0}\dot c^\varepsilon_t$ is a well-defined quadratic form on $T_{x_0}M$. Recall that $\varepsilon=0$ is available, but $t$ must be positive.
We are going to study the asymptotics of the family of quadratic forms
$d^2_{x_0}\dot c^\varepsilon_t$ as $t\to 0$ for fixed $\varepsilon$. This asymptotics is  different
for $\varepsilon>0$ and $\varepsilon=0$. The change reflects the structural difference of the Riemannian
and sub-Riemannian metrics and emphasizes the role of the curvature.
{\review
In this approach, the curvature is encoded in the function $\dot c_t(x)$. A geometrical interpretation of such a function can be found in Appendix~\ref{s:ctdot}.
}

Given $v,w\in T_xM,\ \varepsilon>0$, we denote $\langle v|w\rangle_\varepsilon=\langle G^\varepsilon v,w\rangle$ the inner product generating $\dist^{\eps}$. Recall that $\langle v|v\rangle_\varepsilon$ does not depend on $\varepsilon$ if $v\in \distr$ and
$\langle v|v\rangle_\varepsilon\to\infty\ (\varepsilon\to 0)$ if $v\notin \distr$; we will write
$|v|^2\doteq \langle v|v\rangle_\varepsilon$ in the first case. For fixed $\varepsilon>0$, we have:
\begin{equation}
d^2_{x_0}\dot c^\varepsilon_t(v)=\frac 1{t^2}\langle v|v\rangle_\varepsilon+
\frac 13\langle R^\varepsilon(\dot \gamma^\varepsilon,v)\dot \g^\varepsilon)|v\rangle_\varepsilon+O(t),\qquad
v\in T_{x_0}M,
\end{equation}
where $\dot\g^{\eps} = \dot\g^{\eps}(0)$ and $R^\varepsilon$ is the Riemannian curvature tensor of the metric $\dist^\varepsilon$.
For $\varepsilon=0$, only vectors $v\in \distr$ have a finite length and the above expansion is modified as follows:
\begin{equation}
d^2_{x_0}\dot c^0_t(v)=\frac 1{t^2}\Qz_{\g}(v)+
\frac{1}{3} \RR_{\g}(v)+O(t),\qquad
v\in \distr\cap T_{x_0}M,
\end{equation}
where $\Qz_{\g}(v)\ge|v|^2$ and $\RR_{\g}$ is the \emph{sub-Riemannian curvature} at $x_0$ along the geodesic $\g = \g^0$. Both $\Qz_{\g}$ and $\RR_{\g}$ are quadratic forms on $\distr_{x_0}\doteq  \distr\cap T_{x_0}M$. The principal ``structural'' term $\Qz_{\g}$ has the following properties: let $K_{\g}$  be the linear hyperplane inside $\distr_{x_{0}}$ defined as the $d\omega$-orthogonal to $\dot\g(0)$, namely $K_{\g}=\{v\in \distr_{x_{0}} \, |\, d\omega(v,\dot{\g}(0))=0\}$ and let $K_{\g}^{\perp}$ be its sub-Riemannian orthogonal inside $\distr_{x_{0}}$. Then 
\begin{gather}
\Qz_\g(v)=\begin{cases}
 |v|^{2} & \text{if} \ v\in K_{\gamma},\\
4|v|^{2} & \text{if} \ v\in K_{\gamma}^{\perp}.
\end{cases}
\end{gather}
In other words, the symmetric operator on $\distr_{x_0}$  associated with the quadratic form $\Qz_\gamma$ has eigenvalue $1$
of multiplicity $\dim \distr_{x_0}-1$ and eigenvalue $4$ of multiplicity $1$. The trace of this operator, which, in this case, does not depend on $\gamma$, equals
$\dim \distr_{x_0}+3$. This trace has a simple geometric interpretation, it is equal to the \emph{geodesic dimension}
of the sub-Riemannian space.

The geodesic dimension is defined as follows. Let $\Omega\subset M$ be a bounded and measurable subset of positive volume and let $\Omega_{x_{0},t}$, for $0\le t\le 1$, be a family of  subsets obtained from $\Omega$ by the homothety of $\Omega$ with respect to a fixed point $x_{0}$
along the shortest geodesics connecting $x_{0}$ with the points of $\Omega$, so that $\Omega_{x_{0},0}=\{x_{0}\},\ \Omega_{x_{0},1}=\Omega$. The volume
of $\Omega_{x_{0},t}$ has order $t^{\mathcal{N}_{x_{0}}}$, where $\mathcal{N}_{x_{0}}$ is the geodesic dimension at $x_{0}$ (see Section~\ref{s:gd} for details).

Note that the geodesic dimension is $\dim \distr_{x_0}+3$, while the topological dimension of our contact sub-Riemannian space is $\dim \distr_{x_0}+1$, the Hausdorff dimension is $\dim \distr_{x_0}+2$. All three dimensions are obviously equal for Riemannian or Finsler manifolds. The structure of the
term $\Qz_\g$ and comparison of the asymptotic expansions of $d^2_{x_0}\dot c^\varepsilon_t$ for
$\varepsilon>0$ and $\varepsilon=0$ explains why sectional curvature goes to $-\infty$ for certain sections.

\medskip  The curvature operator which we define can be computed in terms of the symplectic invariants of the Jacobi curve,  a curve in the Lagrange Grassmannian related to the linearisation of the Hamiltonian flow. These symplectic invariants can be computed, via an algorithm which is, however, quite hard to implement. Explicit computations of the contact sub-Riemannian curvature in dimension three appears in Section \ref{s:3Dcomputations}, while the computations of the curvature in the higher dimensional contact case will be the object of a forthcoming paper. The current
paper deals with the presentation of the general setting and the study of the structure of the asymptotic of $c_{t}$ in its generality. All the details are presented in the forthcoming sections but, since the paper is long, we find it worth to briefly describe the main ideas in the introduction (beware to the slightly different notation with respect to the rest of the paper).

Let $M$ be a smooth manifold, $\distr\subset TM$ be a vector distribution (not necessarily contact), $f_{0}$ be a vector field on $M$ and $L:TM\to M$ be a Tonelli Lagrangian (see Section \ref{s:affcs} for precise definitions). \emph{Admissible paths} on $M$ are curves whose velocities belong to the \virg{affine distribution}
$f_{0}+\distr$. Let $\mc{A}_t$ be the space of admissible paths defined on the segment $[0,t]$ and
$N_t=\{(\gamma(0),\gamma(t)): \gamma\in\mc{A}_t\}\subset M\times M$. The optimal cost (or action) function
$\SS_t:N_t\to\mathbb R$ is defined as follows:
\begin{equation}
\SS_t(x,y)=\inf\left\{\int_0^tL(\dot\gamma(\tau))\,d\tau: \gamma\in\mc{A}_t,\ \gamma(0)=x,\
\gamma(t)=y\right\}.
\end{equation}
The space $\mc{A}_t$ equipped with the $W^{1,\infty}$-topology is a smooth Banach manifold. The functional
$J_t:\gamma\mapsto\int_0^t L(\dot\gamma(\tau))\,d\tau$ and the evaluation maps
$F_\tau:\gamma\mapsto\gamma(\tau)$ are smooth on $\mc{A}_t$.

The optimal cost $\SS_t(x,y)$ is the solution of the conditional minimum problem for the functional
$J_t$ under conditions $F_0(\gamma)=x,\ F_t(\gamma)=y$. The Lagrange multipliers rule for this problem
reads:
\begin{equation}\label{eq:1}
d_\gamma J_t=\lambda_t D_\gamma F_t-\lambda_0 D_\gamma F_0.
\end{equation}
Here $\lambda_t$ and $\lambda_0$ are ``Lagrange multipliers'', $\lambda_t\in T^*_{\gamma(t)}M,\
\lambda_0\in T^*_{\gamma(0)}M$. We have:
\begin{equation}
D_\gamma F_t:T_\gamma\mc{A}_t\to T_{\gamma(t)}M,\qquad \lambda_t:T_{\gamma(t)}M\to\mathbb R,
\end{equation}
and the composition $\lambda_t D_{\gamma}F_t$ is a linear functional on $T_\gamma\mc{A}_t$. Moreover, Eq.~\eqref{eq:1} implies that
\begin{equation}\label{eq:2}
d_\gamma J_{\tau}=\lambda_{\tau} D_\gamma F_{\tau}-\lambda_0 D_\gamma F_0,
\end{equation}
for some $\lambda_{\tau}\in T^*_{\gamma(\tau)}M$ and any $\tau\in[0,t]$ (see for instance \cite[Proposition I.2]{cime}). The curve $\tau\mapsto\lambda_{\tau}$
is a trajectory of the Hamiltonian system associated to the Hamiltonian $H:T^*M\to\mathbb R$ defined by
\begin{equation}
H(\lambda)=\max_{v\in f_{0}(x)+\distr_x}\left(\langle\lambda,v\rangle-L(v)\right),\qquad \lam\in T^*_xM,\, x\in M. \label{eq:Hamham}
\end{equation}
Moreover, any trajectory of this Hamiltonian system satisfies relation \eqref{eq:2}, where $\gamma$ is the projection
of the trajectory to $M$. Trajectories of the Hamiltonian system are called \emph{normal extremals} and their projections to $M$ are called  \emph{normal extremal trajectories}.

We recover the sub-Riemannian setting by taking $f_0=0$, and $L(v)=\frac 12\langle Gv,v\rangle$. Then, the optimal cost $S_{t}$ is related with the sub-Riemannian distance $\dist(x,y)$ by $\SS_t(x,y)=\frac 1{2t}\dist^2(x,y)$, and normal extremal trajectories are normal sub-Riemannian geodesics.

Let $\gamma$ be an admissible path. The germ of $\gamma$ at the point $x_{0}=\gamma(0)$ defines a flag in $T_{x_{0}}M$ 
$\{0\}=\DD^0_\gamma\subset\DD^1_\gamma\subset\DD^2_\gamma\subset\ldots\subset T_{x_{0}}M$
in the following way. Let $V$ be a section of the vector distribution $\distr$ such that
$\dot\gamma(t)=f_{0}(\gamma(t))+V(\gamma(t)),\ t\ge 0,$ and $P^{t}$ be the local flow on $M$ generated by the vector field
$f_{0}+V$; then $\gamma(t)=P^{t}(\gamma(0))$. We set:
\begin{equation}
\DD^i_\gamma=\spn\left\{\left.\frac{d^j}{dt^j}\right|_{t=0}P^{-t}_{*}\distr_{\gamma(t)} : j=0,\ldots,i-1\right\}.
\end{equation}
The flag $\DD^i_\gamma$ depends only on the germs of $f_{0}+\distr$ and $\gamma$ at the initial point $x_{0}$.

A normal extremal trajectory $\gamma$ is called \emph{ample} if $\DD^m_\gamma=T_{x_{0}}M$ for some
$m> 0$. If $\gamma$ is  ample, then $J_t(\gamma)=\SS_t(x_{0},\gamma(t))$ for all
sufficiently small $t>0$ and $\SS_t$ is a smooth function in a neighborhood of $(\gamma(0),\gamma(t))$. Moreover,
$\frac{\partial \SS_t}{\partial y}\bigr|_{y=\gamma(t)}=\lambda_t,\ \frac{\partial \SS_t}{\partial x}\bigr|_{x=\gamma(0)}=-\lambda_0,$ where $\lambda_t$ is the normal extremal whose projection is $\gamma$.

We set $c_t(x)\doteq -\SS_t(x,\gamma(t))$; then $d_{x_{0}}c_t=\lambda_0$ for any $t> 0$ and $x_{0}$ is a critical point of the function $\dot c_t$. The Hessian of this function $d^2_{x_{0}}\dot c_t$ is a well-defined
quadratic form on $T_{x_{0}}M$. We are going to write an asymptotic expansion of
$d^2_{x_{0}}\dot c_t\bigr|_{\distr_{x_{0}}}$ as $t\to 0$  (see Theorem~\ref{t:main}):

\begin{equation}
d^2_{x_{0}}\dot c_t(v)=\frac 1{t^2}\Qz_\gamma(v)+\frac 13 \RR_\gamma(v)+O(t),\qquad \forall v\in \distr_{x_{0}}.
\end{equation}

Now we introduce a natural Euclidean structure on $T_{x_{0}}M$. Since $L$ is Tonelli,  $\left. L \right|_{T_{x_0}M}$ is a smooth strictly convex function, and $d^2_w(\left. L \right|_{T_{x_0}M})$ is a positive definite quadratic form on $T_{x_0}M,\ \all w\in T_{x_0}M$. If we
 set $|v|_\gamma^2=d^2_{\dot\gamma(0)}(\left. L \right|_{T_{x_0}M})(v),\ v\in T_{x_{0}}M$ we have the inequality
\begin{equation}
\Qz_{\g}(v)\ge|v|^2_\gamma,\qquad \all v\in \distr_{x_{0}}.
\end{equation}
The inequality
$\Qz_\gamma(v)\ge|v|^2_\gamma$ means that the eigenvalues of the symmetric operator
on $\distr_{x_0}$ associated with the quadratic form $\Qz_\gamma$ with respect to $|\cdot|_{\g}$ are greater or equal than $1$. The quadratic form $\RR_\gamma$ is the \emph{curvature}
of our constrained variational problem along the extremal trajectory $\gamma$. 

A mild regularity assumption allows us to explicitly compute the eigenvalues of $\Qz_{\g}$. We set $\gamma_\varepsilon(t)=\gamma(\varepsilon+t)$ and assume that $\dim\DD^i_{\gamma_\varepsilon}=\dim\DD^i_\gamma$ for all sufficiently small $\varepsilon\geq 0$ and all $i$.
Then $d_i=\dim\DD^i_\gamma-\dim\DD^{i-1}_\gamma$, for $i\geq 1$ is a non-increasing sequence of
natural numbers with $d_1= \dim \distr_{x_{0}} =k$. We draw a Young tableau with $d_i$ blocks in the $i$-th column and we define $n_1,\ldots,n_k$ as the lengths of its rows (that may depend on $\gamma$).
\begin{equation}
\ytableausetup{boxsize=2.2em}
\begin{ytableau}
\none[n_1] & \empty & \empty & \none[\dots] & \empty & \empty \\
\none[n_2] & \empty & \empty & \none[\dots] & \empty & \none[d_m] \\
\none & \none[\vdots] & \none[\vdots] & \none & \none[d_{m-1}] \\
\none[n_{k-1}] & \empty & \empty \\
\none[n_k] & \empty & \none[d_2] \\
\none & \none[d_1]
\end{ytableau}
\end{equation}
The eigenvalues of the symmetric operator $\Qz_\gamma$ are $n_1^2,\ldots,n_{k}^2$ (see Theorem~\ref{t:main2}). 
All $n_{i}$ are equal to 1 in the Riemannian case. In the sub-Riemannian setting, the trace of $\Qz_{\g}$ is
\begin{equation}
\trace\Qz_{\g}=n_1^2+\cdots+n_{k}^2=\sum_{i=1}^{m}(2i-1)d_i,
\end{equation}
along an ample normal sub-Riemannian geodesic. This trace is equal to the geodesic dimension of the space (see
Theorem~\ref{t:main4}).

\medskip The construction of the curvature presented here was preceded by a rather long research line
(see \cite{AAPL,cime,agrafeedback,geometryjacobi1,lizel2,lizel}). For alternative approaches to curvatures, one can see  \cite{garofalob,baudoincontact} and references therein for a heat equation approach to the generalization of the curvature-dimension inequality and \cite{agsmms,lottvillaniannals,sturm1,sturm2} and references therein for an optimal transport approach to the generalization of Ricci curvature to metric measure spaces. These works are in part motivated by the lack of classical Riemannian tools, such as the Levi-Civita connection and the theory of Jacobi fields. For a more recent discussion on these last topics, see \cite{BR-jacobi}.
           
%
          

\section{Structure of the paper}
In Chapters~\ref{c:affcs}--\ref{c:geodcost} we give a detailed exposition of the main constructions in a more general and flexible
setting than in this introduction. Chapter~\ref{c:srg} is devoted to the specification to the case of sub-Riemannian spaces
and to some further results: the proof that ample geodesics always exist (Theorem~\ref{t:bound}), an asymptotic expansion of the sub-Laplacian applied to the square of the distance (Theorem~\ref{t:main3}), the computation of the geodesic dimension (Theorem~\ref{t:main4}). 

Before entering into details of the proofs, we end Chapter~\ref{c:srg} by repeating our construction for one of the simplest sub-Riemannian structures: the Heisenberg group. In particular, we recover by a direct computation the results of Theorems~\ref{t:main}, \ref{t:main2} and \ref{t:main3}.

The proofs of the main results are concentrated in Chapters~\ref{c:jac}--\ref{c:sublapproof} where we introduce the main technical tools: Jacobi curves, their symplectic invariants and Li--Zelenko structural equations.

\section{Statements of the main theorems}
The main results, namely Theorems~\ref{t:main}, \ref{t:main2}, \ref{t:main3} and \ref{t:main4}, are spread in Part I of the paper. For convenience of the reader we collect them here, without any pretence at completeness. To be consistent with the original statements, in this section we express the dependence  of the operators and the scalar product on $\gamma$ through the associated initial covector $\lambda$.

\medskip

Let $\gamma:[0,T]\to M$ be an ample geodesic with initial covector $\lam\in T^{*}_{x_0}M$, and let $\QQ_\lam(t)$ be the symmetric operator associated with the second derivative $d^{2}_{x_{0}}\dot{c}_t$ via the scalar product $\langle \cdot|\cdot\rangle_\lambda$, defined for sufficiently small $t>0$.
\begin{maintheoremintro}[Section~\ref{s:2dctdot}]
The map $t\mapsto t^{2}\QQ_\lam(t)$ can be extended to a smooth family of operators on $\distr_{x_0}$ for small $t\geq 0$, symmetric with respect to $\langle\cdot|\cdot\rangle_\lam$. Moreover,
\begin{equation}
\Qz_{\lam}\doteq \displaystyle \lim_{t\to 0^+}t^{2}\QQ_\lam(t) \geq \id > 0,
\end{equation}
as operators on $(\distr_{x_0},\langle\cdot|\cdot\rangle_\lam)$. Finally
\begin{equation}
\left.\dfrac{d}{dt}\right|_{t=0}t^{2}\QQ_\lam(t)=0.
\end{equation}
\end{maintheoremintro}
The \emph{curvature} is the symmetric operator $\RR_{\lam}:\distr_{x_0}\to \distr_{x_0}$ defined by
\begin{equation}
\RR_{\lam}\doteq\dfrac{3}{2}\left.\dfrac{d^{2}}{dt^{2}}\right|_{t=0}t^2\QQ_\lam(t).
\end{equation}
Moreover, the \emph{Ricci curvature} at $\lam\in T_{x_0}^*M$ is the scalar function defined by $\Ric(\lam)\doteq\trace \RR_{\lam}$. In particular,  we have the following Laurent expansion for the family of symmetric operators $\QQ_\lam(t): \distr_{x_0} \to \distr_{x_0}$
\begin{equation}\label{eq:mainexpintro}
\QQ_\lam(t)= \frac{1}{t^{2}}\Qz_{\lam}+ \frac{1}{3}\RR_{\lam}+O(t), \qquad t >0. \tag{$*$}
\end{equation}
The operators $\Qz_{\lam}$ and $\RR_{\lam}$ satisfy the following homogeneity properties
\bqn
\Qz_{\al\lam}=\Qz_{\lam},\qquad \RR_{\al\lam}=\al^{2}\RR_{\lam},\qquad \all \al>0.
\eqn
{\review 
\begin{remark*}
Eq.~\eqref{eq:mainexpintro} is crucial in our approach to curvature. As we will see, on a Riemannian manifold $\distr_{x_0} = T_{x_0} M$ and $\langle\cdot|\cdot\rangle_\lambda =\langle\cdot|\cdot\rangle$ is the Riemannian scalar product for all $\lambda \in T_{x_0}^*M$. The specialization of Eq.~\eqref{eq:mainexpintro} leads to the following identities:
\begin{equation}
\Qz_{\lam} = \mathbb{I}, \qquad \RR_{\lam}w = R^\nabla(w,v)v, \qquad \forall\, w \in T_{x_0}M,
\end{equation}
where $v = \dot{\gamma}(0)$ is the initial vector of the fixed geodesic dual to the initial covector $\lambda$, while $R^\nabla$ is the Riemannian curvature tensor (see Section~\ref{s:riemann}). The operator $\RR_{\lam}$ is symmetric with respect to the Riemannian scalar product and, seen as a quadratic form on $T_{x_{0}}M$, it computes the sectional curvature of the planes containing the direction of the geodesic. As such it is basic in the Jacobi equation of Riemannian geometry.
\end{remark*}
}

\begin{maintheoremintro}[Section~\ref{s:spec}]
Let  $\gamma:[0,T]\to M$ be an ample and equiregular geodesic. Then the symmetric operator $\Qz_{\lam}: \distr_{x_{0}}\to \distr_{x_{0}}$ satisfies
\bi
\iii[(i)] $\spec \Qz_{\lam}=\{n_{1}^{2},\ldots,n_{k}^{2}\}$,
\iii[(ii)] $\trace \Qz_{\lam}=n_{1}^{2}+\ldots +n_{k}^{2}$.
\ei
\end{maintheoremintro}

Let $M$ be a sub-Riemannian manifold and let $\lapl_{\mu}$ be the sub-Laplacian associated with a smooth volume $\mu$. The next result is an explicit expression for the asymptotics of the sub-Laplacian of the squared distance from a geodesic, computed at the initial point $x_{0}$ of the geodesic $\gamma$. Let $\f_t\doteq \frac{1}{2}\dist^2(\,\cdot\,,\gamma(t))$.

\begin{maintheoremintro}[Section~\ref{s:asymptotic}]
Let $\gamma$ be an equiregular geodesic with initial covector $\lam \in T_{x_0}^*M$. Assume also that $\dim \distr$ is constant in a neighborhood of $x_0$. Then there exists a smooth $n$-form $\omega$ defined along $\gamma$, such that for any volume form $\mu$ on $M$, $\mu_{\gamma(t)} = e^{g(t)}\omega_{\gamma(t)}$, we have
\begin{equation}
\lapl_\mu \f_t|_{x_0} = \trace \Qz_\lam - \dot{g}(0) t -  \frac{1}{3} \Ric(\lam)t^2 + O(t^3).
\end{equation}
\end{maintheoremintro}

 Let $x_0 \in M$ and let $\Sigma_{x_0} \subset M$  be the set of points $x$ such that 
there exists a unique minimizer $\gamma :[0,1]\to M$ joining $x_0$ with $x$, which is not abnormal and $x$ is not conjugate to $x_0$ along $\gamma$. 
\brem A fundamental result states that the set $\Sigma_{x_0}$ is precisely the set of smooth points for the function $x \mapsto \dist^2(x_0,x)$. Another central result asserts that $\Sigma_{x_0}$ is open and dense in $M$ (see \cite{agrachevsmooth,trelatrifford} or also Theorem~\ref{t:d2sr}). This partially answer the question addressed in \cite{montgomerybook}: is the Sard theorem true for the endpoint map? The result just stated only implies that the image of the set of \emph{minimizing} critical points under the endpoint map based at $x_{0}$ is contained in the complement of the open dense set $\Sigma_{x_{0}}$.  It remains a major open problem to determine whether the set $\Sigma_{x_{0}}$ has full measure.
\erem
Let $\Omega_{x_0,t}$ be the homothety of a set $\Omega\subset \Sigma_{x_{0}}$ with respect to $x_{0}$ along the geodesics connecting $x_{0}$ with the points of $\Omega$.
\begin{maintheoremintro}[Section~\ref{s:gd}]
Let $\mu$ be a smooth volume. For any bounded, measurable set $\Omega \subset \Sigma_{x_0}$, with $0<\mu(\Omega)<+\infty$ we have
\begin{equation}
\mu(\Omega_{x_0,t}) \sim t^{\gd_{x_0}},\qquad\text{for } t\to 0.
\end{equation}
where $\gd_{x_0}$ is the geodesic dimension at the point $x_0$.
\end{maintheoremintro}

\section{The Heisenberg group}
Here we specify the result obtained above in the case of the Heisenberg group. All details are presented in Section \ref{s:Heis} and here we present the main computations.

The Heisenberg group $\mathbb{H}$ is the sub-Riemannian structure on $\R^3$ defined by the global orthonormal frame
\begin{equation}
X = \partial_x - \frac{y}{2} \partial_z, \qquad Y = \partial_y + \frac{x}{2} \partial_z.
\end{equation}
Let us introduce the linear on fibers functions $h_x,h_y,h_z:T^*\R^3 \to \mathbb{R}$
\begin{equation}
h_x \doteq  p_x - \frac{y}{2} p_z ,\qquad h_y \doteq  p_y +\frac{x}{2}p_z,\qquad h_z \doteq  p_z,
\end{equation}
where $(x,y,z,p_x,p_y,p_z)$ are canonical coordinates on $T^*\R^3$ induced by coordinates $(x,y,z)$ on $\R^3$. 

The Hamiltonian \eqref{eq:Hamham} takes the form $H = \tfrac{1}{2}(h_x^2 + h_y^2)$ and the coordinates $(x,y,z,h_x,h_y, h_z)$ define a global chart for $T^*\R^{3}$.
It is useful to introduce the identification $\mathbb{R}^3 = \mathbb{C}\times \mathbb{R}$, by defining the complex variable $w \doteq  x+iy$ and the complex ``momentum'' $h_w\doteq  h_x + i h_y$. Let $q = (w, z)$ and $q' = (w',z')$ be two points in $\mathbb{H}$. The Heisenberg group law, in complex coordinates, is given by
\begin{equation}
q \cdot q' = \left(w + w', z+z' - \frac{1}{2}\Im\left(w\overline{w'}\right)\right).
\end{equation}
where $\Im$ denotes the imaginary part of a complex number.
Every non constant geodesic $\g(t)=(w(t),z(t))$ starting from $(w_0,z_0)\in \mb{H}$ corresponds to an initial covector $\lam=(h_{w,0}, h_z)$, with $h_{w,0}\neq 0$. They are explicitly given by
\bqn
\begin{cases}
w(t) = w_0 + \frac{h_{w,0}}{i h_z}\left(e^{i h_z t}-1\right),\\
z(t) = z_0 + \frac{1}{2}\int_0^t \Im(\overline{w} dw),
\end{cases}\qquad \text{if} \quad h_{z}\neq0,
\eqn
or by
\bqn
\begin{cases}
w(t) = w_0 + h_{w,0} t, \\
z(t) = z_0 + \frac{1}{2}\Im(h_{w,0} \overline{w_0}) t,
\end{cases}\qquad \text{if} \quad h_{z}=0.
\eqn
In the first case the component $w(t)$ draw a circle on the complex plane, while in the second one it is a straight line. It is easy to see that in both cases the geodesic is ample with geodesic growth vector $\mathcal{G}_{\g}=\{2,3\}$. Thus the Heisenberg group has geodesic dimension equal to 5.

We are now ready to compute explicitly the asymptotic expansion of $\QQ_\lam$, for  $\lambda = (h_{w,0},h_z) \in T_{x_0}^*M$. Fix $v\in T_{x_0}\R^{3}$ and let $\alpha(s)$ be any curve in $\mb{H}$ such that $\dot{\al}(0)=v$. 
Then we compute the quadratic form $d^{2}_{x_0}\dot c_{t}(v)$ for $t>0$
\begin{equation}
\metr{\QQ_{\lam}(t) v}{v} =d^{2}_{x_0}\dot c_{t}(v)=\left.\frac{\partial^{2}}{\partial s^{2}}\right|_{s=0}\frac{\partial}{\partial t} c_{t}(\alpha(s)) .
\end{equation}
It is possible to compute explicitly the value of $\QQ_{\lam}(t)$ on the orthonormal basis $v\doteq \dot{\g}(0)$ and $v^{\perp}\doteq \dot{\g}(0)^{\perp}$:
\begin{equation}
\metr{\QQ_{\lam}(t) v}{v}=\frac{1}{t^{2}}+O(t),\qquad \metr{\QQ_{\lam}(t)v^{\perp}}{v^{\perp}}= \frac{4}{t^{2}}+\frac{2}{15}h_{z}^{2}+O(t).
\end{equation}
By polarization one also obtain $\metr{\QQ_{\lam}(t) v}{v^{\perp}}=O(t)$. Thus the matrices representing the symmetric operators $\Qz_{\lam}$ and $\RR_{\lam}$ in the basis $\{v^\perp,v\}$ of $\distr_{x_0}$ are
\begin{equation}
\Qz_{\lam}=
\begin{pmatrix}
4&0\\0&1
\end{pmatrix},
\qquad
\RR_{\lam}=\frac{2}{5}
\begin{pmatrix}
h_{z}^{2}&0\\0&0
\end{pmatrix},
\end{equation}
where, we recall, $\lam$ has coordinates $(h_{w,0},h_z)$.

In terms of the orthonormal frame, the sub-Laplacian in the Heisenberg group is expressed as the sum of squares $\lapl=X^{2}+Y^{2}$ and Theorem D reads
\begin{equation}
\lapl\f_t|_{x_0} = 5 - \frac{2}{15}h_{z}^2 t^2 + O(t^3),
\end{equation}
where, we recall, $\f_t\doteq \frac{1}{2}\dist^2(\,\cdot\,,\gamma(t))$ and the initial covector associated with the geodesic $\gamma$ is $\lambda = (h_{w,0},h_z) \in T_{x_0}^*\R^{3}$.


\mainmatter

\part{Statements of the results}

\chapter{General setting}\label{c:affcs}

In this chapter we introduce a general framework that allows to treat smooth control system on a manifold in a coordinate free way, i.e. invariant under state and feedback transformations. For the sake of simplicity, we will restrict our definition to the case of nonlinear affine control systems, although the construction of this section can be extended to any smooth control system (see \cite{cime}). 

\section{Affine control systems} \label{s:affcs}
\bdeff \label{d:cs}
Let $M$ be a connected smooth $n$-dimensional manifold. 
An \emph{affine control system} \index{affine control system}
 on $M$ is a pair $(\cb,f)$ where:
\bi
\iii[$(i)$] $\cb$ is a smooth rank $k$ vector bundle with base $M$ and fiber $\fib_x$ i.e., for every $x\in M$, $\fib_x$ is a $k$-dimensional vector space,
\iii[$(ii)$] $f:\cb\to TM$ is a smooth affine morphism of vector bundles, i.e. the diagram \eqref{eq:diagr1} is commutative and $f$ is \emph{affine} on fibers.
\begin{equation}\label{eq:diagr1}
\xymatrix{
\cb \ar[dr]_{\pi_{\cb}} \ar[r]^{f}
& TM \ar[d]^{\pi} \\
 & M }
\end{equation}
\ei
The maps $\pi_{\cb}$ and $\pi$ are the canonical projections of the vector bundles $\cb$ and $TM$, respectively.
\edeff
We denote points in $\cb$ as pairs $(x,u)$, where $x\in M$ and $u\in \cb_x$ is an element of the fiber. 
According to this notation, the image of the point $(x,u)$ through $f$ is $f(x,u)$ or $f_u(x)$ and we prefer the second one when we want to emphasize $f_{u}$ as a vector on $T_{x}M$.  Finally, let $L^\infty([0,T],\cb)$ be the set of measurable, essentially bounded functions $u:[0,T]\to \cb$.
\bdeff
A Lipschitz curve $\g:[0,T]\to M$ is said to be \emph{admissible} \index{admissible curve} for the control system if there exists a \emph{control} $u \in L^\infty([0,T],\cb)$ such that $\pi_{\cb}\circ u=\gamma$ and
\begin{equation} \label{eq:intcurve}
\dot \g(t)=f(\gamma(t),u(t)), \qquad \text{for a.e. }t\in [0,T].
\end{equation} 
The pair $(\gamma,u)$ of an admissible curve $\gamma$ and its control $u$ is called \emph{admissible pair}.\index{admissible pair}
\edeff

We denote by $\overline{f}:\cb\to TM$ the linear bundle morphism induced by $f$. In other words we write $f(x,u)=f_{0}(x)+\overline f(x,u)$, where $f_{0}(x)\doteq f(x,0)$ is the image of the zero section. In terms of a local frame for $\cb$, $\overline{f}(x,u) = \sum_{i=1}^k u_i f_i(x)$.
\bdeff \label{def:iso0}
The \emph{distribution} \index{distribution} $\distr \subset TM$ is the family of subspaces
\begin{equation} 
\distr=\{ \distr_x\}_{x\in M},\qquad  \tx{where} \qquad \distr_{x}\doteq \overline f(\fib_{x})\subset T_{x}M. 
\end{equation}
The family of \emph{horizontal vector fields}  \index{horizontal vector fields}$\overline{\distr}\subset\VecM$ is
\begin{equation} 
\overline{\distr}=\tx{span}\left\{ \overline f\circ \sigma,\,  \sigma:M\to \cb \tx{\ is\ a\ smooth\  section\  of \ } \cb \right\}. 
\end{equation}
\edeff
Observe that, if the rank of $\overline{f}$ is not constant, $\distr$ is not a sub-bundle of $TM$. Therefore the dimension of $\distr_x$, in general, depends on $x \in M$.

Given a smooth function $\LL:\cb \to \R$,  called a \emph{Lagrangian}, \index{Lagrangian function} the \emph{cost functional at time }$T$, called $\J_{T}:L^{\infty}([0,T],\cb)\to \R$, is defined by \index{cost functional}
\begin{equation} 
\J_{T}(u)\doteq \int_{0}^{T} \LL(\gamma(t),u(t))dt,
\end{equation} 
where $\g(t) = \pi(u(t))$. We are interested in the problem of minimizing the cost among all admissible pairs $(\gamma,u)$ that join two fixed points $x_{0},x_{1}\in M$ in time $T$. This corresponds to the optimal control problem
\begin{equation}\label{eq:ocp}
\begin{aligned}
&\dot{x}=f(x,u)=f_{0}(x)+\sum_{i=1}^k u_i f_i(x),\qquad x \in M, \\
&x(0)=x_{0}, \ x(T)=x_{1}, \qquad J_{T}(u)\to \text{min},
\end{aligned}
\end{equation}
where we have chosen some local trivialization of $\cb$.
\begin{definition}\label{d:value}
Let $M'\subset M$ be an open subset with compact closure. For $x_{0},x_{1}\in M'$ and $T>0$, we define the \emph{value function} \index{value function}
\begin{equation}
\SS_{T}(x_{0},x_{1})\doteq \inf \{\J_{T}(u)\,|\, (\gamma,u) \text{ admissible pair, }\g(0)=x_{0},\,\g(T)=x_{1},\, \g \subset M'\}.
\end{equation}
\end{definition}
The value function depends on the choice of a relatively compact subset $M'\subset M$. This choice, which is purely technical, is related with Theorem~\ref{t:smoothness}, concerning the regularity properties of $\SS$. We stress that all the objects defined in this paper by using the value function do not depend on the choice of $M'$.

\paragraph{Assumptions.} In what follows we make the following general assumptions:
\bi
\iii[(A1)] The affine control system is \emph{bracket generating},\index{bracket generating} namely
\begin{equation}\label{eq:X0pure}
\text{Lie}_{x}\Pg{(\ad\, f_{0})^{i}\,\overline \distr\,|\, i\in \N}=T_{x}M, \qquad \all x \in M,
\end{equation}
where $(\ad\, X) Y=[X,Y]$ is the Lie bracket of two vector fields and $\text{Lie}_{x}\mc{F}$ denotes the Lie algebra generated by a family of vector fields $\mc{F}$, computed at the point $x$. Observe that the vector field $f_{0}$ is not included in the generators of the Lie algebra \eqref{eq:X0pure}.
\iii[(A2)] The function $\LL:\cb \to \R$ is a \emph{Tonelli Lagrangian}, \index{Tonelli Lagrangian} i.e. it satisfies
\bi
\iii[(A2.a)] The Hessian of $\LL|_{\fib_{x}}$ is positive definite for all $x\in M$. In particular, $\LL|_{\fib_{x}}$ is strictly convex.
\iii[(A2.b)] $\LL$ has  superlinear growth, i.e. $\LL(x,u)/|u|\to +\infty$ when $|u|\to +\infty$.
\ei
\ei
Assumptions (A1) and (A2) are necessary conditions in order to have a nontrivial set of strictly normal minimizer and allow us to introduce a well defined smooth Hamiltonian (see Chapter~\ref{c:geodesic}).

\subsection{State-feedback equivalence}
All our considerations will be local. Hence, up to restricting our attention to a trivializable neighbourhood of $M$, we can assume that $\cb\simeq M\times \R^{k}$. By choosing a basis of $\mathbb{R}^k$, we can write $f(x,u)=f_{0}(x)+\sum_{i=1}^{k}u_{i}f_{i}(x)$. Then, a Lipschitz curve $\gamma:[0,T]\to M$ 
is admissible if there exists a measurable, essentially bounded control $u :[0,T]\to \R^{k}$ such that \index{control}
\begin{equation} \label{eq:cs2}
\dot \gamma(t)=f_{0}(\gamma(t))+\sum_{i=1}^{k}u_{i}(t)f_{i}(\gamma(t)), \qquad \text{for a.e.}\  t\in [0,T].
\end{equation}
We use the notation $u \in L^\infty([0,T],\R^k)$ to denote a measurable, essentially bounded control with values in $\R^k$. By choosing another (local) trivialization of $\cb$, or another basis of $\mathbb{R}^k$, we obtain a different \emph{presentation} of the same affine control system. Besides, by acting on the underlying manifold $M$ via diffeomorphisms, we obtain equivalent affine control system starting from a given one. The following definition formalizes the concept of equivalent control systems.

\bdeff\label{d:sf}
Let $(\cb,f)$ and $(\cb',f')$ be two affine control systems on the same manifold $M$. A \emph{state-feedback transformation} \index{state-feedback transformation} is a pair $(\phi,\psi)$, where $\phi:M\to M$ is a diffeomorphism and $\psi:\cb\to \cb'$ an invertible affine bundle map, such that
the following diagram is commutative.
\begin{equation}\label{eq:diagr4}
\xymatrix{
\cb \ar[d]_{\psi} \ar[r]^{f}
& TM \ar[d]^{\phi_{*}}\\  \cb' \ar[r]_{f'} & TM\\}
\end{equation} 
In other words, $\phi_{*}f(x,u)=f'(\phi(x),\psi(x,u))$ for every $(x,u)\in \cb$. In this case $(\cb,f)$ and $(\cb',f')$ are said \emph{state-feedback equivalent}.      
\edeff

Notice that, if $(\cb,f)$ and $(\cb',f')$ are state-feedback equivalent, then $\text{rank}\,\cb=\text{rank}\,\cb'$. Moreover, different presentations of the same control systems are indeed feedback equivalent (i.e. related by a state-feedback transformation with $\phi = \id$). Definition \ref{d:sf} corresponds to the classical notion of point-dependent reparametrization of the controls. The next lemma states that a state-feedback transformation preserves admissible curves.
\begin{lemma} \label{l:admtoamd} 
Let $\gamma_{x_{0},u}$ be the admissible curve starting from $x_{0}$ and associated with $u$. Then 
\begin{equation}
\phi(\gamma_{x_{0},u}(t))=\gamma_{\phi(x_{0}),v}(t),
\end{equation}
where $v(t)=\psi(x(t),u(t))$.
\end{lemma}
\begin{proof} 
Denote $x(t)=\gamma_{x_{0},u}(t)$ and set $y(t)\doteq \phi(x(t))$. Then, by definition, $\dot{x}(t)=f(x(t),u(t))$ and $x(0)=x_{0}$. Hence $y(0)=\phi(x_{0})$ and
\begin{equation}
\dot{y}(t)=\phi_{*} f(x(t),u(t))=f'(\phi(x(t)),\psi(x(t),u(t)))=f'(y(t),v(t)).\qedhere
\end{equation}
\end{proof}

\brem \label{r:sf} 
Notice that every state-feedback transformation $(\phi,\psi)$ can be written as a composition of a pure state one, i.e. with $\psi=\id$, and a pure feedback one, i.e. with $\phi=\id$.
For later convenience, let us discuss how two feedback equivalent systems are related. Consider a presentation of an affine control system
\begin{equation}
\dot{x}=f(x,u)=f_{0}(x)+\sum_{i=1}^{k} u_{i} f_{i}(x) .
\end{equation}
By the commutativity of diagram~\eqref{eq:diagr4}, a feedback transformation writes
\begin{equation}
\begin{cases}
u'=\psi(x,u)\\
x'=\phi(x)
\end{cases}\qquad u'_{i}=\psi_{i}(x,u)=\psi_{i,0}(x)+\sum_{j=1}^{k} \psi_{i,j}(x)u_{j}, \qquad i=1,\ldots,k,
\end{equation}
where $\psi_{i,0}$ and $\psi_{i,j}$ denote, respectively, the affine and the linear part of the $i$-th component of $\psi$. In particular, for a pure feedback transformation, the original system is equivalent to 
\begin{equation} 
\dot{x}=f'(x,u')=f'_{0}(x)+\sum_{i=1}^{k} u'_{i}f'_{i}(x),
\end{equation}
where $f_{0}(x)\doteq f_{0}'(x)+\sum_{i=1}^{k} \psi_{i,0}(x)f_{i}'(x)$ and $f_{i}(x)\doteq \sum_{j=1}^{k}\psi_{j,i}(x)f_{j}'(x)$.
\erem 

We conclude recalling some well known facts about non-autonomous flows. By Caratheodory Theorem, for every control $u\in L^{\infty}([0,T],\R^k)$ and every initial condition $x_{0}\in M$, there exists a unique Lipschitz solution to the Cauchy problem
\begin{equation}\label{eq:ex}
\begin{cases}
\dot \g(t)=f_0(\gamma(t)) + \sum_{i=1}^k u_i(t) f_i(\gamma(t)), \\
\gamma(0)=x_{0},
\end{cases}
\end{equation}
defined for small time (see, e.g. \cite{agrachevbook,pontbook}). We denote such a solution by $\gamma_{x_{0},u}$ (or simply $\gamma_{u}$ when the base point $x_{0}$ is fixed). Moreover, for a fixed control $u \in L^\infty([0,T],\R^k)$, it is well defined the family of diffeomorphisms $P_{0,t} : M \to M$, given by $P_{0,t}(x)\doteq \gamma_{x,u}(t)$, which is  Lipschitz with respect to $t$.
Analogously one can define the flow $P_{s,t}:M\to M$, by solving the Cauchy problem with initial condition given at time $s$.
Notice that $P_{t,t}=\id$ for all $t\in \R$ and $P_{t_{1},t_{2}}\circ P_{t_{0},t_{1}} =P_{t_{0},t_{2}}$, whenever they are defined. In particular $(P_{t_{1},t_{2}})^{-1}=P_{t_{2},t_{1}}$.
\section{End-point map}

In this section, for convenience, we assume to fix some (local) presentation of the affine control system, hence $L^{\infty}([0,T],\cb) \simeq L^{\infty}([0,T],\R^k)$. For a more intrinsic approach see \cite[Sec. 1]{cime}.
 
\bdeff Fix a point $x_{0}\in M$ and $T>0$.  The \emph{end-point map at time $T$}  \index{end-point map}  of the system \eqref{eq:ex} is the map 
\begin{equation} \End_{x_{0},T}: \U\to M,\qquad u\mapsto \gamma_{x_{0},u}(T), \end{equation}
where $\U\subset L^{\infty}([0,T],\R^k)$ is the open subset  of controls such that the solution $t\mapsto \gamma_{x_{0},u}(t)$ of the Cauchy problem \eqref{eq:ex} is defined on the whole interval $[0,T]$. 
\edeff 
The end-point map is smooth. Moreover, its Fr\'echet differential is computed by the following well-known formula (see, e.g. \cite{agrachevbook}).

\bp 
The differential \index{end-point map!differential of}of $E_{x_{0},T}$ at $u\in \U$, i.e. 
$D_{u}\End_{x_0,T}: L^{\infty}([0,T],\R^{k})\to T_{x}M$, where $x=\gamma_{u}(T)$, is 
\begin{equation}\label{eq:duev}
D_{u}\End_{x_{0},T}(v)=\int_{0}^{T}(P_{s,T})_{*}\overline{f}_{v(s)}(\gamma_{u}(s)) ds,\qquad \all v\in L^{\infty}([0,T],\R^{k}).
\end{equation}
\ep
In other words the differential $D_{u}\End_{x_{0},T}$ applied to the control $v$ computes the integral mean of the linear part $\overline f_{v(t)}$ of the vector field $f_{v(t)}$ along the trajectory defined by $u$, by pushing it forward to the final point of the trajectory through the flow $P_{s,T}$ (see Fig.~\ref{fig:diffend}).

\begin{figure}
\centering
\scalebox{0.6} 
{
\begin{pspicture}(0,-2.8988476)(12.48291,2.8788476)
\definecolor{color133}{rgb}{0.0,0.2,1.0}
\definecolor{color222}{rgb}{0.8,0.0,0.0}
\fontsize{14}{0}
\psbezier[linewidth=0.04,dotsize=0.07055555cm 2.0]{*-o}(0.28101563,-2.3211524)(1.0010157,-1.0011524)(1.9688088,-0.15527576)(3.3810155,-0.20115234)(4.7932224,-0.24702893)(5.381016,-0.44115233)(6.0810156,-0.28115234)(6.7810154,-0.12115234)(8.361015,0.95884764)(8.401015,1.3988477)
\usefont{T1}{ptm}{m}{n}
\rput(0.4124707,-2.6761522){$x_0$}
\usefont{T1}{ptm}{m}{n}
\rput(1.4324707,-0.21615234){$\gamma_u(s)$}
\psline[linewidth=0.04cm,linecolor=color222,arrowsize=0.05291667cm 2.0,arrowlength=1.4,arrowinset=0.4]{->}(2.0810156,-0.46115234)(3.8010156,-1.0411524)
\usefont{T1}{ptm}{m}{n}
\rput(4.372471,-1.1761523){$\overline{f}_{v(s)}$}
\pspolygon[linewidth=0.04](5.9210157,0.51884764)(7.2410154,2.8388476)(11.081016,2.8588476)(9.741015,0.51884764)
\psbezier[linewidth=0.04,linecolor=color133,arrowsize=0.05291667cm 2.0,arrowlength=1.4,arrowinset=0.4]{->}(2.6010156,0.55884767)(2.6010156,0.55884767)(3.3010156,1.3588476)(3.9299412,1.4388477)(4.558867,1.5188477)(4.841016,1.4588476)(5.321016,1.3188477)
\usefont{T1}{ptm}{m}{n}
\rput(4.142471,1.8638476){$(P_{s,T})_*$}
\usefont{T1}{ptm}{m}{n}
\rput(9.4247,2.4038476){$(P_{s,T})_*\overline{f}_{v(s)}$}
\usefont{T1}{ptm}{m}{n}
\rput(8.8247,1.2038476){$x$}
\psline[linewidth=0.04cm,linecolor=color222,arrowsize=0.05291667cm 2.0,arrowlength=1.4,arrowinset=0.4]{->}(8.401015,1.3788476)(8.121016,2.5788476)
\usefont{T1}{ptm}{m}{n}
\rput(9.872471,0.20384766){$T_xM$}
\end{pspicture} 
}
\caption{Differential of the end-point map.}\label{fig:diffend}
\end{figure}
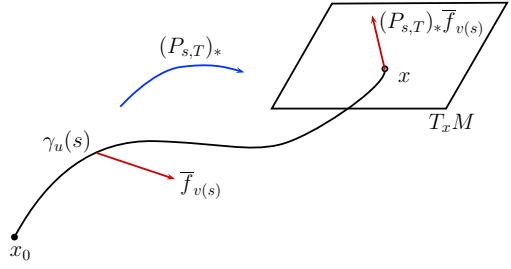

More explicitly, $f(x,u)=f_{0}(x)+\sum_{i=1}^{k}u_{i}f_{i}(x)$, and Eq. \eqref{eq:duev} is rewritten as follows
\begin{equation}  
D_{u}\End_{x_{0},T}(v)=\int_{0}^{T}\sum_{i=1}^{k}v_{i}(s) (P_{s,T})_{*}f_{i}(\gamma_{u}(s)) ds,\qquad \all v\in L^{\infty}([0,T],\R^{k}). 
\end{equation}

\section{Lagrange multipliers rule}
Fix $x_{0},x\in M$. The problem of finding the infimum of the cost $J_T$ for all admissible curves connecting the endpoints $x_0$ and $x$, respectively, in time $T$, can be naturally reformulated via the end-point map as a constrained extremal problem
\begin{equation}\label{eq:ccp}
\SS_{T}(x_{0},x)=\inf\{ \J_{T}(u)\,|\, \End_{x_{0},T}(u)=x\}=\inf_{\End_{x_{0},T}^{-1}(x)} \J_{T} .
\end{equation}
\begin{definition}\label{d:optimalcontrol}
We say that $u \in \U$ is an \emph{optimal control} \index{control!optimal} if it is a solution of Eq.~\eqref{eq:ccp}.
\end{definition}
\begin{remark}\label{r:mincontr}
When $f$ is not injective, a curve $\gamma$ may be associated with multiple controls. Nevertheless, among all the possible controls $u$ associated with the same admissible curve, there exists a unique minimal control $u^*$ which, for a.e. $t \in [0,T]$, minimizes the Lagrangian function. Then, since we are interested in optimal controls, we assume that any admissible curve $\gamma$ is always associated with the control $u^*$ which minimizes the Lagrangian, and in this way we have a one-to-one correspondence between admissible curves and controls. With this observation, we say that the admissible curve $\gamma$ is an \emph{optimal trajectory} (or \emph{minimizer}) if the associated control $u^*$ is optimal according to Definition~\ref{d:optimalcontrol}. \index{control!minimal}
\end{remark}
Notice that, in general, $D_{u}\End_{x_{0},T}$ is not surjective and the set $\End_{x_{0},T}^{-1}(x) \subset M$ is not a smooth submanifold.
The Lagrange multipliers rule provides a necessary condition to be satisfied by a control $u$ which is a constrained critical point for \eqref{eq:ccp}.

\bp\label{p:lmr} \index{Lagrange multipliers rule} Let $u\in \U$ be an optimal control, with $x=\End_{x_{0},T}(u)$. Then (at least) one of the two following statements holds true \bi
\iii[(i)] $\ex \lam_{T} \in T^{*}_{x}M$ s.t. $\lam_{T} \,D_{u}\End_{x_{0},T}=d_{u}\J_{T}$,
\iii[(ii)] $\ex \lam_{T} \in T^{*}_{x}M,\, \lam_{T} \neq 0,$ s.t. $\lam_{T} \,D_{u}\End_{x_{0},T}=0$,
\ei
where $\lam_{T} \,D_{u}\End_{x_{0},T}$ denotes the composition of  linear maps
\begin{equation}\label{eq:diagr2}
\xymatrixcolsep{5pc}\xymatrix{
\hspace{-1cm}L^{\infty}([0,T],\R^{k}) \ar[dr]_-{d_{u}\J_{T}} \ar[r]^-{D_{u}\End_{x_{0},T}}
& T_{x}M \ar[d]^{\lam_{T}} \\
 & \R }
\end{equation}
\ep
\begin{definition} \label{d:normal}
A control $u$, satisfying the necessary conditions for optimality of Proposition~\ref{p:lmr}, is called \emph{normal} in case (i), while it is called \emph{abnormal} in case (ii). We use the same terminology to classify the associated \emph{extremal trajectory}\index{extremal trajectory} $\gamma_u$. \index{control!normal}\index{control!abnormal}\index{extremal trajectory!normal}\index{extremal trajectory!abnormal}
\end{definition}
Notice that a single control $u\in \mc{U}$ can be associated with two different covectors (or \emph{Lagrange multipliers}) such that both (i) and (ii) are satisfied. In other words, an optimal trajectory may be simultaneously normal and abnormal.
We now introduce a key definition for what follows.
\bdeff A normal extremal trajectory $\gamma:[0,T]\to M$ is called \emph{strictly normal}  \index{extremal trajectory!strictly normal}if it is not abnormal. Moreover, if for all $s\in [0,T]$ the restriction $\gamma|_{[0,s]}$ is also strictly normal, then $\gamma$ is called \emph{strongly normal}.\index{extremal trajectory!strongly normal}
\edeff
\brem \label{r:abnormal}
A trajectory is abnormal if and only if the differential $D_{u}\End_{x_{0},T}$ is not surjective. By linearity of the integral, it is easy to show from Eq.~\eqref{eq:duev} that this is equivalent to the relation
\begin{equation} 
\tx{span}\{(P_{s,T})_*\distr_{\gamma(s)}, s\in[0,T]\}\neq T_{\gamma(T)}M. 
\end{equation}
In particular $\gamma$ is strongly normal if and only if a short segment $\gamma|_{[0,\eps]}$ is strongly normal, for some $\eps \leq T$.

\erem

\section{Pontryagin Maximum Principle}\label{s:pmp}
In this section we recall a weak version of the Pontryagin Maximum Principle (PMP) for the optimal control problem, which rewrites the necessary conditions satisfied by normal optimal solutions in the Hamiltonian formalism. In particular it states that every normal optimal trajectory of problem \eqref{eq:ocp} is the projection of a solution of a fixed Hamiltonian system defined on $T^{*}M$. 
 
Let us denote by $\pi: T^{*}M \to M$ the canonical projection of the cotangent bundle, and by $\la\lam,v\ra$ the  pairing between a cotangent vector $\lam\in T^{*}_{x}M$ and a vector $v\in T_{x}M$. 
The Liouville 1-form $\varsigma\in \Lambda^{1}(T^{*}M)$ is defined as follows: $\varsigma_{\lam}=\lam \circ \pi_{*}$, for every $\lam \in T^{*}M$. The canonical symplectic structure on $T^{*}M$ is defined by the non degenerate closed 2-form $\sigma=d\varsigma$. In canonical coordinates $(p,x)\in T^{*}M$ one has
\begin{equation} 
\varsigma= \sum_{i=1}^{n} p_{i} dx_{i}, \qquad \sigma = \sum_{i=1}^{n} dp_{i} \wedge dx_{i}. 
\end{equation}
We denote by $\vec h$ the Hamiltonian vector field associated with a function $h\in C^{\infty}(T^{*}M)$. Namely, $d_{\lam}h=\sigma(\cdot,\vec{h}(\lam))$ for every $\lam\in T^{*}M$ and the coordinates expression of $\vec{h}$ is 
\begin{equation} 
\vec{h}=\sum_{i=1}^{n} \frac{\partial h}{\partial p_i}\frac{\partial}{\partial x_i}-\frac{\partial h}{\partial x_i}\frac{\partial}{\partial p_i}. 
\end{equation}
Let us introduce the smooth control-dependent Hamiltonian on $T^{*}M$: 
\begin{equation}
\HH(\lam,u)=\la \lam, f(x,u)\ra- \LL(x,u), \qquad \lam\in T^{*}M, \ x=\pi(\lam).
\end{equation}
Assumption (A2) guarantees that, for each $\lam \in T^*M$, the restriction $u\mapsto\HH(\lam,u)$ to the fibers of $\cb$ has a unique maximum $\bar{u}(\lambda)$. Moreover, the fiber-wise strong convexity of the Lagrangian and an easy application of the implicit function theorem prove that the map $\lambda \mapsto \bar{u}(\lam)$ is smooth. Therefore, it is well defined the \emph{maximized Hamiltonian} (or simply, \emph{Hamiltonian}) $H: T^*M \to \R$
\begin{equation} \index{Hamiltonian} \index{maximized Hamiltonian}
H(\lam)\doteq  \max_{v \in U_x} \HH(\lam,v) = \HH(\lam,\bar{u}(\lam)), \qquad \lam \in T^*M, x = \pi(\lam).
\end{equation}
\brem
When $f(x,u)=f_{0}(x)+\sum_{i=1}^{k}u_{i}f_{i}(x)$ is written in a local frame, then $\bar u=\bar u(\lam)$ is characterized as the solution of the system
\begin{equation} \label{eq:as}
\frac{\partial \HH}{\partial u_{i}}(\lam,u)=\la\lam,f_{i}(x)\ra-\frac{\partial \LL}{\partial u_{i}}(x,u)=0, \qquad i=1,\ldots,k.
\end{equation}
\erem
\bt[PMP,\cite{agrachevbook,pontbook}]\label{t:pmp} \index{Pontryagin Maximum Principle (PMP)} The admissible curve $\gamma:[0,T]\to M$ is a normal extremal trajectory if and only if there exists a Lipschitz lift $\lam:[0,T] \to T^*M$, such that $\gamma(t) = \pi(\lam(t))$ and
\begin{equation}
\dot{\lam}(t) = \vec{H}(\lam(t)),\qquad t\in [0,T].
\end{equation}
In particular, $\gamma$ and $\lambda$  are smooth. Moreover, the associated control can be recovered from the lift as $u(t) = \bar{u}(\lam(t))$, and the final covector $\lambda_T = \lambda(T)$ is a normal Lagrange multiplier associated with $u$, namely $\lam_T \,D_{u}\End_{x_{0},T}=d_{u}\J_{T}$.
\et
Thus, every normal extremal trajectory $\gamma:[0,T]\to M$ can be written as $\gamma(t)=\pi \circ e^{t\vec{H}}(\lam_{0})$, for some initial covector $\lam_{0}\in T^{*}M$ (although it may be non unique). 
This observation motivates the next definition. For simplicity, and without loss of generality, we assume that $\vec{H}$ is complete.

\bdeff Fix $x_{0}\in M$. The \emph{exponential map} \index{exponential map}with base point $x_{0}$ is the map $\EXP_{x_{0}}:\R^{+}\times T^{*}_{x_{0}}M\to M$, defined by $\EXP_{x_{0}}(t,\lam_{0})=\pi \circ e^{t\vec{H}}(\lam_{0})$. 
\edeff

When the first argument is fixed, we employ the notation $\EXP_{x_0,t} :T_{x_0}^*M \to M$ to denote the exponential map with base point $x_0$ and time $t$, namely $\EXP_{x_0,t}(\lambda) = \EXP_{x_0}(t,\lambda)$. Indeed, the exponential map is smooth.

From now on, we call \emph{geodesic} \index{geodesic} any trajectory that satisfies the normal necessary conditions for optimality. In other words, geodesics are admissible curves associated with a normal Lagrange multiplier or, equivalently, projections of integral curves of the Hamiltonian flow.

\section{Regularity of the value function}

The next well known regularity property of the value function is crucial for the forthcoming sections (see Definition~\ref{d:value}).

\begin{theorem}\label{t:smoothness} \index{value function!regularity of}
Let $\gamma:[0,T] \to M'$ be a strongly normal trajectory. Then there exist  $\eps >0$ and an open neighbourhood $U \subset (0,\eps) \times M' \times M'$ such that:
\begin{itemize}
\item[(i)] $(t,\gamma(0),\gamma(t)) \in U$ for all $t \in (0,\eps)$,
\item[(ii)] For any $(t,x,y) \in U$ there exists a unique (normal) minimizer of the cost functional $J_t$, among all the admissible curves that connect $x$ with $y$ in time $t$, contained in $M'$,
\item[(iii)] The value function $(t,x,y)\mapsto \SS_t(x,y)$ is smooth on $U$.
\end{itemize}
\end{theorem}
According to Definition~\ref{d:value}, the function $\SS$, and henceforth $U$, depend on the choice of a relatively compact $M'\subset M$. For different relatively compacts, the correspondent value functions $\SS$ agree on the intersection of the associated domains $U$: they define the same germ.

The proof of this result can be found in Appendix~\ref{a:proofsmoothness}.
We end this section with a useful lemma about the differential of the value function at a smooth point.

\bl\label{l:lambdat} Let $x_{0},x\in M$ and $T>0$. Assume that the function $x \mapsto S_{T}(x_{0},x)$  is smooth at $x$ and there exists an optimal trajectory $\gamma:[0,T]\to M$ joining $x_{0}$ to $x$. Then 
\bi
\iii[(i)] $\gamma$ is the unique minimizer  of the cost functional $J_T$, among all the admissible curves that connect $x_{0}$ with $x$ in time $T$, and it is strictly normal,
\iii[(ii)]  $d_{x}S_{T}(x_{0},\cdot)=\lambda_{T}$, where $\lambda_{T}$ is the final covector of the normal lift of $\gamma$.
\ei
\el
\begin{proof} Under the above assumptions the function
\begin{equation} 
v\mapsto J_{T}(v)-S_{T}(x_{0},E_{x_{0},T}(v)) , \qquad v \in L^\infty([0,T],\R^k),
\end{equation}
is smooth and non negative. For every optimal trajectory $\gamma$, associated with the control $u$, that connects $x_0$ with $x$ in time $T$, one has  
\begin{equation} 
0=d_{u}\big(J_{T}(\cdot)-S_{T}(x_{0},E_{x_{0},T}(\cdot)\big)=d_{u}J_{T}-d_{x}S_{T}(x_{0},\cdot) \circ D_{u}E_{x_{0},T}.
\end{equation}
Thus, $\gamma$ is a normal extremal trajectory, with Lagrange multiplier $\lambda_{T}=d_{x}S_{T}(x_{0},\cdot)$. By Theorem~\ref{t:pmp}, we can recover $\gamma$ by the formula $\gamma(t) = \pi\circ e^{(t-T)\vec{H}}(\lambda_T)$. Then, $\gamma$ is the unique minimizer of $J_T$ connecting its endpoints.

Next we show that $\gamma$ is not abnormal. For $y$ in a neighbourhood of $x$, consider the map
\begin{equation}
\Theta:y\mapsto e^{-T\vec{H}}(d_{y}S_{T}(x_{0},\cdot)).
\end{equation}
The map $\Theta$, by construction, is a smooth right inverse for the exponential map at time $T$. This implies that $x$ is a regular value for the exponential map and, a fortiori, $u$ is a regular point for the end-point map at time $T$.
\end{proof}

\chapter{Flag and growth vector of an admissible curve}\label{c:geodesic}

For each smooth admissible curve, we introduce a family of subspaces, which is related with a micro-local characterization of the control system along the trajectory itself.

\section{Growth vector of an admissible curve}
Let $\gamma:[0,T]\to M$ be an admissible, smooth curve such that $\gamma(0)=x_{0}$, associated with a smooth control $u$. Let $P_{0,t}$ denote the flow defined by $u$. We define the family of subspaces of $T_{x_{0}}M$
\begin{equation} \label{eq:DD}
\DD_{\gamma}(t)\doteq (P_{0,t})^{-1}_{*}\distr_{\gamma(t)}.
\end{equation}
In other words, the family $\DD_{\gamma}(t)$ is obtained by collecting the distributions along the trajectory at the initial point, by using the flow $P_{0,t}$ (see Fig.~\ref{fig:DD}).
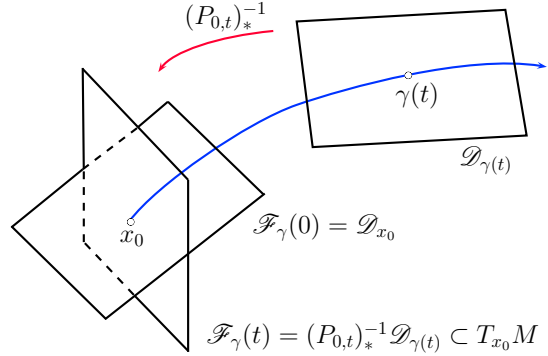
\begin{figure}[ht]
\centering
\scalebox{0.7} 
{
\begin{pspicture}(0,-3.2691991)(10.521894,3.289199)
\definecolor{color236}{rgb}{0.0,0.2,1.0}
\definecolor{color246}{rgb}{1.0,0.0,0.2}
\psbezier[linewidth=0.04,linecolor=color236,arrowsize=0.05291667cm 2.0,arrowlength=1.4,arrowinset=0.4]{->}(2.22,-0.7691992)(2.72,-0.089199215)(4.4285975,1.1068101)(5.36,1.4308008)(6.291403,1.7547915)(8.56,2.3908007)(10.06,2.1508007)
\pspolygon[linewidth=0.04](5.36,3.0108008)(9.28,3.2308009)(9.68,0.85080075)(5.66,0.6308008)(5.66,0.6308008)
\psline[linewidth=0.04cm](1.3594159,0.22425534)(1.3396081,2.130801)
\psline[linewidth=0.04cm](1.34,2.1108007)(3.32,0.010800782)
\psline[linewidth=0.04cm](3.32,0.010800782)(3.32,-3.2491992)
\psline[linewidth=0.04cm,linestyle=dashed,dash=0.16cm 0.16cm](1.36,0.17080078)(1.36,-1.1691992)
\psline[linewidth=0.04cm,linestyle=dashed,dash=0.16cm 0.16cm](1.38,-1.1891992)(2.3,-2.1691992)
\psbezier[linewidth=0.04,linecolor=color246,arrowsize=0.05291667cm 2.0,arrowlength=1.4,arrowinset=0.4]{->}(4.92,2.8308008)(3.86,2.7308009)(3.02,2.4308007)(2.76,2.0508008)
\fontsize{14}{0}
\usefont{T1}{ptm}{m}{n}
\rput(4.031455,3.0958009){$(P_{0,t})_{*}^{-1}$}
\usefont{T1}{ptm}{m}{n}
\rput(5.881455,-0.88419924){$\DD_{\gamma}(0)=\distr_{x_0}$}
\psdots[dotsize=0.14,fillstyle=solid,dotstyle=o](2.24,-0.78919923)
\psdots[dotsize=0.14,fillstyle=solid,dotstyle=o](7.46,1.9908007)
\usefont{T1}{ptm}{m}{n}
\rput(7.631455,1.6158007){$\gamma(t)$}
\usefont{T1}{ptm}{m}{n}
\rput(8.911455,0.3758008){$\distr_{\gamma(t)}$}
\psline[linewidth=0.04cm](1.34,0.21080078)(0.0,-0.8891992)
\psline[linewidth=0.04cm](0.0,-0.8691992)(1.78,-2.6291993)
\psline[linewidth=0.04cm,linestyle=dashed,dash=0.16cm 0.16cm](2.36,0.9908008)(1.34,0.21080078)
\psline[linewidth=0.04cm](4.76,-0.24919921)(1.76,-2.6291993)
\psline[linewidth=0.04cm](2.94,1.4908007)(2.38,1.0108008)
\psline[linewidth=0.04cm](3.32,-3.2291992)(2.32,-2.1891992)
\psline[linewidth=0.04cm](2.94,1.4908007)(4.74,-0.26919922)
\usefont{T1}{ptm}{m}{n}
\rput(2.271455,-1.1041992){$x_0$}
\usefont{T1}{ptm}{m}{n}
\rput(6.7914553,-2.975){$\DD_{\gamma}(t)=(P_{0,t})^{-1}_{*}\distr_{\gamma(t)}\subset T_{x_{0}}M$}
\end{pspicture} 
}
\caption{The family of subspaces $\DD_{\gamma}(t)$.}\label{fig:DD}
\end{figure}

Given a family of subspaces in a linear space it is natural to consider the associated flag.
\bdeff\label{d:flag} 
The \emph{flag of the admissible curve $\gamma$} \index{flag of an admissible curve} \index{admissible curve!flag of} is the sequence of subspaces
\begin{equation}
\DD^{i}_{\gamma}(t)\doteq \spn\left\{\frac{d^{j}}{dt^{j}}\, v(t)\, \bigg| \,v(t)\in \DD_{\gamma}(t) \text{ smooth} ,\, j\leq i  -1 \right\}\subset T_{x_{0}}M, \qquad i\geq 1.
\end{equation}
Notice that, by definition, this is a filtration of $T_{x_{0}}M$, i.e. $\DD^{i}_{\gamma}(t)\subset \DD^{i+1}_{\gamma}(t)$, for all $i \geq 1$.
\edeff

\bdeff \label{d:ample} Let $k_{i}(t)\doteq \dim \DD^{i}_{\gamma}(t)$. The \emph{growth vector of the admissible curve $\gamma$} \index{growth vector of an admissible curve} \index{admissible curve!growth vector}is the sequence of integers $\mc{G}_{\gamma}(t)=\{k_{1}(t),k_{2}(t),\ldots\}$. 

An admissible curve is \emph{ample at $t$} \index{admissible curve!ample} if there exists an integer $m=m(t)$ such that $\DD^{m(t)}_{\gamma}(t)=T_{x_{0}}M$. We call the minimal $m(t)$ such that the curve is ample the \emph{step at $t$ of the admissible curve}. An admissible curve is called \emph{equiregular at $t$} \index{admissible curve!equiregular} if its growth vector is locally constant at $t$. Finally, an admissible curve is \emph{ample} (resp. \emph{equiregular}) if it is ample (resp. equiregular) at each $t\in[0,T]$.
\edeff

\begin{remark}\label{r:durezza}
One can analogously introduce the family of subspaces (and the relevant filtration) at any base point $\gamma(s)$, for every $s\in [0,T]$, by defining the shifted curve $\gamma_s(t)\doteq  \gamma(s+t)$. Then $\DD_{\gamma_s}(t)\doteq (P_{s,s+t})_{*}^{-1}\distr_{\gamma(s+t)}$. Notice that the relation $\DD_{\gamma_s}(t)=(P_{0,s})_{*} \DD_{\gamma}(s+t)$ implies that the growth vector of the original curve at $t$ can be equivalently computed via the growth vector at time $0$ of the curve $\gamma_t$, i.e. $k_{i}(t)=\dim \DD_{\gamma_t}^{i}(0)$, and $\mc{G}_\gamma(t) = \mc{G}_{\gamma_t}(0)$.
\end{remark}

Let us stress that the the family of subspaces \eqref{eq:DD} depends on the choice of the local frame (via the map  $P_{0,t}$). However, we will prove that the flag of an admissible curve at $t=0$ and its growth vector (for all $t$) are invariant by state-feedback transformation and, in particular, independent on the particular presentation of the system (see Section \ref{s:inv}).

\begin{remark} \label{r:di} The following properties of the growth vector of an \emph{ample} admissible curve highlight the analogy with the \virg{classical} growth vector of the distribution.
\begin{itemize}
\iii[(i)] The functions $t\mapsto k_{i}(t)$, for $i=1,\ldots,m(t)$, are lower semicontinuous. In particular, being integer valued functions, this implies that the set of points $t$ such that the growth vector is locally constant is open and dense on $[0,T]$.

\iii[(ii)] The function $t\mapsto m(t)$ is upper semicontinuous. As a consequence, the step of an admissible curve is bounded on $[0,T]$. \iii[(iii)] If the admissible curve is equiregular at $t$, then $k_{1}(t)<\ldots<k_{m}(t)$ is a strictly increasing sequence. Let $i < m$. If $k_{i}(t)=k_{i+1}(t)$ for all $t$ in a open neighbourhood then, using a local frame, it is easy to see that this implies $k_{i}(t)=k_{i+1}(t)=\ldots = k_m(t)$ contradicting the fact that the admissible curve is ample at $t$.
\end{itemize}
\end{remark}
{\review 
\begin{lemma} \label{l:flag0}
Assume that the curve is equiregular with step $m$. For every $i=1,\ldots,m-1$, the derivation of sections of $\DD_{\gamma}(t)$ induces a linear surjective map on the quotients
\begin{equation} 
\delta_{i}:\DD^{i}_{\gamma}(t) / \DD^{i-1}_{\gamma}(t) \longrightarrow \DD^{i+1}_{\gamma} (t)/\DD^{i}_{\gamma}(t),\qquad \all t \in [0,T]. 
\end{equation}
In particular we have the following inequalities for $k_{i}=\dim \DD^{i}_{\gamma}(t)$
\begin{equation} 
k_{i}-k_{i-1}\leq k_{i+1}-k_{i}, \qquad \all i=1,\ldots,m-1.
\end{equation}
\end{lemma}
The proof of Lemma \ref{l:flag0} is contained in Appendix \ref{a:flag}.
\finereview}
Next, we show how the family $\DD_{\gamma}(t)$ can be conveniently employed to characterize strictly and strongly normal geodesics.

\begin{proposition}\label{p:amplestrn}
 Let $\gamma:[0,T]\to M$ be a geodesic. Then
\begin{itemize}
\item[(i)] $\gamma$ is strictly normal if and only if $\tx{span}\{\DD_{\gamma}(s),s\in[0,T]\}=T_{x_{0}}M$,
\item[(ii)] $\gamma$ is strongly normal if and only if $\tx{span}\{\DD_{\gamma}(s),s\in[0,t]\}=T_{x_{0}}M$ for all $0<t\leq T$,
\item[(iii)] If $\gamma$ is ample at $t=0$, then it is strongly normal. 
\end{itemize}
\end{proposition} 
\begin{proof}
Recall that a geodesic $\gamma:[0,T]\to M$ is abnormal on $[0,T]$ if and only if the differential $D_{u}\End_{x_{0},T}$ is not surjective, which implies  (see Remark \ref{r:abnormal})
\begin{equation} 
\tx{span}\{(P_{s,T})_*\distr_{\gamma(s)}, s\in[0,T]\}\neq T_{\gamma(T)}M. 
\end{equation}
By applying the inverse flow $(P_{0,T})^{-1}_{*}:T_{\gamma(T)}M\to T_{\gamma(0)}M$, we obtain
\begin{equation} 
\tx{span}\{\DD_{\gamma}(s),s\in[0,T]\}\neq T_{x_{0}}M. 
\end{equation} 
This proves (i). In particular, this implies that a geodesic is strongly normal if and only if 
\begin{equation} 
\spn\{\DD_{\gamma}(s),s\in[0,t]\}=T_{x_{0}}M, \qquad \all 0<t\leq T,
\end{equation} 
which proves (ii). We now prove (iii). We argue by contradiction. If the geodesic is not strongly normal, there exists some $\lam \in T_{x_0}^*M$ such that $\langle \lam, \DD_\gamma(t)\rangle = 0$, for all $0 < t \leq T$. Then, by taking derivatives at $t=0$, we obtain that $\langle \lam, \DD^i_\gamma(0)\rangle = 0$, for all $i\geq 0$, which is impossible since the curve is ample at $t=0$ by hypothesis.
\end{proof}

\begin{remark}
Ample geodesics play a crucial role in our approach to curvature, as we explain in Chapter~\ref{c:geodcost}. By Proposition~\ref{p:amplestrn}, these geodesics are strongly normal. 
One may wonder whether the generic covector $\lam_0 \in T_{x_0}^*M$ corresponds to a strongly normal (or even ample) geodesic. The answer to this question is trivial when there are no abnormal trajectories (e.g. in Riemannian geometry), but the matter is quite delicate in general. For this reason, in order to define the curvature of an affine control system, we assume in the following that the set of ample geodesics is non empty. Eventually, we address the problem of existence of ample geodesics for linear quadratic control systems and sub-Riemannian geometry. In these cases, we will prove that a generic normal geodesic is ample.
\end{remark}

\section{Linearised control system and growth vector}\label{s:crit}

It is well known that the differential of the end-point map at a point $u\in \U$ is related with the linearisation of the control system along the associated trajectory. The goal of this section is to discuss the relation between the controllability of the linearised system and the ampleness of the geodesic.

\subsection{Linearisation of a control system in \texorpdfstring{$\R^{n}$}{Rn}} \label{s:lins} \index{linearisation of a control system}

We start with some general considerations. Consider the nonlinear control system in $\R^{n}$ 
\begin{equation} 
\dot{x}=f(x,u),\qquad x\in \R^{n},\ u\in \R^{k},
\end{equation} 
where $f:\R^n\times\R^k \to \R^n$ is smooth. Fix $x_{0}\in \R^{n}$, and consider the end-point map $\End_{x_{0},t}:\U\to \R^{n}$ for $t\geq 0$. Consider a smooth solution $x_u(t)$, associated with the control $u(t)$, such that $x_u(0)=x_0$. The differential of the end-point map $D_{u}\End_{x_{0},t}:L^{\infty}([0,T],\R^{k})\to \R^{n}$ at $u$ is related with the end-point map of the \emph{linearised system} at the pair $(x_{u}(t),u(t))$. More precisely, for every $v\in L^{\infty}([0,T],\R^{k})$ the trajectory  $y(t)\doteq D_{u}\End_{x_{0},t}(v)\in \R^{n}$ is the solution of the non-autonomous linear system
\begin{equation} \label{eq:lins}
 \begin{cases}
 \dot{y}(t)=A(t)y(t)+B(t)v(t),\\ 
 y(0)=0,
 \end{cases}
\end{equation}
where $A(t)\doteq \dfrac{\partial f}{\partial x}(x_{u}(t),u(t))$ and $B(t)\doteq \dfrac{\partial f}{\partial u}(x_{u}(t),u(t))$  are smooth families of $n\times n$ and $n\times k$ matrices, respectively. We have the formula
\begin{equation} 
y(t)=D_{u}\End_{x_{0},t}(v)=M(t)\int_{0}^{t}M(s)^{-1}B(s)v(s)ds, 
\end{equation}
where $M(t)$ is the solution of the matrix Cauchy problem $\dot{M}(t)=A(t)M(t)$, with $M(0)=\id$. Indeed the solution $M(t)$ is defined on the whole interval $[0,T]$, and it is invertible therein.
\begin{definition}
The linear control system~\eqref{eq:lins} is \emph{controllable} in time $T>0$ if, for any $y \in \R^n$, there exists $v \in L^\infty([0,T],\R^k)$ such that the associated solution $y_v(t)$ satisfies $y_v(T) = y$.
\end{definition}
Let us recall the following classical controllability condition for a linear non-autonomous system, which is the non-autonomous generalization of the Kalman condition (see e.g. \cite{Coronbook}). \index{Kalman condition}
For a set $\{M_i\}$ of $n\times k$ matrices, we denote with $\spn\{M_i\}$ the vector space generated by the columns of the matrices in $\{M_i\}$.
\begin{proposition} \label{p:Kalman}
Consider the control system \eqref{eq:lins}, with $A(t),B(t)$ smooth, and define
\begin{equation} \label{eq:BB}
B_{1}(t)\doteq B(t),\qquad B_{i+1}(t)\doteq A(t)B_{i}(t)-\dot B_{i}(t).
\end{equation}
Assume that there exist $t\in [0,T]$ and $m>0$ such that $\spn\{B_{1}(t),B_{2}(t),\ldots,B_{m}(t)\}=\R^{n}$. Then the system \eqref{eq:lins} is controllable in time $T$.
\end{proposition}

\begin{remark} \label{r:mt} Notice that, using $M(t)$ as a time-dependent change of variable, the new curve $\zeta(t)\doteq M(t)^{-1}y(t)\in \R^{n}$ satisfies 
\begin{equation} \label{eq:lins2}
 \begin{cases}
\dot{\zeta}(t)=M(t)^{-1}B(t)v(t), \\ \zeta(0)=0. 
 \end{cases}
\end{equation}
\end{remark}
If the controllability condition of Proposition~\eqref{p:Kalman} is satisfied for the pair $(A(t),B(t))$, then it is satisfied also for the pair $(0,C(t))$, with $C(t)=M(t)^{-1}B(t)$, as a consequence of the identity $C^{(i)}(t)=(-1)^{i}M(t)^{-1}B_{i+1}(t)$. Therefore, the controllability conditions for the control systems~\eqref{eq:lins} and~\eqref{eq:lins2} are equivalent. Moreover, both systems are controllable if and only if one of them is controllable.

\subsection{Linearisation of a control system in the general setting}
Let us go back to the general setting. Let $\gamma$ be a smooth admissible trajectory associated with the control $u$ such that $\gamma(0) = x_0$. We are interested in the linearisation of the affine control system at $\gamma$.
Consider the image of a fixed control $v\in L^{\infty}([0,T],\R^{k})$ through the differential of the end-point map $\End_{x_{0},t}$, for every $t \geq 0$:
\begin{equation} 
D_{u}\End_{x_{0},t}: L^{\infty}([0,T],\R^{k})\to T_{\gamma(t)}M,\qquad \gamma(t)=\End_{x_{0},t}(u). 
\end{equation}
In this case, for each $t\geq0$, the image of $v$ belongs to a different tangent space. In order to obtain a well defined differential equation, we collect the family of vectors in a single vector space through the composition with the push forward $(P_{0,t})^{-1}_{*}:T_{\gamma(t)}M\to T_{x_0}M$: 
\begin{equation} 
(P_{0,t})^{-1}_{*} \circ D_{u}\End_{x_{0},t}: L^{\infty}([0,T],\R^{k})\to T_{x_0}M. 
\end{equation}
Using formula \eqref{eq:duev} one easily finds
\begin{equation}  
(P_{0,t})^{-1}_{*} \circ D_{u}\End_{x_{0},t}(v)=\int_{0}^{t}(P_{0,s})^{-1}_{*}\overline{f}_{v(s)}(\gamma(s)) ds.
\end{equation}
Denoting $\zeta(t)\doteq  (P_{0,t})^{-1}_{*} \circ D_{u}\End_{x_{0},t}(v)\in T_{x_{0}}M$ one has that, in a local frame, this curve satisfies
\begin{equation}
\dot{\zeta}(t)=(P_{0,t})^{-1}_{*}\overline{f}_{v(t)}(\gamma(t))=\sum_{i=1}^{k}v_{i}(t) (P_{0,t})^{-1}_{*}f_{i}(\gamma(t)).
\end{equation}
Therefore, $\zeta(t)$ is a solution of the control system
\begin{equation}\label{eq:lins3}
\begin{cases}
\dot{\zeta}(t)=C(t)v(t), \\
\zeta(0) = 0,
\end{cases}
\end{equation}
where the $n\times k$ matrix $C(t)$ has columns $C_{i}(t)\doteq (P_{0,t})^{-1}_{*}f_{i}( \gamma(t))$ for $i=1,\ldots,k$. Eq.~\eqref{eq:lins3} is the \emph{linearised system along the admissible curve $\gamma$}. By hypothesis, $\gamma$ is smooth. Then the linearised system is also smooth.

\begin{remark} \label{r:AB}Notice that the composition of the end-point map with $(P_{0,t})^{-1}_{*}$ corresponds to the time dependent transformation $M(t)^{-1}$ of Remark \ref{r:mt}.
\end{remark}

\subsection{Growth vector and controllability}
From the definition of growth vector of an admissible curve, it follows that
\begin{equation} 
 \DD^{i}_{\gamma}(t)=\spn\{C(t),\dot C (t)\ldots,C^{(i-1)}(t)\},\qquad i\geq 1.
\end{equation}

This gives an efficient criterion to compute the geodesic growth vector of the admissible curve $\gamma_u$ associated with the control $u$.
Define in any local frame $f_{1},\ldots,f_{k}$ and any coordinate system in a neighbourhood of $\gamma$, the $n\times n$ and $n\times k$ matrices, respectively:
\begin{gather}
A(t)\doteq \dfrac{\partial f}{\partial x}(\gamma_{u}(t),u(t))=\frac{\partial f_{0}}{\partial x}(\gamma_{u}(t))+\sum_{i=1}^{k}u_{i}(t)\frac{\partial f_{i}}{\partial x}(\gamma_{u}(t)), \label{eq:AAA} \\
B(t)\doteq \dfrac{\partial f}{\partial u}(\gamma_{u}(t),u(t))=\left[f_{i}(\gamma_{u}(t))\right]_{i=1,\ldots,k}. \label{eq:BBB}
\end{gather}
Denoting by $B_{j}(t)$ the matrices defined as in \eqref{eq:BB}, and recalling Remark~\ref{r:mt}, we have
\begin{equation}
k_{i}(t)=\dim \DD^{i}_{\gamma}(t)=\rank\{B_{1}(t),\ldots,B_{i}(t)\}. 
\end{equation}

Assume now that the admissible curve $\gamma$ is actually a normal geodesic of the optimal control system. As a consequence of this discussion and Proposition~\ref{p:Kalman}, we obtain the following characterisation in terms of the controllability of the linearised system.
{\review \begin{proposition} \label{p:amplecontr} Let $\gamma:[0,T]\to M$ be a geodesic. Then
\begin{itemize}
\item[(i)] $\gamma$ is strictly normal $\Leftrightarrow$ the linearised system is controllable in time $T$,
\item[(ii)] $\gamma$ is strongly normal  $\Leftrightarrow$ the linearised system is controllable in time $t$,  $\all t\in(0, T]$,
\item[(iii)$_{t}$] $\gamma$ is ample at $t\in [0,T]$  $\Leftrightarrow$ the controllability condition of Proposition~\ref{p:Kalman} is satisfied at $t\in [0,T]$.
\end{itemize}
In particular (iii)$_{0}\,\Rightarrow$(ii)$\,\Rightarrow$(i). Moreover (i)$\,\Rightarrow$(ii)$\,\Rightarrow$(iii)$_{t}$ for all $t
\in [0,T]$ in the analytic case.
\end{proposition}
The implications in the analytic case are a classical fact about the controllability of non autonomous analytic linear systems. See, for example, \cite[Sec. 1.3]{Coronbook}.
\finereview}

\section{State-feedback invariance of the flag of an admissible curve}\label{s:inv}

In this section we prove that, albeit the family $\DD_{\gamma}(t)$ depends on the choice of the local trivialization, the flag of an admissible curve at $t=0$ is invariant by state-feedback transformation, hence it does not depend on the presentation. This also implies that the growth vector of the admissible curve is well-defined (for all $t$). In this section we use the shorthand $\DD^{i}_{\gamma}=\DD^{i}_{\gamma}(0)$, when the flag is evaluated at $t=0$.

\begin{proposition} \label{p:durezza} The flag $\DD^{1}_{\gamma}\subset \DD^{2}_{\gamma}\subset \ldots \subset T_{x_{0}}M$ is state-feedback invariant. In particular it does not depend on the presentation of the control system. 
\end{proposition}
The next corollary is a direct consequence of Proposition~\ref{p:durezza} and Remark~\ref{r:durezza}.
\begin{corollary} \label{c:durezza}
The growth vector of an admissible curve $\mc{G}_{\gamma}(t)$ is state-feedback invariant. \index{growth vector of an admissible curve!state-feedback invariance}
\end{corollary}

\begin{proof}[Proof of Proposition \ref{p:durezza}]

Recall that every state-feedback transformation is the composition of pure state and a pure feedback one. For pure state transformations the statement is trivial, since it is tantamount to a change of variables on the manifold. Thus, it is enough to prove the proposition for pure feedback ones. Recall that the subspaces $\DD^{i}_{\gamma}$ are defined, in terms of a given presentation, as
\begin{equation} \label{eq:lemmaB}
\DD^{i}_{\gamma}=\spn\{C(0),\ldots,C^{(i-1)}(0)\},\qquad i\geq 1,
\end{equation}
where the columns of the matrices $C(t)$ are given by the vectors $C_i(t) = (P_{0,t})^{-1}_*f_i(\gamma(t))$. 
A pure feedback transformation corresponds to a change of presentation. Thus, let
\begin{equation} 
\dot{x}=f(x,u)=f_{0}(x)+\sum_{i=1}^{k} u_{i} f_{i}(x), \qquad \dot{x}=f'(x,u')=f'_{0}(x)+\sum_{i=1}^{k} u'_{i}f'_{i}(x),
\end{equation}
related by the pure feedback transformation $u'_{i}=\psi_{i}(x,u)=\psi_{i,0}(x)+\sum_{j=1}^{k} \psi_{i,j}(x)u_{j}$. In particular (see also Remark \ref{r:sf}) 
\begin{equation}\label{eq:ffprimo2}
f_{0}(x)=f_{0}'(x)+\sum_{i=1}^{k} \psi_{i,0}(x)f_{i}'(x),\qquad f_{i}(x)=\sum_{j=1}^{k}\psi_{j,i}(x)f_{j}'(x).
\end{equation}
Denote by $A(t), A'(t)$ and $B(t),B'(t)$ the matrices \eqref{eq:AAA} and \eqref{eq:BBB} associated with the two presentations, in some set of coordinates. According to Remark~\ref{r:mt}, $C(t) = M(t)^{-1} B(t)$, where $M(t)$ is the solution of $\dot{M}(t) = A(t) M(t)$, with $M(0) = \id$, and analogous formulae for the \virg{primed} counterparts. In particular, since $C^{(i)}(t) = (-1)^i M(t)^{-1} B_{i+1}(t)$ and $M(0) = M(0)' = \id$, we get
\begin{equation}\label{eq:flagz}
\DD_\gamma^i = \spn\{B_{1}(0),\ldots, B_{i}(0)\}, \qquad (\DD_\gamma^i)' = \spn\{B'_{1}(0),\ldots, B'_{i}(0)\},
\end{equation}
where $B_i(t)$ and $B_i'(t)$ are the matrices defined in Proposition~\ref{p:Kalman} for the two systems. Notice that Eq.~\eqref{eq:flagz} is true only at $t=0$. We prove the following property, which implies our claim: there exists an invertible matrix $\Psi(t)$ such that
\begin{equation}\label{eq:lemmaBB}
B_{i+1}(t)=B'_{i+1}(t)\Psi(t) \bmod \spn\{B'_{1}(t),\ldots,B'_{i}(t)\},
\end{equation}
where Eq.~\eqref{eq:lemmaBB} is meant column-wise. Indeed, from Eq.~\eqref{eq:ffprimo2} we obtain the relations
\begin{equation} \label{eq:AABBprime}
A(t)=A'(t)+ B'(t) \Phi(t),\qquad B(t)=B'(t)\Psi(t),
\end{equation}
where $\Psi(t)$ and $\Phi(t)$ are $k\times k$ and $k\times n$ matrices, respectively, with components
\begin{equation} 
\Psi(t)_{i\ell}\doteq \psi_{i,\ell}(x(t)) ,\qquad \Phi(t)_{i\ell}\doteq \frac{\partial \psi_{i,0}}{\partial x_\ell}(x(t))+\sum_{j=1}^{k}u_{j}(t) \frac{\partial \psi_{i,j}}{\partial x_\ell}(x(t)).
\end{equation}
Notice that, by definition of feedback transformation, $\Psi(t)$ is invertible. We prove Eq. \eqref{eq:lemmaBB} by induction. For $i=0$, it follows from \eqref{eq:AABBprime}. The induction assumption is (we omit $t$)
\begin{equation}
B_{i}=B'_{i}\Psi + \sum_{j=0}^{i-1} B'_{j}\Theta_{j},\qquad \text{for some time dependent $k\times k$ matrices } \Theta_{j}.
\end{equation}
Let $X\simeq Y$ denote $X=Y \bmod \spn\{B'_{1},\ldots,B'_{i}\}$, column-wise. Then 
\begin{equation}
\begin{split}
B_{i+1} &= AB_{i}-\dot B_{i} \simeq \\
&\simeq(A'B'_{i}-\dot B'_{i})\Psi+\sum_{j=0}^{i-1}(A'B'_{j}-\dot B'_{j})\Theta_{j}\simeq B'_{i+1}\Psi. 
\end{split}
\end{equation}
We used that $A=A' \bmod \spn\{ B'\}$, hence we can replace $A$ by $A'$. Moreover all the terms with the derivatives of $\Theta_{j}$ belong to $\spn\{B'_{1},\ldots,B'_{i}\}$.
\end{proof}

\section{An alternative definition} \label{s:altdef}
In this section we present an alternative definition for the flag of an admissible curve, at $t=0$. 
The idea is that the flag $\DD_\gamma = \DD_\gamma(0)$ of a smooth, admissible trajectory $\gamma$ can be obtained by computing the Lie derivatives along the direction of $\gamma$ of sections of the distribution, namely elements of $\overline{\distr}$. In this sense, the flag of an admissible curve carries informations about the germ of the distribution along the given trajectory.

Let $\gamma:[0,T] \to M$ be a smooth admissible trajectory, such that $x_0 = \gamma(0)$. By definition, this means that there exists a smooth map $u : [0,T] \to \cb$ such that $\dot{\gamma}(t) = f(\gamma(t),u(t))$. 
\begin{definition}
We say that $\tanf \in f_0 + \overline\distr$ is a smooth admissible extension of $\dot{\gamma}$ if there exists a smooth section $\sigma : M \to \cb$ such that $\sigma(\gamma(t)) = u(t)$ and $\tanf = f\circ\sigma$.
\end{definition}
In other words $\tanf$ is a vector field extending $\dot{\gamma}$ obtained through the bundle map $f: \cb \to TM$ from an extension of the control $u$ (seen as a section of $\cb$ over the curve $\gamma$). Notice that, if $\dot{\gamma}(t) = f_0(\gamma(t)) + \sum_{i=1}^k u_i(t)f_i(\gamma(t))$, an admissible extension of $\dot{\gamma}$ is a smooth field of the form $\tanf = f_0 + \sum_{i=1}^k \alpha_i f_i$, where $\alpha_i \in C^{\infty}(M)$ are such that $\alpha_i(\gamma(t)) = u_i(t)$ for all $i=1,\ldots,k$.


With abuse of notation, we employ the same symbol $\DDa_\gamma^i$ for the following alternative definition.

\begin{definition}\label{d:flagalt}
The flag of the admissible curve $\gamma$ \index{flag of an admissible curve} is the sequence of subspaces
\begin{equation}
\DDa_\gamma^i \doteq  \spn\{\mc{L}_\tanf^j (X)|_{x_0}|\, X \in \overline{\distr},\, j \leq i-1\} \subset T_{x_0} M, \qquad i \geq 1,
\end{equation}
where $\mc{L}_{\tanf}$ denotes the Lie derivative in the direction of $\tanf$.
\end{definition}
Notice that, by definition, this is a filtration of $T_{x_0}M$, i.e. $\DDa_\gamma^i \subset \DDa_\gamma^{i+1}$, for all $i \geq 1$. Moreover, $\DDa_\gamma^1 = \distr_{x_0}$. In the rest of this section, we show that Definition~\ref{d:flagalt} is well posed, and is equivalent to the original Definition~\ref{d:flag} at $t=0$.

\begin{proposition}\label{p:doesnotdepend}
Definition~\ref{d:flagalt} does not depend on the admissible extension of $\dot{\gamma}$.
\end{proposition}
\begin{proof}
Let $\DDa_\gamma^i$ and $\wt{\DDa}_\gamma^i$ the subspaces obtained via Definition~\ref{d:flagalt} with two different extensions $\tanf$ and $\wt{\tanf}$ of $\dot\gamma$, respectively. In particular, the field $V\doteq \wt{\tanf}-\tanf \in \overline\distr$ vanishes on the support of $\gamma$. We prove that $\wt{\DDa}_\gamma^i = \DDa_\gamma^i$ by induction. For $i=1$ the statement is trivial. Then, assume $\wt{\DDa}_\gamma^i = \DDa_\gamma^i$. Since $\wt{\DDa}_\gamma^{i+1} =\wt{\DDa}_\gamma^{i} + \spn\{\mc{L}_{\wt{\tanf}}^{i}(X)|_{x_0}|\, X \in \overline\distr\}$, it sufficient to prove that
\begin{equation}\label{eq:induzio}
\mc{L}^i_{\wt{\tanf}}(X) = \mc{L}^i_{\tanf}(X) \bmod \DDa_\gamma^i, \qquad X \in \overline{\distr}.
\end{equation}
Notice that $\mc{L}^{i}_{\wt{\tanf}}(X) = \mc{L}^{i}_{\tanf}(X) + W$, where $W \in \VecM$ is the sum of terms of the form 
\begin{equation}
W = \mc{L}_{\tanf}^\ell([V,Y]),\qquad \text{for some} \quad Y \in \VecM,\quad 0 \leq \ell \leq i-1.
\end{equation}
In terms of a local set of generators $f_1,\ldots,f_k$ of $\overline{\distr}$, $V = \sum_{i=1}^k v_j f_j$, where the functions $v_i$ vanish identically on the support of $\gamma$, namely $v_j(\gamma(t)) = 0$ for $t \in [0,T]$. Then, an application of the binomial formula for derivations leads to
\begin{equation}
W = \sum_{j=1}^k \mc{L}_{\tanf}^\ell(v_j[f_j,Y]) - \mc{L}_{\tanf}^\ell (Y(v_j)f_j) = \sum_{j=1}^k \sum_{h=0}^\ell\binom{\ell}{h} \left(\mc{L}^h_{\tanf}(v_j)\mc{L}^{\ell -h}_{\tanf}([f_i,Y]) - \mc{L}_{\tanf}^h(Y(v_j))\mc{L}_{\tanf}^{\ell -h} (f_i)\right).
\end{equation}
Observe that $\mc{L}^h_{\tanf}(v_j)|_{x_0} = \left.\frac{d^h v_j}{dt^h}\right|_{t=0} =0$, for all $h \geq 0$. Then, if we evaluate $W$ at $x_0$, we obtain
\begin{equation}
W|_{x_0} = -\sum_{j=1}^k \sum_{h=0}^\ell\binom{\ell}{h} \mc{L}_{\tanf}^h(Y(v_j))|_{x_0}\mc{L}_{\tanf}^{\ell -h} (f_i).
\end{equation}
Then, since $0 \leq \ell \leq i-1$, and by the induction hypothesis, $W|_{x_0} \in \DDa_\gamma^i$ and Eq.~\eqref{eq:induzio} follows.
\end{proof}

\begin{proposition}
Definition~\ref{d:flagalt} is equivalent to Definition~\ref{d:flag} at $t=0$.
\end{proposition}
\begin{proof}
Recall that, according to Definition~\ref{d:flag}, at $t=0$
\begin{equation}
\DD^{i}_{\gamma} =\DD^{i}_{\gamma}(0) =\spn\left\{\left.\frac{d^{j}}{dt^{j}}\right|_{t=0}\, v(t)\, \bigg| \,v(t)\in \DD_{\gamma}(t) \text{ smooth} ,\, j\leq i  -1 \right\}\subset T_{x_{0}}M, \qquad i\geq 1.
\end{equation}
where $\DD_\gamma(t) = (P_{0,t})^{-1}_*\distr_{\gamma(t)}$. By Proposition~\ref{p:durezza}, the flag at $t=0$ is state-feedback invariant. Then, up to a (local) pure feedback transformation, we assume that the fixed smooth admissible trajectory $\gamma:[0,T] \to M$ is associated with a constant control, namely $\dot{\gamma}(t) = f_0(\gamma(t)) + \sum_{i=1}^k u_i f_i(\gamma(t))$, where $u \in L^\infty([0,T],\R^k)$ is constant. In this case, the flow $P_{0,t} : M\to M$ is actually the flow of the autonomous vector field $\tanf\doteq f_0 + \sum_{i=1}^k u_i f_i$, that is $P_{0,t} = e^{t\tanf}$. 

Indeed $\tanf \in f_0+\overline{\distr}$ is an admissible extension of $\dot{\gamma}$. Moreover, any smooth $v(t) \in \DD_{\gamma}(t)$ is of the form $v(t) = (P_{0,t})_*^{-1} X|_{\gamma(t)}$, where $X \in \overline{\distr}$. Then
\begin{equation}
\left.\frac{d^j}{dt^j}\right|_{t=0} v(t) = \left.\frac{d^j}{dt^j}\right|_{t=0}(P_{0,t})^{-1}_* X|_{\gamma(t)} = \left.\frac{d^j}{dt^j}\right|_{t=0} e^{-t\tanf}_* X|_{\gamma(t)} = \mc{L}_\tanf^j (X)|_{x_0},
\end{equation}
where in the last equality we have employed the definition of Lie derivative.
\end{proof}

\begin{remark}\label{r:family}
To end this section, observe that, for any equiregular smooth admissible curve $\gamma :[0,T] \to M$, the Lie derivative in the direction of the curve defines surjective linear maps
\begin{equation}
\mc{L}_\tanf : \DD^{i}_{\gamma(t)} / \DD^{i-1}_{\gamma(t)} \to \DD^{i+1}_{\gamma(t)} / \DD^i_{\gamma(t)}, \qquad i \geq 1,
\end{equation}
for any fixed $t \in [0,T]$ as follows. Let $\tanf \in \VecM$ be any admissible extension of $\dot{\gamma}$. Similarly, for $X \in \DD^i_{\gamma(t)}$, consider a smooth extension of $X$ along the curve $\gamma$ such that $X|_{\gamma(s)} \in \DD^{i}_{\gamma(s)}$ for all $s \in [0,T]$. Then we define 
\begin{equation}
\mc{L}_\tanf(X) := [T,X]|_{\gamma(t)} \bmod \DD^{i}_{\gamma(t)}, \qquad t \in [0,T].
\end{equation}
The proof that $\mc{L}_\tanf$ does not depend on the choice of the admissible extension $\tanf$ is the same of Proposition~\ref{p:doesnotdepend} and for this reason we omit it. The fact that it depends only on the value of $X \bmod \DD^{i-1}_{\gamma(t)}$ at the point $\gamma(t)$ is similar, under the equiregularity assumption.

In particular, notice that the maps $\mc{L}_\tanf^i : \DD_{\gamma(t)} \to \DD^{i+1}_{\gamma(t)} / \DD^i_{\gamma(t)}$, for $i \geq 1$, are well defined, surjective linear maps from the distribution $\distr_{\gamma(t)}= \DD_{\gamma(t)}$.
\end{remark}
\chapter{Geodesic cost and its asymptotics} \label{c:geodcost}
In this chapter we define the \emph{geodesic cost function} and we state the main result about the existence of its asymptotics (see Theorems~\ref{t:main}-\ref{t:main2}). We anticipate that, in the Riemannian setting, the cost function is the squared Riemannian distance. In this case one can recover the Riemannian sectional curvature from its asymptotics, as we explain in Section~\ref{s:motivation} (see also the Riemannian example in Section~\ref{s:riemann}). This connection paves the way for the definition of curvature of an affine optimal control system that follows.

\section{Motivation: a Riemannian interlude}\label{s:motivation} \index{geodesic cost!Riemannian case}

Let $M$ be an $n$-dimensional Riemannian manifold. In this case, $\cb = T M$, and $f:TM \to TM$ is the identity bundle map. Let $f_{1},\ldots,f_{n}$ be a local orthonormal frame for the Riemannian structure. Any Lipschitz curve on $M$ is admissible, and is a solution of the control system
\begin{equation} 
\dot{x}=\sum_{i=1}^{n} u_{i} f_{i}(x), \qquad x\in M,\, u \in \R^n.
\end{equation}
The cost functional, whose extremals are the classical Riemannian geodesics, is
\begin{equation} 
J_{T}(u)=\frac{1}{2}\int_{0}^{T}\sum_{i=1}^{n} u_{i}(t)^{2} dt. 
\end{equation}
The value function $S_{T}$ can be written in terms of the Riemannian distance $\dist :M\times M \to \R$ as follows:
\begin{equation} 
S_{T}(x,y)=\frac{1}{2T}\dist^{2}(x,y),  \qquad x,y \in M.
\end{equation}
Let $\gamma_{v}(t)$, $\gamma_{w}(s)$ be two arclength parametrized geodesics, with initial vectors $v,w\in T_{x_0}M$, respectively, starting from $x_0$. Let us define the function $C(t,s)\doteq\frac12 \dist^{2}(\gamma_{v}(t),\gamma_{w}(s))$. It is well known that $C$ is smooth at $(0,0)$ (this is not true in more general settings, such as sub-Riemannian geometry). The next formula, due to Loeper and Villani provides, \emph{a posteriori}, the geometrical motivation of our approach (see Lemma~\ref{l:lll} in Section~\ref{s:riemann} for a proof and more detailed explanation):
\begin{equation}
C(t,s)= \frac{1}{2}\left(t^{2}+s^{2}-2 \la v|w\ra ts\right) - \frac{1}{6}\la R^\nabla(v,w)v|w\ra t^{2}s^{2}+t^{2}s^{2}o(|t|+|s|),
\end{equation}
where $\langle\cdot|\cdot\rangle$ denotes the Riemannian inner product and $R^\nabla$ is the Riemann curvature tensor. In particular, the Riemannian curvature tensor can be recovered from the derivatives of $C(t,s)$:
\begin{equation}
\langle R^\nabla(v,w)v|w\rangle = -\frac{3}{2}\frac{\partial^4 C}{\partial t^2 \partial s^2}(0,0).
\end{equation}
Then \emph{\virg{the Riemannian curvature is the second order term in the Taylor expansion (w.r.t. the variable $t$) of the Hessian of $C(t,s)$ (w.r.t. the variable $s$) computed at $(t,s)=(0,0)$}}.

\section{Geodesic cost}
\begin{definition} \label{d:geodesiccost}
Let $x_{0}\in M$ and consider a strongly normal geodesic $\gamma:[0,T]\to M$ such that $\gamma(0)=x_{0}$. The \emph{geodesic cost} \index{geodesic cost}associated with $\gamma$ is the family of functions 
\begin{equation}
\cc_{t}(x)\doteq-\SS_{t}(x,\gamma(t)),\qquad  x\in M,\,  t>0,
\end{equation}
\end{definition}

\begin{figure}[!ht]
\centering
\scalebox{0.7} 
{
\begin{pspicture}(0,-3.7)(12.561894,3.68)
\definecolor{color193}{rgb}{0.0,0.2,1.0}
\definecolor{color192}{rgb}{0.2,0.0,1.0}
\psbezier[linewidth=0.04,linecolor=color193,linestyle=dashed,dash=0.16cm 0.16cm,dotsize=0.07055555cm 2.0]{-*}(10.54,2.16)(9.14,2.44)(5.3795967,1.3622835)(4.42,1.04)(3.4604032,0.7177165)(1.96,-0.1)(1.74,-1.06)
\psbezier[linewidth=0.04,dotsize=0.07055555cm 2.0]{-*}(11.6,3.66)(11.219311,2.6917758)(9.175264,0.7630423)(7.88,-0.14)(6.5847354,-1.0430423)(5.088138,-1.9508411)(1.7,-2.02)
\fontsize{14}{0}
\usefont{T1}{ptm}{m}{n}
\rput(1.671455,-2.435){$x_0$}
\usefont{T1}{ptm}{m}{n}
\rput(11.591455,1.805){$\gamma(t)$}
\pscircle[linewidth=0.04,linestyle=dotted,dotsep=0.16cm,dimen=scalar](1.71,-1.95){1.71}
\usefont{T1}{ptm}{m}{n}
\rput(6.5214553,-2.575){$x\mapsto -S_{t}(x,\gamma(t))$}
\psdots[dotsize=0.12,linecolor=color193](10.52,2.16)
\usefont{T1}{ptm}{m}{n}
\rput(1.171455,-1.095){$x$}
\end{pspicture} 
}
\caption{The geodesic cost function.}
\end{figure}
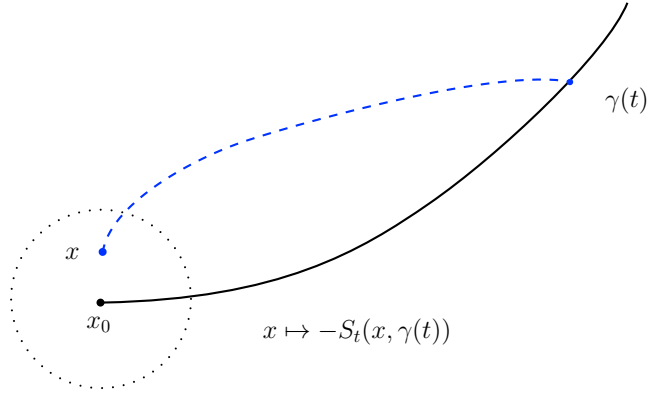
The geodesic cost function is smooth in a neighbourhood of $x_0$, and for $t>0$ sufficiently small. More precisely, Theorem~\ref{t:smoothness}, applied to the geodesic cost, can be rephrased as follows.
\begin{theorem} \label{t:mst} 
Let $x_{0}\in M$ and $\gamma:[0,T]\to M$ be a strongly normal geodesic such that $\gamma(0)=x_{0}$. Then there exist $\eps >0$ and an open set $U\subset (0,\eps) \times M$ such that 
\begin{itemize}
\item[(i)] $(t,x_0) \in U$ for all $t \in (0,\eps)$,
\item[(ii)] The geodesic cost function $(t,x)\mapsto \cc_{t}(x)$ is smooth on $U$.
\end{itemize}
Moreover, for any $(t,x) \in U$, there exists a unique (normal) minimizer of the cost functional $J_t$, among all the admissible curves that connect $x$ with $\gamma(t)$.
\end{theorem}


In the following, $\dot{\cc}_t$ denotes the derivative of the geodesic cost with respect to $t$. 

\begin{proposition} \label{p:critlam0} Under the assumptions above, $d_{x_{0}}c_{t}=\lambda_{0}$, for all $t \in (0,\eps)$. In particular $x_{0}$ is a critical point for the function $\dot{\cc}_{t}$ for all $t \in (0,\eps)$.
\end{proposition}

\begin{proof}
First observe that, in general, if $\gamma(t)$ is an admissible curve for an affine control system, the \virg{reversed} curve  $\til \gamma(t)\doteq\gamma(T-t)$ is no longer admissible. As a consequence, the value function $(x_{0},x_{1})\mapsto S_{T}(x_{0},x_{1})$ is not symmetric and we cannot directly apply Lemma~\ref{l:lambdat} To compute the differential of the value function $x\mapsto -S_{t}(x,\gamma(t))$ at $x_{0}$.  Nevertheless, we can still exploit Lemma~\ref{l:lambdat}, by passing to an associated control problem with reversed dynamic.
\begin{lemma} \label{l:ll} Consider the control system with reversed dynamic
\begin{equation} \label{eq:srcpa}
\begin{aligned}
&\dot{x}=\til{f}(x,u),\qquad x\in M, \qquad \til f(x,u)\doteq-f(x,u), \\
&J_{T}(u)\to \min.
\end{aligned}
\end{equation}
Let $\til S_{T}$ be the value function of this problem. Then $\til S_{T}(x_{0},x_{1})=S_{T}(x_{1},x_{0})$, for all $x_{0},x_{1}\in M$.
\end{lemma}
\begin{proof}[Proof of Lemma \ref{l:ll}]
It is easy to see that the map  $\gamma(t)\mapsto \til\gamma(t)\doteq\gamma(T-t)$ defines a one-to-one correspondence between admissible curves for the two problems. Moreover, if $\gamma$ is associated with the control $u$, then $\til \gamma$ is associated with control $\til u(t)\doteq u(T-t)$. Since the cost is invariant by this transformation, one has $\til S_{T}(x_{1},x_{0})=S_{T}(x_{0},x_{1})$. Notice that this transformation preserves normal and abnormal trajectories and minimizers.
\end{proof}

The Hamiltonian of the reversed system is $\wt H(\lam)=H(-\lam)$. Let $i:T^{*}M\to T^{*}M$ be the fiberwise linear map $\lam\mapsto-\lam$. Then, $i_{*}\vec{\til{H}}(\lam)=-\vec{H}(-\lam)$ (i.e. $\vec{\til{H}}$ is $i_*\text{-related}$ with $-\vec{H}$).
This implies that, if $\lam(t)$ is the lift of the geodesic $\gamma(t)$ for the original system, then $\til{\lam}(t)\doteq-\lam(T-t)$ is the lift of the geodesic $\til{\gamma}(t)=\gamma(T-t)$ for the reversed system. In particular, the final covector of the reversed geodesic $\til{\lam}_T = \til\lam(T)=-\lam(0)=-\lam_{0}$ is equal to minus the initial covector of the original geodesic.
Thus, we can apply Lemma~\ref{l:lambdat} and obtain
\begin{equation} 
d_{x_{0}}c_{T}=-d_{x_{0}}S_{T}(\cdot,\gamma(T))=-d_{x_{0}}(\til S_{T}(\gamma(T),\cdot))=- \tilde{\lam}(T) = \lam_{0}. 
\end{equation}
where $\tilde{\gamma}:[0,T]\to M$ is the unique strictly normal minimizer of the cost functional $\wt{J}_T = J_T$ of the reversed system such that $\tilde{\gamma}(0) = \gamma(T)$ and $\tilde{\gamma}(T) = x_0$.
\end{proof}
 
 
\section{Hamiltonian inner product}\label{s:Hpp} \index{Hamiltonian inner product}

In this section we introduce an inner product on the distribution, which depends on a given geodesic. Namely, it is induced by the second derivative of Hamiltonian of the control system at a point $\lambda\in T^{*}M$, associated with a geodesic. 

A non-negative definite quadratic form, defined on the dual of a vector space $V^*$, induces an inner product on a subspace of $V$ as follows. Recall first that a quadratic form can be defined as a self-adjoint linear map $B: V^* \to V$. $B$ is non-negative definite if, for all $\lam \in V^*$, $\langle \lam , B(\lam)\rangle \geq 0$. Let us define a bilinear map on $\im(B) \subset V$ by the formula
\begin{equation} 
\langle w_{1}|w_{2}\rangle_{B}\doteq\la \lambda_{1},B(\lambda_{2})\ra,\qquad \text{where }w_{i}=B(\lam_{i}).  
\end{equation}
It is easy to prove that $\langle\cdot|\cdot\rangle_{B}$ is symmetric and does not depend on the representatives $\lam_{i}$. Moreover, since $B$ is non-negative definite, $\langle\cdot|\cdot\rangle_{B}$ is an inner product on $\im(B)$.

Now we go back to the general setting. Fix a point $x\in M$, consider the restriction of the Hamiltonian $H$ to the fiber $H_{x}\doteq H|_{T^{*}_{x}M}$ and denote by $d^{2}_{\lam}H_{x}$ its second derivative at the point $\lam\in T^{*}_{x}M$. We show that $d^{2}_{\lam}H_x$ is a non-negative quadratic form and, as a self-adjoint linear map $d^2_\lam H_{x}:T^{*}_{x}M\to T_{x}M$, its image is exactly the distribution at the base point.

\begin{lemma}\label{l:image} 
For every $\lam\in T^{*}_{x}M$,  $d^2_\lam H_x$ is non-negative definite and $\im(d^{2}_{\lam}H_{x} )= \distr_{x}$.
\end{lemma}

\begin{proof}
We prove the result by computing an explicit expression for $d^2_\lam H_x$ in coordinates $\lam=(p,x)$ on $T^*M$. Recall that the maximized Hamiltonian $H$ is defined by the identity 
\begin{equation} 
H(p,x)=\mc{H}(p,x,\bar u)=\la p, f_{0}(x)\ra+\sum_{i=1}^{k} \bar u_{i} \la p,f_{i}(x)\ra- L(x,\bar u),
\end{equation} 
where $\bar u=\bar u (p,x)$ is the solution of the maximality condition
\begin{equation}\label{eq:maxi}
\la p,f_{i}(x)\ra=\frac{\partial \LL}{\partial u_{i}}(x,\bar u(p,x)), \qquad i=1,\ldots,k.
\end{equation}
By the chain rule, we obtain
\begin{equation} 
\frac{\partial H}{\partial p}(p,x)=f_{0}(x)+\sum_{i=1}^{k} \bar u_{i}f_{i}(x)+\underbrace{\frac{\partial \bar u_{i}}{\partial p}\la p,f_{i}(x)\ra-\frac{\partial \LL}{\partial u_{i}} \frac{\partial \bar u_{i}}{\partial p}}_{=0}. 
\end{equation} 
By differentiating Eq.~\eqref{eq:maxi} with respect to $p$, we get
\begin{equation}
f_{i}(x)=\sum_{j=1}^{k}\frac{\partial^{2} \LL}{\partial u_{i}\partial u_{j}}\frac{\partial \bar u_{j}}{\partial p}, \qquad  i=1,\ldots,k.
\end{equation}
Finally, we compute the second derivatives matrix
\begin{equation}\label{eq:hpp}
\frac{\partial^{2} H}{\partial p^{2}}(p,x)=\sum_{i=1}^{k} \frac{\partial \bar u_{i}}{\partial p} f^*_{i}(x)= \sum_{i,j=1}^k f_i(x) \left(\frac{\partial^2 L}{\partial u_i \partial u_j}\right)^{-1} f_j^*(x).
\end{equation}
Since the Hessian of $L$ (with respect to $u$) is positive definite, Eq.~\eqref{eq:hpp} implies that $d^2_\lam H_x$ is non-negative definite and $\im d^2_\lam H_x \subset \distr_x$. Moreover, it is easy to see that $\rank \frac{\partial^{2} H}{\partial p^{2}} = \dim \distr_x$, therefore $\im (d^2_\lam H_x) = \distr_x$.
\end{proof}

\begin{definition}
For any $\lam \in T_x^*M$, the \emph{Hamiltonian inner product} (associated with $\lam$) is the inner product $\langle\cdot |\cdot\rangle_{\lambda}$ induced by $d^{2}_{\lam}H_{x} $ on $\distr_{x}$.
\end{definition}

\begin{remark} \label{r:hamsub} 
We stress that, for any fixed $x \in M$, the subspace $\distr_x \subset T_x M$, where the inner product $\langle\cdot |\cdot\rangle_{\lambda}$ is defined, does not depend on the choice of the element $\lambda$ in the fiber $T_x^* M$. When $H_x$ itself is a quadratic form, $d^{2}_{\lam}H_{x}= 2 H_{x}$ for every $\lambda\in T^{*}_{x}M$. Therefore, the inner product $\langle\cdot|\cdot\rangle_\lam$ does not depend on the choice of $\lambda \in T_x M$. This is the case, for example, of an optimal control system defined by a sub-Riemannian structure, in which the inner product just defined is precisely the sub-Riemannian one (see Chapter~\ref{c:srg}).
\end{remark}

\section{Asymptotics of the geodesic cost function and curvature} \label{s:2dctdot}

Let $f:M\to \R$ be a smooth function defined on a smooth manifold $M$. Its first differential at a point $x\in M$ is the linear map $d_{x}f:T_{x}M\to \R$. The \emph{second differential} \index{second differential} of $f$, as a symmetric bilinear form, is well defined only at a critical point, i.e. at those points $x \in M$ such that $d_{x}f=0$. Indeed, in this case, the map
\begin{equation}
d^{2}_{x}f: T_{x}M\times T_{x}M\to \R,\qquad d^{2}_{x}f(v,w)=V(W(f))(x),
\end{equation}
where $V,W$ are vector fields such that $V(x)=v$ and $W(x)=w$, respectively, is a well defined symmetric bilinear form which does not depend on the choice of the extensions.

The quadratic form associated with the second differential of $f$ at $x$ which, for simplicity, we denote by the same symbol $d^2_x f: T_x M \to \mathbb{R}$, is
\begin{equation}
d^{2}_{x}f(v)=\frac{d^{2}}{dt^{2}}\bigg|_{t=0} f(\gamma(t)),\qquad \gamma(0)=x,\quad \dot \gamma(0)=v.
\end{equation}

Now, for $\lam \in T_{x_0}^* M$, consider the geodesic cost function associated with the strongly normal geodesic $\gamma(t) = \EXP_{x_0}(t,\lam)$, starting from $x_0$. By Proposition \ref{p:critlam0}, for every $t\in (0,\eps)$, the function $x\mapsto \dot c_{t}(x)$ has a critical point at $x_{0}$. 
Hence we can consider the family of quadratic forms defined on the distribution
\begin{equation} 
d^{2}_{x_0}\dot c_{t}\big|_{\distr_{x_0}}: \distr_{x_0} \to \R,\qquad t\in(0,\eps), 
\end{equation}
obtained by the restriction of the second differential of $\dot{c}_t$ to the distribution $\distr_{x_0}$.
Then, using the inner product $\langle\cdot |\cdot\rangle_{\lambda}$ induced by $d^{2}_{\lam}H_{x}$ on $\distr_{x}$ introduced in Section~\ref{s:Hpp}, we associate with this family of quadratic forms the family of symmetric operators on the distribution
$\QQ_\lam(t):\distr_{x_0} \to \distr_{x_0}$ defined by the identity
\begin{equation}\label{eq:prscl}
d^{2}_{x_0}\dot c_{t}(v)\doteq\langle \QQ_\lam(t)v  |v\rangle_{\lam},\qquad   t\in(0,\eps),\, v\in \distr_{x_0}.
\end{equation}
The assumption that the geodesic is strongly normal ensures the smoothness of $\QQ_\lam(t)$ for small $t>0$. If the geodesic is also ample, we have a much stronger statement about the asymptotic behaviour of $\QQ_\lam(t)$ for $t\to 0$.

\begin{maintheorem}\label{t:main} \index{geodesic cost!asymptotics}
Let $\gamma:[0,T]\to M$ be an ample geodesic with initial covector $\lam\in T^{*}_{x_0}M$, and let $\QQ_\lam(t):\distr_{x_0} \to \distr_{x_0}$ be defined by~\eqref{eq:prscl}. Then $t\mapsto t^{2}\QQ_\lam(t)$ can be extended to a smooth family of operators on $\distr_{x_0}$ for small $t\geq 0$, symmetric with respect to $\langle\cdot|\cdot\rangle_\lam$. Moreover,
\begin{equation}
\Qz_{\lam}\doteq\displaystyle \lim_{t\to 0^+}t^{2}\QQ_\lam(t) \geq \id > 0,
\end{equation}
as operators on $(\distr_{x_0},\langle\cdot|\cdot\rangle_\lam)$. Finally
\begin{equation}
\left.\dfrac{d}{dt}\right|_{t=0}t^{2}\QQ_\lam(t)=0.
\end{equation}
\end{maintheorem}
As a consequence of Theorem~\ref{t:main} we are allowed to introduce the following definitions.
\begin{definition}\label{d:curv}
Let $\lam\in T^{*}_{x_0}M$ be the initial covector associated with an ample geodesic. The \emph{curvature} is the symmetric operator $\RR_{\lam}:\distr_{x_0}\to \distr_{x_0}$ defined by
\begin{equation}
\RR_{\lam}\doteq\dfrac{3}{2}\left.\dfrac{d^{2}}{dt^{2}}\right|_{t=0}t^2\QQ_\lam(t).
\end{equation}
The \emph{Ricci curvature} at $\lam\in T_{x_0}^*M$ is defined by $\Ric(\lam)\doteq\trace \RR_{\lam}$. \index{curvature operator} \index{Ricci curvature}
\end{definition}
In particular,  we have the following Laurent expansion for the family of symmetric operators $\QQ_\lam(t): \distr_{x_0} \to \distr_{x_0}$:
\begin{equation}\label{eq:mainexp}
\QQ_\lam(t)= \frac{1}{t^{2}}\Qz_{\lam}+ \frac{1}{3}\RR_{\lam}+O(t),\qquad t>0.
\end{equation}
The normalization factor $1/3$ appearing in \eqref{eq:mainexp} in front of the operator $\RR_{\lam}$ is necessary for recovering the sectional curvature in the case of a control system defined by a Riemannian structure (see Section~\ref{s:riemann}).
We stress that, by construction, $\Qz_{\lam}$ and $\RR_{\lam}$ are operators on the distributions, symmetric with respect to the inner product $\langle \cdot |\cdot\rangle_{\lam}$.

\begin{remark}
Theorem~\ref{t:main} states that the curvature is encoded in the time derivative of the geodesic cost, namely the function $\dot c_t(x)$, for small $t$ and $x$ close to $x_{0}$. A geometrical interpretation of such a function and an insight of its relation with the curvature can be found in Appendix~\ref{s:ctdot}.
\end{remark}

\subsection{Spectrum of \texorpdfstring{$\Qz_{\lam}$}{I} for  equiregular geodesics}\label{s:spec}

Under the assumption that the geodesic is also equiregular, we can completely characterize the operator $\Qz_{\lam}$, namely compute its spectrum.

Let us consider the growth vector $\mc{G}_{\gamma}=\{k_{1},k_{2},\ldots,k_{m}\}$ of the geodesic $\gamma$ which, by the equiregularity assumption, does not depend on $t$. Let $d_{i}\doteq\dim \DD^{i}_{\gamma}-\dim \DD^{i-1}_{\gamma}=k_{i}-k_{i-1}$, for $i=1,\ldots,m$ (where $k_{0}\doteq0$). Recall that $d_{i}$ is a non increasing  sequence (see Lemma~\ref{l:flag0}). Then we can build a tableau with $m$ columns of length $d_{i}$, for $i=1,\ldots,m$, as follows:
\begin{equation}\label{eq:tableaugeneral}
\ytableausetup{boxsize=2em}
\begin{ytableau}
\none[n_1] & \empty & \empty & \none[\dots] & \empty & \empty \\
\none[n_2] & \empty & \empty & \none[\dots] & \empty & \none[d_m] \\
\none & \none[\vdots] & \none[\vdots] & \none & \none[d_{m-1}] \\
\none[n_{k-1}] & \empty & \empty \\
\none[n_k] & \empty & \none[d_2] \\
\none & \none[d_1]
\end{ytableau}
\qquad\qquad \begin{aligned}
\sum_{i=1}^m d_i = n = \dim M, \\
d_1 = k_1 = k \doteq \dim \distr_{x_0}.
\end{aligned}
\end{equation}
Finally, for $j=1,\dots,k$, let $n_j$ be the integers denoting the length of the $j$-th row of the tableau.

\begin{maintheorem} \label{t:main2} 
Let $\gamma:[0,T]\to M$ be an ample and equiregular geodesic with initial covector $\lam\in T^{*}_{x_0}M$. Then the symmetric operator $\Qz_{\lam}: \distr_{x_{0}}\to \distr_{x_{0}}$ satisfies
\begin{itemize}
\iii[(i)] $\spec \Qz_{\lam}=\{n_{1}^{2},\ldots,n_{k}^{2}\}$,
\iii[(ii)] $\trace \Qz_{\lam}=n_{1}^{2}+\ldots +n_{k}^{2}$. 
\end{itemize}
\end{maintheorem}
\begin{remark}
Although the family $\QQ_\lam(t)$ depends on the cost function, the operator $\Qz_\lam$ depends only on the growth vector $\mc{G}_\gamma$, which is a state-feedback invariant (see Section~\ref{s:inv}). Hence the integers $n_{1},\ldots,n_{k}$ do not depend on the cost.
\end{remark}
\begin{remark}\label{r:gd}
By the classical identity $\sum_{i=1}^n (2i-1) = n^2$, we rewrite the trace of $\Qz_\lam$ as follows:
\begin{equation}
\trace \Qz_\lam = \sum_{i=1}^{m}(2i-1)(\dim \DD^{i}_{\gamma}-\dim \DD^{i-1}_{\gamma}).
\end{equation}
Notice that the right hand side of the above equation makes sense also for a non-equiregular (tough still ample) geodesic, where the dimensions are computed at $t=0$. This number also appears in Chapter~\ref{c:srg}, under the name of \emph{geodesic dimension}, in connection with the asymptotics of the volume growth in sub-Riemannian geometry. \index{geodesic dimension}
\end{remark}

The proofs of Theorems~\ref{t:main} and~\ref{t:main2} are postponed to Chapter~\ref{c:proof}, upon the introduction of the required technicals tools. 

\section{Examples}\label{s:examples}

In this section we discuss three relevant examples: Riemannian structures, Finsler structures and an autonomous linear control system on $\R^n$ with quadratic cost. In particular, in the first and second example we show how our construction recovers the classical Riemannian and Finsler flag curvature, respectively. In the third example we show how to compute $\QQ_\lam$ and its expansion, through a direct manipulation of the cost geodesic function. Examples of Sub-Riemannian structures are discussed in Sections~\ref{s:Heis} and~\ref{s:3Dcontact}.

\subsection{Riemannian geometry}\label{s:riemann}
In this example we characterize the family of operators $\QQ_\lam$ and $\Qz_\lam$ for an optimal control system associated with a Riemannian structure. In particular, we show that $I_\lambda$ is the identity operator and $\RR_{\lambda}$ recovers the classical sectional curvature.

Let $M$ be an $n$-dimensional Riemannian manifold. In this case, $\cb = T M$, and $f:TM \to TM$ is the identity bundle map. Let $f_{1},\ldots,f_{n}$ be a local orthonormal frame for the Riemannian structure. Any Lipschitz curve on $M$ is admissible, and is a solution of 
\begin{equation} 
\dot{x}=\sum_{i=1}^{n} u_{i} f_{i}(x), \qquad x\in M,\, u \in \R^n.
\end{equation}
The cost functional, whose extremals are the classical Riemannian geodesics, is
\begin{equation} 
J_{T}(u)=\frac{1}{2}\int_{0}^{T}\sum_{i=1}^{n} u_{i}(t)^{2} dt. 
\end{equation}
Every geodesic is ample and equiregular, and has trivial growth vector $\mc{G}_\gamma =\{n\}$ since, for all $x \in M$, $\distr_{x} = T_{x} M$. 
Then, the tableau associated with $\gamma$ has only one column:
\begin{equation}\label{eq:youngriemann}
\ytableausetup{boxsize=2em}
\begin{ytableau}
\empty \\
\empty \\
\none[\vdots] \\
\empty \\
\end{ytableau}
\end{equation}
and all the rows have length $n_j=1$ for all $j=1,\ldots,\dim M$. Moreover, the Hamiltonian inner product $\langle\cdot|\cdot\rangle_\lambda$ coincides to the Riemannian inner product $\langle\cdot|\cdot\rangle$ for every $\lambda \in T_{x} M$.
As a standard consequence of the Cauchy-Schwartz inequality, and the fact that Riemannian geodesics have constant speed, the value function $S_{T}$ can be written in terms of the Riemannian distance $\dist :M\times M \to \R$ as follows
\begin{equation} 
S_{T}(x,y)=\frac{1}{2T}\dist^{2}(x,y),  \qquad x,y \in M.
\end{equation}
The Riemannian structures realises an isomorphism between $T_{x_0}M$ and $T_{x_0}^*M$, that associates with any $v \in T_{x_0}M$ the covector $\lambda \in T_{x_0}^*M$ such that $\langle\lambda, \cdot\rangle = \langle v|\cdot\rangle$. In particular to any initial covector $\lam \in T_{x_0}^*M$ corresponds an initial vector $v \in T_{x_0}^*M$. We call $\gamma_v:[0,T] \to M$ the associated geodesic, such that $\gamma_v(0) = x_0$ and $\dot{\gamma}_v(0) = v$. Thus, the geodesic cost function associated with $\gamma_v$ is
\begin{equation}
c_t(x) = -\frac{1}{2t}\dist^2(x,\gamma_v(t)).
\end{equation}
Then, in order to compute the operators $\Qz_{\lam}$ and $\RR_{\lam}$ we essentially need an asymptotic expansion of the \virg{squared distance from a geodesic}. We now perform explicitly the expansion of Eq.~\eqref{eq:mainexp}.

Let $\gamma_{v}(t)$, $\gamma_{w}(s)$ be two arclength parametrized geodesics, with initial vectors $v,w\in T_{x_0}M$, respectively, starting from $x_0$. Let us define the function $C(t,s)\doteq\frac{1}{2} \dist^{2}(\gamma_{v}(t),\gamma_{w}(s))$. It is well known that $C$ is smooth at $(0,0)$.

\begin{lemma} \label{l:lll}
The following formula holds true for the Taylor expansion of $C(t,s)$ at $(0,0)$
\begin{equation}\label{eq:villani}
C(t,s)= \frac{1}{2}\left(t^{2}+s^{2}-2 \la v|w\ra ts\right) - \frac{1}{6}\la R^\nabla(v,w)v|w\ra t^{2}s^{2}+t^{2}s^{2}o(|t|+|s|),
\end{equation}
where $\langle\cdot|\cdot\rangle$ denotes the Riemannian inner product and $R^\nabla$ is the Riemann curvature tensor.
\end{lemma}
\begin{proof}
Since the geodesics $\gamma_{v}$ and $\gamma_{w}$ are parametrised by arclength, we have
\begin{equation}\label{eq:no31}
C(t,0)=t^{2}/2, \qquad C(0,s)=s^2/2,\qquad \all t,s\geq0. 
\end{equation}
Moreover, by standard computations, we obtain
\begin{equation}  \label{eq:no32}
\frac{\partial C}{\partial s}(t,0)=-t\la v|w\ra,\qquad \frac{\partial C}{\partial t}(0,s)=-s\la v|w\ra,\qquad \all t,s\geq0. 
\end{equation}
Eqs.~\eqref{eq:no31} and~\eqref{eq:no32} imply that the monomials $t^{n}$, $st^{n}$, $s^{n}$, $ts^{n}$ with $n\geq 2$ do not appear in the Taylor polynomial. The statement is then reduced to the following identity:
\begin{equation}\label{eq:no33}
-\frac{3}{2}\frac{\partial^{4}C}{\partial t^{2}{\partial s^{2}}} (0,0)=\la R^\nabla(w,v)v|w\ra.
\end{equation}
This identity appeared for the first time in \cite[Th. 8.3]{loeper}, in the context of the Ma-Trudinger-Wang curvature tensor, and also in~\cite[Eq. 14.1]{villani}. For a detailed proof one can see also \cite[Prop. 1.5.1]{gallouet}. Essentially, this is the very original definition of curvature introduced by Riemann in his famous \emph{Habilitationsvortrag} (see \cite{Ri}).
\end{proof}
Finally we compute the quadratic form $\langle\QQ_\lam(t)  w |w\rangle = d^{2}_{x_0}\dot c_{t}(w)$ for any $w\in T_{x_0}M$
\begin{equation}\label{eq:mainexpriem}
\begin{split}
d^{2}_{x_0}\dot c_{t}(w)&=\left.\frac{\partial^{2}}{\partial s^{2}}\right|_{s=0}\frac{\partial}{\partial t}\left(-\frac{1}{2t}d^{2}(\gamma_{v}(t),\gamma_{w}(s))\right) 
=\frac{\partial}{\partial t}\left(-\frac1t \frac{\partial^{2}C}{\partial s^{2}}(t,0)\right) = \\
& = \frac{1}{t^{2}}\frac{\partial^{2}C}{\partial s^{2}}(0,0)+\frac13 \left(-\frac32 \frac{\partial^{4}C}{\partial t^{2}{\partial s^{2}}} (0,0)\right)+O(t) = \\
& = \frac{1}{t^2} + \frac{1}{3} \langle R^\nabla(w,v) v|  w\rangle + O(t),
\end{split}
\end{equation}
where, in the first equality, we can exchange the order of derivations by the smoothness of $C(t,s)$ and, in the last equality, we used Eqs.~\eqref{eq:no31}-\eqref{eq:no33}. Now compare Eq.~\eqref{eq:mainexpriem} with the general expansion of Eq.~\eqref{eq:mainexp} and we obtain:
\begin{equation}
\Qz_\lambda  = \mathbb{I}, \qquad  \RR_{\lam} = R^\nabla(\cdot,v)v.
\end{equation}
where $\lam$ is the initial covector associated with the geodesic $\gamma$. For any fixed $\lambda \in T_{x_0}^*M$, $\RR_\lam$ is a linear operator on $T_{x_0} M$, symmetric with respect to the Riemannian scalar product. As a quadratic form on $T_{x_{0}}M$, it computes the sectional curvature of the planes containing the direction of the geodesic, namely
\begin{equation}
\langle \RR_{\lam} w | w\rangle = \|v\|^2\|w\|^2(1-\cos\theta) \mathrm{Sec}(v,w), \qquad \forall w \in T_{x_0}M,
\end{equation}
where $\theta$ is the Riemannian angle between $v$ and $w$. Moreover, since the correspondence $\lambda \leftrightarrow v$ is linear, $\RR_\lambda$ is quadratic with respect to $\lambda$. In particular, it is homogeneous of degree $2$: for any $\alpha >0$ we have $\RR_{\alpha\lambda} = \alpha^2 \RR_\lambda$. The last property remains true for the curvature of any optimal control problem with fiber-wise quadratic Hamiltonian (such as sub-Riemannian structures, see Section~\ref{s:homog}).

Finally, for what concerns the Ricci curvature, we observe that
\begin{equation}
\Ric(\lambda) = \trace \RR_\lambda = \sum_{i=1}^{n} \langle R^\nabla(w_i,v)v| w_i\rangle = \Ric^\nabla(v),
\end{equation}
where $w_1,\ldots,w_n$ is any orthonormal basis of $T_{x_0}M$ and $\Ric^\nabla$ is the classical Ricci curvature associated with the Riemannian structure. Indeed $\Ric(\lambda)$ is homogeneous of degree $2$ in $\lambda$.

\begin{remark}
In Chapter~\ref{c:srg}, we apply our theory to the sub-Riemannian setting, where an analogue approach, leading to the Taylor expansion of Eq.~\eqref{eq:villani} is not possible, for two major differences between the Riemannian and sub-Riemannian setting. First, geodesics cannot be parametrized by their initial tangent vector. Second, and crucial, for every $x_0 \in M$, the sub-Riemannian squared distance $x\mapsto \dist^2(x_0,x)$ is \emph{never} smooth at $x_0$.
\end{remark}

\subsection{Finsler geometry}

The notion of curvature introduced in this paper recovers not only the classical sectional curvature of Riemannian manifolds, but also the notion of \emph{flag curvature} of Finsler manifolds. These structures can be realized as optimal control problems (in the sense of Chapter~\ref{c:affcs}) by the choice $\cb = TM$ and $f: TM \to TM$ equal to the identity bundle map. Moreover the Lagrangian is of the form $L = F^2/2$, where $F \in C^\infty(TM\setminus 0_{TM})$ ($0_{TM}$ is the zero section), is non-negative and positive-homogeneous, i.e. $F(cv) = c F(v)$ for all $v \in TM$ and $c > 0$. Finally $L$ satisfies the Tonelli assumption (A2).

In this setting, it is common to introduce the isomorphism $\tau^* : T^*M \to TM$ (the inverse \emph{Legendre transform}) defined by 
\begin{equation}
\tau^*(\lambda) \doteq d_\lambda H_x, \qquad \lambda \in T_x^* M ,
\end{equation}
where $H_x$ is the restriction to the fiber $T_x^*M$ of the Hamiltonian $H$ of the system.

In this case for all $x \in M$, $\distr_x = T_x M$, hence  every geodesic is ample and equiregular, with trivial growth vector $\mc{G}_\gamma =\{n\}$. The tableau associated with $\gamma$ is the same one as for a Riemannian geodesic~\eqref{eq:youngriemann} with only one column whose rows have length $n_j=1$ for all $j=1,\ldots,\dim M$. The operator $\RR_\lam : T_x M \to T_x M$ can be identified with the Finsler flag curvature operator $R^F_v: T_xM \to T_x M$, where $v = \tau^*(\lambda)$ is the flagpole. A more detailed discussion of Finsler structure and the aforementioned correspondence one can see, for instance, the recent work~\cite[Example 5.1]{Ohta2013}.

\subsection{Sub-Riemannian geometry}
Since sub-Riemannian geometry is extensively treated in the forthcoming Chapter~\ref{c:srg}, we postpone two relevant examples, the Heisenberg group and three-dimensional contact structures, to Sections~\ref{s:Heis} and~\ref{s:3Dcomputations}, respectively.

\subsection{Linear-quadratic control problems} \label{s:lq}
Let us consider a classical linear-quadratic control system. Namely $M = \R^n$, $\cb = \R^n\times \R^k$ and $f(x,u) = Ax + Bu$ is linear both in the state and in the control variables. Admissible curves are solutions of
\begin{equation} \label{eq:lq1}
\dot{x}(t)=Ax(t)+Bu(t), \qquad x\in \R^{n},\,u \in \R^{k},
\end{equation}
where $A$ and $B$ are two $n\times n$ and $n\times k$ matrices, respectively. The cost of an admissible trajectory associated with $u$ is proportional to the square of the $L^{2}$-norm of the control
\begin{equation}\label{eq:lqc}
J_{T}(u)=\frac12 \int_{0}^{T}u(t)^{*}u(t) dt.
\end{equation}
Since $u:[0,T]\to \R^{k}$ is measurable and essentially bounded, the trajectory $x(t;x_{0})$ associated with $u$ such that $x(0;x_{0})=x_{0}$ is explicitly computed by the Cauchy formula
\begin{equation}\label{eq:lq2}
x(t;x_{0})=e^{tA}x_{0}+\int_{0}^{t}e^{(t-s)A}Bu(s)ds.
\end{equation}
In this case, the bracket-generating condition (A1) is the classical Kalman controllability condition:
\begin{equation}\label{eq:lqcontr}
\spn\{B,AB,\ldots,A^{m-1}B\}=\R^{n}.
\end{equation}
Since the system is linear, the linearisation along any admissible trajectory coincides with the system itself. Hence it follows that any geodesic is ample and equiregular. In fact, the geodesic growth vector is the same for any non-trivial geodesic, and is equal to $\mc{G}=\{k_1,\ldots,k_m\}$ where:
\begin{equation} 
k_i=\dim \DD^{i}=\rank\{B,AB,\ldots,A^{i-1}B\},\qquad i = 1,\ldots,m. 
\end{equation}
The associated tableau is the same for any non-trivial geodesic and is built as in~\eqref{eq:tableaugeneral}. The lengths of the rows $n_j$, for $j=1,\ldots,k$ are classically referred to as the \emph{controllability indices} (or Kronecker indices) of the linear control systems (see~\cite[Chapter 9]{agrachevbook} and~\cite[Chapter 1]{Coronbook}).

A standard computation shows that, under the assumption \eqref{eq:lqcontr}, there are no abnormal trajectories. Let us introduce canonical coordinates $(p,x)\in T^{*}\R^{n}\simeq \R^{n*}\times\R^{n}$. Here, it is convenient to treat $p\in\R^{n*}$ as a row vector, and $x\in \R^n$, $u \in \R^k$ as column vectors. The Hamiltonian of the system for normal extremals is
\begin{equation} 
\mc{H}(p,x,u)=pAx+pBu-\frac{1}{2}u^{*}u. 
\end{equation} 
The maximality condition gives $\bar{u}(p,x) = B^{*}p^*$. Then, the maximized Hamiltonian is
\begin{equation} 
H(p,x)=pAx+\frac12 pBB^*p^{*}.
\end{equation}
For a normal trajectory with initial covector $\lam = (p_{0},x_0)$, we have $p(t;x_0,p_0)=p_{0}e^{-tA}$ and
\begin{equation}\label{eq:lq4}
x(t;x_{0},p_{0})=e^{tA} x_{0}+e^{tA} \int_{0}^{t} e^{-sA}BB^{*}e^{-sA^{*}}ds\, p^{*}_{0}.
\end{equation}
Let us denote by $C(t)$ the controllability matrix
\begin{equation} 
C(t)\doteq\int_{0}^{t} e^{-sA}BB^{*}e^{-sA^{*}}ds. 
\end{equation}
By Eq.~\eqref{eq:lq4}, we can compute the optimal cost to reach the point $\wt x(t)=x(t;x_{0},p_{0})$, starting at point $x$ (close to $x_{0}$), in time $t$, as follows
\begin{equation} 
c_{t}(x)=-S_{t}(x,\wt x(t))=-\frac{1}{2}p_{0}C(t)p_{0}^{*}+p_{0}(x-x_{0})-\frac12 (x-x_{0})^{*}C(t)^{-1}(x-x_{0}).
\end{equation}
Thus, $d^{2}_{x}\dot c_{t}=-\frac{d}{dt}C(t)^{-1}$, and the family of quadratic forms $\QQ_\lam$, written in terms of the basis defined by the columns of $B$, is represented by the matrix
\begin{equation} 
\QQ_{\lam}(t)=- B^{*}\frac{d}{dt}C(t)^{-1}B.
\end{equation}
The operator $\Qz_{\lam}$ is completely determined by Theorem \ref{t:main2}. Its eigenvalues coincide with the squares of the Kronecker indices (or controllability indices) of the control system (see \cite{agrachevbook,Coronbook}). Moreover, the curvature $\RR_{\lam}$ is
\begin{equation}
\RR_{\lam} =-\frac{3}{2}\left.\frac{d^{2}}{dt^{2}}\right|_{t=0}\left(t^{2} B^{*}\frac{d}{dt}C(t)^{-1}B\right)=-\frac{3}{2}\left.\frac{d^{2}}{dt^{2}}\right|_{t=0}\left(t B^{*}C(t)^{-1}B\right).
\end{equation}
We stress that, for this specific case, the operators $\Qz_{\lam}$ and $\RR_{\lam}$ do not depend neither on the geodesic nor on the initial point since the system is linear (hence it coincides with its linearisation along any geodesic starting at any point).

\begin{remark}
With straightforward but long computations one can generalize these formulae to the case of a quadratic cost with a potential of the form
\begin{equation}
J_T(u) = \frac{1}{2}\int_{0}^T u(t)^* u(t) + x_u(t)^* Q x_u(t) dt,
\end{equation}
where $Q$ is a symmetric $n\times n$ matrix, and $x_u(t)$ is the trajectory associated with the control $u$.
\end{remark}
\chapter{Sub-Riemannian geometry} \label{c:srg}

In this chapter we focus on the sub-Riemannian setting. After a brief introduction, we discuss the existence of ample geodesics, the regularity  of the geodesic cost and the homogeneity properties of the family $\QQ_{\lam}$. Then we state the main result of this chapter about the sub-Laplacian of the sub-Riemannian distance. Finally, we define the concept of geodesic dimension and we investigate the asymptotic rate of growth of the volume of measurable set under sub-Riemannian geodesic homotheties.

\section{Basic definitions}

Sub-Riemannian structures are particular affine optimal control system, in the sense of Definition~\ref{d:cs}, where the \virg{drift} vector field is zero and the Lagrangian $\LL$ is induced by an Euclidean structure on the control bundle $\cb$. For a general introduction to sub-Riemannian geometry from the control theory viewpoint we refer to \cite{nostrolibro}. Other classical references are \cite{bellaiche,montgomerybook}.
\begin{definition}
Let $M$ be a connected, smooth $n$-dimensional manifold. A \emph{sub-Riemannian structure} \index{sub-Riemannian!structure} on $M$ is a pair $(\cb,f)$ where:
\begin{itemize}
\iii[$(i)$] $\cb$ is a smooth rank $k$ \emph{Euclidean} vector bundle with base $M$ and fiber $\fib_x$, i.e. for every $x\in M$, $\fib_x$ is a $k$-dimensional vector space endowed with an inner product.
\iii[$(ii)$] $f:\cb\to TM$ is a smooth \emph{linear} morphism of vector bundles, i.e. $f$ is \emph{linear} on fibers and the following diagram is commutative:
\begin{equation}
\xymatrix{
\cb \ar[dr]_{\pi_{\cb}} \ar[r]^{f}
& TM \ar[d]^{\pi} \\
 & M }
\end{equation}
\end{itemize}
The maps $\pi_{\cb}$ and $\pi$ are the canonical projections of the vector bundles $\cb$ and $TM$, respectively. Notice that once we have chosen a local trivialization for the vector bundle $\cb$, i.e. $\cb\simeq M\times \R^{k}$, we can choose a basis in the fibers and the map $f$ reads $f(x,u)=\sum_{i=1}^{k}u_{i}f_{i}(x)$.
\end{definition}

\brem There is no assumption on the rank of the function $f$. In other words if we consider, in some choice of the trivialization of $\cb$, the vector fields $f_{1},\ldots,f_{k}$, they could be  linearly dependent at some (or even at every) point. The structure is Riemannian if and only if $\dim \distr_{x}=n$ for all $x\in M$. 
\erem
\begin{remark}[On the notation]
Throughout this chapter, to adhere to the standard notation of the sub-Riemannian literature, we use the notation $X_i = f_i$ for the set of (local) vector fields which define the sub-Riemannian structure.
\end{remark}
The Euclidean structure on the fibers induces a metric structure on the \emph{distribution} $\distr_{x}=f(\fib_x)$  for all $x\in M$ as follows:
\begin{equation}\label{eq:min}
\|v\|_x^{2}\doteq \min\Pg{\|u\|^2\,\bigg| \ v=f(x,u)},\qquad \all v\in \distr_{x}.
\end{equation}
It is possible to show that $\|\cdot\|_x$ is a norm on $\distr_x$ that satisfies the parallelogram law, i.e.
it is actually induced by an inner product $\metr{\cdot}{\cdot}_{x}$ on $\distr_x$. Notice that the minimum in \eqref{eq:min} is always attained since we are minimizing an Euclidean norm in $\R^{k}$ on an affine subspace.

It is always possible to reduce to the case when the control bundle $\cb$ is trivial without changing the sub-Riemannian inner product (see \cite{nostrolibro,noterifford}). In particular it is not restrictive to assume that the vector fields $X_{1},\ldots,X_{k}$ are globally defined. 

An admissible trajectory for the sub-Riemannian structure is also called \emph{horizontal}, i.e. a Lipschitz curve $\gamma:[0,T]\to M$ such that
\begin{equation}
\dot \gamma(t)=f(\gamma(t),u(t)), \qquad \text{a.e.}\  t\in [0,T],
\end{equation}
for some measurable and essentially bounded map $u:[0,T]\to \R^{k}$. 
\begin{remark} Given an admissible trajectory it is pointwise defined its \emph{minimal control} $u:[0,T]\to \R^{k}$ such that $\|\dot\gamma(t)\|^{2}=\|u(t)\|^{2}=\sum_{i=1}^{k} u_{i}^{2}(t)$ for a.e. $t\in [0,T]$. In what follows, whenever we speak about the control associated with a horizontal trajectory, we implicitly assume to consider its minimal control. This is the sub-Riemannian implementation of Remark~\ref{r:mincontr}
\end{remark}

For every admissible curve $\gamma$, it is natural to define its \emph{length} by the formula \index{sub-Riemannian!length}
\begin{equation}
\ell(\gamma)=\int_{0}^{T}\|\dot \gamma (t)\| dt= \int_{0}^{T} \left(\sum_{i=1}^{k} u_{i}^{2}(t)\right)^{1/2} dt.
\end{equation}
Since the length is invariant by reparametrization, we can always assume that $\|\dot\gamma(t)\|$ is constant. The \emph{sub-Riemannian (or Carnot-Carath\'eodory) distance} \index{sub-Riemannian!distance} \index{Carnot-Carath\'eodory distance} between two points $x,y\in M$ is
\begin{equation}
\dist(x,y)\doteq \inf\{\ell(\gamma)\, |\, \gamma \text{ horizontal, }\g(0)=x,\g(T)=y\}.
\end{equation}
It follows from the Cauchy-Schwartz inequality that, if the final time $T$ is fixed, the minima of the length (parametrized with constant speed) coincide with the minima of the energy functional:
\begin{equation}
J_{T}(\gamma)=\frac12\int_{0}^{T}\|\dot \gamma (t)\|^{2} dt=\frac12 \int_{0}^{T}\sum_{i=1}^{k} u_{i}^{2}(t) dt.
\end{equation}
Moreover, if $\gamma$ is a minimizer with constant speed, one has the identity $\ell^{2}(\gamma)=2T J_{T}(\gamma)$.

In particular, the problem of finding the sub-Riemannian geodesics, i.e. curves on $M$ that minimize the distance between two points, coincides with the optimal control problem 
\begin{equation}
\begin{aligned}\label{eq:srocp}
&\dot{x}=\sum_{i=1}^{k} u_{i} X_{i}(x),\qquad x \in M, \\
&x(0)=x_{0}, \, x(T)=x_{1}, \qquad J_{T}(u)\to \min.
\end{aligned}
\end{equation}
Thus, a sub-Riemannian structure corresponds to an affine optimal control problem~\eqref{eq:ocp} where $f_{0} = 0$ and the Lagrangian $\LL(x,u)=\frac12 \|u\|^{2}$ is induced by the Euclidean structure on $\cb$. Extremal trajectories for the sub-Riemannian optimal control problem can be normal or abnormal according to Definition~\ref{d:normal}.

\begin{remark} \label{r:ipsr} The assumption (A1) on the control system in the sub-Riemannian case reads  $\text{Lie}_{x} \overline \distr=T_{x}M$, for every $x\in M$. This is the classical \emph{bracket-generating} (or \emph{H\"ormander}) condition on the distribution $\distr$, which implies the controllability of the system, i.e.  $\dist(x,y)<\infty$ for all $x,y\in M$. Moreover one can show that $\dist$ induces on $M$ the original manifold's topology. When $(M,\dist)$ is complete as a metric space,  Filippov Theorem guarantees the existence of minimizers joining $x$ to $y$, for all $x,y\in M$ (see \cite{agrachevbook,nostrolibro}).
\end{remark}

The maximality condition \eqref{eq:as} of PMP reads $u_{i}(\lam)=\la \lambda,X_{i}(x)\ra$, where $x=\pi(\lam)$.
Thus the maximized Hamiltonian is
\begin{equation} 
H(\lam)=\frac12 \sum_{i=1}^{k}\la\lam,X_{i}(x)\ra^{2},\qquad \lam\in T^{*}M. 
\end{equation}
It is easily seen that $H: T^*M \to \R$ is also characterized as the dual of the norm on the distribution
\begin{equation} 
H(\lam)=\frac12\|\lam\|^{2},\qquad \|\lam\|=\sup \{\la\lam,v\ra \mid v \in \distr_x, \; \|v\|=1\}.
\end{equation}

Since, in this case, $H$ is quadratic on fibers, we obtain immediately the following properties for the exponential map
\begin{equation}
\EXP_{x_0}(t,s\lam_{0})=\EXP_{x_0}(ts,\lam_{0}), \qquad  \lam_0 \in T_{x_0}^*M, \quad t,s\geq 0,
\end{equation}
which is tantamount to the fact that the normal geodesic associated with the covector $\lam_{0}$ is the image of the ray $\{t\lam_{0},t\geq 0\} \subset T_{x_0}^* M$ through the exponential map: $\EXP_{x_0}(1,t\lam_{0})=\gamma(t)$.

\bdeff Let $\gamma(t)=\pi \circ e^{t\vec{H}}(\lam_{0})$ be a strictly normal geodesic. We say that $\gamma(s)$ is \emph{conjugate} to $\gamma(0)$ along $\gamma$ if $\lam_{0}$ is a critical point for $\EXP_{x_0,s}$, i.e. $D_{\lam_0}\EXP_{x_0,s}$ is not surjective.\index{conjugate point}
\edeff
\brem The sub-Riemannian maximized Hamiltonian is a quadratic function on fibers, which implies $d^{2}_{\lam}H_{x}=2H_{x}$, where $H_x = H|_{T_x^*M}$ and $\lam\in T^{*}_{x}M$. In particular $d^2_\lam H_x$ does not depend on $\lam$ and the inner product $\la \cdot |\cdot \ra_{\lam}$ induced on the distribution $\distr_{x}$ coincides with the sub-Riemannian inner product (see Section \ref{s:Hpp}).
\erem
 
The value function \index{value function!sub-Riemannian}at time $T>0$ of the sub-Riemannian optimal control problem~\eqref{eq:srocp} is closely related with the sub-Riemannian distance as follows:
\begin{equation}
\SS_{T}(x,y)=\frac{1}{2T}\dist^{2}(x,y),\qquad x,y\in M,
\end{equation}
Notice that, with respect to Definition~\ref{d:value} of value function, we choose $M'=M$, even if the latter is not compact. Indeed, the proof of the regularity of the value function in Appendix~\ref{a:proofsmoothness} can be adapted by using the fact that small sub-Riemannian balls are compact.

Next, we provide a fundamental characterization for smooth points of the squared distance. Let $x_0 \in M$, and let $\Sigma_{x_{0}}\subset M$  be the set of points $x$ such that 
there exists a unique minimizer $\gamma :[0,1]\to M$ joining $x_0$ with $x$, which is not abnormal and $x$ is not conjugate to $x_0$ along $\gamma$. 
\begin{theorem}[see \cite{agrachevsmooth,trelatrifford}] \label{t:d2sr} \index{}
Let $x_{0}\in M$ and set $\f\doteq 	\frac{1}{2}\dist^{2}(x_{0},\cdot)$. The set $\Sigma_{x_0}$ is open, dense and $\f$ is smooth precisely on $\Sigma_{x_{0}}$.
\end{theorem}
This result can be seen as a ``global'' version of Theorem~\ref{t:smoothness}. Finally, as a consequence of Lemma~\ref{l:lambdat}, if $x\in \Sigma_{x_{0}}$ then $d_{x}\f=\lambda(1)$, where $\lambda(t)$ is the normal lift of $\gamma(t)$.

\subsection{Nilpotent approximation and privileged coordinates} \label{s:deltaeps} \index{nilpotent approximation}
In this section we briefly recall the concept of nilpotent approximation. For more details we refer to \cite{agrachevlocal,agragauthiersubanal,notejean,bellaiche}. See also \cite{mitchell} for equiregular structures. The classical presentation that follows relies on the introduction of a set of privileged coordinates; an intrinsic construction can be found in \cite{nostrolibro}.

Let $M$ be a bracket-generating sub-Riemannian manifold. The \emph{flag} of the distribution at a point $x\in M$ is the sequence of subspaces
$\distr_x^0\subset\distr_x^{1}\subset\distr_x^{2}\subset\ldots\subset T_xM$ defined by
\begin{equation}
\distr^0_x\doteq  \{0\},\qquad\distr_x^{1}\doteq \distr_x,\qquad\distr_x^{i+1}\doteq \distr_x^{i}+[\distr^{i},\distr]_x,
\end{equation}
where, with a standard abuse of notation, we understand that $[\distr^i,\distr]_x$ is the vector space generated by the iterated Lie brackets, up to length $i+1$, of local sections of the distribution, evaluated at $x$. We denote by $\m=\m_{x}$ the \emph{step of the distribution} at $x$, i.e. the smallest integer such that $\distr_x^{\m}=T_{x}M$. The sub-Riemannian structure is called \emph{equiregular} if $\dim \distr^i_x$ does not depend on $x\in M$, for every $i\geq1$.

Let $O_{x}$ be an open neighbourhood of the point $x\in M$. We say that a system of coordinates $\psi: O_{x}\to \R^{n}$ is \emph{linearly adapted} to the flag if, in these coordinates, $\psi(x)=0$ and
\begin{equation} 
\psi_{*}(\distr^{i}_{x})=\R^{h_{1}}\oplus \ldots \oplus \R^{h_{i}}, \qquad \all i=1,\ldots,\m, 
\end{equation}
where $h_i=\tx{dim}\, \distr^{i}_{x}-\tx{dim}\, \distr^{i-1}_{x}$ for $i=1,\ldots,\m$. Indeed $h_{1}+\ldots+h_{\m}=n$.

In these coordinates, $x=(x_{1},\ldots,x_{\m})$, where $x_{i}=(x_{i}^{1},\ldots,x_{i}^{h_{i}})\in \R^{h_{i}}$, and $T_x M =\R^{h_{1}}\oplus \ldots \oplus \R^{h_{\m}}$.
The space of all differential operators in $\R^{n}$ with smooth coefficients forms an associative
algebra with composition of operators as multiplication. The differential operators with
polynomial coefficients form a subalgebra of this algebra with generators $1, x_{i}^{j} ,\partial_{x_{i}^{j}},$ where
$i=1,\ldots,\m;\  j = 1,\ldots,k_{i}$. We define weights of generators as follows:
$\nu(1)=0,\,  \nu(x_{i}^{j})=i,\,  \nu(\partial_{x_{i}^{j}})=-i,$ and the weight of monomials accordingly.
Notice that a polynomial differential operator homogeneous with respect to $\nu$ (i.e. whose monomials are all of same weight) is homogeneous with respect to dilations $\delta_{\al}:\R^{n}\to \R^{n}$ defined by
\begin{equation}\label{eq:dilations0}
 \delta_{\al}(x_{1},\ldots,x_{\m})=(\al x_{1},\al^{2}x_{2},\ldots, \al^{\m}x_{\m}),\qquad \al>0.
\end{equation}
In particular for a homogeneous vector field $X$ of weight $h$ it holds $\delta_{\al*}X=\al^{-h}X
$.

Let $X\in \tx{Vec}(\R^{n})$, and consider its Taylor expansion at the origin as a first order differential operator. Namely, we can write the formal expansion
\begin{equation} 
X\approx \sum_{h=-\m}^{\infty} X^{(h)}, 
\end{equation}
where $X^{(h)}$ is the homogeneous part of degree $h$ of $X$ (notice that every monomial of a first order differential operator has weight not smaller than $-\m$).
Define the filtration of  $\tx{Vec}(\R^{n})$ 
\begin{equation} 
\tx{Vec}^{(h)}(\R^{n})=\{X\in \tx{Vec}(\R^{n}): X^{(i)}=0, \all i<h\}, \qquad h\in \mb{Z}. 
\end{equation}
\bdeff A system of coordinates $\psi: O_{x}\to \R^{n}$ is called \emph{privileged} for the sub-Riemannian structure if they are linearly adapted and $\psi_{*}X_{i}\in \tx{Vec}^{(-1)}(\R^{n})$ for every $i=1,\ldots,k$. \index{privileged coordinates}
\edeff
The existence of privileged coordinates is proved, e.g. in \cite{agrachevlocal,bellaiche}. Notice, however, that privileged coordinates are not unique. Now we are ready to define the sub-Riemannian tangent space of $M$ at $x$.
\bdeff
Given a set of privileged coordinates, the \emph{nilpotent approximation at $x$} is the sub-Riemannian structure on $T_x M = \R^{n}$ defined by the set of vector fields $\wh{X}_{1}, \ldots, \wh{X}_{k}$, where $\wh{X}_{i}\doteq (\psi_{*}X_{i})^{(-1)} \in \mathrm{Vec}(\R^n)$.
\edeff
The definition is well posed, in the sense that the structures obtained by different sets of privileged coordinates are isometric (see \cite[Proposition 5.20]{bellaiche}). Then, in what follows we omit the coordinate map in the notation above, identifying $T_x M = \R^n$ and a vector field with its coordinate expression in $\R^{n}$. The next proposition also justifies the name of the sub-Riemannian tangent space (see \cite[Proposition 5.17]{bellaiche}).
\begin{proposition}
The vector fields $\wh{X}_1,\ldots \wh{X}_k$ generate a nilpotent Lie algebra $\mathrm{Lie}(\wh{X}_1,\ldots,\wh{X}_k)$ of step $\m$. At any point $z \in \R^n$ they satisfy the bracket-generating assumption, namely $\mathrm{Lie}_z(\wh{X}_1,\ldots,\wh{X}_k) = \R^n$. 
\end{proposition}

\begin{remark}
The sub-Riemannian distance $\widehat{\dist}$ on the nilpotent approximation is homogeneous with respect to dilations $\delta_{\al}$, i.e. $\widehat{\dist}(\delta_{\al}(x),\delta_{\al}(y))=\al \,\widehat{\dist}(x,y)$.
\end{remark}

\bdeff 
Let $X_{1},\ldots,X_{k}$ be a set of vector fields which defines the sub-Riemannian structure on $M$ and fix a system of privileged coordinates at $x\in M$. The $\eps$-approximating system at $x$ is the sub-Riemannian structure induced by the vector fields $X_{1}^{\eps},\ldots,X_{k}^{\eps}$ defined by
\begin{equation}
X_{i}^{\eps}\doteq \eps \delta_{1/\eps*}X_{i}, \qquad 
i=1,\ldots,k.
\end{equation}
\edeff
The following lemma is a consequence of the definition of $\eps$-approximating system and privileged coordinates.
\begin{lemma}\label{l:convframe}
$X_{i}^{\eps}\to \wh{X}_{i}$ in the $C^{\infty}$ topology of uniform convergence of all derivatives on compact sets in $\R^{n}$ when $\eps \to 0$, for $i=1,\ldots,k$.
\end{lemma}
Therefore, the  nilpotent approximation $\wh{X}$ of a vector field $X$ at a point $x$ is the \virg{principal part} in the expansion when one considers the blown up coordinates near the point $x$, with rescaled distances. 

\subsection{Approximating trajectories}\label{s:approxtraj}

In this subsection we show, in a system of privileged coordinates $\psi: O_x \to \mathbb{R}^n$, how the normal trajectories of the $\eps$-approximating system converge to corresponding normal trajectories of the nilpotent approximation.

Let $H^\eps :T^*\mathbb{R}^n \to \mathbb{R}$ be the maximized Hamiltonian for the $\eps$-approximating system, and $\EXP^\eps : T^*_0\mathbb{R}^n\to \mathbb{R}^n$ the corresponding exponential map (starting at $0$). We denote by the symbols $\wh{H}$ and $\wh{\EXP}$ the analogous objects for the nilpotent approximation. The $\eps$-approximating normal trajectory $\gamma^\eps(t)$ converges to the corresponding nilpotent trajectory $\wh{\gamma}(t)$.

\begin{proposition}\label{p:catapulta} \index{nilpotent approximation!convergence lemma}
Let $\lam_0 \in T_0^*\mathbb{R}^n$. Let $\gamma^\eps:[0,T] \to \mathbb{R}^n$ and $\wh{\gamma}:[0,T]\to \mathbb{R}^n$ be the normal geodesics associated with $\lam_0$ for the $\eps$-approximating system and for the nilpotent system, respectively. Let $u^\eps :[0,T] \to \mathbb{R}^k$ and $\wh{u}:[0,T] \to \mathbb{R}^k$ be the associated controls. Then there exists a neighbourhood $O_{\lam_0} \subset T_0^*\mathbb{R}^n$ of $\lam_0$ such that for $\eps \to 0$
\begin{itemize}
\item[(i)] $\EXP^\eps \to \widehat{\EXP}$ in the $C^{\infty}$ topology of uniform convergence of all derivatives on $O_{\lam_0}$,
\item[(ii)] $\gamma^\eps \to \wh{\gamma}$ in the $C^{\infty}$ topology of uniform convergence of all derivatives on $[0,T]$,
\item[(iii)] $u^\eps \to \wh{u}$ in the $C^{\infty}$ topology of uniform convergence of all derivatives on $[0,T]$.
\end{itemize}
\end{proposition}
The proof of Proposition~\ref{p:catapulta} is a consequence of a more general statement for the Hamiltonian flow of the approximating systems, which can be found in Appendix~\ref{a:proofcatapulta}.

\section{Existence of ample geodesics}\label{s:ample}
In this section we discuss the properties of the growth vector in the sub-Riemannian setting. Even though we defined the growth vector for any admissible curve, here we restrict our attention to (possibly abnormal) geodesics. Thus, we employ the terminology \emph{geodesic flag} and \emph{geodesic growth vector} to denote the flag and growth vector of a geodesic, respectively. We start with a basic estimate, which is a direct consequence of the alternative definition of the geodesic flag given in Section~\ref{s:altdef}.
\begin{lemma} \label{l:upbound} Let $\g:[0,T]\to M$ be a normal geodesic.
For every $t\in[0,T]$ and every $i\geq 1$ one has 
\begin{equation}
\dim \DD_{\g}^{i}(t)\leq \dim \distr^{i}_{\gamma(t)}.
\end{equation}
\end{lemma}
Next we prove the existence of ample geodesics on every sub-Riemannian manifold. 
\begin{theorem} \label{t:bound} Let $M$ be a sub-Riemannian manifold and $x_{0}\in M$. Then there exists at least one geodesic $\g:[0,T]\to M$ starting at $x_{0}$ that is ample at every $t\in[0,T]$.
\end{theorem}
\begin{proof} 
Consider privileged coordinates on a neighbourhood $O_{x_{0}}$ of $x_{0}$ and let $\lam_{0}\in T^{*}_{x_{0}}M$. As in Proposition \ref{p:catapulta}, for every $\eps>0$ sufficiently small, we define the curve $\g^{\eps}(t)=\EXP^{\eps}(t,\lambda_{0})$ which is a normal geodesic for the $\eps$-approximating system. Let $\wh \gamma=\wh\EXP(t,\lam_{0})$ be the normal geodesic associated with $\lam_{0}$ in the nilpotent approximation at $x_{0}$. Recall that $\gamma^{\eps}\to \wh\gamma$ uniformly with all derivatives on some common neighbourhood of definition $[0,T]$.

\begin{lemma}\label{l:situa3} There always exists $\lam_{0}\in T^{*}_{x_{0}}M$ such that $\wh{\g}(t)=\wh{\EXP}(t,\lam_{0})$ is ample at every $t\in[0,T]$.
\end{lemma}
\begin{proof}[Proof of Lemma \ref{l:situa3}]
The nilpotent approximation at $x_{0}$ is an analytic sub-Riemannian structure. By Proposition~\ref{p:amplecontr}, every strictly normal geodesic is ample at every $t\in [0,T]$. The existence of at least one strictly normal geodesic (on any smooth sub-Riemannian manifold)  follows by Theorem~\ref{t:d2sr}.
\end{proof}

We now show that, for $\eps$ small enough, the growth vector of the geodesic $\wh \gamma$ controls (more precisely, bounds from below) the growth vector of the geodesic $\gamma^{\eps}$ of the $\eps$-approximating system.

\bl \label{l:situa1}
Let $\DD_{\gamma^{\eps}}^{i}(t)$ and $\DD^{i}_{\hat{\gamma}}(t)$ be the $i$-th element of the geodesic flag at time $t$ of $\gamma^{\eps}$ and $\wh{\g}$, respectively. Then, for every $i\geq 1$ and $t\in [0,T]$ we have 
\begin{equation} 
 \dim \DD^{i}_{\hat{\gamma}}(t)\leq \liminf_{\eps\to 0} \dim \DD_{\gamma^{\eps}}^{i}(t).
\end{equation}
\el
\begin{proof}[Proof of Lemma \ref{l:situa1}]
To compute the dimension of the geodesic flag, we use the criterion of Section~\ref{s:crit}. For any normal geodesic $\gamma$, associated with the control $u$, of the control system
\begin{equation} 
\dot{x}=f(x,u)=\sum_{i=1}^{k} u_{i} X_{i}(x), \qquad x\in \R^{n} ,
\end{equation}
we define the matrices
\begin{gather} 
A(t)= \dfrac{\partial f}{\partial x}(\gamma(t),u(t)),\qquad B(t)= \dfrac{\partial f}{\partial u}(\gamma(t),u(t)),  
\end{gather}
which, in turn, define the matrices
\begin{equation} \label{eq:BB0}
B_{1}(t)= B(t),\qquad B_{i+1}(t)= A(t)B_{i}(t)-\dot B_{i}(t),\qquad \all i\geq 1.
\end{equation}
Then
\begin{equation} 
\dim \DD^{i}_{\gamma}(t)=\rank\{B_{1}(t),\ldots,B_{i}(t)\}. 
\end{equation}
We apply the criterion to the geodesics $\g^{\eps}$ and $\wh\gamma$ of the $\eps$-approximating and nilpotent systems, respectively:
\begin{equation} 
\dot{x}=f^{\eps}(x,u)=\sum_{i=1}^{k} u_{i} X^{\eps}_{i}(x),\qquad \dot{x}=\wh{f}(x,u)=\sum_{i=1}^{k} u_{i} \wh{X}_{i}(x),\qquad x\in \R^{n} .
\end{equation}
Lemma~\ref{l:convframe} and Proposition~\ref{p:catapulta} imply that, for $\varepsilon \to 0$ 
\begin{gather} 
A^{\eps}(t)\doteq \dfrac{\partial f^{\eps}}{\partial x}(\gamma^{\eps}(t),u^{\eps}(t))\longrightarrow \wh{A}(t)\doteq \dfrac{\partial \wh{f}}{\partial x}(\wh{\gamma}(t),\wh{u}(t)), \\
B^{\eps}(t)\doteq \dfrac{\partial f^{\eps}}{\partial u}(\gamma^{\eps}(t),u^{\eps}(t))\longrightarrow \wh{B}(t)\doteq \dfrac{\partial \wh{f}}{\partial u}(\wh{\gamma}(t),\wh{u}(t)),
\end{gather}
uniformly with all derivatives on $[0,T]$.
Here $u^\eps$ and $\wh u$ are the controls associated with the geodesics $\gamma^{\eps}$ and $\wh\gamma$, respectively. In particular
\begin{equation}
B^{\eps}_{i}(t)\to \wh B_{i}(t),\qquad \all i\geq 1,
\end{equation}
uniformly on $[0,T]$. As a consequence, the maps $(\eps,t)\mapsto B^{\eps}_{i}(t)$ are continuous on $[0,1]\times [0,T]$. Hence, the functions $(\eps,t)\mapsto \dim \DD^{i}_{\g^{\eps}}(t)$ are lower semicontinuous on the compact set $[0,1]\times [0,T]$ by semicontinuity of the rank of a continuous family on matrices. This implies the statement.
\end{proof}
In the next lemma, we denote by $\delta_{\eps}$ the dilation with parameter $\eps$ defined by \eqref{eq:dilations0}.
\begin{lemma} \label{l:situa2} Fix $\eps>0$ and let $\gamma$ be a normal geodesic for the $\eps$-approximating system. Then the curve $\eta\doteq \delta_{\eps}(\gamma)$ is a normal geodesic for the original system with the same growth vector of $\g$.
\end{lemma}
Lemma~\ref{l:situa2} is a direct consequence of the invariance of the growth vector by the change of coordinates given by $\delta_\varepsilon$. For the reader's convenience we give a detailed proof in Appendix~\ref{a:proofsitua2}.

Let us now apply Lemma \ref{l:situa2} to the family $\gamma^{\eps}$ of geodesics converging to $\wh{\g}$ in the nilpotent approximation. In other words we define the family of curves $\eta_{\eps}\doteq\delta_{\eps}(\gamma^{\eps})$. By Lemma \ref{l:situa2} $\eta_{\eps}$ is a geodesic of the original system with the same growth vector of $\gamma^{\eps}$. Then, by Lemma \ref{l:situa1} we get, for every  $t$
 \begin{equation} 
 \dim \DD^{i}_{\hat{\gamma}}(t)\leq \liminf_{\eps\to 0} \dim \DD_{\gamma^{\eps}}^{i}(t)=\liminf_{\eps\to 0} \dim \DD_{\eta_{\eps}}^{i}(t).
\end{equation}
In particular, there exists $\bar\eps=\bar \eps(t)$ such that 
 \begin{equation} 
 \dim \DD^{i}_{\hat{\gamma}}(t)\leq \dim \DD_{\gamma^{\eps}}^{i}(t)=\dim \DD_{\eta_{\eps}}^{i}(t),\qquad \all \eps\leq\bar \eps.
\end{equation}
Actually, since the map $(\eps,t)\mapsto \dim \DD^{i}_{\g^{\eps}}(t)$ is lower semicontinuous on $[0,1]\times [0,T]$,  $\bar \eps$ can be chosen independent on $t$ (see the proof of Lemma~\ref{l:situa1}).

If we choose, by Lemma \ref{l:situa3}, the geodesic $\wh{\g}$ to be ample at every $t$, it follows that, for $\eps\leq\bar\eps$, the curve $\eta_{\eps}$ is a geodesic for the original sub-Riemannian structure, ample at every $t$ . 
\end{proof}

\subsection{The maximal geodesic growth vector}\label{s:maximalgrowth} \index{growth vector of an admissible curve!maximal} \index{maximal growth vector}

In what follows we are interested in the behaviour of a strongly normal geodesic for small $t$. For this reason we focus on the growth vector at $t=0$. Let us define the maximal geodesic growth vector.

\begin{definition}
Let $x_0 \in M$. The \emph{maximal geodesic growth vector} at $x_0$ is
\begin{equation}
\mathcal{G}_{x_0}\doteq \{k_1(x_0),k_2(x_0),\ldots\}, \qquad k_i(x_0)\doteq \max_\gamma \dim \DD^i_{\gamma}(0), \qquad \all i \geq 0,
\end{equation}
where the maximum is taken over all the geodesics $\gamma$ such that $\gamma(0) = x_0$.
\end{definition}

Indeed $\mathcal{G}_{x_0}$ depends only on the germ of the sub-Riemannian structure at $x_0$. In the proof of Theorem~\ref{t:bound}, we proved more than the simple existence of an ample geodesic: the maximal geodesic growth vector of the nilpotent approximation at $x_0$ controls the maximal geodesic growth vector at $x_0$ of the original structure.

\begin{proposition}\label{p:estimate}
Let $\mathcal{G}_{x_0}$ and $\wh{\mathcal{G}}_{x_0}$ be the maximal geodesic growth vectors at $x_0$ for the sub-Riemannian structure and for its nilpotent approximation at $x_0$, respectively. Then
\begin{equation}
 \wh{\mathcal{G}}_{x_0} \leq \mathcal{G}_{x_0},
\end{equation}
where the inequality between the two sequences of integer numbers is meant element-wise.
\end{proposition}
\begin{proof}
In the final part of the proof of Theorem~\ref{t:bound} we proved that, for any fixed geodesic $\hat\gamma$ in the nilpotent approximation, there exists a geodesic $\gamma$, in the original structure, such that
\begin{equation}
\dim\DD_{\hat\gamma}^i(0) \leq \dim\DD_{\gamma}^i(0), \qquad \all i\geq 0.
\end{equation}
Then, the statement follows by the definition of maximal geodesic growth vector.
\end{proof}

The next proposition implies that the the generic normal geodesic for sub-Riemannian structures is ample, and its geodesic growth vector at $t=0$ is equal to the maximal one.

\begin{proposition}\label{p:ampledense}
The set $A_{x_0} \subseteq T_{x_0}^*M$ of initial covectors such that the corresponding geodesic growth vector (at $t=0$) is maximal is an open, non-empty Zariski subset. In particular, for any $\lambda \in A_{x_0}$, the corresponding geodesic $\gamma$ is ample and has maximal growth vector, namely $\mathcal{G}_\gamma(0) = \mathcal{G}_{x_0}$.
\end{proposition}
\begin{proof}
For any $\lambda \in T_{x_0}^*M$ and $i \geq 0$, let us denote by $\DD^i_\lambda\doteq \DD^i_\gamma(0)$ the flag of the normal geodesic $\gamma$ with initial covector $\lambda$.  Moreover, let $k_i(\lambda) = \dim\DD_\lambda^i(0)$. Thus the maximal geodesic growth vector is
\begin{equation}
\mathcal{G}_{x_0} = \{\bar{k}_1,\bar{k}_2,\ldots\}, \qquad \bar{k}_i = \max\{\dim \DD^i_{\lambda}(0)\mid \lambda \in T_{x_0}^*M \}.
\end{equation}
For all $i\geq 0$, let $K_i \subset T_{x_0}^*M$ be the set of covectors $\lambda$ where $k_i(\lambda)$ is not maximal, namely
\begin{equation}
K_i =\{\lambda \in T_{x_0}^*M \mid k_i(\lambda) <  \bar{k}_i\}.
\end{equation}
By Remark~\ref{r:rational}, the integers $k_i$ are computed as the rank of matrices whose entries are rational in the covector $\lambda$. Thus $K_i$ is a closed Zariski subset of $T_{x_0}^*M$ (that has zero measure). Let $K_i^c = T_{x_0}^*M \setminus K_i$ the complement of $K_i$. Notice that each one of the $K_i^c$ is non-empty. Then consider the set
\begin{equation}
A_{x_0} \doteq \bigcap_{i\geq 0} K_i^c.
\end{equation}
By Theorem~\ref{t:bound}, there always exists at least one geodesic ample at $t=0$. Let $m$ be the geodesic step of such a geodesic. This means that, for any $\lambda \in K_m^c$ we have $k_m(\lambda)= \dim M$. Thus, by definition of growth vector, for all $\lambda \in K_m^c$ also $k_{m+i}(\lambda) = \dim M$ for all $i \geq 0$. Since this is indeed the maximal possible value for the $k_i$, this means that $K_m^c \subseteq K_{m+i}^c$ for all $i\geq 0$. Thus
\begin{equation}
A_{x_0} =  K_1^c\cap \ldots \cap K_m^c.
\end{equation}
It follows that $A_{x_0}$ is Zariski open, non-empty and $\mathcal{G}_{\gamma}(0) = \mathcal{G}_{x_0}$ for every normal geodesic with initial covector $\lambda \in A_{x_0}$.
\end{proof}

\section{Reparametrization and homogeneity of the curvature operator}\label{s:homog}
We already explained that a geodesic is not ample on a proper Zariski closed subset of the fibre. This set includes covectors associated to abnormal geodesics, since $\distr_x^\perp \subset T_x^*M \setminus \A_x$. On the other hand, for $\lam \in \A_x$, the curvature $\RR_\lam$ is well defined.
Observe that $\A_x$ is invariant by rescaling, i.e. if $\lambda \in \A_x$, then for $\alpha\neq 0$, also $\alpha\lambda \in \A_x$. Therefore, we have the following:
\begin{proposition}
The operators $\Qz_{\lam}$ and $\RR_{\lam}$ are homogeneous of degree $0$ and $2$ with respect to $\lam$, respectively. Namely, for $\lam \in \A_x$ and $\alpha >0$
\begin{equation}\label{eq:homog}
\Qz_{\al\lam}=\Qz_{\lam},\qquad \RR_{\al\lam}=\alpha^2\RR_\lambda.
\end{equation}
\end{proposition}

\begin{proof} 
Let $c_{t}^{\lam}$ be the geodesic cost  associated with the covector $\lam\in T^{*}_{x}M$. By homogeneity of the sub-Riemannian Hamiltonian, for $\al > 0$ we have
\begin{equation}
c_t^{\alpha\lambda} = \alpha \,c_{\alpha t}^\lambda.
\end{equation}
In particular, this implies 
$d^{2}_{x}\dot{c}_t^{\alpha\lambda} = \alpha^2 d^{2}_{x}\dot{c}_{\alpha t}^\lambda$.
The same relation is true for the restrictions to the distribution $\distr_{x}$, therefore $\QQ_{\al\lam}(t)=\al^{2}\QQ_{\lam}(\al t)$ as symmetric operators on $\distr_{x}$.
Applying Theorem  \ref{t:main} to both families one obtains
\begin{equation} 
\frac{1}{ t^2}\Qz_{\al\lam} + \frac{1}{3}\RR_{\al\lambda} + O(t)=\al^{2}\left(\frac{1}{\al^{2} t^2}\Qz_{\lam} + \frac{1}{3}\RR_\lambda + O(\al t)\right),
\end{equation}
which, in particular, implies Eq.~\eqref{eq:homog}.
\end{proof}
Notice that the same proof applies also to a general affine optimal control system, such that the  Hamiltonian (or, equivalently, the Lagrangian) is homogeneous of degree two. 

\section{Asymptotics of the sub-Laplacian of the geodesic cost} \label{s:asymptotic}

In this section we discuss the asymptotic behaviour of the sub-Laplacian of the sub-Riemannian geodesic cost.
On a Riemannian manifold, the Laplace-Beltrami operator is defined as the divergence of the gradient. This definition can be easily generalized to the sub-Riemannian setting. We will denote by $\metr{\cdot}{\cdot}$ the inner product defined on the distribution.

\bdeff
Let $f\in C^{\infty}(M)$. The \emph{horizontal gradient} \index{horizontal gradient} \index{sub-Riemannian!gradient} of $f$ is the unique horizontal vector field $\grad f$ such that
\begin{equation}
\metr{\grad f}{X} = X(f), \qquad \all X \in \overline{\distr}.
\end{equation}
\edeff
For $x \in M$,  the restriction of the sub-Riemannian Hamiltonian to the fiber  $H_x : T_x^*M \to \R$ is a quadratic form. Then, as a consequence of the formula $\la d_\lam H_x | X \ra = \la \lam, X\ra$, we obtain
\begin{equation}\label{eq:grad}
\grad f =\sum_{i=1}^k X_i(f) X_i.
\end{equation}
We want to stress that Eq.~\eqref{eq:grad} is true in full generality, also when $\dim \distr_x$ is not constant or the vectors $X_1,\ldots,X_k$ are not independent.

\bdeff
Let $\mu \in \Omega^n(M)$ be a volume form, and $X \in \VecM$. The \emph{$\mu$-divergence} \index{divergence} of $X$ is the smooth function $\dive_\mu (X)$ defined by 
\begin{equation}
\mathcal{L}_X \mu \doteq \dive_\mu (X)\mu ,
\end{equation}
where, we recall, $\mathcal{L}_X$ is the Lie derivative in the direction of $X$.
\edeff
Notice that the definition of divergence does not depend on the orientation of $M$, namely the sign of $\mu$. The divergence measures the rate at which the volume of a region changes under the integral flow of a field. Indeed, for any compact $\Omega\subset M$ and $t$ sufficiently small, let $e^{tX} : \Omega \to M$ be the flow of $X \in \VecM$, then
\begin{equation}
\left.\frac{d}{dt}\right|_{t=0} \int_{e^{tX}(\Omega)} \mu = - \int_\Omega \dive_\mu (X) \mu .
\end{equation}
The next proposition is an easy consequence of the definition of $\mu$-divergence and is sometimes employed as an alternative definition of the latter.
\begin{proposition}\label{t:divtheo}
Let $C^\infty_0(M)$ be the space of smooth functions with compact support. For any $f\in C_0^\infty(M)$ and $X\in \VecM$
\begin{equation}
\int_M f \dive_\mu (X) \mu = - \int_M X(f) \mu.
\end{equation}
\end{proposition}

With a divergence and a gradient at our disposal, we are ready to define the sub-Laplacian associated with the volume form $\mu$.
\bdeff
Let $\mu \in \Omega^n(M)$, $f \in C^{\infty}(M)$. The \emph{sub-Laplacian} associated with $\mu$ is the second order differential operator
\begin{equation}
\lapl_{\mu} f \doteq  \dive_{\mu}\lp \grad f\rp,
\end{equation}
\edeff
On a Riemannian manifold, when $\mu$ is the Riemannian volume, this definition reduces to the Laplace-Beltrami operator. As a consequence of Eq.~\eqref{eq:grad} and the Leibniz rule for the divergence $\dive_{\mu}(fX)=X(f)+f\,\dive_{\mu}(X)$, we can write the sub-Laplacian in terms of the fields $X_1,\dots,X_k$:
\begin{equation}
\dive_{\mu} \lp \grad f \rp = \sum_{i=1}^k \dive_{\mu}\lp X_i(f)X_i\rp =\sum_{i=1}^k X_i(X_i(f)) + \dive_{\mu} (X_i) X_i(f).
\end{equation}
Then \index{sub-Riemannian!Laplacian} \index{sub-Laplacian}
\begin{equation}\label{eq:sublaplframe}
\lapl_{\mu} = \sum_{i=1}^k X_i^2 + \dive_{\mu} (X_i) X_i.
\end{equation}

\begin{remark}
If we apply Proposition~\ref{t:divtheo} to the horizontal gradient $\grad g$, we obtain
\begin{equation}
\int_M f \lapl g \mu= -\int_M \metr{\grad f}{\grad g} \mu, \qquad \all f,g \in C^\infty_0(M).
\end{equation}
Then $\lapl_{\mu}$ is symmetric and negative on $C^\infty_0(M)$. It can be proved that it is also essentially self-adjoint (see \cite{strichartz}). Hence it admits a unique self-adjoint extension to $L^{2}(M,\mu)$.
\end{remark}

Observe that the principal symbol of $\lapl_\mu$, which is a function on $T^*M$, does not depend on the choice of $\mu$, and is proportional to the sub-Riemannian Hamiltonian, namely $2H: T^*M \to \R$.
The sub-Laplacian depends on the choice of the volume $\mu$ according to the following lemma.

\begin{lemma}\label{l:divtrans}
Let $\mu, \mu' \in \Omega^n(M)$ be two volume forms such that $\mu' = e^{a} \mu$ for some $a \in C^{\infty}(M)$. Then
\begin{equation}
\lapl_{\mu'}f   = \lapl_{\mu}f+ \la \grad a|\grad f\ra. 
\end{equation}
\end{lemma}
\begin{proof}
It follows from the Leibniz rule $\mc{L}_X(a\mu)=X(a)\mu+a\mc{L}_X\mu=(X(\log a)+\dive_\mu (X)) a\mu$ for every $a\in C^{\infty}(M)$.
\end{proof}

The sub-Laplacian, computed at critical points, does not depend on the choice of the volume.
\begin{lemma}\label{l:laplcrit}
Let $f \in C^{\infty}(M)$, and let $x\in M$ be a critical point of $f$. Then, for any choice of the volume $\mu$,
\begin{equation}
\lapl_\mu f|_x = \sum_{i=1}^k X_i^2 (f)|_x.
\end{equation}
\end{lemma}
\begin{proof}
The proof follows from Eq.~\eqref{eq:sublaplframe}, and the fact that $X_i(f)|_x = 0$.
\end{proof}
From now on, when computing the sub-Laplacian of a function at a critical point, we employ the notation $\lapl_\mu f|_x = \lapl f|_x$, since it does not depend on the volume.
\bl \label{l:d2lapl}
Let $f \in C^{\infty}(M)$, and let $x \in M$ be a critical point of $f$. Then $\lapl f|_{x} =\trace d^{2}_{x}f|_{\distr_{x}}$.
\el
\begin{proof} Recall that if $x$ is a critical point of $f$, then the second differential $d^{2}_{x}f$ is the quadratic form associated with the symmetric bilinear form
\begin{equation} 
d^{2}_{x}f:T_{x}M\times T_{x}M\to \R, \qquad (X,Y)\mapsto X(Y(f))|_{x}. 
\end{equation}
The restriction of $d^2_x f$ to the distribution can be associated, via the inner product, with a symmetric operator defined on $\distr_{x}$, whose trace is computed in terms of $X_{1},\ldots,X_{k}$ as follows
\begin{equation}\label{eq:hessian} 
\trace d^{2}_{x}f|_{\distr_x}=\sum_{i=1}^{k}X_{i}^{2}(f)|_{x},
\end{equation}
We stress that Eq.~\eqref{eq:hessian} holds true for any set of generators, not necessarily linearly independent, of the sub-Riemannian structure $X_1,\ldots,X_k$ such that $H(\lam) = \frac{1}{2}\sum_{i=1}^k \langle \lam, X_i\rangle^2$.
The statement now is a direct consequence of Lemma~\ref{l:laplcrit}.
\end{proof}

Remember that the derivative of the geodesic cost function $\dot{c}_t$ has a critical point at $x_0 = \gamma(0)$. As a direct consequence of Theorem~\ref{t:main},~\ref{t:main2}, Lemma~\ref{l:d2lapl} and the fact that, in the sub-Riemannian case, the Hamiltonian inner product is the sub-Riemannian one (see Remark~\ref{r:hamsub}), we get the following asymptotic expansion:
\bt
Let $c_t$ be the geodesic cost associated with a geodesic $\gamma$ such that $\gamma(0) = x_0$. Then
\begin{equation}
\lapl \dot c_{t}|_{x_0} =  \frac{\trace \Qz_\lam}{t^{2}}+\frac{1}{3}\Ric(\lam)+O(t),
\end{equation}
where $\Ric(\lam)= \trace \RR_\lam$.
\et

The next result is an explicit expression for the asymptotic of the sub-Laplacian of the geodesic cost computed at the initial point $x_0$ of the geodesic $\gamma$. In the sub-Riemannian case, the geodesic cost is essentially the squared distance from the geodesic, i.e. the function
\begin{equation}
\f_t(\cdot)\doteq  - t c_t(\cdot) = \frac{1}{2}\dist^2(\,\cdot\,,\gamma(t)), \qquad t \in (0,1].
\end{equation}
For this reason, we may state the theorem equivalently in terms of $\f_t$ or the geodesic cost $c_t$. Remember also that, since $x_0$ is not a critical point of $\f_t$, its sub-Laplacian depends on the choice of the volume form $\mu$.

\begin{maintheorem}\label{t:main3} 
Let $\gamma:[0,T]\to M$ be an equiregular geodesic with initial covector $\lam \in T_{x_0}^*M$. Assume also that $\dim \distr$ is constant in a neighbourhood of $x_0$. Then there exists a smooth $n$-form $\omega$, defined along $\gamma$, such that for any volume form $\mu$ on $M$, we have:
\begin{equation}\label{eq:main3}
\lapl_\mu \f_t|_{x_0} = \trace \Qz_\lam - \dot{g}(0) t -  \frac{1}{3} \Ric(\lam)t^2 + O(t^3),	
\end{equation}
where $g:[0,T] \to M$ is a smooth function defined implicitly by $\mu_{\gamma(t)} = e^{g(t)}\omega_{\gamma(t)}$.
\end{maintheorem}

We stress that, in the statement of Theorem~\ref{t:main3}, $\mu$ is the fixed volume form used to define the Laplace operator $\lapl_\mu$, while the $n$-form $\omega$ depends on the choice of the geodesic $\gamma$. As we will see, $\omega$ is obtained by taking the wedge product of a Darboux frame in the cotangent bundle $T^*M$ that is related with a generalization of the parallel transport along the geodesic (see Chapter~\ref{c:sublapproof}). 

On a Riemannian manifold it turns out that $\omega$ is the restriction to $\gamma$ of the Riemannian volume form (up to a sign). Thus, if one chooses $\mu$ as the standard Riemannian volume, $\omega$ coincides with $\mu$ and $g(t) \equiv 0$ for any geodesic. Therefore the first order term in Eq.~\eqref{eq:main3} vanishes. 

This is \emph{not} true, in general, for sub-Riemannian manifolds, where $\omega$ is not the restriction to $\gamma$ of a global volume form (such as, e.g., the Popp's volume defined in Section~\ref{s:poppvolume}).

\begin{remark}\label{r:laplacianexpansion}
As a consequence of Theorem~\ref{t:main3}, for any choice of the volume form $\mu$, we have:
\begin{equation}
\trace \Qz_\lam = \lim_{t\to 0} \lapl_\mu \f_{t}\big|_{x_0}, \qquad
\Ric(\lam) = -\frac{3}{2}\left.\frac{d^{2}}{dt^{2}}\right|_{t=0}\lapl_\mu \f_{t}\big|_{x_0}.
\end{equation}
In particular the zeroth and second order term in $t$ of Eq.~\eqref{eq:main3} do not depend on the choice of $\mu$. On the other hand, the first order term \emph{does} depend on the choice of the volume. Indeed one con prove, using Lemma~\ref{l:divtrans}, that this is actually the \emph{only} term depending on the choice of $\mu$ in the whole expansion. 
\end{remark}

The proof Theorem~\ref{t:main3} is postponed to Chapter~\ref{c:sublapproof}.

\section{Equiregular distributions}\label{s:slowgrowth}

In this section we focus on equiregular sub-Riemannian structures, endowed with a smooth, intrinsic volume form, called Popp's volume. Then we introduce a special class of equiregular distributions, that we call \emph{slow growth}. In this case, we define a family of smooth operators in terms of which the asymptotic expansion of Theorem~\ref{t:main3} (and in particular its linear term) can be expressed explicitly.

Recall that a bracket generating sub-Riemannian manifold $M$ is \emph{equiregular} \index{distribution!equiregular} \index{sub-Riemannian!equiregular structure} if $\dim \distr^i_x$ does not depend on $x\in M$, for every $i\geq0$, where $\distr_x^0\subset\distr_x^{1}\subset\distr_x^{2}\subset\ldots\subset T_xM$ is the flag  of the distribution at a point $x\in M$ (see Chapter~\ref{c:srg}).

\subsection{Popp's volume}\label{s:poppvolume}

In this section we provide the definition of Popp's volume \index{Popp's volume} for an equiregular sub-Riemannian structure. Our presentation follows closely the one of \cite{montgomerybook,nostropopp}. The definition rests on the following lemmas, whose proof is not repeated here.

\begin{lemma} \label{l:mont1} 
Let $E$ be an inner product space, and let $\pi:E\to V$ be a surjective linear map. Then $\pi$ induces an inner product on $V$ such that the norm of $v \in V$ is
\begin{equation} \label{eq:final}
\|v\|_V = \min\{ \|e\|_E \text{ s.t. } \pi(e) = v \}.
\end{equation}
\end{lemma}

\begin{lemma} \label{l:mont2} 
Let $E$ be a vector space of dimension $n$ with a flag of linear subspaces $\{0\} = F^0 \subset F^1\subset F^2 \subset \ldots\subset F^m = E$. Let $\tx{gr}(F) \doteq F^1\oplus F^2/F^1\oplus \ldots \oplus F^m/F^{m-1}$ be the associated graded vector space. Then there is a canonical isomorphism $\theta: \wedge^n E \to \wedge^n \tx{gr}(F)$. 
\end{lemma}

The idea behind Popp's volume is to define an inner product on each $\distr^i_x/\distr^{i-1}_x$ which, in turn, induces an inner product on the orthogonal direct sum 
\begin{equation}
\text{gr}_x(\distr) = \distr_x \oplus \distr_x^2/\distr_x \oplus \ldots \oplus \distr_x^\m/\distr_x^{\m-1}.
\end{equation}
The latter has a natural volume form, which is the canonical volume of an inner product space obtained by wedging the elements an orthonormal dual basis. Then, we employ Lemma~\ref{l:mont2} to define an element of $(\wedge^n T_x M)^*\simeq \wedge^n T_x^* M$, which is Popp's volume form computed at $x$.

Fix $x \in M$. Then, let $v,w \in \distr_x$, and let $V,W$ be any horizontal extensions of $v,w$. Namely, $V,W \in \overline\distr$ and $V(x) =v$, $W(x) = w$. The linear map $\pi : \distr_x\otimes\distr_x \to \distr_x^2/\distr_x$
\begin{equation} \label{eq:ad1}
\pi(v\otimes w)\doteq  [V,W]_x \mod \distr_x,
\end{equation} 
is well defined, and does not depend on the choice the horizontal extensions. 
Similarly, let $1\leq i \leq \m$. The linear maps $\pi_i: \otimes^i \distr_x \to \distr_x^i/\distr_x^{i-1}$
\begin{equation}\label{eq:ad2}
\pi_i(v_1\otimes\dots\otimes v_i) = [V_1,[V_2,\dots,[V_{i-1},V_i]]]_x \mod \distr^{i-1}_x,
\end{equation}
are well defined and do not depend on the choice of the horizontal extensions $V_1,\dots,V_i$ of $v_1,\dots,v_i$.

By the bracket-generating condition, the maps $\pi_i$ are surjective and, by Lemma~\ref{l:mont1}, they induce an inner product space structure on $\distr_x^i/\distr_x^{i-1}$. Therefore, the nilpotentization of the distribution at $x$, namely $\text{gr}_x(\distr)$,
is an inner product space, as the orthogonal direct sum of a finite number of inner product spaces. As such, it is endowed with a canonical volume (defined up to a sign) $\eta_x \in \wedge^n\text{gr}_x(\distr)^*$, which is the volume form obtained by wedging the elements of an orthonormal dual basis.

Finally, Popp's volume (computed at the point $x$) is obtained by transporting the volume of $\text{gr}_x(\distr)$ to $T_x M$ through the map $\theta_x :\wedge^n T_x M \to \wedge^n \text{gr}_x(\distr)$ defined in Lemma~\ref{l:mont2}. Namely 
\begin{equation}\label{eq:popppoint}
\popp_x = \eta_x \circ \theta_x,
\end{equation}
where we employ the canonical identification $(\wedge^n T_x M)^* \simeq \wedge^n T^*_x M$. Eq.~\eqref{eq:popppoint} is defined only in the domain of the chosen local frame. If $M$ is orientable, with a standard argument, these $n$-forms can be glued together to obtain Popp's volume $\popp \in \Omega^n(M)$. Notice that Popp's volume is smooth by construction.

\begin{remark}
From Eq. \eqref{eq:ad1} and \eqref{eq:ad2} it follows that, for any $i\geq0$ and $V\in \distr_{x}$ the linear maps $\ad_x^{i}V:\distr_{x}\to \distr_{x}^{i+1}/\distr^{i}_{x}$ given by
\begin{equation}
\ad^{i}_x V(W)\doteq\underbrace{[V,[V,\dots,[V}_{i \text{ times}},W]]]_x \mod \distr^{i}_x, \qquad  W\in \distr_{x},
\end{equation}
are well-defined.
\end{remark}
\subsection{Slow growth distributions}
Now we are ready to introduce the following class of equiregular distributions.
\begin{definition}
An equiregular distribution is \emph{slow growth} \index{slow growth distribution} \index{distribution!slow growth} at $x\in M$ if there exists a vector $\tanf\in \distr_{x}$ such that the linear map $\ad^{i}_x\tanf$ is surjective for all $i\geq0$.
\end{definition}
This condition is actually generic in $\tanf$, as stated by the following proposition.
\begin{proposition}\label{p:slowgrowthgeneric}
Let $\distr$ be a slow growth distribution at $x$. Then, for $\tanf$ in a non-empty open Zariski subset of $\distr_{x}$, all the linear maps $\ad^{i}_x\tanf$ are surjective. 
\end{proposition}
\begin{proof} 
Let $X_{i}$ be an orthonormal basis for $\distr_x$ and write $\tanf=\sum_{j=1}^{k} \al_{j}X_{j}$, where $k=\dim \distr_{x}$ and the $\alpha_j$ are constant. The definition of slow growth is a maximal rank condition on the operators $\ad_x^{i}\tanf =(\sum_{j=1}^{k}\al_{j}\ad_x X_{j})^{i}$, which is satisfied by at least one element of $\distr_{x}$. Then, the result follows from the fact that $\ad_x^i \tanf$ depends polynomially on the $\al_{j}$.
\end{proof}
We say that a distribution $\distr$ is \emph{slow growth} if it is slow growth at every point $x \in M$. Familiar sub-Riemannian structures such as contact, quasi-contact, fat, Engel, Goursat-Darboux distributions (see~\cite{Extdiffsyst}) are examples of slow growth distributions.

Now, for any fixed equiregular, ample (of step $m$) geodesic $\gamma:[0,T] \to M$, with flag $0=\DD^{0}_{\gamma(t)}\subset \DD^{1}_{\gamma(t)} \subset \ldots \subset \DD_{\gamma(t)}^{m} = T_{\gamma(t)} M$ recall the smooth families of operators 
\begin{equation}
\mc{L}^{i}_\tanf :\DD_{\gamma(t)} \to \DD^{i+1}_{\gamma(t)}/\DD^{i}_{\gamma(t)}, \qquad i=0,\ldots,m-1,
\end{equation}
defined for all $t \in [0,T]$ in terms of an admissible extension $\tanf$ of $\dot\gamma$ (see Remark~\ref{r:family}). If the distribution is slow growth, we have the identities $\mc{L}^{i}_\tanf=\ad^{i}_{\gamma(t)}\tanf$ which, in particular, say that $\mc{L}^{i}_\tanf$ depend only on the value of $\tanf$ at $\gamma(t)$. Moreover, the following growth condition is satisfied
\begin{equation}\label{eq:gc}
\dim \DD^{i}_{\gamma}=\dim \distr^i, \qquad  \all i\geq 0.
\end{equation}
As a consequence of Proposition~\ref{p:slowgrowthgeneric} it follows that, for a non-empty Zariski open set of initial covectors, the corresponding geodesic is ample (of step $m=\m$, the step of the distribution), equiregular and satisfies the growth condition of Eq.~\eqref{eq:gc}.

Next, recall that given $V, W$ inner product spaces, any surjective linear map $L: V \to W$ descends to an isomorphism $L : V / \ker L \to W$. Then, thanks to the inner product structure, we can consider the map $L^* \circ L : V/ \ker L \to V / \ker L$ obtained by composing $L$ with its adjoint $L^*$, which is a symmetric invertible operator. Applying this construction to our setting, we define the smooth families of symmetric operators
\begin{equation}\label{eq:Mt}
M_i(t) \doteq  (\mc{L}_\tanf^{i-1})^* \circ \mc{L}_\tanf^{i-1}: \distr_{\gamma(t)}/ \ker \mc{L}_\tanf^{i-1} \to \distr_{\gamma(t)}/ \ker \mc{L}_\tanf^{i-1}, \qquad i=1,\ldots,m.
\end{equation}

We are now ready to specify Theorem~\ref{t:main3} for any ample, equiregular geodesic satisfying the growth condition of Eq.~\eqref{eq:gc}. First, let us discuss the zeroth order term of the expansion. Recall that the Hausdorff dimension of an equiregular sub-Riemannian manifold is computed by Mitchell's formula (see \cite{mitchell,bellaiche}), namely
\begin{equation}
Q=\sum_{i=1}^{m} i(\dim \distr^{i}-\dim \distr^{i-1}).
\end{equation}
Thus, for a slow growth distribution and a geodesic $\gamma$ with initial covector $\lambda \in T_{x_0}^*M$ satisfying the growth condition of Eq.~\eqref{eq:gc}, we have the following identity (see also Remark \ref{r:gd})
\begin{equation}
\begin{aligned}
\trace \Qz_\lam &=\sum_{i=1}^{m}(2i-1)(\dim \DD^{i}_{\gamma}-\dim \DD^{i-1}_{\gamma}) =\\
&=\sum_{i=1}^{m}(2i-1)(\dim \distr^{i}-\dim \distr^{i-1})=2Q-n.
\end{aligned}
\end{equation}
This formula gives the zeroth order term of the following theorem.
\begin{theorem}\label{t:slowgrowth} Let $M$ be a sub-Riemannian manifold with a slow growth distribution $\distr$.
Let $\gamma$ be an ample, equiregular geodesic with initial covector $\lambda \in T_{x_0}^* M$ satisfying the growth condition~\eqref{eq:gc}. Then
\begin{equation}\label{eq:2Q-n}
\lapl_\mu \f_t|_{x_0} = (2Q-n) - \frac{1}{2}\sum_{i=1}^m \trace\left( M_i(0)^{-1}\dot{M}_{i}(0)\right) t -  \frac{1}{3} \Ric(\lam)t^2 + O(t^3).
\end{equation}
where the smooth families of operators $M_i(t)$ are defined by Eq.~\eqref{eq:Mt}.
\end{theorem}
\begin{remark} Equivalently we can write Eq.~\eqref{eq:2Q-n} in the following form
\begin{equation}
\lapl_\mu \f_t|_{x_0} = (2Q-n) - \frac{1}{2}\left(\left.\frac{d}{ds}\right|_{s=0} \sum_{i=1}^m \log \det M_i(s)\right) t -  \frac{1}{3} \Ric(\lam)t^2 + O(t^3).
\end{equation}
\end{remark}
The proof of Theorem~\ref{t:slowgrowth} is postponed to the end of Chapter~\ref{c:sublapproof}. We end this section with an example. 
\begin{example}[Riemannian structures]
In a Riemannian structure (see Section~\ref{s:riemann}), any non-trivial geodesic has the same flag $\DD_{\gamma(t)} = \distr_{\gamma(t)} = T_{\gamma(t)} M$. In particular, it is a trivial example of slow growth distribution. Notice that Popp's volume reduces to the usual Riemannian volume form. Since every geodesic is ample with step $m=1$, there is only one family of operators associated with $\gamma(t)$, namely the constant operator $M_1(t) = \mathbb{I}|_{T_{\gamma(t)}M}$. Thus, in this case, the linear term of Theorem~\ref{t:slowgrowth} vanishes, and we obtain
\begin{equation}
\lapl \f_t|_{x_0} = n -  \frac{1}{3} \Ric(\lam)t^2 + O(t^3),
\end{equation}
where $\Ric(\lam)$ is the classical Ricci curvature in the direction of the geodesic.
\end{example}

In Section~\ref{s:Heis} we compute explicitly the asymptotic expansion of Theorem~\ref{t:slowgrowth} in the case of the Heisenberg group, endowed with its canonical volume. A more general class of slow growth sub-Riemannian distributions are contact structures, where the operators $M_i(t)$ are not trivial and can be computed explicitly.

\section{Geodesic dimension and sub-Riemannian homotheties}\label{s:gd}

In this section, $M$ is a complete, connected, orientable sub-Riemannian manifold, endowed with a smooth volume form $\mu$. With a slight abuse of notation, we denote by the same symbol the induced measure on $M$.  We are interested in sub-Riemannian homotheties, namely contractions along geodesics.
To this end, let us fix $x_0 \in M$, which will be the center of the homothety. Recall that $\Sigma_{x_0}$ is the set of points $x$ such that 
there exists a unique minimizer $\gamma :[0,1]\to M$ joining $x_0$ with $x$, which is not abnormal and $x$ is not conjugate to $x_0$ along $\gamma$. Recall also that, by Theorem~\ref{t:d2sr}, $\Sigma_{x_0} \subset M$ is the open and dense set where the function $\f=\frac{1}{2}\dist^2(x_0,\cdot)$ is smooth.

\begin{definition}
For any $x \in \Sigma_{x_0}$ and $t \in [0,1]$, the \emph{sub-Riemannian geodesic homothety of center $x_0$ at time $t$} is the map $\phi_t : \Sigma_{x_0} \to M$ that associates $x$ with the point at time $t$ of the unique geodesic connecting $x_0$ with $x$. \index{sub-Riemannian!geodesic homothety}
\end{definition}

As a consequence of Theorem~\ref{t:d2sr} and the smooth dependence on initial data, it is easy to prove that $(t,x) \mapsto \phi_t(x)$ is smooth on $[0,1]\times \Sigma_{x_0}$, and is given by the explicit formula
\begin{equation}\label{eq:homothetymap}
\phi_t(x) = \pi \circ e^{(t-1)\vec{H}}(d_x \f).
\end{equation}

Let now $\Omega \subset \Sigma_{x_0}$ be a bounded, measurable set, with $0<\mu(\Omega)<+\infty$, and let $\Omega_{x_0,t} \doteq  \phi_t(\Omega)$. The map $t\mapsto\mu(\Omega_{x_0,t})$ is smooth on $[0,1]$. As shown in Fig.~\ref{fig:homothety}, the homothety shrinks $\Omega$ to the center $x_0$. Indeed $\Omega_{x_0,0} = \{x_0\}$, and $\mu(\Omega_{x_0,t}) \to 0$ for $t\to 0$. 
For a Riemannian structure, a standard computation in terms of Jacobi fields shows that 
\begin{equation}\label{eq:powerlaw}
\mu(\Omega_{x_0,t}) \sim t^{\dim M},\qquad \text{for } t \to 0,
\end{equation}
where we write $f(t) \sim g(t)$ if there exists $C \neq 0$ such that $f(t) = g(t)(C+o(1))$.
\begin{figure}
\centering
\psscalebox{1 1} 
{
\begin{pspicture}(0,-2)(9.805715,2.1720803)
\definecolor{colour0}{rgb}{0.8,0.8,1.0}
\psbezier[linecolor=black, linewidth=0.03, fillstyle=solid,fillcolor=colour0](4.103176,-0.14610136)(3.906352,-0.60185575)(2.4073117,-0.38586113)(2.3245804,0.09984902)(2.2418492,0.5855592)(2.640162,0.53313017)(3.1615665,0.49336964)(3.6829712,0.45360908)(4.3,0.30965307)(4.103176,-0.14610136)
\psbezier[linecolor=black, linewidth=0.03, fillstyle=solid,fillcolor=colour0](8.9,0.7981442)(8.523747,-0.12837279)(5.6581507,0.31072927)(5.5,1.2981442)(5.3418493,2.2855592)(6.103272,2.1789746)(7.1,2.0981443)(8.096728,2.0173137)(9.276253,1.7246612)(8.9,0.7981442)
\psbezier[linecolor=black, linewidth=0.02, linestyle=dashed, dash=0.17638889cm 0.10583334cm](0.2,-1.7018558)(1.5,-0.5018558)(5.0,1.1981442)(7.2,1.2113022)
\rput[bl](-0.3,-1.8018558){$x_0$}
\rput[bl](7.3485713,1.13243){$x$}
\rput[bl](8.325714,-0.03614149){$\Omega$}
\rput[bl](3.0514287,-0.30757006){$\phi_t(x)$}
\psbezier[linecolor=black, linewidth=0.02, linestyle=dashed, dash=0.17638889cm 0.10583334cm](0.2,-1.7018558)(1.14,-0.60185575)(1.88,0.098144226)(2.46,0.49814424)(3.04,0.89814425)(5.18,1.8981442)(5.9,2.0981443)
\psbezier[linecolor=black, linewidth=0.02, linestyle=dashed, dash=0.17638889cm 0.10583334cm](0.2,-1.7018558)(1.6068493,-1.0018557)(2.6268957,-0.6202739)(3.68,-0.3818558)(4.733104,-0.14343767)(5.96,0.29814422)(7.64,0.27814424)
\psdots[linecolor=black, dotsize=0.08](7.2,1.1981442)
\psdots[linecolor=black, dotsize=0.08](0.2,-1.7018558)
\psdots[linecolor=black, dotsize=0.08](3.22,0.13814422)
\rput[bl](3.78,-0.93899864){$\Omega_{x_0,t} = \phi_t(\Omega)$}
\end{pspicture}
}
\caption{Sub-Riemannian homothety of the set $\Omega$ with center $x_0$.}\label{fig:homothety}
\end{figure}
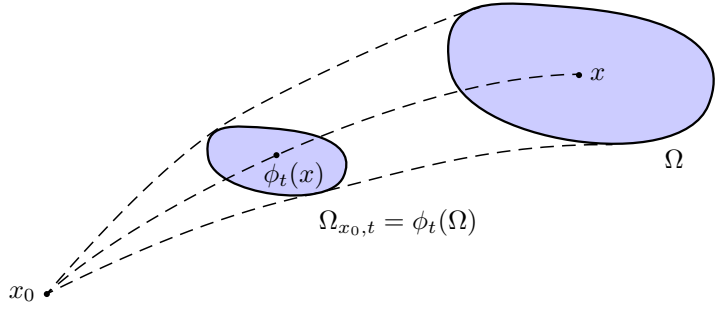

In the sub-Riemannian case, we have a similar power-law behaviour, but the exponent is a different dimensional invariant, which we call \emph{geodesic dimension}. The main result of this section is a formula for the geodesic dimension, in terms of the growth vector of the geodesic.

\begin{definition}
Let $\lam \in T_{x_0}^*M$. Assume that the corresponding geodesic $\gamma:[0,1]\to M$ is ample (at $t=0$) of step $m$, with growth vector $\mc{G}_\g=\{k_1,k_2,\ldots,k_m\}$ (at $t=0$). Then we define
\begin{equation}\label{eq:gdformula}
\gd_\lam \doteq  \sum_{i=1}^m (2i-1)(k_i - k_{i-1}) = \sum_{i=1}^m(2i-1)d_i,
\end{equation}
and $\gd_\lam \doteq  +\infty$ if the geodesic is not ample. 
\end{definition}
Observe that Eq.~\eqref{eq:gdformula} closely resembles the formula for Hausdorff dimension of an equiregular sub-Riemannian manifold (see~\cite{bellaiche,notejean}). In the latter, each direction has a weight according to the flag of the distribution, while in Eq.~\eqref{eq:gdformula}, the weights depend on the flag of the geodesic.
\begin{remark}\label{r:gd=trI}
Assume that $\lam$ is associated with an equiregular geodesic $\gamma$. Then, by Remark~\ref{r:gd} and Eq.~\eqref{eq:gdformula} it follows that
\begin{equation}
\n_{\lam} = \trace \Qz_\lam.
\end{equation}
Moreover, as a consequence of Theorem~\ref{t:main3} (see Remark~\ref{r:laplacianexpansion}), under these assumption $\gd_\lam$ can be recovered from the sub-Laplacian of $\f_t$ by the following formula:
\begin{equation}
\n_\lam = \lim_{t\to 0} \lapl_\mu \f_{t}\big|_{x_0}.
\end{equation}
\end{remark}
Recall that $A_{x_0} \subset T_{x_0}^*M$ is the set of initial covectors such that the corresponding geodesic is ample, with maximal geodesic growth vector (see Section~\ref{s:maximalgrowth}). The next proposition is a direct consequence of Proposition~\ref{p:ampledense}.
\begin{proposition}\label{p:costanza}
The function $\lam \mapsto \gd_\lam$ is constant on the open Zariski set $A_{x_0}\subset T_{x_0}^*M$, assuming its minimum value. \end{proposition}
Proposition~\ref{p:costanza} motivates the next definition.
\begin{definition}
Let $M$ be a sub-Riemannian manifold. The \emph{geodesic dimension} \index{geodesic dimension}  at $x_0 \in M$ is
\begin{equation}
\gd_{x_0}\doteq  \min \{\gd_\lam \mid \lam \in T_{x_0}^*M\} <+\infty.
\end{equation}
\end{definition}
\begin{remark}
As a consequence of Proposition~\ref{p:costanza} we notice that, in order to compute $\gd_{x_0}$ it is sufficient to employ formula~\eqref{eq:gdformula} for the \emph{generic} choice of the covector $\lambda$, namely for $\lambda \in A_{x_0}$.
\end{remark}

For every $x_0 \in M$ we have the inequality $\gd_{x_0} \geq \dim M$ and the equality holds if and only if the structure is Riemannian at $x_{0}$. Notice that, if the distribution is equiregular at $x_0$, it follows from Lemma~\ref{l:upbound} and Mitchell's formula for Hausdorff dimension (see \cite{mitchell}) that $\gd_{x_0} > \dim_{\mathcal{H}}M$.
We summarize these statements in the following proposition.
\begin{proposition}
Let $M$ be an equiregular sub-Riemannian manifold. Let $\dim M$ be its topological dimension and $\dim_\mathcal{H}M$ its Hausdorff dimension. For any point $x_0 \in M$ we have the following inequality:
\begin{equation}
\gd_{x_0} \geq \dim_{\mathcal{H}} M \geq \dim M,
\end{equation}
and the equality holds if and only if the structure is Riemannian at $x_0$.
\end{proposition}
For genuine sub-Riemannian structures then, the geodesic dimension is a new invariant, related with the structure of the distribution along geodesics.

The geodesic dimension is the exponent of the sub-Riemannian analogue of Eq.~\eqref{eq:powerlaw}: namely it represents the critical exponents that describes the contraction of volumes along geodesic homotheties.
\begin{maintheorem}\label{t:main4}
Let $\mu$ be a smooth volume. For any bounded, measurable set $\Omega \subset \Sigma_{x_0}$, with $0<\mu(\Omega)<+\infty$ we have
\begin{equation}
\mu(\Omega_{x_0,t}) \sim t^{\gd_{x_0}},\qquad\text{for } t\to 0.
\end{equation}
\end{maintheorem}
Observe also that homotheties with different center may have different asymptotic exponents. This can happen, for example, in non-equiregular sub-Riemannian structures.

The proof of Proposition~\ref{p:costanza} and Theorem~\ref{t:main4} is postponed to the end of Chapter~\ref{c:jac}.

\begin{example}[Geodesic dimension in contact structures]
Let $(M,\distr,\metr{\cdot}{\cdot})$ be a contact sub-Riemannian structure. In this case, for any $x_0 \in M$, $\dim M = 2\ell+1$ and $\dim \distr_{x_0} = 2\ell$. Any non-trivial geodesic $\g$ is ample with the same growth vector $\mc{G}_{\g} = \{2\ell,2\ell+1\}$. Therefore, by Eq.~\eqref{eq:gdformula}, $\n_{x_0} = 2\ell+3$ (notice that it does not depend on $x_0$). Theorem~\ref{t:main4} is an asymptotic generalization of the results obtained in \cite{Juillet}, where the exponent $2\ell+3$ appears in the context of measure contraction property in the Heisenberg group. For a more recent overview on measure contraction property in Carnot groups, see~\cite{RiffordCarnot}.
\end{example}

\section{Heisenberg group}\label{s:Heis} \index{Heisenberg group}
Before entering into details of the proofs, we repeat the construction introduced in the previous sections for one of the simplest sub-Riemannian structures: the Heisenberg group.  We provide an explicit expression for the geodesic cost function and, applying Definition~\ref{d:curv}, we obtain a formula for the operators $\Qz_\lambda$ and $\RR_\lambda$. 
In particular, we recover by a direct computation the results of Theorems~\ref{t:main}, \ref{t:main2} and \ref{t:main3}.

The Heisenberg group $\mathbb{H}$ is the equiregular sub-Riemannian structure on $\R^3$ defined by the global (orthonormal) frame
\begin{equation}\label{eq:hframe}
X = \partial_x - \frac{y}{2} \partial_z, \qquad Y = \partial_y + \frac{x}{2} \partial_z.
\end{equation}
Notice that the distribution is bracket-generating, for $Z\doteq [X,Y] = \partial_z$. Let us introduce the linear on fibers functions $h_x,h_y,h_z:T^*\R^3 \to \mathbb{R}$
\begin{equation}
h_x \doteq  p_x - \frac{y}{2} p_z ,\qquad h_y \doteq  p_y +\frac{x}{2}p_z,\qquad h_z \doteq  p_z,
\end{equation}
where $(x,y,z,p_x,p_y,p_z)$ are canonical coordinates on $T^*\R^3$ induced by coordinates $(x,y,z)$ on $\R^3$. Notice that $h_x,h_y,h_z$ are the linear on fibers functions associated with the fields $X,Y,Z$, respectively (i.e. $h_x(\lambda) = \langle\lambda,X\rangle$, and analogously for $h_y,h_z$).

The sub-Riemannian Hamiltonian is $H = \tfrac{1}{2}(h_x^2 + h_y^2)$ and the coordinates $(x,y,z,h_x,h_y, h_z)$ define a global chart for $T^*M$.
It is useful to introduce the identification $\mathbb{R}^3 = \mathbb{C}\times \mathbb{R}$, by defining the complex variable $w \doteq  x+iy$ and the complex ``momentum'' $h_w\doteq  h_x + i h_y$. Let $q = (w, z)$ and $q' = (w',z')$ be two points in $\mathbb{H}$. The Heisenberg group law, in complex coordinates, is given by
\begin{equation}\label{eq:grouplaw}
q \cdot q' = \left(w + w', z+z' - \frac{1}{2}\Im\left(w\overline{w'}\right)\right).
\end{equation}
Observe that the frame~\eqref{eq:hframe} is left-invariant for the group action defined by Eq.~\eqref{eq:grouplaw}. Notice also that $h_{z}$ is constant along any geodesic due to the identity $[X,Z]=[Y,Z]=0$.

The geodesic $\g(t)=(w(t),z(t))$ starting from $(w_0,z_0)\in \mb{H}$ and corresponding to the initial covector $(h_{w,0}, h_z)$, with $h_z\neq 0$ is given by
\begin{gather}
w(t) = w_0 + \frac{h_{w,0}}{i h_z}\left(e^{i h_z t}-1\right),\\
z(t) = z_0 + \frac{1}{2}\int_0^t \Im(\overline{w} dw).
\end{gather}
In the following, we assume that the geodesic is parametrized by arc length, i.e. $|h_{w,0}|^2 = 1$. We fix $h_{w,0} = i e^{i \phi}$, i.e. $\phi$ parametrizes the (unit) velocity of the geodesic $\dot{\gamma}(0) = -\sin\phi X + \cos\phi Y$.
Finally, the geodesics corresponding to covectors with $h_z = 0$ are straight lines
\begin{gather}
w(t) = w_0 + h_{w,0} t, \\
z(t) = z_0 + \frac{1}{2}\Im(h_{w,0} \overline{w_0}) t.
\end{gather}
In the following, we employ both real $(x,y,z,h_x, h_y,h_z)$ and complex $(w,z,h_w,h_z)$ coordinates when convenient.

\subsection{Distance in the Heisenberg group} \index{Heisenberg group!distance}

Let $\dist_{0}=\dist(0,\cdot):\mathbb{H}\to \mathbb{R}$ be the sub-Riemannian distance from the origin and introduce cylindrical coordinates $(r,\varphi,z)$ on $\mb{H}$ defined by $x= r \cos\varphi$, $y=r\sin \varphi$. In order to write an explicit formula for $\dist$
recall that 
\begin{itemize}
\item[(i)] $\dist_{0}^2 (r,\varphi,z)$ does not depend on $\varphi$.
\item[(ii)] $\dist_{0}^2(\alpha r,\varphi,\alpha^2 z) =\alpha^2 \dist_{0}^2 (r,\varphi,z)$, where $\alpha > 0$.
\end{itemize}
Then, for $r\neq0$, one has
\begin{equation}\label{eq:d0}
\dist_{0}^2(r,\varphi, z) = r^2 \dist_{0}^2\left(1,0,\frac{z}{r^2}\right).
\end{equation}
It is then sufficient to compute the squared distance of the point $q=(1,0,\xi)$ from the origin.

Consider the minimizing geodesic joining the origin with the point $(1,0,\xi)$. Its projection on the $xy$-plane is an arc of circle with radius $\rho$, connecting the origin with the point $(1,0)$. In what follows we refer to notation of Fig.~\ref{f:heisss}.
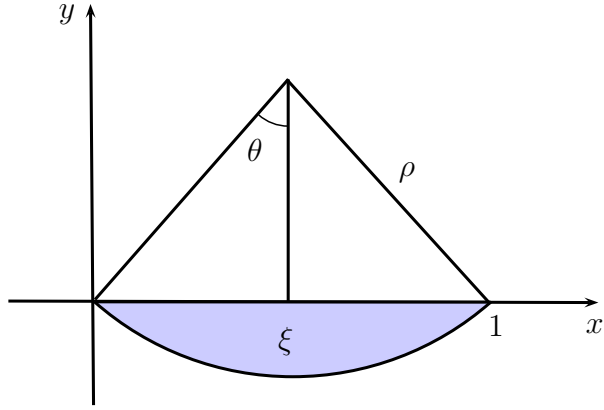
\begin{figure}[!ht]
\centering
\scalebox{1} 
{
\begin{pspicture}(0,-3.878287)(7.96,1.5201114)
\definecolor{color31b}{rgb}{0.8,0.8,1.0}
\psline[linewidth=0.04cm,arrowsize=0.05291667cm 2.0,arrowlength=1.4,arrowinset=0.4]{->}(1.3,-3.858287)(1.26,1.461713)
\psline[linewidth=0.04cm,arrowsize=0.05291667cm 2.0,arrowlength=1.4,arrowinset=0.4]{->}(0.18,-2.478287)(7.94,-2.498287)
\fontsize{12}{0}
\usefont{T1}{ptm}{m}{n}
\rput(7.891455,-2.8053287){$x$}
\usefont{T1}{ptm}{m}{n}
\rput(6.591455,-2.8053287){$1$}
\usefont{T1}{ptm}{m}{n}
\rput(0.95145507,1.326713){$y$}
\psarc[linewidth=0.04,fillstyle=solid,fillcolor=color31b](3.92,0.43828705){3.92}{228.29869}{311.60147}
\psline[linewidth=0.04](1.3122299,-2.4884748)(6.5226655,-2.4930153)
\psline[linewidth=0.04cm](3.86,-2.498287)(3.86,0.46171296)
\psline[linewidth=0.04cm](3.86,0.46171296)(1.32,-2.458287)
\psline[linewidth=0.04cm](3.84,0.46171296)(6.52,-2.518287)
\usefont{T1}{ptm}{m}{n}
\rput(3.831455,-3.013287){$\xi$}
\psarc[linewidth=0.02](3.86,0.42171296){0.58}{228.01279}{270.0}
\usefont{T1}{ptm}{m}{n}
\rput(3.4214551,-0.47328705){$\theta$}
\usefont{T1}{ptm}{m}{n}
\rput(5.4214551,-0.77328705){$\rho$}
\end{pspicture} 
}
\caption{Projection of the geodesic joining the origin with $(1,0,\xi)$ in $\mb{H}$.} \label{f:heisss}
\end{figure}

The highlighted circle segment has area equal to $\xi$. Observe that $\theta \in (-\pi,\pi)$, with $\theta = 0$ corresponding to $\xi = 0$ and $\theta \to \pm \pi$ corresponding to $\xi \to \pm\infty$. Then
\begin{equation}
\xi = \theta \rho^2 - \frac{\rho \cos\theta}{2}.
\end{equation}
Since $2\rho \sin\theta = 1$, we obtain the following equation
\begin{equation}\label{eq:deftheta}
4\xi = \frac{\theta}{\sin^2\theta}-\cot\theta.
\end{equation}
The right hand side of Eq.~\eqref{eq:deftheta} is a smooth and strictly monotone function of $\theta$, for $\theta\in(-\pi,\pi)$. Therefore the function $\theta:\xi\mapsto\theta(\xi)$ is well defined and  smooth. Moreover
$\theta$ is an odd function and, by  Eq.~\eqref{eq:deftheta}, 
it satisfies the following differential equation
\begin{equation}
\frac{d}{d \xi}\left(\frac{\theta^2}{\sin^2\theta}\right) = 4\theta.
\end{equation}
Finally, the squared distance from the origin of the point $(1,0,\xi)$ is the Euclidean squared length of the arc, i.e.
\begin{equation}\label{eq:d1}
\dist_{0}^2(1,0,\xi) 
= \frac{\theta^2(\xi)}{\sin^2\theta(\xi)}.
\end{equation}
Plugging Eq.~\eqref{eq:d1} in Eq.~\eqref{eq:d0}, we obtain the formula for the squared distance:
\begin{equation}\label{d2}
\dist_{0}^2(r,\phi, z) = r^2 \frac{\theta^2(z/r^2)}{\sin^2\theta(z/r^2)}.
\end{equation}

\begin{figure}[!ht]
\begin{center}
\includegraphics[width=12cm]{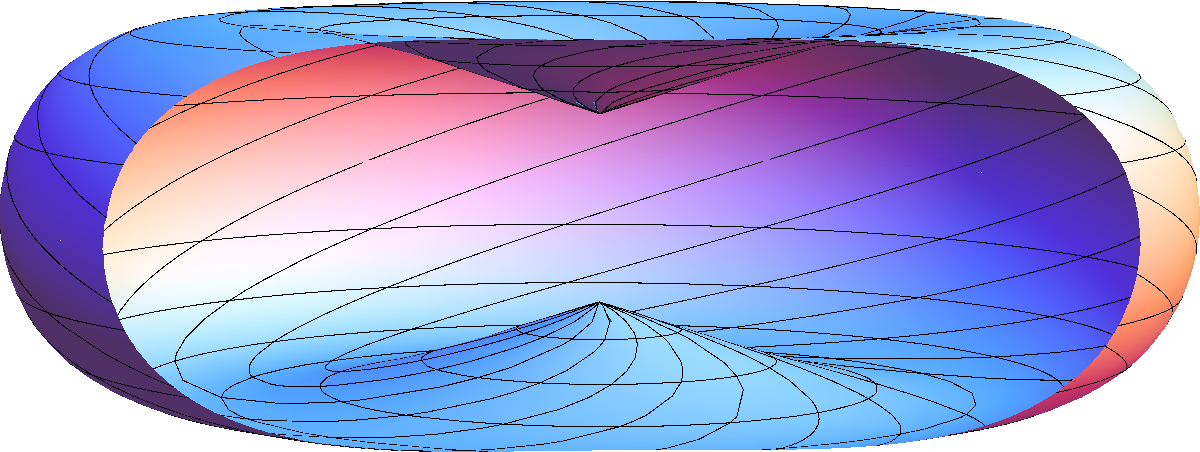}
\caption{A picture of the sub-Riemannian sphere defined by $\dist_0 = 1$.}
\label{f:sfera}
\end{center}
\end{figure} 

\subsection{Asymptotic expansion of the distance}
Next we investigate, for two given geodesics $\gamma_{1},\gamma_{2}$ in $\mb{H}$ starting from the origin and associated with covectors $\lam_{1},\lambda_2 \in T_0^*M$, the regularity of the function
\begin{equation}
C(t,s)\doteq \frac{1}{2}\dist^{2}(\gamma_{1}(t),\gamma_{2}(s)),
\end{equation}
in a neighbourhood of $(t,s)=(0,0)$. By left-invariance, one has
\begin{equation}
C(t,s) = \frac{1}{2}\dist_{0}^2(\gt^{-1}\cdot \gs).
\end{equation}
Let $(W_{t,s}, Z_{t,s})$ be the complex coordinates for the point $\gt^{-1}\cdot \gs\in \mb{H}$. Moreover, let  $R^2_{t,s}\doteq |W_{t,s}|^2$, and $\xi_{t,s}\doteq Z_{t,s}/R^2_{t,s}$. Then, by Eq.~\eqref{d2},
\begin{equation}
C(t,s) = \frac{1}{2} R^2_{t,s} \frac{\theta^2(\xi_{t,s})}{\sin^2\theta(\xi_{t,s})}.
\end{equation}
A long computation, that is sketched in Appendix~\ref{a:proofbrute}, leads to the following result.
\begin{proposition}\label{p:brute}
The function $C(t,s)$ is $C^1$ in a neighbourhood of the origin, but not $C^2$. In particular, the function $\partial_{ss} C(t,0)$ is not continuous at the origin. However, the singularity at $t=0$ is removable, and the following expansion holds, for $t> 0$
\begin{multline}
\frac{\partial^2 C}{\partial s^2}(t,0) = 1 + 3 \sin^2(\phi_2- \phi_1) + \frac{1}{2}[2h_{z,2}\sin(\phi_2-\phi_1)-h_{z,1}\sin(2\phi_2 - 2\phi_1)]t - \\ - \frac{2}{15}h_{z,1}^2\sin^2(\phi_2-\phi_1)t^2+O(t^3).
\end{multline}
If the geodesic $\g_{2}$ is chosen to be a straight line (i.e. $h_{z,2}=0$), then
\begin{equation}\label{eq:dev}
\frac{\partial^2 C}{\partial s^2}(t,0) = 1 + 3 \sin^2(\phi_2- \phi_1) - \frac{h_{z,1}}{2}\sin(2\phi_2 - 2\phi_1)t -\frac{2}{15}h_{z,1}^2\sin^2(\phi_2-\phi_1)t^2+O(t^3),
\end{equation}
where $\lambda_j = (-\sin\phi_j,\cos\phi_j,h_{z,j}) \in T_0^*M$ is the initial covector of the geodesic $\gamma_j$.
\end{proposition}
We stress once again that, for a Riemannian structure, the function $C(t,s)$ (which can be defined in a completely analogous way as the squared distance between two Riemannian geodesics) is smooth at the origin.
	
\subsection{Second differential of the geodesic cost}
We are now ready to compute explicitly the asymptotic expansion of $\QQ_\lam$. Fix $w\in T_{x_0}M$ and let $\alpha(s)$ be any geodesic in $\mb{H}$ such that $\dot{\al}(0)=w$. 
Then we compute the quadratic form $d^{2}_{x_0}\dot c_{t}(w)$ for $t>0$
\begin{equation}
\begin{split}
\metr{\QQ_{\lam}(t) w}{w}&=d^{2}_{x_0}\dot c_{t}(w)=\left.\frac{\partial^{2}}{\partial s^{2}}\right|_{s=0}\frac{\partial}{\partial t} c_{t}(\alpha(s)) =\\
&=\left.\frac{\partial^{2}}{\partial s^{2}}\right|_{s=0}\frac{\partial}{\partial t}\left(-\frac{1}{2t}\dist^{2}(\gamma(t),\al(s))\right) 
=\frac{\partial}{\partial t}\left(-\frac1t \frac{\partial^2 C}{\partial s^2}(t,0)\right) = \\
& = \frac{1}{t^{2}}\left( \lim_{t\to 0^+}\frac{\partial^{2}C}{\partial s^{2}}(t,0)\right)+\frac13 \left(-\frac32 \lim_{t\to 0^+}\frac{\partial^{4}C}{\partial t^{2}{\partial s^{2}}} (t,0)\right)+O(t),
\end{split}
\end{equation}
where, in the second line, we exchanged the order of derivations by smoothness of $C(t,s)$ for $t>0$. It is enough to compute the value of $\QQ_{\lam}(t)$ on an orthonormal basis $v\doteq \dot{\g}(0)$ and $v^{\perp}\doteq \dot{\g}(0)^{\perp}$. By using the results of Proposition~\ref{p:brute}, we obtain
\begin{equation}
\metr{\QQ_{\lam}(t) v}{v}=\frac{1}{t^{2}}+O(t),\qquad \metr{\QQ_{\lam}(t)v^{\perp}}{v^{\perp}}= \frac{4}{t^{2}}+\frac{2}{15}h_{z}^{2}+O(t).
\end{equation}
By polarization we obtain $\metr{\QQ_{\lam}(t) v}{v^{\perp}}=O(t)$. Thus the matrices representing the symmetric operators $\Qz_{\lam}$ and $\RR_{\lam}$ in the basis $\{v^\perp,v\}$ of $\distr_{x_0}$ are
\begin{equation}\label{eq:qzandcurv}
\Qz_{\lam}=
\begin{pmatrix}
4&0\\0&1
\end{pmatrix},
\qquad
\RR_{\lam}=\frac{2}{5}
\begin{pmatrix}
h_{z}^{2}&0\\0&0
\end{pmatrix},
\end{equation}
where, we recall, $\lam$ has coordinates $(h_x,h_y,h_z)$.

Another way to obtain Eq.~\eqref{eq:qzandcurv} is to exploit the connection between the curvature operator and the invariants of the Jacobi curves obtained in the proof of Theorem~\ref{t:main2} (see Eqs.~\eqref{eq:Qlam}--\eqref{eq:Rlam}), in terms of a canonical frame. The latter is not easy to compute, even though, in principle, an algorithmic construction is possible.

\subsection{Sub-Laplacian of the geodesic cost}
By using the results of Proposition~\ref{p:brute}, we explicitly compute the asymptotics of the sub-Laplacian $\lapl_\mu$ of the function $\f_{t} = \frac{1}{2}\dist^2(\cdot,\g(t))$ at $x_{0}$, at the second order in $t$.  In the Heisenberg group, we fix $\mu = dx \wedge dy \wedge dz$ (i.e. the Popp's volume of $\mathbb{H}$), and we suppress the explicit dependence of $\lapl_\mu$ from the volume form.

Since the sub-Riemannian structure of the Heisenberg group is left-invariant, we can reduce the computation of the asymptotic of $\lapl \f_t$ to the case of a geodesic $\gamma$ starting from the origin. Indeed, let us denote by $L_{g}:\mb{H}\to \mb{H}$ the left multiplication by $g\in \mb{H}$. It is easy to show that if $\gamma(t)=\EXP_{x_{0}}(t,\lam)$ is a geodesic, then $\til{\gamma}(t)\doteq L_{g}(\gamma(t))$ is a geodesic too. If $\f_{t}$ and $\til{\f}_{t}$ denote the squared distance along the geodesics $\gamma$ and $\til{\gamma}$, respectively, we have
\begin{equation}
\til{\f}_{t}(L_{g}(x))=\frac{1}{2}\dist^{2}(L_{g}(x),\til{\gamma}(t))=\frac{1}{2}\dist^{2}(L_{g}(x),L_{g}(\gamma(t)))=\frac{1}{2}\dist^{2}(x,\gamma(t))=\f_{t}(x).
\end{equation}
Moreover, by using Proposition~\ref{p:critlam0}, and recalling the relation $c_t = -t \f_t$, it is easy to show that
\begin{equation}
\til{\gamma}(t)=\EXP_{y_{0}}(t,\eta),\qquad \text{where} \qquad y_{0}=L_{g}(x_{0}),\quad \eta=(L_{g}^*)^{-1}\lam \in T_{y_0}^*M.
\end{equation}
Moreover $\lapl$ is left-invariant hence $\lapl (f\circ L_{g})=\lapl f \circ L_{g}$ for every $f\in C^{\infty}(M)$, and we have
\begin{equation}
\lapl \til{\f}_{t}|_{y_{0}}=\lapl \f_{t}|_{x_{0}}.
\end{equation}
In terms of an orthonormal frame, the sub-Laplacian is $\lapl=X^{2}+Y^{2}$ hence
\begin{equation}\label{eq:finH}
\lapl \f_{t}|_{x_0}=\left.\frac{d^{2}}{ds^{2}}\right|_{s=0}\f_{t}(e^{sX}(x_0)) +\left.\frac{d^{2}}{ds^{2}}\right|_{s=0}\f_{t}( e^{sY}(x_0)),
\end{equation}
where $e^{sX}(x_0)$ denote the integral curve of the vector field $X$ starting from $x_0$ (and similarly for $Y$). Observe that the integral curves of the vector fields $X$ and $Y$, starting from the origin, are two orthogonal straight lines contained in the $xy$-plane. Thus we can compute Eq.~\eqref{eq:finH} (where $x_0 = 0$) by summing two copies of Eq.~\eqref{eq:dev} for $\phi_{2}=-\pi/2$ and $\phi_{2}=0$ respectively. By left-invariance we immediately find, for any $x_0 \in \mathbb{H}$
\begin{equation}
\lapl\f_t|_{x_0} = 5 - \frac{2}{15}h_{z}^2 t^2 + O(t^3),
\end{equation}
where, we recall, the initial covector associated with the geodesic $\gamma$ is $\lambda = (h_x,h_y,h_z) \in T_{x_0}^*M$.

\medskip

Another interesting class of examples, of which Heisenberg is the simplest model, are three dimensional contact sub-Riemannian structures.
Clearly, the direct computation of the curvature, analogue to the one carried out for the Heisenberg group, is extremely difficult when there is no general explicit formula for the distance function. Nevertheless, one can still compute it in these cases using the techniques introduced in Chapters~\ref{c:jac} and~\ref{c:proof}.
For this reason, the complete discussion for 3D contact structures is postponed to Section~\ref{s:3Dcomputations}. Explicit computations of higher-dimensional contact sub-Riemannian curvature can be found in \cite{ABR-Contact}.

\section{On the ``meaning'' of constant curvature}\label{s:3Dcontact}
In Riemannian geometry the vanishing of curvature has a basic significance: the metric is locally Euclidean. One can wonder whether a similar interpretation exists in our setting, where one should also take into account the presence of the non-trivial operator $\Qz_{\lam}$. 

For Riemannian structures we proved the formulae
\begin{equation}\label{eq:strascico}
\Qz_\lambda  = \mathbb{I}, \qquad  \RR_{\lam} = R^\nabla(\dot\g,\cdot)\dot\g.
\end{equation}
where $\lambda$ is the initial covector of a geodesic $\g$.
The classical meaning of ``constant curvature'' is  $\RR_{\lam}(w) =kw$ for some $k\in \R$ and every $w\perp \dot \gamma$. In other words $\Qz_\lambda$ and $\RR_{\lam}$ are always constant as a function of $\lam$. 

What about the Heisenberg group? We have proved that the matrices representing the symmetric operators $\Qz_{\lam}$ and $\RR_{\lam}$ in the basis $\{\dot \g^{\perp},\dot\g\}$ of $\distr_{x_0}$ are
\begin{equation} 
\Qz_{\lam}=
\begin{pmatrix}
4&0\\0&1
\end{pmatrix},
\qquad
\RR_{\lam}=\frac{2}{5}
\begin{pmatrix}
h_{z}^{2}&0\\0&0
\end{pmatrix},
\end{equation}
where, we recall, $\lam$ has coordinates $(h_x,h_y,h_{z})$. In particular $\Qz_{\lam}$ is the same for any non-trivial geodesic, but $\RR_{\lam}$ is an operator that depends on $\lam$. For example $\RR_{\lam}=0$ for those $\lam$ corresponding to straight lines (i.e. when $h_{z}=0$), but its norm is unbounded with respect to $\lam$.

This situation carries on to more general settings. In fact, in Section \ref{s:3Dcomputations} we prove the following formula for 3D contact sub-Riemannian structures: 
\begin{equation} 
\Qz_{\lam}=
\begin{pmatrix}
4&0\\0&1
\end{pmatrix},
\qquad
\RR_{\lam}=\frac{2}{5}
\begin{pmatrix}
r_{\lam}&0\\0&0
\end{pmatrix},
\end{equation}
Observe that $r_{\lam}$ is proportional to the Ricci curvature associated with $\RR_{\lam}$:
\begin{equation}
\Ric(\lambda) = \trace\RR_\lambda = \frac{2}{5}r_\lambda.
\end{equation}
We are not interested in an explicit formula for $r_\lambda$ right now (one can find it in Section~\ref{s:3Dcomputations}); we only anticipate that $r_\lambda$, a priori defined only for covectors associated with ample geodesics, can be extended to a well defined quadratic form $\lambda\mapsto r_{\lambda}$ on the whole fiber $T^{*}_{x}M$, where $x=\pi(\lam)$.

It turns out that the quadratic form $r_{\lam}$ is positive when evaluated on the kernel of the Hamiltonian $\ker H_{x}$. In particular, this defines a splitting of the fiber
\begin{equation}
T^{*}_{x}M=\ker H_{x} \oplus (\ker H_{x})^{\perp},
\end{equation}
where $(\ker H_{x})^{\perp}$ is the orthogonal complement of $\ker H_x$ with respect to the quadratic form $r_{\lam}$. Notice that $\ker H_{x}$ is a one-dimensional subspace and we can define a normalized basis $\alpha_{x}$ of it by requiring that $r_{\al_{x}}=1$.
As a matter of fact, this splitting induces the dual splitting of $T_{x}M$
\begin{equation}
T_{x}M=\mathscr{V}_{x} \oplus \distr_{x},
\end{equation}
where $\distr_{x}$ is the distribution of the sub-Riemannian structure at the point $x$ and $\mathscr{V}_{x}$ is a one-dimensional subspace of $T_{x}M$ that is transversal to $\distr_{x}$. This splitting is smooth with respect to $x$. The vector $X_{0} \in \mathscr{V}_{x}$ normalized such that $\alpha_x(X_0) = 1$, for every $x\in M$, is called the \emph{Reeb vector field}. Indeed $\alpha$ is the normalized \emph{contact form}.

Let us now consider the restriction  $r_{\lam}|_{\distr_{x}^{*}}$ of the quadratic form $r_{\lam}$ on the two dimensional Euclidean plane $\distr_{x}^{*}\doteq (\ker H_{x})^{\perp}$, endowed with the dual inner product induced by the Hamiltonian $H_x$. By construction its trace and its discriminant are two metric invariant of the structure
\begin{equation}
\trace\left(r_{\lam}|_{\distr_{x}^{*}}\right), \qquad \discr\left(r_{\lam}|_{\distr_{x}^{*}}\right).
\end{equation}
Recall that the discriminant of an operator $Q$ defined on a two-dimensional space, is the square of the difference of its eigenvalues, and is computed by the formula $\discr(Q)=\trace^{2}(Q)-4\det(Q)$.

One can prove that the Reeb vector field $X_{0}$ generates a flow of isometries for the sub-Riemannian metric (i.e. it preserves $H$) if and only if $\discr\left(r_{\lam}|_{\distr_{x}^{*}}\right)=0$ for all $x\in M$. 

Under this assumption one can check that the quotient of $M$ by the action of $X_{0}$ defines a two dimensional manifold $N$ (at least locally). Then the projection $\pi:M\to N$ defines a principal bundle and the distribution $\distr$ defines a connection on this bundle. Moreover the sub-Riemannian structure on $M$ induces, by projection, a Riemannian structure on $N$  and the curvature associated with the connection $\distr$ over $M$ coincides with the area form on $N$ defined by the Riemannian structure. In this case, the invariant $\trace\left(r_{\lam}|_{\distr_{x}^{*}}\right)$ is constant along the flow of $X_{0}$ and hence descends to a well-defined function on $N$, that is its \emph{Gaussian curvature} (up to a constant factor).

For these reasons, under the assumption $\discr\left(r_{\lam}|_{\distr_{x}^{*}}\right)=0$ for all $x\in M$, one has that the sub-Riemannian structure is locally isometric to the one defined by the Dido's isoperimetric problem on a Riemannian surface $M$ (see \cite{agrexp}).  In the case when $\trace\left(r_{\lam}|_{\distr_{x}^{*}}\right)$ is constant on all $M$ we have the following result.

\begin{proposition} \label{p:3dh}
Let $M$ be a complete and simply connected 3D contact sub-Riemannian
manifold, and assume that $\discr\left(r_{\lam}|_{\distr_{x}^{*}}\right)=0$ and $\trace\left(r_{\lam}|_{\distr_{x}^{*}}\right)$ is constant on $M$. Then, up to dilations of the metric
\begin{itemize}
\item[(i)] if $\trace\left(r_{\lam}|_{\distr_{x}^{*}}\right)=0$, then $M$ is isometric to the Heisenberg group, 
\item[(ii)] if $\trace\left(r_{\lam}|_{\distr_{x}^{*}}\right)>0$, then $M$ is isometric to the group $SU(2)$ with Killing metric,
\item[(iii)] if $\trace\left(r_{\lam}|_{\distr_{x}^{*}}\right)<0$, then $M$ is isometric to the universal covering of $SL(2)$ with the Killing metric.
\end{itemize}
\end{proposition}
Proposition \ref{p:3dh} can be found in \cite[Thm. 11]{agricm} (see also \cite[Cor. 2]{miosr3d}), where it is stated with different language in terms of the invariants $\chi,\kappa$ of a 3D contact sub-Riemannian structure. See Section~\ref{s:interpretation} for a detailed discussion about the curvature of 3D contact sub-Riemannian structure and its relation with these invariants.

Despite the rigidity result stated in Proposition~\ref{p:3dh}, one can wonder it the ``constant curvature'' is achieved in the following sense: does it exist a sub-Riemannian structure such that the curvature operator $\RR_{\lam}$ does not depend on $\lam$? Indeed this would be the real analogue of the Riemannian constant curvature condition (see also Eq.~\eqref{eq:strascico}).
It turns out that, at least in the class of 3D contact sub-Riemannian structures, the curvature operator always depends non-trivially on $\lambda$ (see Proposition~\ref{p:sfiniti} in Section~\ref{s:chikappa}). This suggests that no sub-Riemannian structure has constant curvature in this sense. Still, the computation of our curvature in dimension higher than $3$ is a challenging task.

Even if there are no sub-Riemannian structures with constant curvature in the sense specified above, it is still possible to achieve constant curvature in the larger class of affine optimal control problems. Indeed, as proved in Section~\ref{s:lq}, the operators $\Qz_{\lam}$ and $\RR_{\lam}$ are constant for the so-called linear quadratic optimal control problems.
At the present stage, it is not straightforward how to use these structures as models to investigate purely geometrical aspects of sub-Riemannian manifolds, such as comparison theorem for volumes, distances etc. For other type of comparison, relating curvature bounds to existence (and estimates) of conjugate points along sub-Riemannian geodesics, a connection is possible and has been investigated in \cite{brconj}. In this case, the role of constant curvature models is played by linear quadratic optimal control problems, for which complete conditions for occurrence of conjugate points is well understood (see \cite{ARS-LQ}).


\part{Technical tools and proofs}

\chapter{Jacobi curves} \label{c:jac}
In this chapter we introduce the notion of Jacobi curve associated with a normal geodesic, that is a curve of Lagrangian subspaces in a symplectic vector space. This curve arises naturally from the geometric interpretation of the second derivative of the geodesic cost, and is closely related with the asymptotic expansion of Theorem \ref{t:main}.

We start with a brief description of the properties of curves in the Lagrange Grassmannian. For more details, see \cite{geometryjacobi1,lizel,agrafeedback}.

\section{Curves in the Lagrange Grassmannian} \index{Lagrange Grassmannian} \index{curves in Lagrange Grassmannian}

Let $(\Sigma, \sigma)$ be a $2n$-dimensional symplectic vector space. A subspace $\Lambda  \subset \Sigma$ is called 
\emph{Lagrangian} if it has dimension $n$ and $\sigma|_{\Lambda}\equiv 0.$  The \emph{Lagrange Grassmannian} $L(\Sigma)$ is the set of all $n$-dimensional Lagrangian subspaces of $\Sigma$. 

\bp\label{p:lagrass} $L(\Sigma)$ is a compact $n(n+1)/2$-dimensional submanifold of the Grassmannian of $n$-planes in $\Sigma$.
\ep
\begin{proof}
Let $\Delta \in L(\Sigma)$, and consider the set $\Delta^{\pitchfork}\doteq \{\Lambda \in L(\Sigma)\,|\,  \Lambda \cap \Delta =0 \}$ of all Lagrangian subspaces transversal to $\Delta$.
Clearly, the collection of these sets for all $\Delta \in L(\Sigma)$ is an open cover of $L(\Sigma)$. Then it is  sufficient to find submanifold coordinates on each $\Delta^{\pitchfork}$. 

Let us fix any Lagrangian complement $\Pi$ of $\Delta$ (which always exists, though it is not unique).
Every $n$-dimensional subspace $\Lambda \subset \Sigma$ that is transversal to $\Delta$ is the graph of a linear map from $\Pi$ to $\Delta$. Choose an adapted Darboux basis on $\Sigma$, namely a basis $\{e_i,f_i\}_{i=1}^n$ such that
\begin{gather}
\Delta = \spn\{f_1,\ldots,f_n\}, \qquad \Pi = \spn\{e_1,\ldots,e_n\}, \\
\sigma(e_i,f_j) - \delta_{	ij} = \sigma(f_i,f_j) = \sigma(e_i,e_j) = 0, \qquad i,j=1,\ldots,n.
\end{gather} 
In these coordinates, the linear map is represented by a matrix $S_{\Lambda}$ such that
\begin{equation} 
\Lambda \cap \Delta=0 \Leftrightarrow \Lambda=\{z=(p,S_{\Lambda} p), \,p \in \Pi\simeq \mathbb{R}^{n}\}.  
\end{equation}
Moreover it is easily seen that
$\Lambda \in L(\Sigma) $ if and only if $ S_{\Lambda}=S_{\Lambda}^{*}$.
Hence, the open set $\Delta^{\pitchfork}$ of all Lagrangian subspaces transversal to $\Delta$ is parametrized by the set of symmetric matrices, and this gives smooth submanifold coordinates on $\Delta^\pitchfork$.
This also proves that the dimension of $L(\Sigma)$ is $n(n+1)/2$. Finally, as a closed subset of a compact manifold, $L(\Sigma)$ is compact.
\end{proof}

Fix now $\Lambda \in L(\Sigma)$. The tangent space $T_{\Lambda }L(\Sigma)$ to the Lagrange Grassmannian at the point $\Lambda$ can be canonically identified with the set of quadratic forms on the space $\Lambda$ itself, namely
\begin{equation}
T_{\Lambda}L(\Sigma)\simeq Q(\Lambda).
\end{equation}
Indeed, consider a smooth curve $\Lambda(\cdot)$ in $L(\Sigma)$ such that $\Lambda(0)=\Lambda$, and denote by $\dot{\Lambda}\in T_{\Lambda}L(\Sigma)$ its tangent vector. For any point $z\in \Lambda$ and any smooth extension $z(t)\in \Lambda(t)$, we define the quadratic form
\begin{equation} 
\dot{\Lambda}\doteq  z \mapsto \sigma(z,\dot{z}), 
\end{equation} 
where $\dot{z}\doteq \dot{z}(0)$.
A simple check shows that the definition does not depend on the extension $z(t)$. Finally, if in local coordinates $\Lambda(t)=\{(p,\,S(t)p),\,p\in \R^{n}\}$, the quadratic form $\dot{\Lambda}$ is represented by the matrix $\dot S(0)$. In other words, if $z \in \Lambda$ has coordinates $p \in \mathbb{R}^n$, then $\dot{\Lambda}: p \mapsto p^*\dot{S}(0)p$.

\subsection{Ample, equiregular, monotone curves} \index{curves in Lagrange Grassmannian!ample} \index{curves in Lagrange Grassmannian!equiregular} \index{curves in Lagrange Grassmannian!monotone}

Let $J(\cdot)\in L(\Sigma)$ be a smooth curve in the Lagrange Grassmannian. For $i \in \mathbb{N}$, consider 
\begin{equation}
J^{(i)}(t)=\tx{span}\left\{\frac{d^{j}}{dt^{j}}\ell(t)\bigg| \ \ell(t)\in J(t),\, \ell(t) \text{ smooth},\, 0\leq j \leq i\right\}\subset \Sigma, \qquad i\geq 0.
\end{equation}
\bdeff \label{d:amplestar}
The subspace $J^{(i)}(t)$ is the \emph{i-th extension} of the curve $J(\cdot)$ at $t$. The flag
\begin{equation}
J(t) = J^{(0)}(t)\subset J^{(1)}(t)\subset J^{(2)}(t)\subset \ldots\subset \Sigma,
\end{equation}
is the \emph{associated flag of the curve} at the point $t$. The curve $J(\cdot)$ is called:
\bi
\iii[(i)] \emph{equiregular} at $t$ if $\text{dim }J^{(i)}(\cdot)$ is locally constant at $t$, for all $i \in \N$,
\iii[(ii)] \emph{ample} at $t$ if there exists $N\in \N$ such that $J^{(N)}(t)=\Sigma$,
\iii[(iii)] \emph{monotone increasing} (resp. \emph{decreasing}) at $t$ if $\dot{J}(t)$ is non-negative (resp. non-positive) as a quadratic form.
\ei
The \emph{step} of the curve at $t$ is the minimal $N \in \N$ such that $J^{(N)}(t) = \Sigma$.
\edeff

In coordinates, $J(t)=\{(p,S(t)p)|\  p\in \R^{n}\}$ for some smooth family of symmetric matrices $S(t)$. The curve is ample at $t$ if and only if there exists $N\in \N$ such that 
\begin{equation}
\text{rank} \{\dot S(t), \ddot S(t),\ldots, S^{(N)}(t)\}=n.
\end{equation}
The \emph{rank} of the curve at $t$ is the rank of $\dot{J}(t)$ as a quadratic form (or, equivalently, the rank of $\dot{S}(t)$). 
We say that the curve is equiregular, ample or monotone (increasing or decreasing) if it is equiregular, ample or monotone for all $t$ in the domain of the curve.

In the subsequent sections we show that with any ample (resp. equiregular) geodesic, we can associate in a natural way an ample (resp. equiregular) curve in an appropriate Lagrange Grassmannian. This justifies the terminology introduced in Definition~\ref{d:amplestar}. 

An important property of ample, monotone curves is described in the following lemma.
\begin{lemma}\label{l:ampletrasv}
Let $J(\cdot) \in L(\Sigma)$ be a monotone, ample curve at $t_0$. Then, there exists $\eps > 0$ such that $J(t) \cap J(t_0) = \{0\}$ for $0<|t - t_0| <\eps$.
\end{lemma}
\begin{proof}
Without loss of generality, assume $t_0 =0 $. Choose a Lagrangian splitting $\Sigma = \Lambda \oplus \Pi$, with $\Lambda = J(0)$. For $|t|<\eps$, the curve is contained in the chart defined by such a splitting. In coordinates, $J(t)=\{(p,S(t)p)|\  p\in \R^{n}\}$, with $S(t)$ symmetric and $S(0) = 0$. The curve is monotone, then $\dot{S}(t)$ is a semidefinite symmetric matrix. It follows that $S(t)$ is semidefinite too.

Suppose that, for some $ \tau$, $J(\tau) \cap J(0) \neq \{0\}$ (w.l.o.g. assume $\tau > 0$). This means that $\exists p \in \mathbb{R}^n$ such that $S(\tau)p = 0$. Indeed also $p^* S(\tau) p = 0$.  The function $t \mapsto p^* S(t) p = 0$ is monotone, vanishing at $t = 0$ and $t = \tau$. Therefore $p^* S(t) p = 0$ for all $0\leq t \leq \tau$. Being a semidefinite, symmetric matrix, $p^* S(t) p = 0$ if and only if $S(t)p = 0$. Therefore, we conclude that $p \in \ker S(t)$ for $0\leq t \leq \tau$. This implies that, for any $i \in \N$, $p \in \ker S^{(i)}(0)$, which is a contradiction, since the curve is ample at $0$.
\end{proof}

\brem \index{curves in Lagrange Grassmannian!regular}
Ample curves with $N=1$ are also called \emph{regular}. See in particular \cite{agrafeedback,geometryjacobi1}, where the authors discuss geometric invariants of these curves.
Notice that a curve $J(\cdot)$ is regular at $t$ if and only if its tangent vector at $t$ is a non degenerate quadratic form, i.e. the matrix $\dot S(t)$ is invertible. 
\erem

\subsection{The Young diagram of an equiregular curve} \index{Young diagram} 

Let $J(\cdot) \in L(\Sigma)$ be smooth, ample and equiregular. We can associate in a standard way a Young diagram with the curve $J(\cdot)$ as follows. Consider the restriction of the curve to a neighbourhood of $t$ such that, for all $i \in N$, $\dim J^{(i)}(\cdot)$ is constant. Let $h_i\doteq  \dim J^{(i)}(\cdot)$. By hypothesis, there exists a minimal $N \in \N$ such that $h_i = \dim \Sigma$ for all $i \geq N$. 
{\review
\bl \label{l:flag}
Let $J(\cdot)\in L(\Sigma)$ be smooth, ample and equiregular and denote $h_i= \dim J^{(i)}(\cdot)$. Then we have the inequalities 
\begin{equation} 
h_{i+1} - h_i \leq h_i - h_{i-1},\qquad \all i\geq 0.
\end{equation}  
\el
These inequalities are valid for any equiregular curve in the Grassmannian of a vector space. The proof of Lemma \ref{l:flag} is in Appendix \ref{a:flag}.
\finereview}

Then, we build a Young diagram with $N$ columns, with $h_i - h_{i-1}$ boxes in the $i$-th column. This is the \emph{Young diagram of the curve $J(\cdot)$}. In particular, notice that the number of boxes in the first column is equal to the rank of $J(\cdot)$.

\section{The Jacobi curve and the second differential of the geodesic cost}
Recall that $T^*M$ has a natural structure of symplectic manifold, with the canonical symplectic form defined as the differential of the Liouville form, namely $\sigma = d\varsigma$. In particular, for any $\lam \in T^* M$, $T_\lam (T^*M)$ is a symplectic vector space with the canonical symplectic form $\sigma$. Therefore, we can specify the construction above to $\Sigma\doteq  T_\lam(T^*M)$. In this section we show that the second derivative of the geodesic cost (associated with an ample geodesic $\gamma$ with initial covector $\lambda \in T^*M$) can be naturally interpreted as a curve in the Lagrange Grassmannian of $T_\lambda (T^*M)$, which is ample in the sense of Definition~\ref{d:amplestar}.

\subsection{Second differential at a non critical point}

Let $f \in C^{\infty}(M)$. As we explained in Section \ref{s:2dctdot}, the second differential \index{second differential} of $f$, which is a symmetric bilinear form on the tangent space, is well defined only at critical points of $f$.
If $x \in M$ is not a critical point, it is still possible to define the second differential of $f$, as the differential of $df$, thought as a section of $T^*M$.

\begin{definition}\label{d:secdif}
Let $f \in C^\infty(M)$, and
\begin{equation}
df: M\to T^{*}M, \qquad df: x\mapsto d_{x}f.
\end{equation}
Fix $x \in M$, and let $\lambda\doteq  d_x f \in T^*M$. The \emph{second differential}  of $f$ at $x \in M$ is the linear map
\begin{equation} \label{eq:2dxx}
d^{2}_{x}f\doteq d_{x}(df): T_{x}M\to T_{\lam}(T^{*}M),\qquad d^2_x f: v\mapsto \frac{d}{ds}\bigg|_{s=0} d_{\gamma(s)}f,
\end{equation}
where $\gamma(\cdot)$ is a curve on $M$ such that $\gamma(0)=x$ and $\dot \gamma(0)=v$.
\end{definition}
Definition~\ref{d:secdif} generalizes the concept of \virg{second derivatives} of $f$, as the linearisation of the differential.


\brem The image of the differential $df: M\to T^{*}M$ is a Lagrangian submanifold of $T^*M$. Thus, by definition, the image of the second differential $d^{2}_{x}f (T_x M)$ at a point $x$ is the tangent space of $df(M)$ at $\lam=d_{x}f$, which is an $n$-dimensional Lagrangian subspace of $T_{\lam}(T^{*}M)$ transversal to the vertical subspace $T_\lam(T^*_x M)$. 
\erem

By a dimensional argument and the fact that $\pi\circ df=\id_{M}$ (hence $\pi_{*}\circ d^{2}_{x}f=\id_{T_{x}M}$), we obtain the following formula for the image of a subspace through the second differential.
\bl 
Let $f:M\to \R$ and $W\subset T_{x}M$. Then $d^{2}_{x}f(W)=d^{2}_{x}f(T_{x}M)\cap \pi_{*}^{-1}(W)$.
\el

The next lemma describes the affine structure on the space of second differentials.

\bl
Let $\lam \in T_x^*M$. The set $\mc{L}_\lambda\doteq \{d_x^2f |\, f \in C^{\infty}(M), d_x f = \lambda\}$ is an affine space over the vector space $Q(T_{x}M)$ of the quadratic forms over $T_{x}M$.
\el
\begin{proof}
Consider two functions $f_1, f_2$ such that $d_x f_1 = d_x f_2 =\lam$. Then $f_1 - f_2$ has a critical point at $x$. We define the difference between $d^2_x f_1$  and $d^2_x f_2$ as the quadratic form $d^2_x (f_1 - f_2)$.
\end{proof}

\brem When $\lam = 0 \in T_x^* M$, $\mc{L}_\lam$ is the space of the second derivatives of the functions with a critical point at $x$. In this case we can fix a canonical origin in $\mc{L}_\lam$, namely the second differential of any constant function. This gives the identification of $\mc{L}_\lam$ with the space of quadratic forms on $T_xM$, recovering the standard notion of Hessian discussed in Section~\ref{s:2dctdot}.
\erem

\subsection{Second differential of the geodesic cost function}

Let $\gamma: [0,T] \to M$ be a strongly normal geodesic. Let $x = \gamma(0)$. Without loss of generality, we can choose $T$ sufficiently small so that the geodesic cost function $(t,x) \to c_t(x)$ is smooth in a neighbourhood of $(0,T)\times \{x\}\subset \R\times M$, and $d_{x} c_t = \lam$ is the initial covector associated with $\gamma$ (see Definition~\ref{d:geodesiccost}, Theorem~\ref{t:mst} and Proposition~\ref{p:critlam0}).

The second differential of $c_t$ defines a curve in the Lagrange Grassmannian $L( T_\lam (T^* M))$. For any $\lambda \in T^*M$, $\pi(\lambda) = x$, we denote with the symbol $\ve_\lam = T_\lam(T^*_{x} M) \subset T_\lam (T^*M)$ the vertical subspace, namely the tangent space to the fiber $T_{x}^*M$. Observe that, if $\pi : T^*M \to M$ is the bundle projection, $\ve_\lam = \ker \pi_*$.
\begin{definition}
The \emph{Jacobi curve} \index{Jacobi curve} associated with $\gamma$ is the smooth curve $J_\lam: [0,T] \to L(T_\lam (T^* M))$ defined by
\begin{equation}
J_{\lam}(t)\doteq  d^{2}_{x}c_{t}(T_{x}M),
\end{equation}
for $t \in (0,T]$, and $J_\lam(0)\doteq  \ve_\lam$.
\end{definition}
The Jacobi curve is smooth as a consequence of the next proposition, which provides an equivalent characterization of the Jacobi curve in terms of the Hamiltonian flow on $T^*M$.

\bp\label{p:Jproperties} Let $\lambda: [0,T] \to T^*M$ be the unique lift of $\gamma$ such that $\lambda(t) = e^{t\vec{H}} (\lambda)$. Then the associated Jacobi curve satisfies the following properties for all $t,s$ such that both sides of the statements are defined:
\bi
\iii[(i)] $J_{\lam}(t)=e^{-t\vec{H}}_{*} \ve_{\lambda(t)}$, 
\iii[(ii)] $J_{\lam}(t+s)=e^{-t\vec{H}}_{*} J_{\lam(t)}(s)$,
\iii[(iii)] $\dot J_\lam(0) = - d^2_\lam H_x$ as quadratic forms on $\ve_\lam \simeq T^*_x M$.
\ei
\ep
\begin{proof} In order to prove (i) it is sufficient to show that $\pi_{*}\circ e^{t\vec{H}}_{*}\circ d^{2}_{x}c_{t}=0$. Then, let $v \in T_x M$, and $\alpha(\cdot)$ a smooth arc such that $\alpha(0) = x$, $\dot{\alpha}(0) = v$. Recall that, for $s$ sufficiently small, $d_{\alpha(s)} c_t$ is the initial covector of the unique normal geodesic which connects $\alpha(s)$ with $\gamma(t)$ in time $t$, i.e. $\pi\circ e^{t\vec{H}} \circ d_{\alpha(s)}c_{t}=\gamma(t)$. Then
\begin{equation}
\pi_{*}\circ e^{t\vec{H}}_{*}\circ d^{2}_{x}c_{t} (v) =\frac{d}{ds}\bigg|_{s=0} \pi\circ e^{t\vec{H}} \circ d_{\alpha(s)}c_{t}=0.
\end{equation}
Statement (ii) follows from (i) and the group property of the Hamiltonian flow.
To prove (iii), introduce canonical coordinates $(p,x)$ in the cotangent bundle. Let $\xi\in \ve_\lam$, such that $\xi =  \sum_{i=1}^n \xi_i \partial_{p_i}|_\lam$. By (i), the smooth family of vectors in $\ve_\lam$ defined by
\begin{equation}
\xi(t)\doteq  e^{-t\vec{H}}_*\left(\sum_{i=1}^n \xi^i \partial_{p_i}|_{\lambda(t)}\right),
\end{equation}
satisfies $\xi(0) = \xi$ and $\xi(t)\in J_\lam(t)$. Therefore
\begin{equation}
\dot J_{\lam}(0) \xi=\sigma(\xi,\dot\xi)=-\sum_{i,j=1}^n\frac{\partial^{2} H}{\partial p_i \partial p_j }\xi^i \xi^j = - \langle \xi, (d^2_\lam H_x) \xi \rangle,
\end{equation}
where the last equality follows from the definition of $d^2_\lambda H_x$ after the identification $\ve_\lambda \simeq T^*_xM$ (see Section~\ref{s:Hpp}).
\end{proof}
\begin{remark}
Point (i) of Proposition~\ref{p:Jproperties} can be used to associate a Jacobi curve with any integral curve of the Hamiltonian flow, without any further assumptions on the underlying trajectory on the manifold. In particular we associate with any initial covector $\lambda \in T_x M$ the Jacobi curve $J_\lam(t)\doteq e^{-t\vec{H}}\ve_{\lam(t)}$. Observe that, in general, $\gamma(\cdot)\doteq \pi\circ\lambda(\cdot)$ may be also abnormal.
\end{remark}

Proposition~\ref{p:Jproperties} and the fact that the quadratic form $d^2_\lam H_x$ is non-negative  imply the next corollary.
\begin{corollary} \index{Jacobi curve!monotone}
The Jacobi curve $J_\lam$ is monotone decreasing for every $\lambda \in T^*M$.
\end{corollary}

The following proposition provides the connection between the flag of a normal geodesic and the flag of the associated Jacobi curve.

\begin{proposition}\label{p:twoflags}
Let $\gamma(t) = \pi\circ e^{t \vec{H}}(\lam)$ be a normal geodesic associated with the initial covector $\lam$. The flag of the Jacobi curve $J_\lam$ projects to the flag of the geodesic $\gamma$ at $t=0$, namely
\begin{equation}\label{eq:projection}
\pi_* J^{(i)}_{\lam} (0)  = \DD^i_{\gamma}(0),\qquad \all i \in \N.
\end{equation} 
Moreover, $\dim J^{(i)}_\lam(t) = n + \dim \DD^i_{\gamma}(t)$. Therefore $\gamma$ is ample of step $m$ (resp. equiregular) if and only if $J_\lam$ is ample of step $m$ (resp. equiregular).
\end{proposition}
\begin{proof}
The last statement follows directly from Eq.~\eqref{eq:projection}, Proposition~\ref{p:Jproperties} (point (ii)) and the definition of $\DD_{\gamma(s)}(t) = \left(P_{s,s+t}\right)^{-1}_*\distr_{\gamma(s+t)}$. In order to prove Eq.~\eqref{eq:projection}, let $\uu: T^*M \to L^\infty([0,T],\mathbb{R}^k)$ be the map that associates to any covector the corresponding normal control:
\begin{equation}
\uu_i(\lambda)(\cdot) = \langle e^{\cdot\vec{H}}(\lambda), f_i\rangle,\qquad i = 1,\ldots,k,
\end{equation}
where we assume, without loss of generality, that the Hamiltonian field $\vec{H}$ is complete. For any control $v \in L^\infty([0,T],\mathbb{R}^k)$ and initial point $x \in M$, consider the non-autonomous flow $P^v_{0,t}(x)$. We have the following identity, for any $\lambda \in T^*M$ and $t \in [0,T]$
\begin{equation}
\pi \circ e^{t\vec{H}}(\lambda) = P_{0,t}^{\uu(\lambda)} (\pi(\lam)).
\end{equation}
Remember that, as a function of the control, $P^v_{0,t}(x) = E_{x,t}(v)$ (i.e. the endpoint map with basepoint $x$ and endtime $t$). Therefore, by taking the differential at $\lambda$ (such that $\pi(\lambda) = x$), we obtain
\begin{equation}
\pi_* \circ e^{t\vec{H}}_*|_\lambda = \left(  P_{0,t}^{\uu(\lambda)} \right)_* \circ \pi_* |_\lambda +  D_{\uu(\lambda)} E_{x,t} \circ \uu_*|_\lam,
\end{equation}
Then, by the explicit formula for the differential of the endpoint map, we obtain, for any vertical field $\xi(t) \in \ve_{e^{t\vec{H}}(\lam)}$
\begin{equation}
\pi_* \circ e^{-t\vec{H}}_* \xi(t)= - \int_0^t (P_{0,\tau})^{-1}_* \overline{f}(v(t,\tau),\gamma(t))d\tau,
\end{equation}
where $\gamma(t) = \pi \circ e^{t\vec{H}}(\lam)$ is the normal geodesic with initial covector $\lambda$ and, for any $t \in [0,T]$,
\begin{equation}
v_i(t,\cdot) \doteq \uu_* \circ e^{-t\vec{H}}_* \xi(t) = \left(\uu \circ e^{-t\vec{H}}\right)_* \xi(t), \qquad v(t,\cdot) \in L^\infty([0,T],\R^k).
\end{equation}
More precisely, $v(t,\cdot)$ has components 
\begin{equation}
v_i(t,\tau) = \left.\frac{d}{d\varepsilon}\right|_{\varepsilon = 0} \langle e^{(\tau-t)\vec{H}}(\lambda(t) + \varepsilon \xi(t)), f_i\rangle,\qquad i=1,\ldots,k,
\end{equation}
where $\lam(t) = e^{t\vec{H}}(\lam)$, and we identified $\ve_{e^{t\vec{H}}(\lam)} \simeq T_{\gamma(t)}^*M$. Observe that, on the diagonal, $v_i(t,t) = \langle\xi(t),f_i\rangle = \xi_i(t)$. It is now easy to show that, for any positive $i \in \mathbb{N}$
\begin{equation}\label{eq:derivatio}
\left.\frac{d^i}{dt^i}\right|_{t=0} \pi_* \circ e^{-t\vec{H}}_* \xi(t)= -\left.\frac{d^{i-1}}{dt^{i-1}}\right|_{t=0}\left[\left(P_{0,t}\right)^{-1}_* \sum_{j=1}^k \xi_j(t) \overline{f}_j(\gamma(t))\right] \mod \DD_{\gamma}^{i-1}(0).
\end{equation}
By point (i) of Proposition~\ref{p:Jproperties}, any smooth family $\ell(t) \in J_\lam(t)$ is of the form $e^{-t\vec{H}}_* \xi(t)$ for some smooth $\xi(t) \in \ve_{e^{t\vec{H}}(\lam)}$. Therefore, Eq.~\eqref{eq:derivatio} for $i=1$ implies that $J_\lam^{(1)} = \DD_{\gamma}^{1}(0)$. The same equation and an easy induction argument, together with the definitions of the flags show that $J_\lam^{(i)}(0) = \DD_{\gamma}^{i}(0)$ for any positive $i \in \mathbb{N}$.
\end{proof}

\begin{remark}
If $\gamma$ is equiregular, ample of step $m$ with growth vector $\mc{G}_\lam = (k_1,k_2,\ldots,k_m)$, the Young diagram of $J_\lam$ has $m$ columns, with $d_i\doteq  k_i - k_{i-1}$ boxes in the $i$-th column (recall that $k_0 = \dim \DD^0_{\gamma}(t) = 0$).
\end{remark}

\begin{remark}\label{r:rational}
Notice that, by the coordinate representation of $J_\lambda^{(i)}(t)$ and Proposition~\ref{p:twoflags}, we have the following formula:
\begin{equation}
\dim \DD_\gamma^i(0) = \rank\{\dot{S}_\lambda(0),\ddot{S}_\lambda(0),\ldots,S_{\lambda}^{(i)}(0)\}, \qquad \all i \geq 0.
\end{equation}
By point (i) of Proposition~\ref{p:Jproperties} it follows that, for any fibre-wise polynomial Hamiltonian, $S^{(i)}_\lam(0)$ is a rational function of the initial covector $\lam \in T_x^* M$, for any $i \in \mathbb{N}$. In particular, the integer numbers $k_i = \dim \DD_\gamma^i(0)$ are obtained as the rank of a matrix whose entries are rational in the covector $\lambda$.

Finally, we stress that the curve is ample at $t=0$ if and only if there exists $N \in \mathbb{N}$ such that
\begin{equation}
\rank \{\dot{S}_\lam(0),\ddot{S}_\lam(0),\ldots, S^{(N)}_\lam(0)\} = n.
\end{equation}
Therefore, under this polynomial assumption (which is true, for example, in the sub-Riemannian case), $J_\lam(\cdot)$ is ample on an open Zariski subset of the fibre $T_x^* M$.
\end{remark}

\section{The Jacobi curve and the Hamiltonian inner product}
The following is an elementary, albeit very useful property of the symplectic form $\sigma$.
\begin{lemma}\label{l:simplettico}
Let $\xi \in \ve_\lam$ a vertical vector. Then, for any $\eta \in T_\lam(T^*M)$
\begin{equation}
\sigma(\xi,\eta) = \langle \xi, \pi_* \eta \rangle,
\end{equation}
where we employed the canonical identification $\ve_\lam = T_x^* M$.
\end{lemma}
\begin{proof}
In any Darboux basis induced by canonical local coordinates $(p,x)$ on $T^*M$, we have $\sigma = \sum_{i=1}^n dp_i \wedge dx_i$ and $\xi = \sum_{i=1}^n \xi^i \partial_{p_i}$. The result follows immediately.
\end{proof}
In Section~\ref{s:Hpp} we introduced the Hamiltonian inner product on $\distr_x$, which, in general, depends on $\lambda$. Such an inner product is defined by the quadratic form $d^2_\lam H_x : T_x^*M \to T_x M$ on $\distr_x = \text{Im}(d^2_\lam H_x)$. The following lemma allows the practical computation of the Hamiltonian inner product through the Jacobi curve.
\begin{lemma}\label{l:pistarra}
Let $\xi \in T_x^*M$. Then
\begin{equation}
d_\lam^2 H_x (\xi) = -\pi_* \dot \xi,
\end{equation}
where $\dot\xi$ is the derivative, at $t=0$, of any extension $\xi(t)$ of $\xi$ such that $\xi(0) = \xi$ and $\xi(t) \in J_\lam(t)$.
\end{lemma}
\begin{proof}
By point (iii) of Proposition~\ref{p:Jproperties}, $d^2_\lam H_x = - \dot{J}_\lam(0)$. By definition of $\dot{J}_\lam(0): \ve_\lam \to \R$ as a quadratic form, $\dot{J}_\lam(0)(\xi) = \sigma(\xi,\dot{\xi})$. Then, by Lemma~\ref{l:simplettico}, $\dot{J}_\lam(0)(\xi) = \langle \xi, \pi_* \dot \xi\rangle$. This implies the statement after identifying again the quadratic form with the associated symmetric map.
\end{proof}
By Lemma~\ref{l:pistarra}, for any $v \in \distr_x$ there exists a $\xi \in \ve_\lam$ such that, for any extension $\xi(t) \in J_\lam(t)$, with $\xi(0) = \xi$, we have $v = \pi_* \dot \xi$. Indeed $\xi$ may not be unique. Besides, if $v = \pi_* \dot \xi$ and $w = \pi_* \dot \eta$, the Hamiltonian inner product rewrites \index{Hamiltonian inner product}
\begin{equation}\label{eq:scalprod}
\langle v| w \rangle_\lam =\sigma(\xi,\dot\eta) = -\sigma(\eta,\dot\xi).
\end{equation}
We now have all the tools required for the proof of Theorem~\ref{t:main}.

\section{Proof of Theorem~\ref{t:main}}

The statement of Theorem~\ref{t:main} is related with the analytic properties of the functions $t \mapsto \langle \QQ_\lam(t)v|v\rangle_\lam$ for $v \in \distr_x$. By definition, $\langle \QQ_\lam(t)v|v\rangle_\lam = d^2_x\dot{c}_t(v)$.

As a first step, we compute a coordinate formula for such a function in terms of a splitting $\Sigma = \ve_{\lam} \oplus \hor_{\lam}$, where $\ve_\lam$ is the vertical space and $\hor_\lam$ is any Lagrangian complement.
Observe that $\ve_{\lam} = J_\lam(0)= \ker \pi_*$ and $\pi_{*}$ induces an isomorphism between $\hor_{\lam}$ and $T_{x}M$.  $J_\lam(t)$ is the graph of a linear map $S(t) : \ve_{\lam} \to \hor_{\lam}$. Equivalently, by Lemma~\ref{l:ampletrasv}, for $0<t<\eps$, $J_\lam(t)$ is the graph of  $S(t)^{-1}: \hor_{\lam} \to \ve_{\lam}$. Once a Darboux basis (adapted to the splitting) is fixed, as usual one can identify these maps with the representative matrices.

Fix $v \in \distr_{x}\subset T_xM$ and let $\wt{v} \in \hor_\lam$ be the unique horizontal lift such that $\pi_* \wt{v} = v$. Then, by definition of Jacobi curve, and the standard identification $\ve_\lam \simeq T_x^*M$
\begin{equation}\label{eq:sing}
\langle \QQ_\lam(t) v|v\rangle _\lam = \frac{d}{dt}\sigma( S(t)^{-1} \wt{v},\wt{v}).
\end{equation}

Since $J_{\lam}(0)=\ve_{\lam}$, it follows that $S(t)^{-1}$ is singular at $t=0$. In what follows we prove Theorem~\ref{t:main}, by computing the asymptotic expansion of the matrix $S(t)^{-1}$.  More precisely, from \eqref{eq:sing} it is clear that we need only a \virg{block} of $S(t)^{-1}$ since it acts only on vectors $\til{v}\in \pi_{*}^{-1}(\distr_{x})\cap \hor_{\lam}$. In what follows we build natural coordinates on the space $\Sigma$ in such a way that Eq.~\eqref{eq:sing} is given by the derivative of the first $k\times k$ block of $S(t)^{-1}$ where, we recall, $k=\dim \distr_x$. Notice that this restriction is crucial in the proof since only the aforementioned block has a simple pole. This is not true, in general, for the whole matrix $S(t)^{-1}$. 

\subsection{Coordinate presentation of the Jacobi curve} \index{Jacobi curve!coordinate presentation}
In order to obtain a convenient expression for the matrix $S(t)$ we introduce a  set of coordinates $(p,x)$ induced by a particular Darboux frame adapted to the splitting $\Sigma = \ve_\lam \oplus \hor_\lam$. Namely 
\begin{equation} 
\Sigma=\{(p,x)|\,p,x\in \R^{n}\},\qquad \ve_{\lam}=\{(p,0)|\,p\in \R^{n}\},\qquad \hor_{\lam}=\{(0,x)|\,x\in \R^{n}\}.
\end{equation}
Besides, if $\xi = (p,x)$, $\bar{\xi} = (\bar{p},\bar{x}) \in \Sigma$ the symplectic product is $\sigma(\xi,\bar{\xi}) = p^* \bar{x}-\bar{p}^* x $. In these coordinates, $J_{\lambda}(t)=\{(p,S(t)p)|\, p\in \R^{n}\}$, and $S(0)=0$. The symmetric matrix $S(t)$ represents a monotone Jacobi curve, hence $\dot{S}(t) \leq 0$. Moreover, since the curve is ample, by Lemma \ref{l:ampletrasv}, $S(t)< 0$ for $0< t < \eps$. Moreover we introduce the coordinate splitting $\R^{n}=\R^{k} \oplus \R^{n-k}$ (accordingly we write $p=(p_{1},p_{2})$ and $x=(x_{1},x_{2})$), such that $\pi_*(\R^k) = \distr_x$. In blocks notation
\begin{equation}\label{eq:II}
S(t)=
\begin{pmatrix}
S_{11}(t)&S_{12}(t)\\
S_{12}^{*}(t)&S_{22}(t)\\
\end{pmatrix}
, \qquad \text{with}\quad S_{11}(t), S_{22}(t)<0 \quad \text{for}\quad 0<t<\eps.
\end{equation}
By point (iii) of Proposition~\ref{p:Jproperties}, in these coordinates we also have
\begin{equation}
\dot{S}(0) = \begin{pmatrix}\dot{S}_{11}(0) & 0 \\ 0 & 0\end{pmatrix}, \qquad\text{with}\quad \rank \dot{S}_{11}(0)=\dim \distr_{x}.
\end{equation}

Therefore, we obtain the following coordinate formula for the Hamiltonian inner product. Let $v,w \in \distr_x$, with coordinates $v = (v_1,0)$, $w = (w_1,0)$ then 
\begin{equation}\label{eq:scalprodcoord}
\la v | w \ra_\lam = - v_1^* \dot{S}_{11}(0)^{-1} w_1, \qquad v_1,w_1 \in \R^k,
\end{equation}
\brem \label{r:iden}
In other words, the quadratic form associated with the operator $\id:\distr_x \to \distr_x$ via the Hamiltonian inner product is represented by the matrix $-\dot{S}_{11}(0)^{-1}$.
\erem

Moreover the horizontal lift of $v$ is $\til{v} = ((0,0),(v_1,0))$ and analogously for $w$. Thus, by \eqref{eq:sing}
\begin{equation}\label{eq:Qcoord}
\la \QQ_\lam(t) v|w \ra_\lam = \frac{d}{dt} v^*_1 [S(t)^{-1}]_{11} w_1, \qquad v_1,w_1 \in \R^k, \quad t>0.
\end{equation}
For convenience, for $t>0$, we introduce the smooth family of $k\times k$ matrices $\Sred(t)$ defined by 
\begin{equation}
\Sred(t)^{-1}\doteq [S(t)^{-1}]_{11},\qquad t > 0.
\end{equation}
Then, the quadratic form associated with the operator $\QQ_\lam(t) : \distr_x \to \distr_x$ via the Hamiltonian inner is represented by the matrix $\frac{d}{dt}\Sred(t)^{-1}$.

The proof of Theorem~\ref{t:main} is based upon the following result.

\bt \label{t:aared} 
The map $t\mapsto \Sred(t)^{-1}$ has a simple pole at $t=0$.
\et
\begin{proof}
The expression of $\Sred(t)$ in terms of the blocks of $S(t)$ is given by the following lemma.
\bl \label{l:A11} 
Let $A=\left(\begin{smallmatrix}A_{11}&A_{12}\\A_{21}&A_{22}\end{smallmatrix}\right)$ be a sign definite matrix, and denote by $[A^{-1}]_{11}$ the first block of the inverse of $A$. Then $[A^{-1}]_{11}=(A_{11}-A_{12}A^{-1}_{22}A_{21})^{-1}$.
\el
Then, by definition of $\Sred$, we have the following formula (where we suppress $t$):
\begin{equation}\label{eq:SGamma}
\Sred=S_{11}-S_{12}S_{22}^{-1}S_{12}^{*}.
\end{equation}
\begin{lemma}\label{l:SGamma2}
As quadratic forms on $\R^{k}$, $S_{11}(t) \leq \Sred(t) < 0$ for $t>0$.
\end{lemma}
\begin{proof}[Proof of Lemma \ref{l:SGamma2}]
Let $t>0$. $S(t)$ is symmetric and negative, then also its inverse $S(t)^{-1}$ is symmetric and negative. This implies that $\Sred(t)^{-1}=[S(t)^{-1}]_{11} < 0$ and so is $\Sred(t)$. This proves the right inequality.
By Eq.~\eqref{eq:SGamma} and the fact that $S_{22}(t)$ is negative definite (and so is $S_{22}^{-1}(t)$) one also gets (we suppress $t >0$)
\begin{equation} 
p_1^* (S_{11}-\Sred)p_{1} =p_1^* S_{12}S_{22}^{-1}S_{12}^{*}p_{1} =  (S^*_{12} p_1)^* S_{22}^{-1} (S_{12}^*p_1) \leq 0, \qquad p_1 \in \R^k. \qedhere
\end{equation}

\end{proof}
\begin{lemma}
The map $t \mapsto \Sred(t)$ can be extended by smoothness at $t=0$.
\end{lemma}
\begin{proof}
Indeed, by the coordinate expression of Eq.~\eqref{eq:SGamma}, it follows that the only term that can give rise to singularities is the inverse matrix $S_{22}^{-1}(t)$. Since, by assumption, the curve is ample, $t\mapsto  \det S_{22}(t)$ has a finite order zero at $t=0$, thus the singularity can be only a finite order pole. On the other hand $S(t)\to 0$ for $t \to 0$, thus $S_{11}(t)\to 0$ as well. Then, by Lemma~\ref{l:SGamma2}, $\Sred(t)\to 0$  for $t\to 0$, hence can be extended by smoothness at $t=0$. 
\end{proof}

We are now ready to prove that $t\mapsto \Sred(t)^{-1}$ has a simple pole at $t=0$. As a byproduct, we obtain an explicit form for its residue. As usual, for $i>0$, we set $k_i\doteq  \dim J_\lam^{(i)}(0) - n$, and $d_i\doteq  k_i - k_{i-1}$. In coordinates, this means that
\begin{equation}
\rank\{\dot{S}(0), \ldots, S^{(i)}(0) \} = k_i, \qquad i=1,\ldots,m.
\end{equation}
By hypothesis, the curve is ample at $t=0$, then there exists $m$ such that $k_m =n$. Since we are only interested in Taylor expansions, we may assume $S(t)$ to be real-analytic in $[0,\varepsilon]$ by replacing, if necessary, $S(t)$ with its Taylor polynomial of sufficient high order. Then, let us consider the analytic family of symmetric matrices $\dot{S}(t)$. For $i=1,\ldots,n$, the family $w_i(t)$ of eigenvectors of $\dot{S}(t)$ (and the relative eigenvalues) are an analytic family (see \cite[Theorem 6.1, Chapter II]{Kato}). Therefore, $\dot{S}(t) = W(t)D(t) W(t)^*$, where $W(t)$ is the $n\times n$ matrix whose columns are the vectors $w_i(t)$, and $D(t)$ is a diagonal matrix. Recall that $\dot{S}(t)$ is non-positive. Then $\dot{S}(t) = -V(t)V(t)^*$, for some analytic family of $n\times n$ matrices $V(t)$. Let $v_i(t)$ denote the columns of $V(t)$.

Now, let us consider the flag $E_1\subset E_2 \subset \ldots \subset E_m = \mathbb{R}^n$ defined as follows
\begin{equation}
E_i = \spn\{v_j^{(\ell)}(0), \, 1\leq j \leq n, \, 0 \leq \ell \leq i-1\}.
\end{equation}
Let $\spn\{A\}$ denote the column space of a matrix $A$. Indeed $\spn\{\dot{S}(t)\} \subseteq \spn\{V(t)\}$. Besides, $\rank\{\dot{S}(t)\} = \rank\{V(t)V(t)^*\} = \rank\{V(t)\} = \dim\spn\{V(t)\}$. Therefore, $\spn\{\dot{S}(t)\} = \spn\{V(t)\}$, for all $|t|<\varepsilon$.
Thus, for $i=1,\ldots,m$
\begin{equation}
E_i = \spn\{ V(0), V^{(1)}(0), \ldots,V^{(i-1)}(0)\} = \spn\{ \dot{S}(0), \ldots, S^{(i)}(0)\}.
\end{equation}
Therefore $\dim E_i = k_i$. Choose coordinates in $\mathbb{R}^n$ adapted to this flag, i.e. $\spn\{e_1,\ldots,e_{k_i}\} = E_i$. In these coordinates, $V(t)$ has a peculiar structure, namely
\begin{equation}\label{eq:Vorder}
V(t) = \begin{pmatrix} \wh{v}_1 \\
t \wh{v}_2\\
\vdots \\
t^{m-1} \wh{v}_m
\end{pmatrix}+ \begin{pmatrix} O(t) \\
O(t^2) \\
\vdots \\
 O(t^m)\end{pmatrix},
\end{equation}
where $\wh{v}_i$ is a $d_i\times n$ matrix of maximal rank (notice that the $\wh{v}_i$ are not directly related with the columns $v_i(t)$ of $V(t)$). Let $\wh{V}(t)$ denote the ``principal part'' of $V(t)$. In other words, $\wh{V}(t) = (\wh v_1, t \wh v_2,\ldots, t^{m-1} \wh v_m)^*$. Then, remember that $S(0) =0$ and 
\begin{equation}
S(t) = \int_0^t \dot S(\tau) d\tau = -\int_0^t V(\tau)V(\tau)^* d\tau = -\int_0^t \wh{V}(\tau)\wh{V}(\tau)^* d \tau + r(t),
\end{equation}
where $r(t)$ is a remainder term. Observe that the matrix
\begin{equation}
\wh{S}(t) =  - \int_0^t \wh{V}(\tau)\wh{V}(\tau)^* d \tau
\end{equation}
is negative definite for $t>0$. In fact, a non trivial kernel for some $t>0$ would contradict the hypothesis $\spn\{V(0),V^{(1)}(0),\ldots,V^{(m-1)}(0)\} = \mathbb{R}^n$. 
In components, we write $S(t)$ as a $m\times m$ block matrix, $S_{ij}(t)$ being a $d_i\times d_j$ block, as follows:
\begin{equation}
S_{ij}(t) = \int_0^t \dot{S}_{ij}(\tau) d\tau = -\left(\frac{ \wh v_i \wh v_j^*}{i+j-1}\right) t^{i+j-1} + O(t^{i+j}) = \chi_{ij}t^{i+j-1} + O(t^{i+j}),
\end{equation}
where we introduced the negative definite constant matrix $\chi\doteq \wh{S}(1) <0$. By computing the determinant of $\wh{S}(t)$, we obtain
\begin{equation}\label{eq:cappuccio}
\det \wh{S}(t) = \det \begin{pmatrix}
t \chi_{11} & t^2 \chi_{12} & \cdots & t^{m} \chi_{1m} \\
t^2 \chi_{21} & t^3\chi_{22} & \cdots & t^{m+1} \chi_{2m} \\
\vdots & \vdots & \ddots & \vdots \\
t^m \chi_{m1} & t^{m+1} \chi_{m2} & \cdots & t^{2m-1} \chi_{mm} 
\end{pmatrix} = t^{d_1 + 3d_2 + \ldots + (2m-1)d_m} \det \chi.
\end{equation}
We now compute the inverse of $S(t)$. First, the inverse of the principal part $\wh{S}(t)$ is
\begin{equation}
\wh{S}(t)^{-1}_{ij} = \frac{(\chi^{-1})_{ij}}{t^{i+j-1}},
\end{equation}
as we readily check:
\begin{equation}
\sum_{\ell=1}^m \wh{S}(t)^{-1}_{i\ell} \wh{S}(t)_{\ell j} = \sum_{\ell=1}^m (\chi^{-1})_{i\ell}\chi_{\ell j}\frac{t^{\ell+j-1}}{t^{i+\ell-1}} = \sum_{\ell=1}^m (\chi^{-1})_{i\ell}\chi_{\ell j} t^{j - i}= \delta_{ij}.
\end{equation}
The (block-wise) principal part of the inverse $S(t)^{-1}$ is equal to the inverse of the (block-wise) principal part of $S(t)$. Then we obtain, in blocks notation, for $i=1,\ldots,m$ 
\begin{equation}
[S(t)^{-1}]_{ij} =  \frac{(\chi^{-1})_{ij}}{t^{i+j-1}} + O\left(\frac{1}{t^{i+j-2}}\right).
\end{equation}
Finally, by definition, $(\Sred)^{-1} = [S^{-1}]_{11}$. Thus
\begin{equation}
\Sred(t)^{-1} = \frac{(\chi^{-1})_{11} }{t} + O(1).
\end{equation}
Thus $\Sred(t)^{-1}$ has a simple pole at $t=0$, with a negative definite residue, as claimed.
\end{proof}

\brem As a consequence of Eq.~\eqref{eq:cappuccio}, the order of $\det S(t)$ at $t=0$ is equal to the order of its principal part $\wh{S}(t)$. Namely
\begin{equation}\label{eq:ordS}
\det S(t) \sim \det \wh{S}(t) \sim t^\mathcal{N},\qquad \mathcal{N} = \sum_{i=1}^m (2i-1)d_i. 
\end{equation}
\erem

\begin{proof}[Proof of the Theorem \ref{t:main}]
It is now clear that, in coordinates
\begin{equation}
\QQ_\lam(t) =  \frac{d}{dt} \Sred(t)^{-1},
\end{equation}
as quadratic forms on $(\distr_x,\la\cdot,\cdot\ra_\lam)$ (see Eq.~\eqref{eq:Qcoord}). By Theorem~\ref{t:aared}, the map $t \mapsto \Sred(t)^{-1}$ has a simple pole at $t=0$, and its residue is a negative definite matrix. Then, $\QQ_\lam(t)$ has a second order pole at $t=0$, and $t^2\QQ_\lam(t)$ can be extended smoothly also at $t=0$. In particular, $\Qz_\lam\doteq \lim_{t\to 0^+}t^2 \QQ_\lam(t)> 0$. 

Besides, by Lemma~\ref{l:SGamma2}, $S_{11}(t) \leq \Sred(t) < 0$, which implies $\Sred(t)^{-1} \leq S_{11}(t)^{-1} <0$. Then,
\begin{equation}
\Qz_\lam = \lim_{t\to 0^+} t^2 \frac{d}{dt} \Sred(t)^{-1} = -\lim_{t\to 0^+} t \Sred(t)^{-1} \geq -\lim_{t\to 0^+} t S_{11}(t)^{-1} = -\dot{S}_{11}(0)^{-1} > 0,
\end{equation}
which, according to Remark \ref{r:iden}, implies $\Qz_\lam \geq \id >0$ as operators on $\distr_x$.

Finally, $\QQ_\lam(t)$ cannot have a term of order $-1$ in the Laurent expansion, which is tantamount to $\left.\frac{d}{dt}\right|_{t=0} t^2\QQ_\lam(t) =0$.
\end{proof}

\section{Proof of Theorem~\ref{t:main4}}
The purpose of this section is the proof of the main result of Section~\ref{s:gd}, namely a formula for the exponent of the asymptotic volume growth of geodesic homotheties. 

Fix $x_0 \in M$ and let $\gamma:[0,1] \to M$ be the geodesic associated with the covector $\lambda \in T_{x_0}^*M$. Moreover, let $J_\lam$ be the associated Jacobi curve. As usual, we  fix a Lagrangian splitting $T_\lam(T^*M) = \ve_\lam \oplus \hor_\lam$, in terms of which $J_\lam(t)$ is the graph of the map $S(t):\ve_\lam \to \hor_\lam$.  The reader can easily check that  the statements that follow do not depend on the choice of the Lagrangian subspaces $\hor_\lam$. 
The following lemma relates $\gd_{\lam}$ with the Jacobi curve.

\begin{lemma}\label{l:ordS}
Assume that $\gamma$ is ample, of step $m$, with growth vector $\mc{G}_\lam = \{k_1,\ldots,k_m\}$ (at $t=0$). Then the order of $\det S(t)$ at $t=0$ is
\begin{equation}
\det S(t) \sim t^{\gd_{\lam}},\qquad \gd_{\lam} = \sum_{i=1}^m (2i-1)(k_i-k_{i-1}).
\end{equation}
If $\gamma$ is not ample, the order of $\det S(t)$ at $t=0$ is $+\infty$.
\end{lemma}
\begin{proof}
Indeed the order of $\det S(t)$ does not depend on the choice of the horizontal complement $\hor_\lam$ and Darboux coordinates. Then, for an ample curve, the statement is precisely Eq.~\eqref{eq:ordS}. Finally, if $\gamma$ is not ample, the Taylor polynomial of arbitrary order of $S(t)$ is singular, thus the order of $\det S(t)$ at $t=0$ is $+\infty$.
\end{proof}


We are now ready to prove the main result of Section~\ref{s:gd}.

\begin{proof}[Proof of Theorem~\ref{t:main4}]

Without loss of generality, we can assume that $\Omega$ is contained in a single coordinate patch $\{x_i\}_{i=1}^n$. In terms of such coordinates, $\mu = e^a dx^1\wedge\ldots\wedge dx^n$ and
\begin{equation}\label{eq:ultimaformula}
\mu(\Omega_{x_0,t}) = \int_\Omega |\det(d_x \phi_t)| e^{a\circ\phi_t(x)}d x.
\end{equation}
By smoothness, it is clear that the order of $\mu(\Omega_{x_0,t})$ at $t=0$ is equal to the order of the map $t\mapsto\det(d_x \phi_t)$. In the following, $\EXP_{x_0}:T_{x_0}^*M \to M$ denotes the sub-Riemannian exponential map at time $1$. Let us define $\Sigma_{x_0}^*\doteq \EXP_{x_0}^{-1}(\Sigma_{x_0})\subset T_{x_0}^*M$. Indeed, if $\lam \in \Sigma_{x_0}^*$, the associated geodesic $\gamma(t)= \EXP_{x_0}(t\lam)$ is the unique one connecting $x_0$ with $x = \EXP_{x_0}(\lam)$. We now compute the order of the map $t\mapsto\det(d_x \phi_t)$.
%
%
%
\bl For every $x\in \Sigma_{x_{0}}$ the order of $t\mapsto\det(d_x \phi_t)$ is equal to $\gd_{\lam}$, where $\lam=\EXP^{-1}_{x_{0}}(x)$.
\el
\begin{proof}
Recall that the order of a family of linear maps does not depend on the choice of the representative matrices. By Eq.~\eqref{eq:homothetymap},
\begin{equation}
d_x \phi_t = \pi_* \circ e^{(t-1)\vec{H}}_* \circ d^2_x \f.
\end{equation}
Let us focus on the linear map $e^{(t-1)\vec{H}}_* \circ d^2_x \f : T_x M \to T_{\lam(t)}(T^*M)$, where $\lam(t) = e^{t\vec{H}}(\lam)$ is the normal lift of $\gamma$. Let us choose a smooth family of Darboux bases $\{E_i|_{\lam(t)},F_i|_{\lam(t)}\}_{i=1}^n$ of $T_{\lam(t)}(T^*M)$, such that $\ve_{\lam(t)} = \spn\{E_i|_{\lam(t)}\}_{i=1}^n$ and $\hor_{\lam(t)} = \spn\{F_i|_{\lam(t)}\}_{i=1}^n$. Let us define the column vectors $E|_{\lam(t)}\doteq (E_1|_{\lam(t)},\ldots,E_n|_{\lam(t)})^*$ and $F|_{\lam(t)}\doteq (F_1|_{\lam(t)},\ldots,F_n|_{\lam(t)})^*$. Observe that the elements of $\pi_* F|_{\lam(t)}$ are a smooth family of bases for $T_{\gamma(t)} M$. Then
\begin{equation}\label{eq:scomposazza0}
e^{(t-1)\vec{H}}_*\circ  d^2_x \f( \pi_*F|_{\lam(1)}) = A(t) E|_{\lam(t)}  +  B(t) F|_{\lam(t)} ,
\end{equation}
for some smooth families of $n\times n$ matrices $A(t)$ and $B(t)$. 
Then, by definition, the order of the map $t\mapsto\det(d_x \phi_t)$ is the order of $\det B(t)$ at $t=0$. By acting with $e^{-t\vec{H}}_*$ in Eq.~\eqref{eq:scomposazza0}, we obtain
\begin{equation}\label{eq:scomposazza}
A(t) e^{-t\vec{H}}_* E|_{\lam(t)} =  e^{-\vec{H}}_*\circ  d^2_x \f(\pi_* F|_{\lam(1)}) - B(t) e^{-t\vec{H}}_*F|_{\lam(t)}.
\end{equation}
Notice that $A(0)$ is nonsingular. Then, for $t$ sufficiently close to $0$, the l.h.s. of Eq.~\eqref{eq:scomposazza} is a smooth basis for the Jacobi curve $J_\lam$. We rewrite the r.h.s. of Eq.~\eqref{eq:scomposazza} in terms of the fixed basis $\{E|_{\lam(0)},F|_{\lam(0)}\}$. To this end, observe that 
\begin{gather}
e^{-t\vec{H}}_*F|_{\lam(t)} = C(t) E|_{\lam(0)} + D(t) F|_{\lam(0)}, \\
e^{-\vec{H}}_*\circ  d^2_x\f(\pi_* F|_{\lam(1)}) = G E|_{\lam(0)}.
\end{gather}
For some $n\times n$ smooth matrices $C(t),D(t),G$. Observe that $C(0) = 0$ and $D(t)$ is nonsingular for $t$ sufficiently close to $0$. Moreover, since $x \in \Sigma_{x_0}$ is a regular value for the sub-Riemannian exponential map $\EXP_{x_0} = \pi \circ e^{\vec{H}}$, $G$ is nonsingular. Then
\begin{equation}\label{eq:scomposazza2}
A(t) e^{-t\vec{H}}_* E|_{\lam(t)} = [G - B(t)C(t)]E|_{\lam(0)} - B(t)D(t) F|_{\lam(0)}.
\end{equation}
Therefore, the representative matrix of $J_\lam(t)$ in terms of the basis $\{E|_{\lam(0)},F|_{\lam(0)}\}$ is
\begin{equation}\label{eq:scomposazza3}
S(t) = -[G-B(t)C(t)]^{-1}B(t)D(t), \qquad |t|<\varepsilon.
\end{equation}
By the properties of the matrices $G$, $C(t)$ and $D(t)$ for sufficiently small $t$, $\det S(t) \sim \det B(t)$, and the two determinants have the same order. Then the statement follows from Lemma \ref{l:ordS}.
\end{proof}
By Proposition \ref{p:costanza}, $\gd_{\lam}=\gd_{x_{0}}$ a.e. on $T^{*}_{x_{0}}M$. Then the order of $t\mapsto\det(d_x \phi_t)$ is equal to $\gd_{x_{0}}$ up to a zero measure set on $\Sigma_{x_0}$ and the statement of Theorem~\ref{t:main4} follows from \eqref{eq:ultimaformula}, since $\mu(\Omega)>0$.
\end{proof}

\chapter{Asymptotics of the Jacobi curve: equiregular case}\label{c:proof} \index{Jacobi curve!asymptotics}

In this chapter, we introduce a key technical tool, the so-called \emph{canonical frame}, associated with a monotone, ample, equiregular curve in the Lagrange Grassmannian $L(\Sigma)$. This is a special moving frame in the symplectic space $\Sigma$ which satisfies a set of differential equations encoding the dynamics of the underlying curve, which has been introduced for the first time in \cite{lizel}.

The main result of this chapter is an asymptotic formula for the curve, written in coordinates induced by the canonical frame. Finally, we exploit this result to prove Theorem~\ref{t:main2}.

\section{The canonical frame} \index{canonical frame}
Let $J(\cdot) \subset L(\Sigma)$ be an ample, monotone nonincreasing, equiregular curve of rank $k$. Suppose that its Young diagram $D$ has $k$ rows, of length $n_a$, for $a=1,\dots,k$. Let us fix some terminology about the frames, indexed by the boxes of the Young diagram $D$. Each box of the diagram is labelled \virg{$ai$}, where $a = 1,\dots,k$ is the row index, and $i=1,\dots,n_a$ is the progressive box number, starting from the left, in the specified row. Indeed $n_a$ is the length of the $a$-th row, and $n_1+\dots+n_k =n = \dim\Sigma$. Briefly, the notation $ai\in D$ denotes a generic box of the diagram.

From now on, we employ letters from the beginning of the alphabet $a,b,c,d,\dots$ for rows, and letters from the middle of the alphabet $i,j,h,k,\dots$ for the position of the box in the row.
According to this notation, a frame $\{E_{ai},F_{ai}\}_{ai \in D}$ for $\Sigma$ is Darboux if, for any $ai,bj \in D$,
\begin{equation}
\sigma(E_{ai},E_{bj}) = \sigma(F_{ai},F_{bj}) = \sigma(E_{ai},F_{bj}) - \delta_{ab}\delta_{ij} = 0,
\end{equation}
where $\delta_{ab}\delta_{ij}$ is the Kronecker delta defined on $D\times D$.

\subsection{A remark on the notation}
Any Darboux frame indexed by the boxes of the Young diagram defines a Lagrangian splitting $\Sigma = \ve \oplus \hor$, where
\begin{equation}
\ve = \spn\{E_{ai}\}_{ai\in D}, \qquad
\hor = \spn\{F_{ai}\}_{ai\in D}.
\end{equation}
In the following, we deal with linear maps $S: \ve \to \hor$ (and their inverses), written in coordinates induced by the frame. The corresponding matrices have a peculiar block structure, associated with the Young diagram. The $F_{bj}$ component of $S (E_{ai})$ is denoted by $S_{ab,ij}$. As a matrix, $S$ can be naturally thought as a $k\times k$ block matrix. The block $ab$ is a $n_a \times n_b$ matrix. This structure is the key of the  calculations that follow, and we provide an example. Consider the Young diagram $D$, together with the \virg{reflected} diagram $\overline{D}$ in Fig.~\ref{f:y0}.
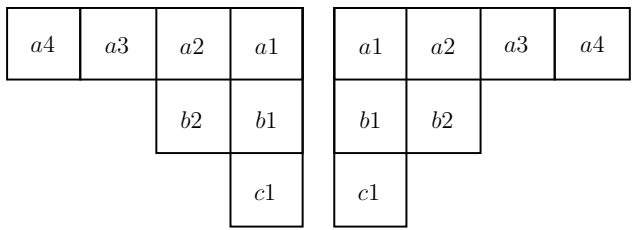
\begin{figure}[h!t]
\begin{center}
\scalebox{0.65} 
{
\begin{pspicture}(0,-1.85)(12.64,2.25)
\rput{-90.0}(5.27,5.27){\psframe[linewidth=0.04,dimen=outer](7.52,0.75)(3.02,-0.75)}
\psframe[linewidth=0.04,dimen=outer](8.12,2.25)(6.62,-2.25)
\psframe[linewidth=0.04,dimen=outer](9.62,2.25)(6.62,-0.75)
\psframe[linewidth=0.04,dimen=outer](11.12,2.25)(6.62,0.75)
\rput{-90.0}(2.25,5.25){\psframe[linewidth=0.04,dimen=outer](4.5,3.77)(3.0,-0.77)}
\rput{-90.0}(3.77,5.27){\psframe[linewidth=0.04,dimen=outer](6.02,2.25)(3.02,-0.75)}
\fontsize{14}{0}
\usefont{T1}{ptm}{m}{n}
\usefont{T1}{ptm}{m}{n}
\psframe[linewidth=0.04,dimen=outer](1.52,2.25)(0.0,0.75)
\psframe[linewidth=0.04,dimen=outer](12.64,2.25)(11.08,0.75)
\usefont{T1}{ptm}{m}{n}
\rput(7.381455,1.455){$a1$}
\usefont{T1}{ptm}{m}{n}
\rput(7.371455,-0.045){$b1$}
\usefont{T1}{ptm}{m}{n}
\rput(7.341455,-1.525){$c1$}
\usefont{T1}{ptm}{m}{n}
\rput(5.281455,1.455){$a1$}
\usefont{T1}{ptm}{m}{n}
\rput(5.2714553,-0.045){$b1$}
\usefont{T1}{ptm}{m}{n}
\rput(5.241455,-1.525){$c1$}
\usefont{T1}{ptm}{m}{n}
\rput(8.841455,1.455){$a2$}
\usefont{T1}{ptm}{m}{n}
\rput(8.831455,-0.045){$b2$}
\usefont{T1}{ptm}{m}{n}
\rput(3.761455,1.455){$a2$}
\usefont{T1}{ptm}{m}{n}
\rput(3.751455,-0.045){$b2$}
\usefont{T1}{ptm}{m}{n}
\rput(2.241455,1.475){$a3$}
\usefont{T1}{ptm}{m}{n}
\rput(10.321455,1.495){$a3$}
\usefont{T1}{ptm}{m}{n}
\rput(0.7414551,1.495){$a4$}
\usefont{T1}{ptm}{m}{n}
\rput(11.861455,1.495){$a4$}
\end{pspicture} 
}
\end{center}
\caption{The Young diagrams $\overline{D}$ (left) and $D$ (right).} \label{f:y0}
\end{figure}
We labelled the boxes of the diagrams according to the convention introduced above. It is useful to think at each box of the diagram $D$ as a one dimensional subspace of $\ve$, and at each box of the diagram $\overline{D}$ as a one dimensional subspace of $\hor$. Namely, the box $ai \in D$ corresponds to the subspace $\mathbb{R} E_{ai}$ (respectively, the box $bj \in \overline{D}$ corresponds to the subspace $\mathbb{R} F_{bj}$). Then the matrix $S$ has the following block structure. 
\begin{equation}
S = \begin{pmatrix}
S_{\textcolor{black}{a}\textcolor{black}{a}} & S_{\textcolor{black}{a}\textcolor{black}{b}} & S_{\textcolor{black}{a}\textcolor{black}{c}} \\
S_{\textcolor{black}{b}\textcolor{black}{a}} & S_{\textcolor{black}{b}\textcolor{black}{b}} & S_{\textcolor{black}{b}\textcolor{black}{c}}\\
S_{\textcolor{black}{c}\textcolor{black}{a}} & S_{\textcolor{black}{c}\textcolor{black}{b}} & S_{\textcolor{black}{c}\textcolor{black}{c}}
\end{pmatrix},
\end{equation}
where each block is a matrix of the appropriate dimension, e.g. $S_{ab}$ is a $4\times 2$ matrix as explained pictorially in Fig.~\ref{f:y1}.

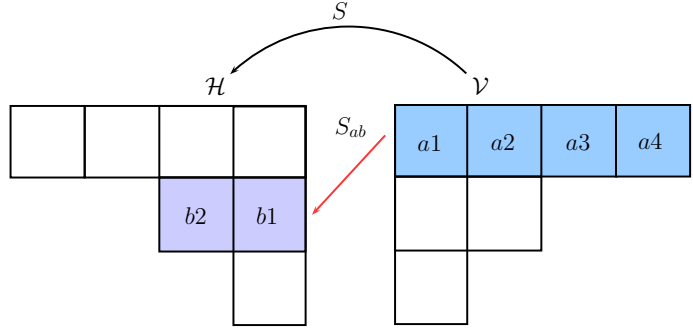
\begin{figure}[h!t]
\begin{center}
\scalebox{0.65} 
{
\begin{pspicture}(0,-2.4091992)(13.79,3.21191992)
\definecolor{color2367b}{rgb}{0.6,0.8,1.0}
\definecolor{color2368b}{rgb}{0.8,0.8,1.0}
\definecolor{color2512}{rgb}{1.0,0.2,0.2}
\rput{-270.0}(10.910801,-9.169199){\psframe[linewidth=0.04,dimen=outer,fillstyle=solid,fillcolor=color2367b](10.79,3.1308007)(9.29,-1.3891993)}
\psframe[linewidth=0.04,dimen=outer,fillstyle=solid,fillcolor=color2367b](13.79,1.6208007)(12.25,0.12080078)
\rput{-270.0}(10.150801,-8.409199){\psframe[linewidth=0.04,dimen=outer,fillstyle=solid,fillcolor=color2367b](10.03,2.3708007)(8.53,-0.6291992)}
\rput{-270.0}(4.6058006,-2.9041991){\psframe[linewidth=0.04,dimen=outer](4.505,3.1058009)(3.005,-1.4041992)}
\rput{-270.0}(3.8758008,-5.1441994){\psframe[linewidth=0.04,dimen=outer,fillstyle=solid,fillcolor=color2368b](5.285,0.8658008)(3.735,-2.1341991)}
\rput{-270.0}(4.600801,-5.919199){\psframe[linewidth=0.04,dimen=outer](7.51,0.090800785)(3.01,-1.4091992)}
\rput{-270.0}(5.3608007,-3.6591992){\psframe[linewidth=0.04,dimen=outer](5.26,2.3508008)(3.76,-0.64919925)}
\psframe[linewidth=0.04,dimen=outer](1.54,1.6008008)(0.0,0.10080078)
\rput{-270.0}(8.665801,-9.894199){\psframe[linewidth=0.04,dimen=outer](10.055,0.8858008)(8.505,-2.1141992)}
\rput{-270.0}(7.8958006,-9.164199){\psframe[linewidth=0.04,dimen=outer](10.785,0.11580078)(6.275,-1.3841993)}
\fontsize{14}{0}
\usefont{T1}{ptm}{m}{n}
\rput(6.681455,3.5258007){$S$}
\usefont{T1}{ptm}{m}{n}
\rput(4.171455,1.9858007){$\mathcal{H}$}
\usefont{T1}{ptm}{m}{n}
\rput(9.531455,1.9858007){$\mathcal{V}$}
\psarc[linewidth=0.04,arrowsize=0.05291667cm 2.0,arrowlength=1.4,arrowinset=0.4]{->}(6.86,-0.23919922){3.44}{46.206047}{134.17566}
\psline[linewidth=0.04cm,linecolor=color2512,arrowsize=0.05291667cm 2.0,arrowlength=1.4,arrowinset=0.4]{->}(7.6,0.9808008)(6.12,-0.6391992)
\usefont{T1}{ptm}{m}{n}
\rput(6.901455,1.1458008){$S_{ab}$}
\usefont{T1}{ptm}{m}{n}
\rput(8.501455,0.8258008){$a1$}
\usefont{T1}{ptm}{m}{n}
\rput(5.211455,-0.65419924){$b1$}
\usefont{T1}{ptm}{m}{n}
\rput(3.771455,-0.65419924){$b2$}
\usefont{T1}{ptm}{m}{n}
\rput(9.981455,0.84580076){$a2$}
\usefont{T1}{ptm}{m}{n}
\rput(11.501455,0.84580076){$a3$}
\usefont{T1}{ptm}{m}{n}
\rput(12.961455,0.8658008){$a4$}
\end{pspicture}
} 
\end{center}
\caption{The $4\times 2$ block $S_{ab}$ of the map $S$.} \label{f:y1}
\end{figure}

\begin{definition}
A smooth family of Darboux frames $\{E_{ai}(t),F_{ai}(t)\}_{ai \in D}$ is called a \emph{moving frame} of a monotonically nonincreasing curve $J(\cdot)$ with Young diagram $D$ if $J(t) = \spn\{E_{ai}(t)\}_{ai\in D}$ for any $t$, and there exists a one-parametric family of $n \times n$ symmetric matrices $R(t)$ such that the moving frame satisfies the \emph{structural equations}
\begin{align}
\dot{E}_{ai}(t) &= E_{a(i-1)}(t), &a = 1,\dots,k, &\,  i = 2,\dots, n_a, \\
\dot{E}_{a1}(t) &= - F_{a1}(t), &a = 1,\dots,k, \\
\dot{F}_{ai}(t) &=  \displaystyle\sum_{b=1}^k \sum_{j=1}^{n_b} R_{ab,ij}(t)E_{bj}(t) - F_{a(i+1)}(t), & a=1,\dots,k, &\, i = 1,\dots,n_a-1,\\
\dot{F}_{an_a}(t) &= \displaystyle\sum_{b=1}^k \sum_{j=1}^{n_b} R_{ab,n_aj}(t)  E_{bj}(t),  & a = 1,\dots,k.
\end{align}
\end{definition}
Notice that the matrix $R(t)$ is labelled according to the convention introduced above. At the end of this section, we also find a formula which connects the curvature operator $\RR_\lam$ of Definition~\ref{d:curv} with some of the symplectic invariants $R(t)$ of the Jacobi curve (see Eq.~\eqref{eq:Rlam}).

\subsection{On the existence and uniqueness of the moving frame} \label{s:uniquenesscanonical} \index{canonical frame!existence and uniqueness}
{\review 
The moving frame for curves in a Lagrange Grassmannian has been introduced for the first time in \cite{lizel}. In the aforementioned reference, the authors prove that such a frame always exists. Moreover, by requiring some algebraic condition on the family $R(t)$, the authors also proved that the moving frame is unique up to orthogonal transformations which, in a sense, preserve the structure of the Young diagram. In this case, the family $R(t)$ (which is said to be \emph{normal}) can be associated with a well defined operator which, together with the Young diagram $D$, completely classify the curve up to symplectic transformations. 
\begin{definition}\label{d:normalmov}
A moving frame $\{E_{ai}(t),F_{ai}(t)\}_{ai \in D}$ such that the family of symmetric matrices $R(t)$ is normal in the sense of \cite{lizel} is called \emph{canonical frame} (or \emph{normal moving frame}).\index{canonical frame!normal conditions}
\end{definition}
See Appendix \ref{a:normal} for the explicit statement of the normal conditions on the family $R(t)$.

In order to state more precisely the uniqueness property of the canonical frame we need to introduce the superboxes of a Young diagram. We say that two boxes $ai,bj \in D$ belong to the same \emph{superbox} of the Young diagram $D$ if and only if $ai$ and $bj$ are in the same column of $D$ and in possibly distinct row but with same length, i.e. if and only if $i=j$ and $n_a = n_b$. We use greek letters $\alpha,\beta,\dots$ to denote superboxes. The \emph{size} of a superbox $\alpha$ is the number of boxes included in $\alpha$. The Young diagram $D$ is then partitioned into superboxes of (possibly) different sizes. See Fig.~\ref{f:Yd2} for an example of such a partition in superboxes.
\begin{figure}
\psset{unit=0.5} 
{
\begin{pspicture}(0,-3.25)(19.888594,3.25)
\definecolor{colour0}{rgb}{1.0,0.4,0.4}
\definecolor{colour1}{rgb}{0.6,0.8,1.0}
\definecolor{colour2}{rgb}{0.4,1.0,0.4}
\psframe[linecolor=black, linewidth=0.1, fillstyle=solid,fillcolor=colour0, dimen=middle](4.4,0.0)(2.8,-1.6)
\psframe[linecolor=black, linewidth=0.1, fillstyle=solid,fillcolor=colour1, dimen=middle](4.4,3.2)(2.8,0.0)
\psframe[linecolor=black, linewidth=0.1, fillstyle=solid,fillcolor=colour1, dimen=middle](2.8,3.2)(1.2,0.0)
\psframe[linecolor=black, linewidth=0.1, fillstyle=solid,fillcolor=colour0, dimen=middle](2.8,0.0)(1.2,-1.6)
\psframe[linecolor=black, linewidth=0.1, fillstyle=solid,fillcolor=colour1, dimen=middle](6.0,3.2)(4.4,0.0)
\psframe[linecolor=black, linewidth=0.1, fillstyle=solid,fillcolor=colour2, dimen=middle](2.8,-1.6)(1.2,-3.2)
\psline[linecolor=black, linewidth=0.02](1.2,1.6)(6.0,1.6)(6.0,1.6)(6.0,1.6)
\rput(0.0,0.0){a)}
\psline[linecolor=black, linewidth=0.02](6.4,3.2)(6.8,3.2)(6.8,0.0)(6.4,0.0)(6.8,0.0)(6.8,-1.6)(6.4,-1.6)(6.8,-1.6)(6.8,-3.2)(6.4,-3.2)
\rput(7.6,1.6){2}
\rput(7.6,-0.8){1}
\rput(7.6,-2.4){$1$}
\psframe[linecolor=black, linewidth=0.02, fillstyle=solid,fillcolor=colour1, dimen=middle](16.4,3.2)(14.8,1.6)
\psframe[linecolor=black, linewidth=0.02, dimen=middle](18.0,3.2)(16.4,1.6)
\psframe[linecolor=black, linewidth=0.02, fillstyle=solid,fillcolor=colour1, dimen=middle](18.0,3.2)(16.4,1.6)
\psframe[linecolor=black, linewidth=0.02, fillstyle=solid,fillcolor=colour0, dimen=middle](16.4,0.0)(14.8,-1.6)
\psframe[linecolor=black, linewidth=0.02, fillstyle=solid,fillcolor=colour0, dimen=middle](16.4,1.6)(14.8,0.0)
\psframe[linecolor=black, linewidth=0.02, fillstyle=solid,fillcolor=colour0, dimen=middle](16.4,-1.6)(14.8,-3.2)
\psline[linecolor=black, linewidth=0.1](14.8,1.6)(18.0,1.6)(18.0,3.2)(16.4,3.2)(16.4,-3.2)(14.8,-3.2)(14.8,3.2)(16.4,3.2)
\rput(13.6,0.0){b)}
\psline[linecolor=black, linewidth=0.02](18.4,3.2)(18.8,3.2)(18.8,1.6)(18.4,1.6)(18.8,1.6)(18.8,-3.2)(18.4,-3.2)(18.8,-3.2)(18.8,-3.2)
\rput(19.6,-0.8){$3$}
\rput(19.6,2.4){1}
\end{pspicture}
}
\caption{Examples of superboxes of a Young diagram for a growth vector a) $\mc{G}_{\gamma}=\{4,5,7\}$ and b) $\mc{G}_\gamma =\{4,5\}$. Superboxes are the groups of boxes delimited by a thick boundary. Superboxes with different size, displayed on the right of each diagram, are painted with different colours.}\label{f:Yd2}
\end{figure}
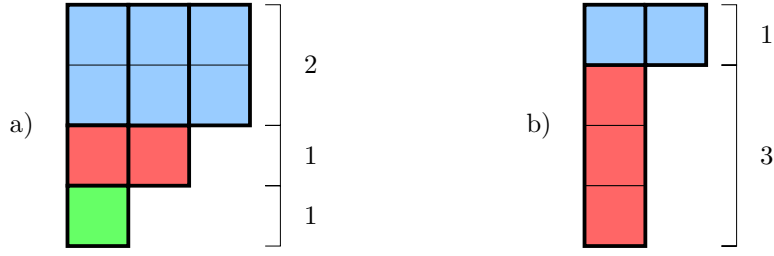

\begin{theorem}[see {\cite[Theorem 1]{lizel}}]\label{t:canuniqueness}
For any monotone nonincreasing ample and equiregular curve $J(\cdot)$ in the Lagrange Grassmannian with Young diagram $D$ there exists a normal moving frame $\{E_{ai}(t), F_{ai}(t)\}_{ai\in D}$. A moving frame $\{\til{E}_{ai}(t), \til{F}_{ai}(t)\}_{ai\in D}$ is a normal moving frame of the curve $J(\cdot)$ if and only if for any superbox $\alpha$ of size $r$ there exists a constant orthogonal $r\times r$ matrix $O^\alpha$ such that
\begin{equation}
\wt{E}_{ai}(t) = \sum_{bj \in \alpha} O^\alpha_{ai,bj} E_{bj}(t), \qquad \wt{F}_{ai}(t) = \sum_{bj \in \alpha} O^\alpha_{ai,bj} F_{bj}(t), \qquad \all ai \in \alpha.
\end{equation}
\end{theorem}
Thus, the canonical frame is unique up to orthogonal transformations that preserve the superboxes of the Young diagram.
\finereview}

\section{Main result}

Fix a canonical frame, associated with $J(\cdot)$. Let $\ve= \spn\{E_{ai}(0)\}_{ai \in D}$ be the \emph{vertical} subspace, and $\hor=\spn\{F_{bj}(0)\}_{bj\in   D}$ be the \emph{horizontal} subspace of $\Sigma$. Observe that $\ve = J(0)$. 
The splitting $\Sigma = \ve \oplus \hor$ induces a coordinate chart in $L(\Sigma)$, such that $J(t) = \{(p,S(t)p)|\, p\in \mathbb{R}^n\}$. Recall that $S(0) =0$ and, being the curve ample, $S(t)$ is invertible for $|t|< \eps$ (see Lemma~\ref{l:ampletrasv}).



We introduce the constant $n\times n$ symmetric matrices, $\wh{S}$, its inverse $\wh{S}^{-1}$ and $C$, defined by
\begin{gather}
\wh{S}_{ab,ij} =\frac{ \delta_{ab} (-1)^{i+j-1}}{(i-1)!(j-1)!(i+j-1)},  \\
\wh{S}^{-1}_{ab,ij} = \frac{ - \delta_{ab}}{i+j-1} \binom{n_a + i -1}{i-1}\binom{n_b + j -1}{j-1}  \frac{(n_a)! (n_b)!}{(n_a-i)!(n_b-j)!}, \\
C_{ab,ij} =  \frac{(-1)^{i+j}(i+j+2)}{(i-1)!(j-1)!(i+j+1)(i+1)(j+1)}.
\end{gather}
where, as usual, $a,b=1,\dots,k$, $i=1,\dots,n_a$, $j=1,\dots,n_b$.
\begin{theorem}\label{t:Sasymptotic}
Let $J(\cdot)$ be a monotone, ample, equiregular curve of rank $k$, with a given Young diagram $D$ with $k$ rows, of length $n_a$, for $a=1,\dots,k$. Then, for $|t| < \eps$
\begin{equation}\label{eq:S}
S_{ab,ij}(t) = \wh{S}_{ab,ij} t^{i+j-1} - R_{ab,11}(0) C_{ab,ij} t^{i+j+1} + O(t^{i+j+2}).
\end{equation}
Moreover, for $0 < |t| < \eps$, the following asymptotic expansion holds for the inverse matrix:
\begin{equation}\label{eq:Sinv}
S^{-1}_{ab,ij}(t) =  \frac{\wh{S}^{-1}_{ab,ij}}{t^{i+j-1}} + R_{ab,11}(0)  \frac{(\wh{S}^{-1}C \wh{S}^{-1})_{ab,ij}}{t^{i+j-3}} + O\left(\frac{1}{t^{i+j-4}}\right).
\end{equation}
\end{theorem}
Eqs.~\eqref{eq:S} and~\eqref{eq:Sinv} highlight the block structure of the $S$ matrix and its inverse at the leading orders. In particular, they give the leading order of the principal part of $S^{-1}$ on the diagonal blocks (i.e. when $a=b$). The leading order terms of the diagonal blocks of $S$ (and its inverse $S^{-1}$) only depend on the structure of the given Young diagram. Indeed the dependence on $R(t)$ appears in the higher order terms of Eqs.~\eqref{eq:S} and~\eqref{eq:Sinv}.

\subsection{Restriction}
At the end of this section, we apply Theorem~\ref{t:Sasymptotic} to compute the expansion of the family of operators $\QQ_\lam(t)$. According to the discussion that follows Eq.~\eqref{eq:sing}, we only need a block of the matrix $S(t)^{-1}$, namely $\Sred(t)^{-1}$. As we explain below, it turns out that this corresponds to consider only the restriction of $S^{-1}$ to the first columns of the Young diagram $D$ and $\overline{D}$ (see Fig.~\ref{fig:reduction}).
\begin{figure}[h!t]
\begin{center}
\scalebox{0.65} 
{
\begin{pspicture}(0,-4.1380467)(13.791894,4.1380467)
\definecolor{color2367b}{rgb}{0.6,0.8,1.0}
\definecolor{color2368b}{rgb}{0.8,0.8,1.0}
\rput{-90.0}(5.5603514,4.9796486){\psframe[linewidth=0.04,dimen=outer,fillstyle=solid,fillcolor=color2367b](7.52,0.45964843)(3.02,-1.0403515)}
\psframe[linewidth=0.04,dimen=outer,fillstyle=solid,fillcolor=color2368b](8.12,1.9596485)(6.62,-2.5403516)
\psframe[linewidth=0.04,dimen=outer](9.62,1.9596485)(6.62,-1.0403515)
\psframe[linewidth=0.04,dimen=outer](11.12,1.9596485)(6.62,0.45964843)
\rput{-90.0}(2.5403516,4.9596486){\psframe[linewidth=0.04,dimen=outer](4.5,3.4796484)(3.0,-1.0603516)}
\rput{-90.0}(4.0603514,4.9796486){\psframe[linewidth=0.04,dimen=outer](6.02,1.9596485)(3.02,-1.0403515)}
\rput{-180.0}(12.78,-5.1007032){\psarc[linewidth=0.04,arrowsize=0.05291667cm 2.0,arrowlength=1.4,arrowinset=0.4]{->}(6.39,-2.5503516){0.95}{25.974394}{154.0256}}
\fontsize{14}{0}
\usefont{T1}{ptm}{m}{n}
\rput(6.511455,-3.9153516){$(\Sred)^{-1}$}
\usefont{T1}{ptm}{m}{n}
\rput(6.371455,3.9446485){$S^{-1}$}
\usefont{T1}{ptm}{m}{n}
\rput(3.751455,2.3846483){$\mathcal{H}$}
\usefont{T1}{ptm}{m}{n}
\rput(9.041455,2.3846483){$\mathcal{V}$}
\psarc[linewidth=0.04,arrowsize=0.05291667cm 2.0,arrowlength=1.4,arrowinset=0.4]{<-}(6.44,0.15964843){3.44}{46.206047}{134.17566}
\psframe[linewidth=0.04,dimen=outer](1.52,1.9596485)(0.0,0.45964843)
\psframe[linewidth=0.04,dimen=outer](12.64,1.9596485)(11.08,0.45964843)
\usefont{T1}{ptm}{m}{n}
\rput(7.381455,1.1246485){$a1$}
\usefont{T1}{ptm}{m}{n}
\rput(7.371455,-0.37535155){$b1$}
\usefont{T1}{ptm}{m}{n}
\rput(7.341455,-1.8553516){$c1$}
\usefont{T1}{ptm}{m}{n}
\rput(5.281455,1.1046485){$a1$}
\usefont{T1}{ptm}{m}{n}
\rput(5.2714553,-0.39535156){$b1$}
\usefont{T1}{ptm}{m}{n}
\rput(5.241455,-1.8753515){$c1$}
\end{pspicture} 
}
\caption{The block $\Sred(t)^{-1}$ of the map $S(t)^{-1}$. Namely $(\Sred)^{-1}_{ab}=S^{-1}_{ab,11}$.}\label{fig:reduction}
\end{center}
\end{figure}
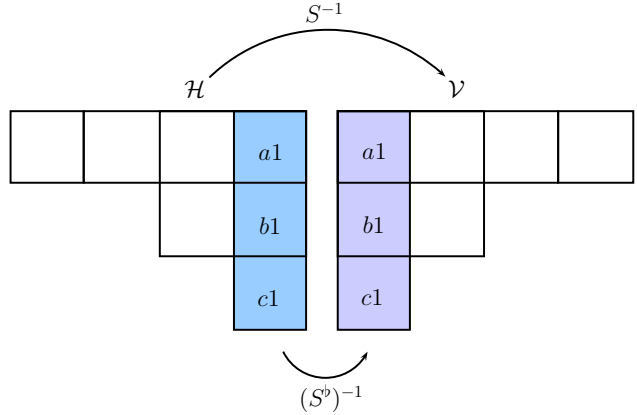
In terms of the frame $\{F_{a1}(0),E_{a1}(0)\}_{a=1}^k$, the map $\Sred(t)^{-1}$ is a $k\times k$ matrix, with entries $\Sred(t)^{-1}_{ab}=(S^{-1})_{ab,11}$. The following corollary is a consequence of Theorem~\ref{t:Sasymptotic}, and gives the principal part of the aforementioned block.

\begin{corollary}\label{c:Sridasymptotic}
Let $J(\cdot)$ be a monotone, ample, equiregular curve of rank $k$, with a given Young diagram $D$ with $k$ rows, of length $n_a$, for $a=1,\dots,k$. 
Then, for $0 < |t| < \eps$
\begin{equation}\label{eq:Sridinv}
\Sred(t)^{-1}_{ab} = -\delta_{ab} \frac{n_a^2}{t} + R_{ab,11}(0) \Omega(n_a,n_b) t + O(t^2),
\end{equation}
where
\begin{equation}\label{eq:coefficient}
\Omega(n_a,n_b)= \begin{cases}
0 & |n_a-n_b|\geq 2,  \\
\frac{1}{4(n_a+n_b)} & |n_a-n_b| = 1, \\
\frac{n_a}{4n_a^2-1} & n_a = n_b.
\end{cases}
\end{equation}
\end{corollary}
\begin{remark}
If the Young diagram consists in a single column, with $n$ boxes, $n_a = 1$ for all $a=1,\dots,n$ and
\begin{equation}
\Sred(t)^{-1}_{ab} = - \frac{\delta_{ab}}{t} + \frac{1}{3}R_{ab}(0) t + O(t^2).
\end{equation} 
\end{remark}

\subsection{A remark on the coefficients}\label{s:remarkcoefficients}

Let us discuss the consequences of the peculiar form of the coefficients of Eq.~\eqref{eq:coefficient}. If $|n_a-n_b|\geq 2$, $\Omega(n_a,n_b) = 0$ and the corresponding $R_{ab,11}$ does not appear in the first order asymptotic. Nevertheless, if we assume that $R(t)$ is a \emph{normal family} in the sense of \cite{lizel}, the \virg{missing} entries are precisely the ones that vanish due to the assumptions on $R(t)$. It is natural to expect that some of the $R_{ab,ij}$ do not appear also in the higher orders of the asymptotic expansion. This may suggest the algebraic conditions to enforce on a generic family $R_{ab,ij}$ in order to obtain a truly canonical moving frame for the Jacobi curve (see also Section~\ref{s:uniquenesscanonical}). 


\subsection{Examples}
In this section we provide two practical examples of the asymptotic form of $\Sred(t)^{-1}$. We suppress the subscript \virg{$_{11}$} and the evaluation at $t=0$ from each entry $R_{ab,11}(0)$.
\paragraph{A)}
\ytableausetup{smalltableaux}
Consider the $3$-dimensional Jacobi curve with Young diagram: \ytableaushort{\empty\empty,\empty}\begin{equation}
\Sred(t)^{-1} = - \frac{1}{t}\begin{pmatrix}
4 & 0 \\
0 & 1
\end{pmatrix} + \frac{1}{3}\begin{pmatrix} 
\frac{2}{5}R_{11} & \frac{1}{4}R_{12} \\
\frac{1}{4}R_{21} & R_{22}
\end{pmatrix} t
+ O(t^2).
\end{equation}
This corresponds to the case of the Jacobi curve associated with the geodesics of a 3D contact sub-Riemannian structure (see Section~\ref{s:3Dcomputations}).
\paragraph{B)} 
Consider the diagram: \ytableaushort{\empty\empty\empty,\empty\empty,\empty} 
\begin{equation}
\Sred(t)^{-1} = - \frac{1}{t}\begin{pmatrix}
9 & 0 & 0 \\
0 & 4 & 0 \\
0 & 0 & 1
\end{pmatrix} + \frac{1}{3}\begin{pmatrix}
\frac{9}{35}R_{11} & \frac{3}{20}R_{12} & 0 \\
\frac{3}{20}R_{21} & \frac{2}{5}R_{22} & \frac{1}{4}R_{23} \\
0 & \frac{1}{4}R_{23} &  R_{33}
\end{pmatrix} t
+ O(t^2).
\end{equation}
This corresponds to the case of the Jacobi curve associated with a generic ample geodesics of a $(3,6)$ Carnot group. In this example we can appreciate that some of the $R_{ab,11}$ do not appear in the linear term of the reduced matrix. 

\section{Proof of Theorem~\ref{t:Sasymptotic}}
The proof boils down to a careful manipulation of the structural equations, and matrices inversions. We prove Theorem~\ref{t:Sasymptotic} in three steps. 
\begin{enumerate}
\item First, we consider the case of a rank $1$ curve, and we assume $R(t) = 0$. In this case, the Young diagram is a single row and the structural equations are very simple. The canonical frame at time $t$ is a polynomial in terms of the canonical frame at $t=0$, and we compute explicitly the matrix $S(t)$ and its inverse. 

\item Then, we consider a general rank $1$ curve. The canonical frame at time $t$ is no longer a polynomial in terms of the canonical frame at $t=0$, but we can control the higher order terms. The non-vanishing $R(t)$ gives a contribution of higher order in $t$ in each entry of the matrix $S(t)$ and its inverse.

\item Finally, we consider a general rank $k$ curve. We show that, at the leading orders, we can ``split'' the curve in $k$ rank $1$ curves, and employ the results of the previous steps.
\end{enumerate}

\subsection{Rank \texorpdfstring{$1$}{1} curve with vanishing \texorpdfstring{$R(t)$}{R(t)}}

With these assumptions, the canonical frame is $\{E_i(t),F_i(t)\}_{i=1}^n$ (we suppress the row index, as $D$ has a single row). The structural equations are
\begin{align*}
\dot{E}_1(t) &=  - F_1(t),  & \dot{F}_1(t) &= -F_2(t),\\
\dot{E}_2(t) &= E_1(t), & \dot{F}_2(t) &= -F_3(t),\\
&\vdots & &  \vdots \\
\dot{E}_n(t) &= E_{n-1}(t),  & \dot{F}_n(t) &= 0.
\end{align*}
Pictorially, in the double Young diagram the derivative shifts each element of the frame to the left by one box (see Fig.~\ref{fig:shift}).
 
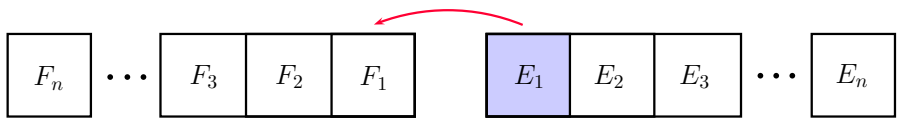
\begin{figure}[h!t]
\centering
\scalebox{0.75} 
{
\begin{pspicture}(0,1.423047)(15.6,3.46)
\definecolor{color2368b}{rgb}{0.8,0.8,1.0}
\definecolor{color883}{rgb}{1.0,0.0,0.2}
\psframe[linewidth=0.04,dimen=outer,fillstyle=solid,fillcolor=color2368b](9.9,3.04)(8.4,1.54)
\psframe[linewidth=0.04,dimen=outer](11.38,3.04)(8.4,1.54)
\psframe[linewidth=0.04,dimen=outer](12.9,3.04)(8.4,1.54)
\psarc[linewidth=0.04,linecolor=color883,arrowsize=0.05291667cm 2.0,arrowlength=1.4,arrowinset=0.4]{->}(7.74,0.0){3.44}{67.583855}{112.16634}
\psframe[linewidth=0.04,dimen=outer](7.18,3.04)(5.68,1.54)
\psframe[linewidth=0.04,dimen=outer](7.18,3.04)(4.18,1.54)
\psframe[linewidth=0.04,dimen=outer](7.18,3.04)(2.68,1.54)
\fontsize{14}{0}
\usefont{T1}{ptm}{m}{n}
\rput(9.171455,2.265){$E_1$}
\usefont{T1}{ptm}{m}{n}
\rput(10.571455,2.265){$E_2$}
\usefont{T1}{ptm}{m}{n}
\rput(12.071455,2.265){$E_3$}
\psdots[dotsize=0.1](13.22,2.28)
\psdots[dotsize=0.1](13.5,2.28)
\psdots[dotsize=0.1](13.78,2.28)
\psdots[dotsize=0.1](2.4,2.26)
\psdots[dotsize=0.1](2.12,2.26)
\psdots[dotsize=0.1](1.82,2.26)
\psframe[linewidth=0.04,dimen=outer](1.5,3.04)(0.0,1.54)
\psframe[linewidth=0.04,dimen=outer](15.6,3.04)(14.1,1.54)
\usefont{T1}{ptm}{m}{n}
\rput(14.831455,2.285){$E_n$}
\usefont{T1}{ptm}{m}{n}
\rput(6.441455,2.245){$F_1$}
\usefont{T1}{ptm}{m}{n}
\rput(4.961455,2.265){$F_2$}
\usefont{T1}{ptm}{m}{n}
\rput(3.461455,2.265){$F_3$}
\usefont{T1}{ptm}{m}{n}
\rput(0.7414551,2.245){$F_n$}
\end{pspicture} 
}
\caption{The action of the derivative on $E_1$.}\label{fig:shift}
\end{figure}
 
Let $E(t) = (E_1,\dots,E_n)^*$ and $F(t)=(F_1,\dots,F_n)^*$, where each element is computed at $t$. Then there exist one parameter families of $n\times n$ matrices $A(t), B(t)$ such that
\begin{equation}
E(t) = A(t) E(0)  +  B(t) F(0).
\end{equation}
$A(t)$ and $B(t)$ have monomial entries w.r.t. $t$. For $i,j = 1,\dots,n$
\begin{equation}\label{Arank1}
A_{ij}(t) = \frac{t^{i-j}}{(i-j)!} = \wh{A}_{ij} t^{i-j}, \qquad (i\geq j),
\end{equation}
\begin{equation}\label{Brank1}
B_{ij}(t) = \frac{(-1)^{j}t^{i+j-1}}{(i+j-1)!} = \wh{B}_{ij} t^{i+j-1}.
\end{equation}
Observe that $A$ is a lower triangular matrix. A straightforward computation shows that
\begin{equation}\label{Ainvrank1}
A^{-1}_{ij}(t) = \frac{(-1)^{i-j}t^{i-j}}{(i-j)!} = \wh{A}^{-1}_{ij} t^{i-j}, \qquad (i\geq j).
\end{equation}
Eqs.~\eqref{Arank1},~\eqref{Brank1} and~\eqref{Ainvrank1} implicitly define the constant matrices $\wh{A}$, $\wh{B}$ and $\wh{A}^{-1}$. The matrix $S(t)$ can be computed directly in terms of $A(t)$ and $B(t)$. Indeed $S(t) = A(t)^{-1}B(t)$.
\bp[Special case of Theorem~\ref{t:Sasymptotic}]\label{prop:rank1}
Let $J(\cdot)$ a curve of rank $1$, with vanishing $R(t)$. The matrix $S(t)$, in terms of a canonical frame, is
\begin{equation}\label{Srank1}
S(t)_{ij} = \frac{(-1)^{i+j-1}}{(i-1)!(j-1)!}\frac{t^{i+j-1}}{(i+j-1)} = \wh{S}_{ij} t^{i+j-1}.
\end{equation}
Its inverse is
\begin{equation}\label{Sinvrank1}
S^{-1}(t)_{ij} = 
 \frac{ - 1}{i+j-1} \binom{n + i -1}{i-1}\binom{n+ j -1}{j-1}  \frac{(n!)^2}{(n-i)!(n-j)!} t^{-i-j+1}= \frac{\wh{S}^{-1}_{ij}}{t^{i+j-1}}.
\end{equation}
\ep
As expected, $S(t)$ is symmetric, since the canonical frame is Darboux. The proof of Proposition~\ref{prop:rank1} is a straightforward but long computation, which can be found in \appendixname~\ref{prop:rank1_proof}.  Eqs.~\eqref{Srank1}~\eqref{Sinvrank1} implicitly define the constant matrix $\wh{S}$ and its inverse $\wh{S}^{-1}$. Observe that the entries of the latter depend explicitly on the dimension $n$.

\subsection{General rank \texorpdfstring{$1$}{1} curve}
Now consider a general rank $1$ curve. Its Young diagram is still a single row but, in general, $R(t) \neq 0$. As a consequence, the elements of the moving frame are no longer polynomial in $t$. However, we can still expand each $E_i(t)$ and obtain a Taylor approximation of its components w.r.t. the frame at $t=0$. 
Each derivative at $t=0$, up to order $i-1$, is still a vertical vector
\begin{equation}
\frac{d^{k} E_i}{d t^k}(0) = E_{i-k}(0),\qquad k=0,\dots, i-1.
\end{equation}
The $i$-th derivative at $t=0$ gives the lowest order horizontal term, i.e.
\begin{equation}
\frac{d^i E_i}{d t^i}(0) =  - F_1(0).
\end{equation}
Henceforth, each additional derivative, computed at $t=0$, gives higher order horizontal terms, but also new vertical terms, depending on $R(t)$. Let us see a particular example, for $E_1(t)$. $\dot{E}_1(0) =  - F_1(0)$, and $\ddot{E}_1(0) = F_2(0) - \sum_{j=1}^n R_{1j}(0)E_j(0)$ (see Fig.~\ref{fig:rnonzero}).

\begin{figure}[h!t]
\begin{center}
\scalebox{0.75} 
{
\begin{pspicture}(0,2.8223048)(15.6,5.6)
\definecolor{color2368b}{rgb}{0.8,0.8,1.0}
\definecolor{color883}{rgb}{1.0,0.0,0.2}
\psarc[linewidth=0.04,linecolor=color883,arrowsize=0.05291667cm 2.0,arrowlength=1.4,arrowinset=0.4]{<-}(7.8,1.4){3.44}{67.83366}{112.416145}
\psarc[linewidth=0.04,linecolor=color883,arrowsize=0.05291667cm 2.0,arrowlength=1.4,arrowinset=0.4]{<-}(8.4,1.22){3.88}{59.743565}{119.65911}
\psarc[linewidth=0.04,linecolor=color883,arrowsize=0.05291667cm 2.0,arrowlength=1.4,arrowinset=0.4]{<-}(9.23,0.41){4.99}{56.118736}{122.94922}
\psarc[linewidth=0.04,linecolor=color883,arrowsize=0.05291667cm 2.0,arrowlength=1.4,arrowinset=0.4]{->}(5.69,3.59){1.27}{52.471558}{129.34052}
\psframe[linewidth=0.04,dimen=outer](9.9,4.48)(8.4,2.98)
\psframe[linewidth=0.04,dimen=outer](11.38,4.48)(8.4,2.98)
\psframe[linewidth=0.04,dimen=outer](12.9,4.48)(8.4,2.98)
\fontsize{14}{0}
\psframe[linewidth=0.04,dimen=outer,fillstyle=solid,fillcolor=color2368b](7.18,4.48)(5.68,2.98)
\psframe[linewidth=0.04,dimen=outer](7.18,4.48)(4.18,2.98)
\psframe[linewidth=0.04,dimen=outer](7.18,4.48)(2.68,2.98)
\usefont{T1}{ptm}{m}{n}
\rput(9.171455,3.705){$E_1$}
\usefont{T1}{ptm}{m}{n}
\rput(10.651455,3.725){$E_2$}
\usefont{T1}{ptm}{m}{n}
\rput(12.071455,3.705){$E_3$}
\psdots[dotsize=0.1](13.22,3.72)
\psdots[dotsize=0.1](13.5,3.72)
\psdots[dotsize=0.1](13.78,3.72)
\psdots[dotsize=0.1](2.4,3.7)
\psdots[dotsize=0.1](2.12,3.7)
\psdots[dotsize=0.1](1.82,3.7)
\psframe[linewidth=0.04,dimen=outer](1.5,4.48)(0.0,2.98)
\psframe[linewidth=0.04,dimen=outer](15.6,4.48)(14.1,2.98)
\usefont{T1}{ptm}{m}{n}
\rput(14.831455,3.725){$E_n$}
\usefont{T1}{ptm}{m}{n}
\rput(6.441455,3.685){$F_1$}
\usefont{T1}{ptm}{m}{n}
\rput(4.961455,3.705){$F_2$}
\usefont{T1}{ptm}{m}{n}
\rput(3.461455,3.705){$F_3$}
\usefont{T1}{ptm}{m}{n}
\rput(0.7414551,3.685){$F_n$}
\psarc[linewidth=0.04,linecolor=color883,arrowsize=0.05291667cm 2.0,arrowlength=1.4,arrowinset=0.4]{<-}(10.48,0.0){6.08}{48.2397}{130.8777}
\end{pspicture} 
}
\caption{The action of the derivative of an horizontal element of the frame when $R\neq 0$.}\label{fig:rnonzero}
\end{center}
\end{figure}
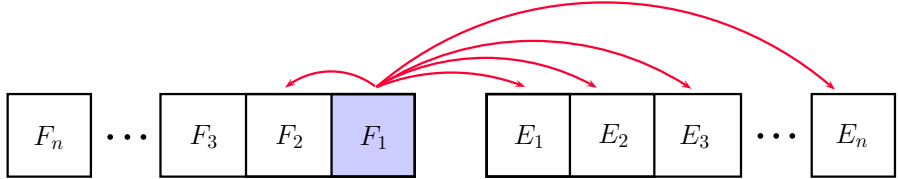

Indeed $E_1(t)$ has a zeroth order term (w.r.t. the variable $t$) in the direction $E_1(0)$. The next term in the direction $E_1(0)$ is of order $2$ or more. Besides, $E_1(t)$ has vanishing zeroth order term in each other vertical direction (i.e. $E_j(0)$, $j\neq 1$), but non vanishing components in each other vertical direction can appear, at orders greater or equal than $2$.
Let us turn to the horizontal components. $E_1(t)$ has a first order term in the direction $F_1(0)$. The next term in the same direction can appear only after two additional derivatives, or more. Therefore, the next term in the direction $F_1(0)$ is of order $3$ or more in $t$. The \virg{gaps} in the orders appearing in a given directions are precisely the key to the proof.

Let $E(t) = (E_1,\dots,E_n)^*$ and $F(t)=(F_1,\dots,F_n)^*$, where each element is computed at $t$. Then, as in the previous step, there exist one parameter families of $n\times n$ matrices $A(t)$, $B(t)$ such that
\begin{equation}
E(t) = A(t) E(0)  + B(t)  F(0).
\end{equation}
The discussion above, and a careful application of the structural equations give us asymptotic formulae for the matrices $A(t)$ and $B(t)$. Let $\wh{A}$ and $\wh{B}$ defined as in Eqs.~\eqref{Arank1}-\eqref{Brank1}, corresponding to the case of a rank $1$ curve with vanishing $R(t)$. Then, for $i,j=1,\ldots,n$
\begin{equation}\label{Arank1R}
A(t)_{ij} = \wh{A}_{ij}t^{i-j} - R_{1j}(0)\frac{ t^{i+1}}{(i+1)!} + O(t^{i+2}),
\end{equation}
\begin{equation}\label{Brank1R}
B(t)_{ij} = \wh{B}_{ij}t^{i+j-1} + R_{11}(0) \frac{(-1)^{j+1} t^{i+j+1}}{(i+j+1)!} +O(t^{i+j+2}).
\end{equation}
The matrix $A$ is no longer triangular, due to the presence of higher order terms in each entry. Besides, the order of the remainder grows only with the row index for $A(t)$ and it grows with both the column and row indices for $B(t)$. This reflects the different role played by the horizontal and vertical terms in the structural equations. We are now ready to consider the general case.



\subsection{General rank \texorpdfstring{$k$}{k} curve}

The last step, which concludes the proof of the theorem, is built upon the previous cases. It is convenient to split a frame in subframes, relative to the rows of the Young diagram. For $a=1,\dots,k$, the symbol $E_a$ denotes the $n_a$-dimensional column vector
\begin{equation}
E_a = (E_{a1},E_{a2},\dots,E_{an_a})^* \in \Sigma^{n_a},
\end{equation}
and analogously for $F_a$. Similarly, the symbol $E$ denotes the $n$-dimensional column vector
\begin{equation}
E = (E_1,\dots,E_k)^* \in \Sigma^{n},
\end{equation}
and similarly for $F$. Once again, we express the elements of the Jacobi curves $E(t)$ in terms of the canonical frame at $t=0$. With the notation introduced above
\begin{equation}
E(t)=A(t)E(0)+ B(t)F(0).
\end{equation}
This time, $A(t)$ and $B(t)$ are $k\times k$ block matrices, the $ab$ block being a $n_{a}\times n_b$ matrix. For $a,b=1,\dots,k$, $i=1,\dots,n_a$, $j=1,\dots,n_b$
\begin{equation}\label{Arankk}
A(t)_{ab,ij} = \delta_{ab} \wh{A}_{ij}t^{i-j}  - R_{ab,1j}(0)\frac{t^{i+1}}{(i+1)!}+ O(t^{i+2}),
\end{equation}
\begin{equation}\label{Brankk}
B(t)_{ab,ij} = \delta_{ab} \wh{B}_{ij}t^{i+j-1} + R_{ab,11}(0) \frac{(-1)^{j+1} t^{i+j+1}}{(i+j+1)!} + O(t^{i+j+2}),
\end{equation}
where, once again, the constant matrices $\wh{A}$, $\wh{B}$ correspond to the matrices defined for the rank $1$ and $R(t)=0$ case, of the appropriate dimension. 
Notice that we do not need explicitly the leading terms on the off-diagonal blocks. The knowledge of the leading terms on the diagonal blocks is sufficient for our purposes.

Remember that $S(t) = A(t)^{-1}B(t)$. 
In order to compute the inverse of $A(t)$ at the relevant order, we rewrite the matrix $A(t)$ as 
\begin{equation}
A(t) = \wh{A}(t) - M(t),
\end{equation}
where $\wh{A}(t)$ is the matrix corresponding to a rank $k$ curve with vanishing $R(t)$, namely
\begin{equation}
\wh{A}(t)_{ab,ij} = \delta_{ab} \wh{A}_{ij} t^{i-j},\qquad i=1,\ldots,n_a,\quad j=1,\ldots,n_b,
\end{equation}
and, from Eq.~\eqref{Arankk}, we get
\begin{equation}
M(t)_{ab,ij} =  R_{ab,1j}(0) \frac{t^{i+1}}{(i+1)!} + O(t^{i+2}).
\end{equation}
A standard inversion of the Neumann series leads to
\begin{equation}
A(t)^{-1} =  \wh{A}(t)^{-1} + \wh{A}(t)^{-1} M(t) \wh{A}(t)^{-1} + \sum_{n=2}^\infty \left(\wh{A}(t)^{-1}M(t)\right)^n \wh{A}(t)^{-1},
\end{equation}
where the reminder term in the r.h.s. converges uniformly in the operator norm small $t$. Then, a long computation gives
\begin{equation}\label{Ainvrankk}
A(t)^{-1}_{ab,ij} = \delta_{ab} \wh{A}^{-1}_{ij}t^{i-j} - R_{ab,11}(0) \frac{(-1)^i t^{i+1}}{(i+1)(i-1)!} + O(t^{i+2}).
\end{equation}
The matrix $S(t)$ can be computed explicitly, at the leading order, by the usual formula $S(t) = A(t)^{-1}B(t)$, and we obtain, for $a,b=1,\dots,k$, $i=1,\dots,n_a$, $j=1,\dots,n_b$,
\begin{equation}\label{Srankk}
S(t)_{ab,ij} = \wh{S}_{ab,ij} t^{i+j-1} - R_{ab,11}(0) C_{ab,ij} t^{i+j+1} + O(t^{i+j+2}),
\end{equation}
where $\wh{S}_{ab,ij} =\delta_{ab} \wh{S}_{ij}$ of the appropriate dimension, and 
\begin{gather}
C_{ab,ij} =  \frac{(-1)^{i+j}(i+j+2)}{(i-1)!(j-1)!(i+j+1)(i+1)(j+1)},\qquad i=1,\ldots,n_a,\quad j=1,\ldots,n_b.
\end{gather}
The computation of $S(t)^{-1}$ follows from another inversion of the Neumann series, and a careful estimate of the remainder. We obtain
\begin{equation}\label{Sinvrankk}
S^{-1}_{ab,ij}(t) =  \frac{\wh{S}^{-1}_{ab,ij}}{t^{i+j-1}} + R_{ab,11}(0)  \frac{(\wh{S}^{-1}C \wh{S}^{-1})_{ab,ij}}{t^{i+j-3}} + O\left(\frac{1}{t^{i+j-4}}\right),
\end{equation}
where
\begin{equation}
\wh{S}^{-1}_{ab,ij} = \frac{-\delta_{ab}}{i+j-1} \binom{n_a + i -1}{i-1}\binom{n_b+ j -1}{j-1} \frac{n_a! n_b!}{(n_a-i)!(n_b-j)!}.
\end{equation}
This concludes the proof of Theorem~\ref{t:Sasymptotic}. \qed

\subsection{Proof of Corollary~\ref{c:Sridasymptotic}}
Corollary~\ref{c:Sridasymptotic} follows easily from Theorem~\ref{t:Sasymptotic}. The only non-trivial part is the explicit form of the coefficient $\Omega(n_a,n_b)$ in Eq.~\eqref{eq:Sridinv}. By the results of Theorem~\ref{t:Sasymptotic},
\begin{equation}
\Omega(n_a,n_b) = (\wh{S}^{-1}C\wh{S}^{-1})_{ab,11}.
\end{equation}
By replacing the explicit expression of $\wh{S}^{-1}$ and $C$, the proof of Corollary~\ref{c:Sridasymptotic} is reduced to the following lemma, which we prove in Appendix~\ref{s:lemmacoeff}.
\begin{lemma}\label{l:lemmacoeff}
Let $\Omega(n,m)$ be defined by the formula
\begin{equation}
\Omega(n,m) = \frac{n m}{(n+1)(m+1)} \sum_{j=1}^{n}  \sum_{i=1}^{m} (-1)^{i+j}\binom{n+i-1}{i-1}\binom{n+1}{i+1}\binom{m+j-1}{j-1}\binom{m+1}{j+1} \frac{i+j+2}{i+j+1}.
\end{equation}
Then
\begin{equation}
\Omega(n,m)= \begin{cases}
0 & |n-m|\geq 2,  \\
\frac{1}{4(n+m)} & |n-m| = 1, \\
\frac{n}{4n^2-1} & n = m.
\end{cases}
\end{equation}
\end{lemma}
The proof of Corollary~\ref{c:Sridasymptotic} is now complete.\qed

\section{Proof of Theorem~\ref{t:main2}}\label{s:main2}

In this section $J_\lam: [0,T] \to L(T_\lam(T^*M))$ is the Jacobi curve associated with an ample, equiregular geodesic $\gamma$, with initial covector $\lam \in T^*_x M$. The next lemma shows that
the projection of the horizontal part of the canonical frame corresponding to the first column of the Young diagram is an orthonormal basis for the Hamiltonian product on the distribution.

\begin{lemma}\label{l:orthframe}
Let $X_a\doteq \pi_* F_{a1}(0) \in T_x M$. Then, the set $\{X_a\}_{a=1}^{k}$ is an orthonormal basis for $(\distr_x,\langle\cdot|\cdot\rangle_\lam)$. 
\end{lemma}
\begin{proof}
First, recall that $F_{a1}(0) = -\dot{E}_{a1}(0)$. Therefore $X_a = - \pi_* \dot{E}_{a1}(0)$. Then, by Eq.~\eqref{eq:scalprod}
\begin{equation}
\langle X_a| X_b \rangle_\lam = - \sigma (E_{a1}(0),\dot E_{b1}(0)) = \sigma (E_{a1}(0), F_{b1}(0)) = \delta_{ab}.  
\end{equation}
where we used the structural equations and the fact that the canonical frame is Darboux.
\end{proof}

We are now ready to prove one of the main results of Section~\ref{s:2dctdot}, namely the one concerning the spectrum of the operator $\Qz_\lam : \distr_x\to \distr_x$.
\begin{proof}[Proof of Theorem~\ref{t:main2}]
Actually, we prove something more: we use the basis $\{X_a\}_{a=1}^k$ obtained above to compute an asymptotic formula for the family $\QQ_\lam(t)$ introduced in Section~\ref{s:2dctdot}. 

Let $\Sigma = \ve_\lam \oplus \hor_\lam$ be the splitting induced by the canonical frame in $\Sigma = T_\lam (T^* M)$. Let $S(t) : \ve_\lam \to \hor_\lam$ be the map which represents the Jacobi curve in terms of the canonical splitting. Then, by definition of Jacobi curve, it follows that, for any $v \in T_x M$ (see also Eq.~\eqref{eq:sing}),
\begin{equation}
\langle \QQ_\lam(t) v|v\rangle_\lam = \frac{d}{dt}\sigma( S(t)^{-1}\wt{v},\wt{v}).
\end{equation}
where $\tilde{v}\in \hor_\lam$ is the unique horizontal lift such that $\pi_*\wt{v} = v$. In particular, if $v = \sum_{a=1}^k v_a X_a \in \distr_x$, we have $\wt{v} = \sum_{a=1}^k v_a F_{a1}(0)$. Thus, 
\begin{equation}
\langle \QQ_\lam(t)v|v\rangle_\lam = \frac{d}{dt} \sum_{a,b=1}^k S(t)^{-1}_{ab,11} v_a v_b = \frac{d}{dt}\sum_{a,b=1}^k \Sred(t)^{-1}_{ab} v_a v_b.
\end{equation}
By Corollary~\ref{c:Sridasymptotic}, we obtain the following asymptotic formula for $\QQ_\lam(t)$.
\begin{equation}\label{eq:Qasympt}
\langle \QQ_\lam(t) v|v\rangle_\lam = \sum_{a,b=1}^k \left(\delta_{ab}\frac{n_a^2}{t^2} +R_{ab,11}(0)\Omega(n_a,n_b)\right)v_a v_b + O(t).
\end{equation}
Equation~\eqref{eq:Qasympt}, together with Lemma~\ref{l:orthframe} imply that, for $a,b=1,\ldots,k$,
\begin{gather}
\Qz_\lam X_a = n_a^2 X_a, \label{eq:Qlam}\\
\RR_\lam X_a = \sum_{b=1}^k 3 R_{ab,11}(0)\Omega(n_a,n_b) X_b. \label{eq:Rlam}
\end{gather}
Equation~\eqref{eq:Qlam} completely characterizes the spectrum and the eigenvectors of $\Qz_\lam$.
\end{proof}
Equation~\eqref{eq:Rlam} is the anticipated formula which connects the curvature operator of Definition~\ref{d:curv} with some of the symplectic invariants of the Jacobi curve, namely the elements of the matrix $R_{ab,ij}$ corresponding to the first column of the Young diagram.

{\review
\section{A worked out example: 3D contact sub-Riemannian structures} \label{s:3Dcomputations} \index{sub-Riemannian!3D contact}

In this section we go through our construction for 3D contact sub-Riemannian structures. The canonical frame and the curvature for these structures have been first explicitly computed in \cite{AAPL}. For the reader's convenience, we report here the details of this construction, following our notation. In particular, we compute the canonical frame associated with ample geodesics and we present an explicit formula for the symplectic invariants $R(t)$ of the canonical frame. In turn, this recovers also the curvature operator $\RR_\lam$. Finally, we discuss the relation of the curvature with the metric invariants of a 3D contact sub-Riemannian structure, studied in \cite{agricm,agrexp,miosr3d,falbel,hughen}.

Let $M$ be a smooth manifold of dimension $\dim M = 3$. A smooth one form $\alpha$ defines a two-dimensional distribution $\distr \doteq \ker \alpha$. We say that $\alpha$ is a \emph{contact form} if $d\alpha|_{\distr}$ is not degenerate. In this case, $\distr$ is called \emph{contact distribution}. The triple $(M,\distr,\langle\cdot|\cdot\rangle)$, where $\distr$ is a contact distribution and $\langle\cdot|\cdot\rangle$ is a smooth scalar product on $\distr$ is called a (3D) \emph{contact sub-Riemannian manifold}. The non-degeneracy assumption implies that $\distr$ has constant rank and that the sub-Riemannian structure defined by $(M,\distr,\langle\cdot|\cdot\rangle)$ satisfies H\"ormander condition.

\begin{definition}
The \emph{Reeb vector field} \index{Reeb vector field} of the contact structure is the unique vector field $X_0 \in \VecM$ such that
\begin{equation}
d\alpha(X_0,\cdot) = 0, \qquad \alpha(X_0) = 1.
\end{equation}
\end{definition}
\begin{remark}
Indeed the contact form $f\alpha$ obtained by rescaling $\alpha$ with any non-vanishing $f \in C^\infty(M)$ defines the same contact distribution.  Then we choose $\alpha$ in such a way that $d\alpha|_{\distr}$ coincides with the volume form induced by the scalar product. This fixes $\alpha$ up to a global sign.
\end{remark}

\subsection{Geodesic flag and growth vector} Let $\gamma$ be any smooth admissible curve. In particular $\dot{\gamma}(t) \in \distr_{\gamma(t)}$ for all $t$. Let $\tanf$ any horizontal extension of the velocity vector $\dot{\gamma}$. Then, for any horizontal section $X \in \Gamma(\distr)$, Cartan's formula leads to
\begin{equation}
\alpha(\mc{L}_\tanf(X)) = \alpha ([\tanf,X]) = -d\alpha(\tanf,X).
\end{equation}
By definition of Reeb vector field, we obtain
\begin{equation}
\mc{L}_\tanf(X) = -d\alpha(\tanf,X) X_0 \mod \distr.
\end{equation}
The non-degeneracy assumption implies that there always exists some smooth section $X$ such that $d\alpha(\tanf,X)\neq 0$. Then, according to the alternative definition of Section~\ref{s:altdef}, the flag of any smooth admissible curve (and, in turn, of any normal geodesic) is
\begin{equation}
\DD^1_{\gamma(t)} = \distr_{\gamma(t)}, \qquad \DD^2_{\gamma(t)} = T_{\gamma(t)}M, \qquad \all t.
\end{equation}
Then, the growth vector is $\mc{G}_{\gamma(t)} = \{2,3\}$ for all $t$. In particular, any non-trivial normal geodesic is ample and equiregular, with geodesic step $m = 2$.

\subsection{The operator \texorpdfstring{$\Qz_\lambda$}{I} and geodesic dimension}
According to the above computations, any non-trivial geodesic has the following Young diagram: 
\begin{equation}
\ytableausetup{nosmalltableaux} \ytableaushort{\empty\empty,\empty},
\end{equation}
with two rows, with length $n_1 =2$ and $n_2 = 1$ respectively. By Theorem~\ref{t:main2}, we readily compute
\begin{equation}
\spec \Qz_{\lambda} = \{4,1\},
\end{equation}
for all $\lambda$ (with $H(\lambda) \neq 0$). Then the geodesic dimension is (see Section~\ref{s:gd})
\begin{equation}
\mathcal{N}_{x_0} = 5, \qquad \all x_0 \in M.
\end{equation}

\subsection{A collection of vector fields}
For any 3D contact sub-Riemannian structure we can choose a (local) orthonormal frame $X_1,X_2$ of horizontal sections, namely
\begin{equation}
\alpha(X_i) = 0, \qquad \langle X_i|X_j\rangle = \delta_{ij}, \qquad i,j=1,2.
\end{equation}
We assume that $X_1,X_2$ is oriented, namely $d\alpha(X_1,X_2) = 1$. Since $X_0$ is always transversal to the distribution, $\{X_0,X_1,X_2\}$ is a local frame of vector fields. In terms of this frame, we define the \emph{structural functions} $c_{ij}^k \in C^\infty(M)$, with $i,j,k=0,1,2$ as follows:
\begin{equation}\label{eq:strcst}
[X_i,X_j] = \sum_{k=0}^2 c_{ij}^k X_k.
\end{equation}
Observe that the following identities hold true for the structural functions as a consequence of the normalization for the contact form and the definition of Reeb vector field
\begin{equation}
c_{12}^0  =  -1 , \qquad c_{i0}^0 = 0, \qquad i=0,1,2.
\end{equation}
Consider the dual frame $\nu_0,\nu_1,\nu_2$ of one-forms. This induces coordinates $h_0,h_1,h_2$ on each fiber of $T^*M$
\begin{equation}
\lambda = (h_0,h_1,h_2) \qquad \Longleftrightarrow \qquad  \lambda = h_0 \nu_0 + h_1\nu_1 + h_2 \nu_2,
\end{equation}
where $h_i(\lambda) = \langle \lambda, X_i \rangle$ are the linear-on-fibers functions associated with $X_i$, for $i=0,1,2$.

Let $\vec{h}_i \in \mathrm{Vec}(T^*M)$ be the Hamiltonian vector fields associated with $h_i\in C^{\infty}(T^*M)$ for $i=0,1,2$, respectively. Moreover, consider the vertical vector fields $\partial_{h_i} \in \mathrm{Vec}(T^*M)$, for $i=0,1,2$. The vector fields
\begin{equation}
\vec{h}_0,\vec{h}_1,\vec{h}_2, \partial_{h_0},  \partial_{h_1},  \partial_{h_2},
\end{equation}
are a local frame of vector fields of $T^*M$. Equivalently, we can introduce cylindrical coordinates $h_0,\rho,\theta$ on each fiber of $T^*M$ by
\begin{equation}
h_1 = \rho \cos\theta, \qquad
h_2 = \rho \sin \theta,
\end{equation}
and employ instead the local frame
\begin{equation}
\vec{h}_0,\vec{h}_1,\vec{h}_2, \partial_{h_0},  \partial_{\theta}, \partial_{\rho}.
\end{equation}
Finally, let the \emph{Euler vector field} be
\begin{equation}
\euler := \sum_{i=0}^2 h_i \partial_{h_i} = \rho \partial_{\rho} +h_0 \partial_{h_0}.
\end{equation}
Notice that $\euler$ is a vertical field on $T^*M$, i.e. $\pi_* \euler = 0$, and is the generator of the dilations $\lambda \mapsto c \lambda$ along the fibers of $T^*M$.
The sub-Riemannian Hamiltonian is
\begin{equation}
H = \frac{1}{2} \left(h_1^2 + h_2^2\right).
\end{equation}
and, therefore, the Hamiltonian vector field is
\begin{equation}
\vec{H} = h_1 \vec{h}_1 + h_2 \vec{h}_2 = \rho\cos\theta \vec{h}_1 + \rho\sin\theta\vec{h}_2. 
\end{equation}
Recall that the Hamiltonian vector fields $\vec{h}_i$ associated with the functions $h_i$ are defined by the formula $dh_i = \sigma(\cdot,\vec{h}_i)$. Thanks to the structural we can write the explicit expression
\begin{equation}\label{eq:harrow}
\vec{h}_i = \wt{X}_i + \sum_{j,k=0}^2 c_{ij}^k h_k \partial_{h_j}.
\end{equation}
Finally, we introduce the following vector field $\vec{H}' \in \mathrm{Vec}(T^*M)$:
\begin{equation}
\vec{H}'\doteq [\partial_\theta, \vec{H}].
\end{equation}
A straightforward but long computation provides an explicit expression for $\vec{H}'$:
\begin{equation}\label{eq:hprime}
\vec{H}' = h_2\vec{h}_1 - h_1\vec{h}_2 - \left(\sum_{j=1}^2 c_{12}^j(\partial_\theta h_j) \right) \partial_\theta + \left(\sum_{i,j=1}^2 h_i c_{i0}^j (\partial_\theta h_j) \right) \partial_{h_0}.
\end{equation}

\subsection{The canonical frame}

We are now ready to compute the normal moving frame for 3D contact structure. Let $\lambda$ be the initial covector of some non-trivial geodesic (that is $H(\lambda) \neq 0$). We employ a lighter notation for labelling the elements of the canonical frame, different from the one introduced in Chapter~\ref{c:proof}. Instead of labelling the elements with respect to their row and columns we employ the following convention:
\begin{equation}\label{eq:notation}
\ytableausetup{nosmalltableaux} \ytableaushort{{a1}{a2},{b1}} \qquad \Rightarrow \qquad \ytableaushort{ac,b}
\end{equation}
Thus, for such a Young diagram, a canonical frame is a smooth family 
\begin{equation}
\{E_a(t),E_b(t),E_c(t),F_a(t),F_b(t),F_c(t)\} \in T_{\lambda}(T^*M),
\end{equation}
with the following properties:
\begin{itemize}
\item[(i)] it is attached to the Jacobi curve, namely $\spn\{E_a(t),E_b(t),E_c(t)\} = J_\lambda(t)$. Notice that, by definition of Jacobi curve, this implies 
\begin{equation}
\pi_* \circ e^{t\vec{H}}_* E_a(t) = \pi_* \circ e^{t\vec{H}}_* E_b(t) =\pi_* \circ e^{t\vec{H}}_* E_c(t)=  0.
\end{equation}

\item[(ii)] They satisfy the structural equations:
\begin{align}
\dot{E}_a(t) &= -F_a(t), \\
\dot{E}_b(t) & = -F_b(t), \\
\dot{E}_c(t) & = E_a(t), \\
\dot{F}_a(t) & = R_{aa}(t) E_a(t) + R_{ab}(t) E_b(t) + R_{ac}(t) E_c(t) - F_c(t), \\
\dot{F}_b(t) & = R_{ba}(t) E_a(t) + R_{bb}(t) E_b(t) + R_{bc}(t) E_c(t), \\
\dot{F}_c(t) & = R_{ca}(t) E_a(t) + R_{cb}(t) E_b(t) + R_{cc}(t) E_c(t).
\end{align}
\item[(iii)] The family of symmetric matrices $R(t)$ is \emph{normal} in the sense of \cite{lizel}. In the 3D contact case, the normality condition is:
\begin{equation}
R_{ac}(t) = R_{ca}(t) = 0.
\end{equation}
\end{itemize}
Once the canonical frame is computed, the symplectic invariants of the Jacobi curve can be obtained through the formula
\begin{equation}\label{eq:formulacurvs}
R_{ij}(t) = \sigma(\dot{F}_i(t),F_j(t)), \qquad i,j \in \{a,b,c\}.
\end{equation}

\begin{remark} In this case, all the superboxes have size $1$, and by Theorem~\ref{t:canuniqueness} the canonical frame is uniquely defined up to a sign. More precisely, a sign for the components labelled with $a,c$ and one for the components labelled with $b$, that can be chosen independently.
\end{remark}
We compute the canonical frame following the general algorithm in \cite{lizel}. \index{canonical frame!3D contact case}
\begin{proposition}\label{p:canonicalframe}
The canonical frame for a 3D contact structure is
\begin{align}
E_c(t) & = \frac{1}{\sqrt{2H}}e^{-t\vec{H}}_*\partial_{h_0}, & 
F_c(t) & = \frac{1}{\sqrt{2H}}e^{-t\vec{H}}_*\left(-[\vec{H},\vec{H}'] + R_{aa}(t) \partial_\theta\right), \\
E_a(t) & = \frac{1}{\sqrt{2H}}e^{-t\vec{H}}_*\partial_\theta, &
F_a(t) & = \frac{1}{\sqrt{2H}}e^{-t\vec{H}}_* \vec{H}', \\ 
E_b(t) & = \frac{1}{\sqrt{2H}}e^{-t\vec{H}}_* \euler, &
F_b(t) & = \frac{1}{\sqrt{2H}}e^{-t\vec{H}}_* \vec{H}.
\end{align}
The only non-vanishing entries of $R(t)$ are
\begin{align}
R_{aa}(t) & = \frac{1}{2H}\sigma([\vec{H},\vec{H}'],\vec{H}'),\\
R_{bb}(t) & =\frac{1}{2H} \sigma([\vec{H},[\vec{H},\vec{H}']],[\vec{H},\vec{H}'])- \frac{1}{(2H)^2}\sigma([\vec{H},\vec{H}'],\vec{H}')^2. \end{align}
where everything is computed along a normal extremal $\lambda(t)$.
\end{proposition}
\begin{remark}
As a consequence of the identity $[\vec{H},\euler] = -\vec{H}$ (that holds true for any quadratic-on-fibres Hamiltonian), we can rewrite
\begin{equation}
E_b(t) = \frac{1}{\sqrt{2H}} (\euler - t\vec{H}), \qquad F_b(t) = \frac{1}{\sqrt{2H}} \vec{H}.
\end{equation}
In particular, we observe that the Jacobi curve $J_\lambda(t) = \spn\{E_a(t),E_b(t),E_c(t)\}$ splits in the $\sigma$-orthogonal direct sum of two curves of subspaces of smaller dimension:
\begin{equation}
J_\lambda(t) = \spn\{\euler - t\vec{H}\} \oplus \spn\{E_a(t),E_c(t)\}.
\end{equation}
\end{remark}
\begin{proof}

The computation is presented through a sequence of lemmas. We start by proving some useful identities.
\begin{lemma}
The following identities hold true:
\begin{gather}
[\vec{H},\partial_{h_0}]  = \partial_\theta, \label{eq:comp1} \\
[\vec{H},\euler]  = -\vec{H}. \label{eq:comp2}
\end{gather}
\end{lemma}
\begin{proof}
We start with Eq.~\eqref{eq:comp1}. By using the explicit expression for $\vec{h}_i$ of Eq.~\eqref{eq:harrow} and the properties of the Lie bracket, we obtain
\begin{equation}
[\vec{H},\partial_{h_0}] = \sum_{i=1}^2 h_i[\vec{h}_i,\partial_{h_0}] = \sum_{i=1}^2 \sum_{j,k=0}^2 h_i c_{ij}^k [h_k \partial_{h_j},\partial_{h_0}] = -\sum_{i=1}^2 \sum_{j,k=0}^2 h_i c_{ij}^0 \partial_{h_j} =  h_1 \partial_{h_2} -h_2 \partial_{h_1} = \partial_\theta.
\end{equation}
For what concerns Eq.~\eqref{eq:comp2} we have
\begin{equation}
\begin{aligned}
[\vec{H},\euler] &  = \sum_{i=1}^2 \sum_{j=0}^2 [h_i\vec{h}_i,h_j\partial_{h_j}] 
 =  \sum_{i=1}^2 \sum_{j=0}^2 h_i\vec{h}_i(h_j) \partial_{h_j} - h_j\partial_{h_j}(h_i)\vec{h}_i + h_i h_j [\vec{h}_i,\partial_{h_j}]  \\
& =  \sum_{i=1}^2 \sum_{j,k=0}^2 h_i c_{ij}^k h_k \partial_{h_j} - \sum_{i=1}^2 h_i \vec{h}_i - \sum_{i=1}^2 \sum_{j,k=0}^2 h_i h_j c_{ik}^j \partial_{h_k}  = - \vec{H}.\\
\end{aligned}
\end{equation}
A more elegant proof using the fact that $\vec{H}$ is homogeneous and $\euler$ is the generator of fiber dilations is indeed possible, and can be found on \cite{nostrolibro}.
\end{proof}
\begin{lemma}
$E_c(t)$ is uniquely specified (up to a sign) by the following conditions:
\begin{itemize}
\item[(i)] $E_c(t) \in J_\lambda(t)$,
\item[(ii)] $\dot{E}_c(t) \in J_\lambda(t)$,
\item[(iii)] $\sigma(\ddot{E}_c(t),\dot{E}_c(t)) = 1$,
\end{itemize}
and, by choosing the positive sign, is given by
\begin{equation}
E_c(t) = \frac{1}{\sqrt{2H}} e^{-t\vec{H}}_* \partial_{h_0}.
\end{equation}
Moreover, one also has
\begin{equation}
E_a(t) = \frac{1}{\sqrt{2H}}e^{-t\vec{H}}_* \partial_\theta.
\end{equation}
and
\begin{equation}
F_a(t) = \frac{1}{\sqrt{2H}} e^{-t\vec{H}}_* \vec{H}'.
\end{equation}
\end{lemma}
\begin{proof}
Condition (i) and the definition of Jacobi curve $J_\lambda(t) = e^{-t\vec{H}} \ve_{\lambda(t)}$ imply that
\begin{equation}
E_c(t) = e^{-t\vec{H}}_*  \sum_{i=0}^2 a_i(t) \partial_{h_i},
\end{equation}
for some smooth functions $a_i(t)$, with $i=0,1,2$. We compute the derivative:
\begin{equation}
\dot{E}_c(t) = e^{-t\vec{H}}_* \left(\sum_{i=0}^2 a_i(t) [\vec{H},\partial_{h_i}] + \dot{a}_i(t)\partial_{h_i}\right).
\end{equation}
Condition (ii) is tantamount to $\pi_* \circ e^{t\vec{H}}_* \dot{E}_c(t) = 0$. Since $\pi_* \partial_{h_i} =0$, we obtain
\begin{equation}\label{eq:Ea1}
0 = \pi_* \sum_{i=0}^2 a_i(t) [\vec{H},\partial_{h_i}].
\end{equation}
Indeed we have, for all $i=0,1,2$
\begin{equation}
[\vec{H},\partial_{h_i}] = \sum_{j=1}^2 [h_j\vec{h}_j,\partial_{h_i}] = \sum_{j=1}^2 h_j[\vec{h}_j,\partial_{h_i}] - \sum_{i=1}^2 \delta_{ij}\vec{h}_j.
\end{equation}
Notice that any Hamiltonian vector field $\vec{h}_i$ is $\pi_*$-related with the corresponding $X_i$ (namely $\pi_* \vec{h}_i = X_i$). Moreover $\pi_*\partial_{h_i} =0$. Then we obtain
\begin{equation}
\pi_* [\vec{H},\partial_{h_i}] = \begin{cases} - X_i & i=1,2, \\ 0 & i=0. \end{cases}
\end{equation}
In particular Eq.~\eqref{eq:Ea1} implies $a_1(t) = a_2(t) = 0$. The remaining function $a_0(t)$ is obtained by condition (iii). Indeed
\begin{equation}
e^{t\vec{H}}_* \dot{E}_c(t) =  a_0(t) [\vec{H},\partial_{h_0}] + \dot{a}_0(t)\partial_{h_0} = a_0(t) \partial_\theta + \dot{a}_0(t)\partial_{h_0},
\end{equation}
where we used Eq.~\eqref{eq:comp1}. Moreover
\begin{equation}
e^{t\vec{H}}_* \ddot{E}_c(t)  = \ddot{a}_0(t) \partial_{h_0} + 2\dot{a}_0(t)  \partial_\theta - a_0(t) \vec{H}'.
\end{equation}
where we used the definition of $\vec{H}' = [\partial_\theta,\vec{H}]$. By using the explicit expression of $\vec{H}'$ of Eq.~\eqref{eq:hprime}, we rewrite condition (iii), after tedious computations, as
\begin{equation}
1 = \sigma_\lambda(\ddot{E}_c(t),\dot{E}_c(t)) = \sigma_{\lambda(t)}(e^{t\vec{H}}_* \ddot{E}_c(t),e^{t\vec{H}}_* \ddot{E}_c(t)) = a_0(t)^2 2H,
\end{equation}
where $2H$ is evaluated on the extremal $\lambda(t)$. This implies
\begin{equation}
E_c(t) = \pm\frac{1}{\sqrt{2H}} e^{-t\vec{H}}_* \partial_{h_0}.
\end{equation}
The explicit expression for $E_a(t)$ and $F_a(t)$ follows directly from the structural equations, indeed
\begin{equation}
E_a(t) = \dot{E}_c(t) = \frac{1}{\sqrt{2H}} e^{-t\vec{H}}_* [\vec{H},\partial_{h_0}] = \frac{1}{\sqrt{2H}} e^{-t\vec{H}}_* \partial_\theta ,
\end{equation}
and
\begin{equation}
F_a(t) = - \dot{E}_a(t) = - \frac{1}{\sqrt{2H}} e^{-t\vec{H}}_* [\vec{H},\partial_\theta] = \frac{1}{\sqrt{2H}} e^{-t\vec{H}}_* \vec{H}'. \qedhere
\end{equation}
\end{proof}

\begin{lemma}
$E_b(t)$ is uniquely specified (up to a sign) by the conditions
\begin{itemize}
\item[(i)] $E_b(t) \in J_\lambda(t)$,
\item[(ii)] $E_b(t) \in \spn\{F_a(t),\dot{F}_a(t)\}^\angle$,
\item[(iii)] $\sigma(\dot{E}_b(t),E_b(t)) = 1$,
\end{itemize}
and, choosing the positive sign, is given by 
\begin{equation}
E_b(t) = \frac{1}{\sqrt{2H}}e^{-t\vec{H}}_* \euler = \frac{1}{\sqrt{2H}}\left(\euler -t\vec{H}\right).
\end{equation}
This, in turn, implies also that
\begin{equation}
F_b(t) = -\dot{E}_b(t) = \frac{1}{\sqrt{2H}} \vec{H}.
\end{equation}
\end{lemma}
\begin{proof}
Condition (i) and the definition of Jacobi curve $J_\lambda(t) = e^{-t\vec{H}} \ve_{\lambda(t)}$ imply that
\begin{equation}
e^{t\vec{H}}_* E_b(t) = a_\theta(t) \partial_\theta + a_\euler(t) \euler + a_0(t) \partial_{h_0},
\end{equation}
for some smooth functions $a_\theta(t),a_\euler(t),a_0(t)$. Condition (ii) then implies
\begin{equation}
0 = \sigma_\lambda(F_a(t),E_b(t)) = \sigma_{\lambda(t)}(\vec{H}',a_\theta(t) \partial_\theta + a_\euler(t) \euler + a_0(t) \partial_{h_0}).
\end{equation}
A tedius computation using the explicit form of $\vec{H}'$ of Eq.~\eqref{eq:hprime} gives
\begin{equation}
\sigma(\vec{H}',\partial_{h_0}) = \sigma(\vec{H}',\euler) = 0, \qquad \sigma(\vec{H}',\partial_\theta) = 2H. 
\end{equation}
Thus we obtain
\begin{equation}
a_\theta(t) = 0.
\end{equation}
Moreover, again condition (ii) implies
\begin{equation}\label{eq:demmerda}
0 = \sigma_\lambda(\dot{F}_a(t),E_b(t)) = \sigma_{\lambda(t)}([\vec{H},\vec{H}'],a_\euler(t) \euler + a_0(t) \partial_{h_0}).
\end{equation}
An explicit computation shows that
\begin{equation}
[\vec{H},\vec{H}'] = -2H(\vec{h}_0 + c_{12}^1 \vec{h}_1 + c_{12}^2 \vec{h}_2) + \left( c_{12}^1 h_2 - c_{12}^2 h_1 \right)\vec{H}' \mod \ve.
\end{equation}
By replacing this expression in Eq.~\eqref{eq:demmerda}, we obtain after straightforward computation that $a_0(t) = 0$. Then $E_b(t) = a_\euler(t) \euler$. Condition (iii) implies
\begin{equation}
1 = \sigma(\dot{a}_\euler(t) \euler - a_\euler(t) \vec{H},a_\euler(t) \euler) = a_\euler(t)^2 \sigma(\euler,\vec{H}) = a_\euler(t)^2 2H.
\end{equation}
where everything is evaluated along the extremal $\lambda(t)$. Then, by choosing the positive sign
\begin{equation}
E_b(t) =  \frac{1}{\sqrt{2H}} e^{-t\vec{H}}_* \euler.
\end{equation}
Moreover, by the structural equations, we have
\begin{equation}
F_b(t) = -\dot{E}_b(t) = -\frac{1}{\sqrt{2H}} e^{-t\vec{H}}_* [\vec{H},\euler] = \vec{H},
\end{equation}
where we used Eq.~\eqref{eq:comp2}.
\end{proof}

Notice that $\dot{F}_b(t) = \frac{d}{dt} e^{-t\vec{H}}_* \vec{H} = [\vec{H},\vec{H}] = 0$. In particular, this implies, by the structural equations, that the following entries of $R(t)$ vanish:
\begin{equation}
R_{ba}(t) = R_{bb}(t) = R_{bc}(t) = 0.
\end{equation}
Thus, together with the normal condition $R_{bc}(t) = 0$, we observe that $R(t)$ has the following form
\begin{equation}
R(t) = \begin{pmatrix}
R_{aa}(t) & 0 & 0 \\ 0 & R_{cc}(t) & 0 \\ 0 & 0 & 0
\end{pmatrix}.
\end{equation}

With the elements of the canonical frame computed so far, namely $E_a(t),E_b(t),E_c(t),F_a(t),F_b(t)$, it is easy to compute the first non-trivial entry $R_{aa}(t)$. Indeed, using formula~\eqref{eq:formulacurvs}, we have
\begin{equation}
R_{aa}(t) =  \sigma_\lambda(\dot{F}_a(t),F_a(t)) = \sigma_{\lambda(t)}(e^{t\vec{H}}_* \dot{F}_a(t),e^{t\vec{H}}_*F_a(t)) = \frac{1}{2H} \sigma_{\lambda(t)}([\vec{H},\vec{H}'],\vec{H}').
\end{equation}

The normal condition $R_{ac}(t) = 0$ and the structural equations uniquely define the final element of the canonical frame:
\begin{equation}
F_c(t) = -\dot{F}_a(t) + R_{aa}(t) E_a(t) =  \frac{1}{\sqrt{2H}} e^{-t\vec{H}}_*\left(-[\vec{H},\vec{H}'] + R_{aa}(t) \partial_\theta\right),
\end{equation}
where we replaced the explicit expressions of $F_a(t)$ and $E_a(t)$. To obtain the second (and last) non-trivial entry of $R(t)$, we apply once again formula~\eqref{eq:formulacurvs}:
\begin{equation}
\begin{aligned}
R_{cc}(t) & =  \sigma_\lambda(\dot{F}_c(t),F_c(t)) = \sigma_{\lambda(t)}(e^{t\vec{H}}_* \dot{F}_c(t),e^{t\vec{H}}_*F_c(t)) = \\
& = \frac{1}{2H}\sigma_{\lambda(t)}([\vec{H},[\vec{H},\vec{H}']],[\vec{H},\vec{H}'])-\frac{1}{(2H)^2}\sigma_{\lambda(t)}([\vec{H},\vec{H}'],\vec{H}')^2.
\end{aligned}\qedhere
\end{equation}
\end{proof}

\subsection{The curvature of 3D contact structures}
Proposition~\ref{p:canonicalframe} gives the expression of the symplectic invariants $R(t)$ in terms of Lie brackets with the Hamiltonian vector field. Now we use Eq.~\eqref{eq:Rlam} to compute the curvature operator $\RR_\lambda:\distr_{x_0} \to \distr_{x_0}$. The latter, in terms of the notation~\eqref{eq:notation} is:
\begin{equation}\label{eq:Rlam3D}
\RR_\lambda X_i = \sum_{j\in \{a,b\}} 3 \Omega(n_i,n_j)R_{ij}(0) X_j, \qquad i \in \{a,b\}.
\end{equation}
By direct inspection, the orthonormal basis $\{X_a,X_b\}$ for $\distr_{x_0}$ obtained by projection of the canonical frame is
\begin{equation}
X_a = \pi_* F_a(0) = \frac{\dot{\gamma}(0)^\perp}{\|\dot\gamma(0)\|^{\phantom{\perp}}}, \qquad X_b = \pi_* F_b(0) = \frac{\dot{\gamma}(0)}{\|\dot\gamma(0)\|},
\end{equation}
where $\gamma$ is the ample geodesic associated with the initial covector $\lambda$. Thus, replacing formula~\eqref{eq:coefficient} for the coefficients $\Omega(n_i,n_j)$ and the expressions for $R(t)$ obtained in Proposition~\ref{p:canonicalframe}, we finally obtain
\begin{equation}
\RR_{\lambda} \dot{\gamma}  = 0, \qquad \RR_{\lambda} \dot\gamma^\perp = \frac25 r_\lambda \dot\gamma^\perp,
\end{equation}
where we suppressed the explicit evaluation at $t=0$ and we have introduced the shorthand
\begin{equation}
r_\lambda:=\frac{1}{2H}\sigma_\lambda([\vec{H},\vec{H}'],\vec{H}').
\end{equation}
In particular, the matrix representing the operator $\RR_{\lambda}:\distr_{x_0} \to \distr_{x_0}$ in terms of the basis $\{\dot{\gamma}^\perp,\dot{\gamma}\}$ is
\begin{equation}\label{eq:rlam3D}
\RR_\lambda = \frac25\begin{pmatrix}
r_\lambda & 0 \\ 0 & 0
\end{pmatrix}.
\end{equation}

\subsection{Relation with the metric invariants} \label{s:chikappa}

In this section we express the curvature $\RR_\lambda$ in terms of the metric invariants $\chi$, $\kappa$ of 3D contact sub-Riemannian structures, first introduced in \cite{agricm} (where $\kappa$ is called $\rho$).
These invariants have been subsequently employed in \cite{agrexp} to describe  the asymptotic expansion of the exponential map of a 3D contact sub-Riemannian structure and in \cite{miosr3d} in the classification of 3D left-invariant sub-Riemannian structures.

The sub-Riemannian Hamiltonian $H$ and the  linear-on-fibers function $h_{0}$ associated with the Reeb vector field are both independent on the choice of the (local) orthornormal frame of the sub-Riemannian structure.
Thus, their Poisson bracket $\{H,h_0\}$ is an invariant of the sub-Riemannian structure. Moreover, by definition, $\{H,h_0\}$ vanishes everywhere if and only if the flow of the Reeb vector field $e^{tX_{0}}$ is a one-parameter family of sub-Riemannian isometries. A standard computation gives 
\begin{equation} \label{eq:parentesihh0}
\{H,h_0\}=c_{10}^1h_1^2+(c_{10}^2+c_{20}^1)h_1h_2+c_{20}^2h_2^2.
\end{equation}
For every $x\in M$, the restriction of $\{H,h_0\}$ to $T^{*}_{x}M$, that we denote by $\{H,h_0\}_{x}$, is a quadratic form on the dual of the distribution $\distr^{*}_{x}\simeq T^{*}_{x}M/\distr_{x}^{\perp}$, where $\distr_{x}^{\perp}$ is the annihilator of $\distr_x$. Hence $\{H,h_0\}_{x}$ can be interpreted as a symmetric operator on $\distr_{x}$, via the inner product. In particular its determinant and its trace are well defined. Moreover one can show that $\trace \{H,h_0\}_x=c_{10}^1+c_{20}^2=0$, for every $x\in M$.

\begin{remark} 
Notice that here we employ a different sign convention with respect to \cite{agrexp,miosr3d}. This leads to different expressions of the invariants $\chi$ and $\kappa$.
\end{remark}

\begin{definition} 
The first invariant $\chi\in C^{\infty}(M)$ is defined as the positive eigenvalue of $\{H,h_0\}_{x}$:
\begin{equation} \label{eq:defchi}
 \chi \doteq \sqrt{-\det\{H,h_0\}_x}\geq0.
\end{equation}
In terms of the structural functions $\chi$ is written as follows
\begin{equation} \label{eq:defchi2}
 \chi = \sqrt{(c_{01}^{1})^{2}+\frac{1}{4}(c_{02}^{1}+c_{01}^{2})^{2}}.
\end{equation}
The second invariant $\kappa\in C^{\infty}(M)$ is defined via the structural functions \eqref{eq:strcst} as follows:
\begin{equation} \label{eq:defkappa}
\kappa \doteq X_{1}(c_{12}^2)-X_{2}(c_{12}^1)-(c_{12}^1)^2-(c_{12}^2)^2+
\frac{1}{2}(c_{02}^1-c_{01}^2).
\end{equation}  
\end{definition}
One can prove that the expression \eqref{eq:defkappa} is invariant by rotation of the orthonormal frame. 

In the next definition, we employ the above identification of $\{H,h_0\}_x$ with a quadratic form on the distribution $\distr_x$ to define a convenient local frame. Recall that a local orthonormal frame $X_1,X_2$ is oriented (with respect to the given 3D contact sub-Riemannian structure) if $d\alpha(X_1,X_2) = 1$.

\begin{definition} 
We say that an oriented local orthonormal frame $X_{1},X_{2}$, defined in a neighbourhood $U$ of $x_{0}$, is an \emph{isotropic frame} if
\begin{equation}
\{H,h_0\}_x (X_1) = \{H,h_0\}_x(X_2) =0, \qquad \all x \in U,
\end{equation}
and the quadratic form $\{H,h_0\}_{x}$ is positive at $X_{1}|_{x}+X_{2}|_{x}$ for all $x \in U$.
\end{definition}
As showed in \cite[Sec. 4]{agrexp} (see also \cite[Prop. 13]{miosr3d}), under the assumption $\chi(x_{0})\neq 0$, it is always possible to find an isotropic frame, and such a frame is unique, up to a global sign. In terms of an isotropic frame, one has the useful simplification
\begin{equation}
\{H,h_0\}_{x}=2\chi h_{1}h_{2}.
\end{equation}
Observe that, when $\chi=0$ on $M$, the last formula automatically holds for every orthonormal frame (indeed, in this case, any oriented orthonormal frame is isotropic).

Next we provide a formula that expresses the curvature introduced here with the invariants of a 3D contact structure. By Eq.~\eqref{eq:rlam3D}, we only need to compute the symplectic product $\sigma([\vec{H},\vec{H}'],\vec{H}')$ in terms of the structural functions.
\begin{proposition} \label{p:sfiniti} The following formula holds true
\begin{equation}
r_\lambda=h_{0}^{2}+2H \kappa+\frac{3}{2}\partial_{\theta}\{H,h_{0}\}.
\end{equation}
Moreover, in terms of an isotropic frame $X_{1},X_{2}$, the above formula becomes:
\begin{equation}
r_\lambda=h_{0}^{2}+\kappa(h_{1}^{2}+h_{2}^{2})+3\chi (h_{1}^{2}-h_{2}^{2}).
\end{equation}
\end{proposition}

Proposition \ref{p:sfiniti} follows by a long but straightforward computation, using the explicit expressions of $\vec{H}$ and $\vec{H}'$ computed in the previous section. A proof of this fact, using a slightly different notation, can be found in \cite{AAPL} (see also \cite{nostrolibro}).

\subsection{Relation of the curvature with cut and conjugate loci} \label{s:interpretation}
In this section we provide a brief interpretation of the role of the two metric invariants in the small time asymptotics of the exponential map for three-dimensional contact structure. In particular we show how the structure of the cut and the conjugate locus is encoded in the curvature. For more details and proofs of the statement appearing here one can refer to \cite{agrexp} and \cite{nostrolibro}.

Let us fix a point $x_{0}\in M$ and let us parametrize normal geodesics starting from $x_{0}$ by their initial covector $\lam=(h_0,h_{1},h_{2})=(h_0,\rho,\theta)$. In what follows we will consider only length-parametrized geodesic, i.e. with $\rho=2H(\lam)=1$. For every pair $(h_0,\theta)$ we denote by $\text{Con}_{x_0}(h_0,\theta)$ (resp. $\text{Cut}_{x_0}(h_0,\theta)$) the first conjugate (resp. cut) point on the geodesic with initial covector $\lam=(h_0,1,\theta)$ starting at $x_0$.
Recall that $\text{Con}_{x_0}(h_0,\theta)$ is the first singular value of the exponential map along the geodesic with initial covector $\lam=(h_0,1,\theta)$. Moreover, $\text{Cut}_{x_0}(h_0,\theta)$ is defined as the point where the geodesic loses global optimality.
We stress also that on a contact sub-Riemannian manifold, due to the absence of non-trivial abnormal minimizers, the \emph{cut locus}, defined as \index{cut locus} \index{conjugate locus}
\begin{equation}
\text{Cut}_{x_0}=\{\text{Cut}_{x_0}(h_0,\theta) \mid  (h_{0},\theta)\in \R\times S^1\},
\end{equation}
coincides with the set of points where the function $\f=\frac{1}{2}\dist^{2}(x_{0},\cdot)$ is not smooth. Rephrasing, one has
\begin{equation}
\Sigma_{x_{0}}=M\setminus (\text{Cut}_{x_0}\cup\{x_0\}),
\end{equation}
where $\Sigma_{x_0}$ is the set of smooth points of $\f$ (see Theorem~\ref{t:d2sr}).
\begin{theorem} \label{t:3dasconj}
Assume $\chi(x_0)\neq 0$. In any set of coordinates, and in terms of an isotropic frame $X_{1},X_{2}$, we have the following asymptotic expansion
\begin{equation}\label{eq:3dconjpart}
\mathrm{Con}_{x_0}(h_0,\theta)= x_{0}\pm \frac{\pi}{|h_{0}|^{2}} X_{0}|_{x_0} \pm \frac{2\pi\chi(x_0)}{|h_{0}|^{3}}(\cos^{3}\theta X_{2}|_{x_0}-\sin^{3}\theta X_{1}|_{x_0})+O\left(\frac{1}{|h_{0}|^{4}}\right), \qquad h_{0} \to \pm\infty.
\end{equation}
Moreover for the conjugate length we have the expansion
\begin{equation}\label{eq:3dconjl}
\ell_{\text{con}}(\theta, h_{0})= \frac{2\pi}{|h_{0}|}-\frac{\pi \kappa(x_0)}{|h_{0}|^{3}}+O\left(\frac{1}{|h_{0}|^{4}}\right),\qquad h_{0} \to \pm\infty.
\end{equation}
\end{theorem}
Analogous formulae can be obtained for the asymptotics of the cut locus at a point $x_{0}$.   
 
\begin{theorem}\label{t:cut3d}
Assume $\chi(x_0)\neq 0$. In any set of coordinates, and in terms of an isotropic frame $X_{1},X_{2}$, we have the following asymptotic expansion
\begin{equation}
\mathrm{Cut}_{x_0}(h_0,\theta)= x_{0}\pm  \frac{\pi}{|h_{0}|^{2}} X_{0}|_{x_0}\pm\frac{2\pi\chi(x_0)}{|h_{0}|^{3}} \cos\theta X_{1}|_{x_0} +O\left(\frac{1}{|h_{0}|^{4}}\right), \qquad h_0 \to \pm\infty.
\end{equation}
Finally the cut length satisfies
\begin{equation}
\ell_{\text{cut}}(h_0,\theta)=  \frac{2\pi}{|h_{0}|}-\frac{\pi}{|h_{0}|^{3}}(\kappa(x_0)+2\chi(x_0) \sin^{2}\theta)+O\left(\frac{1}{|h_{0}|^{4}}\right), \qquad h_{0} \to \pm\infty.
\end{equation}
\end{theorem}

We draw a picture of the asymptotic conjugate and cut loci in Figure \ref{f:cutconj3d}.
Indeed all geometrical information about the structure of these sets is encoded in a pair of quadratic forms defined on $T_{x_0}^*M$: the restriction of the sub-Riemannian Hamiltonian $H$ to the fiber $T^{*}_{x_{0}}M$ and the curvature $\RR_{\lam}$, seen as the quadratic form $\lam\mapsto r_{\lam}$.

\begin{figure}[!htbp]
\centering
\scalebox{0.8}{\input{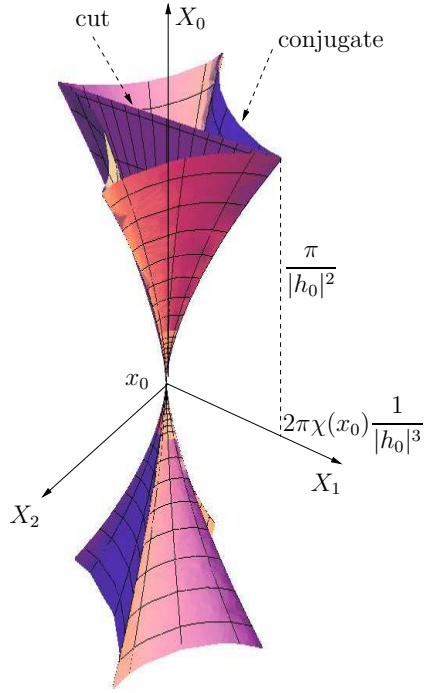}}
\caption{Asymptotic structure of cut and conjugate locus.}
\label{f:cutconj3d}
\end{figure}

Let us consider the kernel of the restriction of sub-Riemannian Hamiltonian to the fiber $T_x^*M$
\begin{equation}
\ker H_{x}= \{\lam\in T^{*}_{x}M \mid \la \lam,v \ra=0, \ \all v\in \distr_{x}\}=\distr^{\perp}_{x}.
\end{equation}
The restriction of $r_{\lam}$ to the 1-dimensional subspace $\distr^{\perp}_{x}$, for every $x\in M$, is the strictly positive quadratic form $r_\lambda|_{\distr^\perp_x} = h_0^2$. Moreover it is equal to $1$ when evaluated on the Reeb vector field. Hence $r_{\lam}$ encodes both the contact form $\alpha$ and its normalization.

Let us consider the orthogonal complement $\distr_{x}^{*}$ of $\distr^{\perp}_{x}$ in the fiber with respect to $r_{\lam}$ (this is indeed isomorphic to the space  of linear functionals defined on $\distr_{x}$). This induces the well-defined splitting
\bqn
T^{*}_{x}M= \distr_{x}^{\perp}\oplus \distr_{x}^{*}=\{\nu_{0}\} \oplus \text{span}\{\nu_{1},\nu_{2}\},
\eqn  
where $\nu_{0}=\alpha$ and $\nu_{1},\nu_{2}$ form a dual basis of $X_{0},X_{1},X_{2}$ (where $X_{1},X_{2}$ is an isotropic frame). Indeed the restriction of $r_{\lam}$ to $\distr^{*}_{x}$ is 
\bqn
r_{\lam}|_{\distr_{x}^{*}}=(\kappa+3\chi)h_{1}^{2}+(\kappa-3\chi)h_{2}^{2}.
\eqn
By using the Euclidean metric induced by $H_x$ on $\distr_x$, it can be identified with a symmetric operator.

From this formulae it is easy to recover the two invariants $\chi, \kappa$ 
\begin{equation}
\trace\left(r_{\lam}|_{\distr_{x}^{*}}\right)=2\kappa, \qquad \discr\left(r_{\lam}|_{\distr_{x}^{*}}\right)=36\chi^{2},
\end{equation}
where the discriminant of an operator $Q$, defined on a two-dimensional space, is defined as the square of the difference of its eigenvalues, and is computed by the formula $\discr(Q)=\trace^{2}(Q)-4\det(Q)$.

The cubic term of the conjugate locus (for a fixed value of $h_{0}$) parametrizes an astroid. The cuspidal directions of the astroid are given by the eigenvectors of $r_{\lam}$ (that correspond to the isotropic directions $X_1,X_2$), and the cut locus intersects the conjugate locus exactly at the cuspidal points in the direction of the eigenvector of $r_{\lam}$ corresponding to the larger eigenvalue (that is $X_1$).
Finally the ``size'' of the cut locus increases for larger values of $\chi$, while $\kappa$ is involved in the length of curves arriving at cut/conjugate locus.

The reader interested in the case when $\chi$ vanishes at $x_{0}$ (but is not constant) is referred to \cite{agrexp}.

\subsection{Final comments}

The study of complete sets of invariants, connected with the problem of equivalence of 3D contact structures, has been considered in different works and contexts with different languages \cite{miosr3d,hughen,falbel}.

Let us introduce a canonical Riemannian metric $g$ on $M$, defined by declaring the Reeb vector field $X_{0}$ to be orthogonal to the distribution and of unit norm. In other words, the metric $g$ satisfies
\begin{equation}
g(X_{i},X_{j})=\delta_{ij}, \qquad \all i,j=0,1,2.
\end{equation}
The purpose of this section is to show how the invariants $\chi$ and $\kappa$ introduced above are related with the curvature of this canonical Riemannian metric and briefly discuss their relation with others invariants introduced in the aforementioned references.

Denote by $\nabla$ the Levi-Civita connection associated with the Riemannian metric $g$. The Christoffel symbols $\Gamma_{ij}^{k}$ of the connections are defined by
\begin{equation}\label{eq:covariant}
\nabla_{X_{i}}X_{j}=\Gamma_{ij}^{k}X_{k}, \qquad \all i,j=0,1,2,
\end{equation} 
and related with the structural functions of the frame by the following formulae:
\begin{equation}
\Gamma_{ij}^{k}=\frac{1}{2}(c_{ij}^{k}-c_{jk}^{i}+c_{ki}^{j}).
\end{equation}
Let us denote by $\mathrm{Sec}(\Pi_x)$ the sectional curvature of the plane $\Pi_x$ generated by two vectors $v,w\in T_{x}M$.
\begin{proposition} \label{p:secR}
The sectional curvature of the plane $\Pi_x = \distr_x$ is
\begin{equation}\label{eq:sec}
\mathrm{Sec}(\distr_{x})=\kappa+\chi^{2}-\frac34.
\end{equation}
\end{proposition}
\begin{proof}
It is a long but straightforward computation, using the explicit expression of the covariant derivatives \eqref{eq:covariant}. In terms of an orthonormal frame $X_{1},X_{2}$ for the distribution $\distr_{x}$ we have 
\begin{align*}
\mathrm{Sec}(\distr_{x})&=g(\nabla_{X_{1}}\nabla_{X_{2}}X_{2}-\nabla_{X_{2}}\nabla_{X_{1}}X_{2}
-\nabla_{[X_{1},X_{2}]}X_{2},X_{1})\\
&=X_{1}(c_{12}^{2})-X_{2}(c_{12}^{1})-(c_{12}^{1})^{2}-(c_{12}^{2})^{2}+\frac{1}{2}(c_{02}^{1}-c_{01}^{2})
+(c_{01}^{1})^{2}+\frac{1}{4}(c_{02}^{1}+c_{01}^{2})^{2}-\frac{3}{4},
\end{align*}
and \eqref{eq:sec} follows from the explicit expressions \eqref{eq:defchi2} and \eqref{eq:defkappa} of $\chi$ and $\kappa$.
\end{proof}
In \cite{hughen}, using the Cartan's moving frame method, the author introduces the family of generating invariants $a_{1},a_{2}, K \in C^\infty(M)$. In terms of these invariants one has 
\begin{equation}
\mathrm{Sec}(\distr_{x})=K+a_{1}^{2}+a_{2}^{2}-\frac34.
\end{equation}
The author also observe that $K=4W$, where $W$ is the Tanaka-Webster curvature of the CR structure associated with the sub-Riemannian one, see \cite[p.15]{hughen}. Notice that also that $\kappa=4W$ (see \cite{AAPL}), hence $\kappa=K$. This, together with Proposition \ref{p:secR}, gives the following relation between the metric invariants: 
\begin{equation}\label{eq:hugen}
\kappa=K,\qquad \chi =\sqrt{a_{1}^{2}+a_{2}^{2}}.
\end{equation}
With these invariants, the author in \cite{hughen} proved Bonnet-Myers type results for 3D contact structures.

Another approach to the classification problem of 3D sub-Riemannian structures is the one of \cite{falbel}, where the authors employ the existence of a canonical linear connection (with non zero torsion) associated with the sub-Riemannian structure. The authors introduce the family of generating invariants $K,\tau_{0},W_{1},W_{2}$, associated with this connection. It is possible to show that the first two invariants coincides with $\kappa$ and $\chi$ respectively. In the case of left-invariant structures with $\chi>0$, the remaining two invariants can be used to distinguish non-isometric structures with same (constant) value of $\chi$ and $\kappa$, see \cite{falbel,miosr3d}.

\finereview} 
\chapter{Sub-Laplacian and Jacobi curves}\label{c:sublapproof}

Throughout this chapter, we assume $M$ to be an equiregular sub-Riemannian manifold (that is, the rank of the distribution $\distr$ is constant, equal to $k$). Nevertheless, most of the statements of this chapter hold true in the general case, by replacing the sub-Riemannian inner product on $\distr$ with the Hamiltonian inner product. The final goal of this chapter is the proof of Theorem~\ref{t:main3}, that is an asymptotic formula for the sub-Laplacian of the cost function. We start with a general discussion about the computation of the sub-Laplacian at a fixed point.

Let $f \in C^{\infty}(M)$, $x \in M$ and $\lam = d_x f \in T_{x}^* M$. Moreover, let $X_1,\dots,X_k$ be a local orthonormal frame for the sub-Riemannian structure. All our considerations are local, then we assume without loss of generality that the frame $X_1,\dots,X_{n}$ is globally defined. Then, by Eq.~\eqref{eq:sublaplframe}, the sub-Laplacian associated with the volume form $\mu$ writes
\begin{equation}
\lapl_\mu f = \sum_{i=1}^k X_i^2(f) + \dive_\mu (X_i)X_i(f).
\end{equation}
As one can see, the sub-Laplacian is the sum of two terms. The first term, $\sum_{i=1}^k X_i^2(f)$, is a \virg{sum of squares} which does not depend on the choice of the volume form. On the other hand, the second term, namely $\sum_{i=1}^k \dive_\mu (X_i) X_i(f)$ depends on $\mu$ through the divergence operator. When $x$ is a critical point for $f$, the second term vanishes, and the sub-Laplacian can be computed by taking the trace of the ordinary second differential of $f$ (see Lemma~\ref{l:d2lapl}). On the other hand, if $x$ is non-critical, we need to compute both terms explicitly. 

We start with the second term. Let $\theta_1,\dots,\theta_n$ be the coframe dual to $X_1,\dots,X_n$. Namely $\theta_i(X_j) = \delta_{ij}$. Then, there exists a smooth function $g \in C^\infty(M)$ such that $\mu = e^{g} \theta_1\wedge \ldots \wedge \theta_n$. Finally, let $c_{ij}^k \in C^\infty(M)$ be the \emph{structure functions} defined by $[X_i,X_j] = \sum_{k=1}^n c_{ij}^k X_k$. A standard computation using the definition of divergence gives
\begin{equation}
\dive_\mu (X_i) = X_i(g) - \sum_{j=1}^n c_{ij}^j.
\end{equation}
Thus, the second term of the sub-Laplacian is
\begin{equation}\label{eq:slap2}
\sum_{i=1}^k \dive_\mu(X_i)X_i(f) = \langle \grad f| \grad g\rangle - \sum_{i=1}^k\sum_{j=1}^n c_{ij}^j X_i(f).
\end{equation}

The first term of the sub-Laplacian can be computed through the generalized second differential introduced with Definition~\ref{d:secdif}. Recall that the second differential at a non critical point $x$ is a linear map $d_x^2 f : T_x M \to T_\lam (T^*M)$.

\section{Coordinate lift of a local frame}

We introduce a special basis of $T_\lam (T^*M)$, associated with a choice of the local frame $X_1,\dots,X_n$, which is a powerful tool for explicit calculations. We define an associated frame on $T^*M$ as follows. For $i = 1,\dots,n$ let $h_i :T^*M \to \mathbb{R}$ be the linear-on-fibres function defined by $\lambda \mapsto h_i(\lambda) \doteq  \langle \lambda, X_i\rangle$. The action of the derivations on $T^*M$ is completely determined by the action on affine functions, namely functions $a \in C^\infty(T^*M)$ such that $a(\lambda) = \langle \lambda, Y \rangle + \pi^* g$ for some $Y\in\VecM$, $g\in C^\infty(M)$. Then, we define the \emph{coordinate lift of a field} $X \in \VecM$ as the field $\wt{X} \in \mathrm{Vec}(T^*M)$ such that $\wt{X}(h_i) = 0$ for $i=1,\dots,n$ and $\wt{X}(\pi^*g) = X(g)$. This, together with Leibniz rule, characterize the action of $\wt{X}$ on affine functions, and then completely define $\wt{X}$. Indeed, by definition, $\pi_* \wt{X} = X$. On the other hand, we define the (vertical) fields $\partial_{h_i}$ such that $\partial_{h_i} (\pi^* g) = 0$, and $\partial_{h_i} (h_j) = \delta_{ij}$. 
It is easy to check that $\{ \partial_{h_i}, \wt{X}_i\}_{i=0}^{n}$ is a frame on $T^*M$. We call such a frame the \emph{coordinate lifted frame}, and we employ the shorthand $\partial_i\doteq  \partial_{h_i}$. \index{coordinate lifted frame}
Observe that, by the same procedure, we can define the coordinate lift of a vector $X \in T_x M$ (i.e. not necessarily a field) at any point $\lam  \in T_x^*M$.
\begin{remark}\label{r:lift}
Remember  that we  require $X_1,\dots,X_n$ to be \emph{fields} (and not simple vectors in $T_x M$) in order to define the coordinate lift. In particular, the lift $\wt{X}|_\lam \in T_\lam (T^*M)$ depends on the germ at $x$ of the chosen frame $X_1,\dots,X_n$. On the other hand, $\partial_i|_\lam$ depends only on the value of $X_1,\dots,X_n$ at $x$.
\end{remark}

\begin{lemma}\label{l:coordlift}
Let $X \in T_x M$. In terms of a coordinate lifted frame,
\begin{equation}
d^2_x f (X) = \wt{X} + \sum_{i=1}^n X(X_i(f)) \partial_i,
\end{equation}
where $X(X_i(f))$ is understood to be computed at $x$ and $\wt{X},\partial_i \in T_\lam (T^*M)$.
\end{lemma}
\begin{proof}
We explicitly compute the action of the vector $d^2_xf (X) \in T_\lam (T^*M)$ on affine functions. First, for any $g \in C^\infty(M)$, $d^2_x f(X) (\pi^*g) = \pi_* \circ d^2_x f (X) (g) = X(g)$. Moreover, $d^2_x f(X)(h_i) = X ( h_i \circ d f) = X(\langle d f, X_i\rangle) = X(X_i(f))$.   
\end{proof}
Lemma~\ref{l:coordlift}, when applied to the vectors $X_1,\ldots,X_k$, completely characterize the second order component of the sub-Laplacian, in terms of the second differential $d^2_x f$.

\section{Sub-Laplacian of the geodesic cost}

Assume $f = c_t$, that is the geodesic cost associated with an ample, equiregular geodesic $\gamma:[0,T]\to M$. As usual, let $x = \gamma(0)$ be the initial point, $\lam = d_x c_t$ the initial covector, and $J_\lam(\cdot)$ the associated Jacobi curve, with Young diagram $D$. As discussed in Chapter~\ref{c:proof}, there is a class of preferred frames in $T_\lam (T^*M)$, namely the canonical moving frame $\{E_{ai}(t),F_{ai}(t)\}_{ai \in D}$. In order to employ the results of Theorem~\ref{t:Sasymptotic} for the computation of $\lapl c_t$, we first relate the canonical frame with a coordinate lifted frame. As a first step, we need the following lemma, which is an extension of Lemma~\ref{l:orthframe} along the geodesic.

\begin{lemma}\label{l:hopefully}
Let $\{E_{ai}(t),F_{ai}(t)\}_{ai \in D}$ be a canonical moving frame for $J_\lam(\cdot)$ and consider the following vector fields along $\gamma$:
\begin{equation}
X_{ai}(t)\doteq \pi_*\circ e^{t\vec{H}}_* F_{ai}(t) \in T_{\gamma(t)} M, \qquad ai \in D.
\end{equation}
The set $\{X_{ai}(t)\}_{ai\in D}$ is a basis for $T_{\gamma(t)} M$. Moreover $\{X_{a1}(t)\}_{a=1}^k$ is an orthonormal basis for $\distr_{\gamma(t)}$ along the geodesic. Finally, consider any smooth extension of $\{X_{ai}(t)\}_{ai\in D}$ in a neighbourhood of $\gamma$, and the associated coordinate lifted frame. Then
\begin{equation}
E_{ai}(t) = e^{-t\vec{H}}_* \partial_{ai}|_{\lambda(t)},
\end{equation}
\end{lemma}
Lemma~\ref{l:hopefully} states that the projection of the horizontal elements of the canonical frame (the ``$F$''s) corresponding to the first column of the Young diagram are an orthonormal frame for the sub-Riemannian distribution along the geodesic. Moreover, if we complete the frame with the projections of the other horizontal elements, and we introduce the associated coordinate lifted frame along the extremal $e^{t\vec{H}}(\lam)$, the vertical elements of the canonical frame (the ``$E$''s) have a simple expression. Observe that, according to Remark~\ref{r:lift}, the last statement of the lemma does not depend on the choice of the extension of the vectors $X_{ai}(t)$ in a neighbourhood of $\gamma$.

\begin{proof}
Assume first that the statement is true at $t=0$. Then, let $0<t < T$. Point (ii) of Proposition~\ref{p:Jproperties} gives the relation between the Jacobi curves \virg{attached} at different points $\lambda(t)= e^{t\vec{H}}(\lam)$ along the lift of $\gamma$. Namely
\begin{equation}
J_{\lam(t)}(\cdot) = e^{t\vec{H}}_* J _\lam(t+\cdot).
\end{equation}
As a consequence of this, and the definition of canonical frame, if $\{E_{ai}(\cdot),F_{ai}(\cdot)\}_{ai \in D}$ is a canonical frame for the Jacobi curve $J_\lam(\cdot )$, it follows that, for any fixed $t$,
\begin{gather}
\wt{E}_{ai}(\cdot)\doteq  e^{t\vec{H}}_* E_{ai}(t+\cdot),\\
\wt{F}_{ai}(\cdot)\doteq  e^{t\vec{H}}_* F_{ai}(t+\cdot),
\end{gather}
is a canonical frame for the Jacobi curve $J_{\lam(t)}(\cdot)$. In particular, $X_{ai}(t) = \pi_* \wt{F}_{ai}(0)$, and the statements now follow from the assumption that the lemma is true at the initial time of the Jacobi curve $J_{\lam(t)}(\cdot)$.

Then, we only need to prove the statement at $t=0$. For clarity, we suppress the explicit evaluation at $t=0$. As usual, let $\hor_\lam = \spn\{F_{ai}\}_{ai\in D}$ be the horizontal subspace and $\ve_\lam = \spn\{E_{ai}\}_{ai \in D}$ be the vertical subspace. By definition of canonical frame, $T_\lam (T^*M) = \hor_\lam \oplus \ve_\lam$. Since $\ve_\lam = \ker \pi_*$, and $\pi_*$ is a submersion, $\pi_* \hor_\lam = T_x M$. Thus $\{X_{ai}\}_{ai \in D}$ is a basis for $T_x M$. By Lemma~\ref{l:orthframe}, the set $\{X_{a1}\}_{a=1}^k$ is an orthonormal frame for the Hamiltonian inner product $\langle\cdot \, |\cdot\, \rangle_\lam$ which, in the sub-Riemannian case, does not depend on $\lam$ and coincides with the sub-Riemannian inner product (see Remark~\ref{r:hamsub}). Now, we show that $E_{ai} = \partial_{ai}|_\lam$. Since the canonical frame is Darboux, this is equivalent to $\sigma(\partial_{ai},F_{bj}) = \delta_{ab}\delta_{ij}$. Indeed, in terms of the coframe $\{\theta_{ai}\}_{ai \in D}$, dual to $\{X_{ai}\}_{ai \in D}$
\begin{equation}
\sigma = \sum_{ai \in D} dh_{ai} \wedge \pi^* \theta_{ai} + h_{ai} \pi^* d\theta_{ai}.
\end{equation}
Therefore
\begin{equation}
\sigma(\partial_{ai},F_{bj}) = \theta_{ai}(\pi_* F_{bj}) = \theta_{ai}(X_{bj}) = \delta_{ab}\delta_{ij}.\qedhere
\end{equation}
\end{proof}

\section{Proof of Theorem~\ref{t:main3}}
We now have all the tools we need in order to prove Theorem~\ref{t:main3}, concerning the asymptotic behaviour of $\lapl c_t$.

The idea is the compute the \virg{hard} term of $\lapl c_t$, namely the sum of squares term, through the coordinate representation of the Jacobi curve. By Lemma~\ref{l:coordlift}, written in terms of the frame $X_{ai}\doteq  X_{ai}(0) = \pi_* F_{ai}(0)$ of $T_x M$, and its coordinate lift, we have
\begin{equation}\label{eq:slap1early}
d^2_x c_t(X_{\rho}) = \wt{X}_{\rho}  +\sum_{\nu \in D} X_{\rho}(X_{\nu}(c_t)) \partial_{\nu},
\end{equation}
where we used greek letters as a shorthand for boxes of the Young diagram $D$. When $\rho$ belongs to the first column of the Young diagram $D$, namely $\rho = a1$ (in this case, we simply write $a$), we have, as a consequence of Lemma~\ref{l:hopefully} and the structural equations
\begin{equation}
F_{a}(0) = -\dot{E}_{a}(0) = - [\vec{H},\partial_{a}] = \wt{X}_{a} +  \sum_{\nu \in D}\left(\sum_{\kappa \in D} c_{a\nu}^\kappa h_\kappa +\sum_{b=1}^k h_b c_{b\nu}^a \right) \partial_\nu,
\end{equation}
where everything is evaluated at $\lam$. Therefore, from Eq.~\eqref{eq:slap1early}, we obtain
\begin{equation}
d^2_x c_t(X_a) = F_a(0) + \sum_{\nu \in D} \left( X_a (X_\nu(c_t))-\sum_{\kappa \in D}c_{a\nu}^\kappa h_\kappa - \sum_{b=1}^k h_b c_{b\nu}^a \right) E_\nu(0).
\end{equation}
Recall that $S(t)^{-1}: \hor_\lam \to \ve_\lam$ is the matrix that represents the Jacobi curve in the coordinates induced by the canonical frame (at $t=0$). More explicitly
\begin{equation}
d^2_x c_t(X_\rho) = F_\rho(0) + \sum_{\nu \in D} S(t)^{-1}_{\rho\nu} E_\nu(0).
\end{equation}
Moreover, since we restricted $d^2_x c_t$ to elements of $\distr_x$, we obtain
\begin{equation}\label{eq:slap1almost}
\sum_{a= 1}^k X_a^2(c_t) =  \sum_{a=1}^k \Sred(t)^{-1}_{aa} + \sum_{a=1}^k \sum_{b=1}^k h_a c_{ab}^b.
\end{equation}
Now observe that, if $\rho$ \emph{does not} belong to the first column of the Young diagram, we have
\begin{equation}
\dot{E}_\rho(0) = [\vec{H},\partial_\rho] = \sum_{a=1}^k\sum_{\nu \in D} h_a c_{a\nu}^\rho E_\nu(0).
\end{equation}
On the other hand, by the structural equations, $\dot{E}_\rho(0)$ is a vertical vector that does not have $E_\rho(0)$ components. Then, when $\rho$ is not in the first column of $D$, $\sum_{a=1}^k h_a c_{a\rho}^\rho = 0$. Thus we rewrite Eq.~\eqref{eq:slap1almost} as
\begin{equation}\label{eq:slap1}
\sum_{a= 1}^k X_a^2(c_t) =  \sum_{a=1}^k \Sred(t)^{-1}_{aa} + \sum_{a=1}^k \sum_{\rho \in D} h_a c_{a\rho}^\rho.
\end{equation}
By taking the sum of Eq.~\eqref{eq:slap2} and Eq.~\eqref{eq:slap1}, we obtain
\begin{equation}
\lapl_{\mu} c_t|_x =  \sum_{a=1}^k \Sred(t)^{-1}_{aa} + \langle \nabla_x c_t| \nabla_x g \rangle,
\end{equation}
where we recall that the function $g$ is implicitly defined (in a neighbourhood of $\gamma$) by $\mu = e^g \theta_1\wedge\ldots\wedge\theta_n$. Remember that, at $x = \gamma(0)$, $\nabla_x c_t = \dot{\gamma}(0)$. Then 
\begin{equation}
\lapl_{\mu} c_t|_x =  \sum_{a=1}^k \Sred(t)^{-1}_{aa} + \left.\frac{d}{dt}\right|_{t=0} g(\gamma(t)).
\end{equation}
\begin{remark}\label{r:parallelotopo}
Observe that if $P_t\doteq X_1(t)\wedge\ldots\wedge X_n(t) \in \bigwedge^n T_{\gamma(t)}M$ is the parallelotope whose edges are the elements of the frame $\{X_i(t)\}_{i=1}^n$, then $g(\gamma(t)) = \log |\mu(P_t)|$, that is the logarithm of the volume of the parallelotope $P_t$.
\end{remark}
Thus, by replacing the results of Corollary~\ref{c:Sridasymptotic} about the asymptotics of the reduced Jacobi curve, we obtain
\begin{equation}
\lapl_{\mu} c_t|_x = - \frac{\trace \Qz_\lam}{t} + \dot{g}(0)  +  \frac{1}{3} \Ric(\lam)t + O(t^2),
\end{equation}
where $\dot{g}(0)\doteq \left.\frac{d}{dt}\right|_{t=0} g(\gamma(t))$. Since $\f_t = -t c_t$, we obtain
\begin{equation}
\lapl_{\mu} \f_t|_x = \trace \Qz_\lam - \dot{g}(0) t -\frac{1}{3}\Ric(\lam)t^2 + O(t^3),
\end{equation}
which is the sought expansion, valid for small $t$.

\subsection{Computation of the linear term}
Recall that, for any equiregular smooth admissible curve $\gamma :[0,T] \to M$, the Lie derivative in the direction of the curve defines surjective linear maps
\begin{equation}
\mc{L}_\tanf : \DD^{i}_{\gamma(t)} / \DD^{i-1}_{\gamma(t)} \to \DD^{i+1}_{\gamma(t)} / \DD^i_{\gamma(t)}, \qquad i \geq 1,
\end{equation}
as defined in Section~\ref{s:slowgrowth}. In particular, notice that $\mc{L}_\tanf^i : \distr_{\gamma(t)} \to \DD^{i+1}_{\gamma(t)} / \DD^i_{\gamma(t)}$, for $i \geq 1$ is a well defined, surjective linear map from the distribution (see also Lemma~\ref{l:flag0}).
\begin{lemma}\label{l:adrecover}
For $t \in [0,T]$, we recover the projections $X_{ai}(t)= e^{t\vec{H}}_*F_{ai}(t) \in T_{\gamma(t)} M$ as
\begin{equation}
X_{ai}(t) = (-1)^{i-1} \mc{L}^{i-1}_\tanf(X_{a1}(t)) \bmod \DD^{i-1}_{\gamma(t)}, \qquad a=1,\ldots,k, \quad i = 1,\ldots,n_a.
\end{equation}
\end{lemma}
\begin{proof}
Fix $a=1,\ldots,k$. For $i=1$ the statement is trivial. Assume the statement to be true for $j\leq i$. Recall that we can see $F_{ai}|_{\lambda(t)} = e^{t\vec{H}}_* F_{ai}(t)$ as a field along the extremal $\lambda(t)$. Then, by the structural equations for the canonical frame, $X_{a(i+1)} = -\pi_*[\vec{H},F_{ai}]$. A quick computation in terms of a coordinate lifted frame proves that
\begin{equation}
X_{a(i+1)}(t) = -[\tanf, X_{ai}]|_{\gamma(t)} \bmod \DD^{i}_{\gamma(t)},
\end{equation}
for an admissible extension $\tanf$ of $\dot{\gamma}$. Thus, by induction, we obtain the statement.
\end{proof}

\begin{proof}[Proof of Theorem~\ref{t:slowgrowth}]
We consider equiregular distributions and ample geodesics $\gamma$ that obey the growth condition
\begin{equation}\label{eq:gcproof}
\dim \DD^i_{\gamma(t)}= \dim \distr^i, \qquad \all i \geq 0.
\end{equation}
We only need to compute explicitly the term $\dot{g}(0)$ of the asymptotic expansion in Theorem~\ref{t:main3}. Recall that, according to the proof of Theorem~\ref{t:main3}, the coefficient of the linear term is given by the following formula (see Remark~\ref{r:parallelotopo})
\begin{equation}\label{eq:linearterm}
\dot{g}(0) = \left.\frac{d}{dt}\right|_{t=0} \log|\mu(P_t)|,
\end{equation}
where $P_t$ is the parallelotope whose edges are the projections $\{X_{ai}(t)\}_{ai \in D}$ of the horizontal part of the canonical frame $X_{ai} = \pi_* \circ e^{t\vec{H}}_* F_{ai}(t) \in T_{\gamma(t)} M$, namely
\begin{equation}\label{eq:parallelotope}
P_t=\bigwedge_{ai \in D} X_{ai}(t).
\end{equation}
By definition of canonical frame, Proposition~\ref{p:twoflags}, and the growth condition~\eqref{eq:gcproof} we have that the elements $\{X_{ai}(t)\}_{ai \in D}$ are a frame along the curve $\gamma(t)$ adapted to the flag of the distribution. More precisely
\begin{equation}
\distr^i_{\gamma(t)} = \spn\{X_{aj}(t)|\,aj \in D,\,1 \leq j \leq i \}.
\end{equation}
By Lemma~\ref{l:adrecover} we can write the adapted frame $\{X_{ai}\}_{ai \in D}$ in terms of the smooth linear maps $\mc{L}_\tanf$, and we obtain the following formula for the parallelotope
\begin{equation}
P_t = \bigwedge_{i=1}^m \bigwedge_{a_i = 1}^{d_i} X_{a_i i}(t) = \bigwedge_{i=1}^m \bigwedge_{a_i = 1}^{d_i} \mc{L}^{i-1}_\tanf( X_{a_i 1}(t)).
\end{equation}
Then, a standard linear algebra argument and the very definition of Popp's volume leads to
\begin{equation}
|\mu(P_t)| = \sqrt{\prod_{i=1}^m \det M_i(t)},
\end{equation}
where the smooth families of operators $M_i(t)$, for $i=1,\ldots,m$ are the one defined in Eq.~\eqref{eq:Mt}. This, together with Eq.~\eqref{eq:parallelotope} completes the computation of the linear term of Theorem~\ref{t:main3} for any ample geodesic satisfying the growth condition~\eqref{eq:gcproof}.
\end{proof}


\part{Appendix}

\appendix

\chapter{Smoothness of value function (Theorem~\ref{t:smoothness})} \label{a:proofsmoothness}

The goal of this section is to prove Theorem~\ref{t:smoothness} on the smoothness of the value function. All the relevant definitions can be found in Chapter~\ref{c:affcs}.
As a first step, we generalize the classical definition of conjugate points to our setting.
\begin{definition}
Let $\gamma:[0,T] \to M$ be a strictly normal trajectory, such that $x_0 = \gamma(0)$ and $\gamma(t) = \EXP_{x_0}(t,\lam_0)$. We say that $\gamma(t)$ is \emph{conjugate with $x_0$ along $\gamma$} if $\lam_0$ is a critical point for $\EXP_{x_0,t}$.
\end{definition}
Observe that the relation \virg{being conjugate with} is not reflexive in general. Indeed, even if $\gamma(t)$ is conjugate with $x_0$, there might not even exist an admissible curve starting from $\gamma(t)$ and ending at $x_0$.

We stress that, if $\gamma$ is also abnormal, any $\gamma(t)$ is a critical value of the sub-Riemannian exponential map. Indeed, this is a consequence of the inclusion $\im D_{\lam_0} \EXP_{x_0,t} \subset \im D_u E_{x_0,t} \neq T_{x_0} M$ for abnormal trajectories; being strongly normal is a necessary condition for the absence of critical values along a normal trajectory. Actually, a converse of this statement is true.
\begin{proposition}\label{p:strongnoconj}
Let $\gamma:[0,T] \to M$ be a strongly normal trajectory. Then, there exists an $\eps >0$ such that $\gamma(t)$ is not conjugate with $\gamma(0)$ along $\gamma$ for all $t \in (0,\eps)$.
\end{proposition}
The proof of Proposition~\ref{p:strongnoconj} in the sub-Riemannian setting can be found in~\cite{nostrolibro} and can be adapted to a general affine optimal control system. See also \cite{agrachevbook} for a more general approach.

We are now ready to prove Theorem~\ref{t:smoothness} about smoothness of the value function which, for the reader's convenience, we restate here. Recall that $M'\subset M$ is the relatively compact subset chosen for the definition of the value function.
\begin{theorem*}
Let $\gamma:[0,T] \to M'$ be a strongly normal trajectory. Then there exists an $\eps >0$ and an open neighbourhood $U \subset (0,\eps) \times M' \times M'$ such that:
\begin{itemize}
\item[(i)] $(t,\gamma(0),\gamma(t)) \in U$ for all $t \in (0,\eps)$,
\item[(ii)] For any $(t,x,y) \in U$ there exists a unique (normal) minimizer of the cost functional $J_t$, among all the admissible curves that connect $x$ with $y$ in time $t$, contained in $M'$,
\item[(iii)] The value function $(t,x,y)\mapsto S_t(x,y)$ is smooth on $U$.
\end{itemize}
\end{theorem*}
\begin{proof}
We first prove the theorem in the case $M'=M$ compact. We need the following sufficient condition for optimality of normal trajectory. Let $a \in C^\infty(M)$. The graph of its differential is a smooth  submanifold $\mc{L}_0 \doteq  \{d_x a|\, x \in M\} \subset T^*M$, $\dim \mc{L}_0 = \dim M$. Translations of $\mc{L}_0$ by the flow of the Hamiltonian field $\mc{L}_\tau = e^{\tau\vec{H}}(\mc{L}_0)$ are also smooth submanifolds of the same dimension.
\begin{lemma}[see {\cite[Theorem 17.1]{agrachevbook}}]\label{l:suffcond} 
Assume that the restriction $\pi : \mc{L}_\tau \to M$ is a diffeomorphism for any $\tau \in [0,\eps]$. Then, for any $\lam_0 \in \mc{L}_0$, the normal trajectory
\begin{equation}
\gamma(\tau) = \pi \circ e^{\tau\vec{H}}(\lam_0),\qquad \tau \in [0,\eps],
\end{equation}
is a strict minimum of the cost functional $J_{\eps}$ among all admissible trajectories connecting $\gamma(0)$ with $\gamma(\eps)$ in time $\eps$.
\end{lemma}
Lemma~\ref{l:suffcond} is a sufficient condition for the optimality of a single normal trajectory. By building a suitable family of smooth functions $a \in C^\infty(M)$, one can prove that, for any sufficiently small compact set $K \subset T^*M$, we can find a $\eps= \eps(K) >0$ sufficiently small such that, for any $\lam_0 \in K$,  and for any $t \leq \eps$, the normal trajectory
\begin{equation}
\gamma(\tau) = \pi \circ e^{\tau\vec{H}}(\lam_0),\qquad \tau \in [0,t], \qquad t \leq \eps
\end{equation}
is a strict minimum of the cost functional $J_{t}$ among all admissible curves connecting $\gamma(0)$ with $\gamma(t)$ in time $t$.	

We sketch the explicit construction of such a family. Let $K \subset T^*M$ sufficiently small such that it is contained in a trivial neighbourhood $\R^n \times U\subset T^*M$. Let $(p,x)$ be coordinates on $K$ induced by a choice of coordinates $x$ on $O \subset M$. Then, consider the function $a:K \times O \to \R$, defined in coordinates by $a(p_0,x_0;y) = p_0^* y$. Extend such a function to $a:K\times M \to \R$. For any $\lam_0 \in K$, denote by $a^{(\lam_0)} = a(\lam_0;\cdot) \in C^\infty (M)$. Indeed, for $x_0 = \pi(\lam_0)$, we have $\lam_0 = d_{x_0} a^{(\lam_0)}$. In other words we can recover any initial covector in $K$ by taking the differential at $x_0$ of an appropriate element of the family. Therefore, let $\mc{L}_0^{(\lam_0)}\doteq \{d_x a^{(\lam_0)}|\,x\in M\}$, and $\mc{L}_\tau^{(\lam_0)}\doteq e^{\tau\vec{H}}(\mc{L}_0^{(\lam_0)})$. $M$ is compact, then there exists $\eps(K) = \sup\{\tau\geq 0|\,\pi:\mc{L}^{(\lam_0)}_s \to M \text{ is a diffeomorphism for all } s \in [0,\tau],\, \lam_0 \in K  \} > 0$.

Let us go back to the proof. Set $x_0 = \gamma(0)$, and let $\gamma(t) = \EXP_{x_0}(t,\lam_0)$. By Proposition~\ref{p:strongnoconj}, we can assume that $\gamma(t)$ is not conjugate with $\gamma(0)$ along $\gamma$ for all $t \in (0,\eps)$. In particular, $D_{\lam_0}\EXP_{x_0,t}$ has maximal rank for all $t \in (0,\eps)$. Without loss of generality, assume that $\vec{H}$ is complete. Then, consider the map $\phi: \R^+\times T^*M \to \R^+ \times M \times M$, defined by
\begin{equation}
\phi(t,\lam) = (t,\pi(\lam),\EXP_{\pi(\lam)}(t,\lam)).
\end{equation}
The differential of $\phi$, computed at $(t,\lam_0)$, is
\begin{equation}
D_{(t,\lam_0)}\phi = \begin{pmatrix}
1 & 0 & 0 \\
0 & \id & 0 \\
* & * & D_{\lam_0}\EXP_{x_0,t}
\end{pmatrix}, \qquad \all t \in (0,\eps),
\end{equation}
which has maximal rank. Therefore, by the inverse function theorem, for each $t \in (0,\eps)$, there exist an interval $I_t$ and open sets $W_t,U_t,V_t$ such that
\begin{gather}
t \in I_t \subset (0,\eps),\qquad \lam_0 \in W_t \subset T^*M,\qquad \gamma(0) \in U_t \subset M,\qquad \gamma(t) \in V_t \subset M,
\end{gather}
and such that the restriction
\begin{equation}
\phi: I_t\times W_t \to I_t \times U_t \times V_t
\end{equation}
is a smooth diffeomorphism. In particular, for any $(\tau,x,y) \in I_t \times U_t\times V_t$ there exists an unique initial covector $\lam_0(\tau,x,y)\doteq \phi^{-1}(\tau,x,y)$ such that the corresponding normal trajectory starts from $x$ and arrives at $y$ in time $\tau$, i.e. $\EXP_x(\tau,\lam_0(\tau,x,y)) = y$. Moreover, we can choose $W_t \subset K$. Then such a normal trajectory is also a strict minimizer of $J_\tau$ among all the admissible curves connecting $x$ with $y$ in time $\tau$. In particular, it is unique.

As a consequence of the smoothness of the local inverse, the value function $(t,x,y) \mapsto S_t(x,y)$ is smooth on each open set $I_t \times U_t\times V_t$. Indeed, for any $(\tau,x,y) \in I_t \times U_t\times V_t$, $S_t(x,y)$ is equal to the cost $J_\tau$ of the unique (normal) minimizer connecting $x$ with $y$ in time $\tau$, namely
\begin{equation}
S_\tau(x,y) = \int_0^\tau L(\EXP_{x_0}(s,\lam_0(\tau,x,y)), \bar{u}(e^{s\vec{H}}(\lam_0(\tau,x,y))))ds, \qquad (\tau,x,y) \in I_t \times U_t\times V_t,
\end{equation}
where $\bar{u}:T^*M \to \R^k$ is the smooth map which recovers the control associated with the lift on $T^*M$ of the trajectory (see Theorem~\ref{t:pmp}). Therefore the value function is smooth on $I_t \times U_t\times V_t$, as a composition of smooth functions. We conclude the proof by defining the open set
\begin{equation}
U\doteq  \bigcup_{t \in (0,\eps)} I_t \times U_t \times V_t \subset (0,\eps) \times M \times M,
\end{equation}
which is indeed open and contains $(t,\gamma(0),\gamma(t))$ for all $t \in (0,\eps)$.

In the general case the proof follows the same lines, although the optimality of small segments of geodesics is only among all the trajectories not leaving $M'$. If we choose a different relatively compact $M''\subset M$, we find a common $\varepsilon$ such that the restriction to the interval $[0,\varepsilon]$ of all the normal geodesics with initial covector in $K$ is a strict minimum of the cost function among all the admissible trajectories not leaving $M''\cup M'$. Therefore, the value functions associated with the two different choices of the relatively compact subset agree on the intersection of the associated domains $U$.

\end{proof}

\chapter{Convergence of approximating Hamiltonian systems (Proposition~\ref{p:catapulta})}\label{a:proofcatapulta}

The goal of this section is the proof of Proposition~\ref{p:catapulta}. Actually, we discuss a more general statement for the associated Hamiltonian system. All the relevant definitions can be found in Section~\ref{s:approxtraj}.

Let $\lambda = (p,x) \in T^*\mathbb{R}^n = \mathbb{R}^{2n}$ any initial datum. Let $\phi^\eps$ and $\wh{\phi}$, respectively, the Hamiltonian flow of the $\eps$-approximated system and of the nilpotent system, respectively. A priori, these local flows are defined in a neighbourhood of the initial condition and for small time which, in general, depend on $\eps$. Notice that, by abuse of notation $\phi^0 = \wh{\phi}$.
\begin{lemma*}
For $\eps\geq 0$ sufficiently small, there exist common neighbourhood $I_0 \subset \mathbb{R}$ of $0$ and $O_{\lam _0}\subset \mathbb{R}^{2n}$ of $\lambda_0$, such that $\phi^\eps: I_0\times O_{\lam_0} \to \mathbb{R}^{2n}$ is well defined. Moreover, $\phi^\eps \to \wh{\phi}$ in the $C^{\infty}$ topology of uniform convergence of all derivatives on $I_0\times O_{\lam_0}$.
\end{lemma*}
\begin{proof}
Indeed, for any $\eps \geq 0$, the Hamiltonian flow $\phi^\eps$ is associated with the Cauchy problem
\begin{equation}
\dot{\lambda}(t) = H^\eps(\lambda(t)),\qquad \lambda(0) = \lambda_0.
\end{equation}
Moreover, $\phi^\eps$ is well defined and smooth in a neighbourhood $I^\eps_0 \times O_{\lam_0}^\eps \subset\mathbb{R}\times \mathbb{R}^{2n}$ (that depends on $\eps$). To find a common domain of definition, consider the associated Cauchy problem in $\mathbb{R}^{2n+1}$.
\begin{equation}\label{eq:cauchy}
\begin{pmatrix}
\dot{\lambda}(t) \\
\dot{\eps}(t)
\end{pmatrix} 
=
\begin{pmatrix}
H(\eps(t),\lambda(t)) \\
0
\end{pmatrix}, \qquad \begin{pmatrix} \lambda(0) \\ \eps(0)\end{pmatrix} = \begin{pmatrix} \lambda_0 \\ \eps_0 \end{pmatrix},
\end{equation}
where $H(\eps,\lam) \doteq H^\eps(\lam)$ is smooth in both variables by construction. We denote by $\Phi(t;\lambda_0,\eps_0)$ the flow associated with the Cauchy problem \eqref{eq:cauchy}. By classical ODE theory, there exists a neighbourhood $I_0 \subset \mathbb{R}$ of $0$ and $U_{\lambda_0,\eps_0} \subset \mathbb{R}^{2n+1}$ of $(\lam_0,\eps_0)$ such that $\Phi: I_0 \times U_{\lambda_0,\eps_0} \to \mathbb{R}^{2n+1}$ is well defined and smooth. Indeed $\Phi(t;\lambda_0,\eps) = \phi^\eps(t;\lambda_0)$ and $\Phi(t;\lambda_0,0) = \wh{\phi}(t;\lambda_0)$. Then, we can find an open neighbourhood $O_{\lam_0} \subset \mathbb{R}^{2n}$ of $\lam_0$ such that $O_{\lam_0} \times [0,\delta] \subset U_{\lambda_0,0}$. Thus, the sought common domain of definition for all the $\phi^\eps$, with $0\leq \eps\leq\delta$, is $I_0 \times O_{\lam_0}$.

Finally, $\Phi$ is smooth on $I_0 \times U_{\lam_0,0}$. Then $\phi^\eps$ (and all its derivatives) converge to $\wh{\phi}$ (and all the corresponding derivatives) on $I_0\times O_{\lam_0}$. Up to restricting the domain of definition of $\Phi$, we can always assume $I_0$ and $O_\lambda$ to be compact, hence the convergence is also uniform.
\end{proof}

Without loss of generality, by homogeneity, we can always reduce to $I_0 = [0,T]$. Now Proposition~\ref{p:catapulta} easily follows, since the exponential map is the projection of the Hamiltonian flow, restricted to the fiber $T_0^*\mathbb{R}^n$.

\chapter{Invariance of geodesic growth vector by dilations (Lemma~\ref{l:situa2})}\label{a:proofsitua2}

{\review
For the reader's convenience, we recall the statement of Lemma~\ref{l:situa2}. We refer to Section~\ref{s:ample} for all the relevant definitions.

\begin{lemma*} Fix $\eps>0$ and let $\gamma$ be a normal geodesic for the $\eps$-approximating system. Then the curve $\eta:=\delta_{\eps}(\gamma)$ is a normal geodesic for the original system with the same growth vector of $\g$.
\end{lemma*}
\begin{proof} The map $\delta_{\eps}$ maps admissible curves of the $\eps$-approximating system into admissible curves of the original one. Indeed if $\gamma$ is an admissible curve for the $\eps$-approximating system, associated with the control $u$, namely
\bqn
\dot \gamma(t)=\sum_{i=1}^{k}u_{i}(t)X^{\eps}_{i}(\gamma(t)),
\eqn
then the curve $\eta(t):=\delta_{\eps}(\gamma(t))$ satisfies
\bqn
\dot \eta(t)=\sum_{i=1}^{k}u_{i}(t)(\delta_{\eps*}X^{\eps}_{i})(\delta_{\eps}\gamma(t))=\sum_{i=1}^{k}\eps u_{i}(t)X_{i}(\eta(t)),
\eqn
where we used the identity $X_{i}^{\eps}=\eps \delta_{1/\eps*}X_{i}$.
In particular, if $\gamma$ is associated with the control $u$ in $\eps$-approximating  system, then $\eta$ is associated with the control $\eps u$ in the original one. Moreover\begin{equation}
J_{T}(\eta)=J_{T}(\delta_{\eps}\g)=\eps^{2}J_{T}(\g).
\end{equation}
It follows that $\delta_{\eps}$ is a one-to-one map between normal (resp. abnormal) geodesics of the $\eps$-approximating system and normal (resp. abnormal) geodesics of the original one. 

To show that $\g$ and $\eta$ have the same growth vector we proceed as in the proof of Lemma \ref{l:situa1}. Let us introduce the matrices $A^{\gamma}(t)$ and $B^{\gamma}(t)$ (resp. $A^{\eta}(t)$ and $B^{\eta}(t)$) associated with the two curves. We prove that there exists a matrix $M=M(\eps)$ such that, for all $t$, we have
\begin{equation}\label{eq:relations}
A^{\gamma}(t)=M A^{\eta}(t)M^{-1},\qquad B^{\gamma}(t)=MB^{\eta}(t).
\end{equation}  
We denote by $b^{\gamma}_{i}(t)$ (resp. $b^{\eta}_{i}(t)$) the columns of $B^{\gamma}(t)$ (resp. $B^{\eta}(t)$). Namely
\begin{gather}
B^{\gamma}(t)=\{b^{\gamma}_{1}(t),\ldots,b^{\gamma}_{k}(t)\},\qquad 
b^{\gamma}_{i}(t)=X^{\eps}_{i}(\gamma(t)),\\
B^{\eta}(t)=\{b^{\eta}_{1}(t),\ldots,b^{\eta}_{k}(t)\},\qquad 
b^{\eta}_{i}(t)=X_{i}(\eta(t)).
\end{gather}
We prove the second relation of \eqref{eq:relations} by a direct computation:
\begin{align}
b_{i}^{\gamma}(t)&= X_{i}^{\eps}(\gamma(t))=\eps(\delta_{1/\eps*}X_{i})(\delta_{1/\eps}\eta(t))=\eps \delta_{1/\eps*}b_{i}^{\eta}(t)=M b_{i}^{\eta}(t).
\end{align}
where $M$ is the matrix representing the invertible linear map $\eps \delta_{1/\eps*}$. We stress that $M$ does not depend on $t$.
We now prove the first relation of \eqref{eq:relations}:  
\begin{align}
A^{\gamma}(t)&= \sum_{i=1}^{k}u_{i}(t)\frac{\partial X^{\eps}_{i}}{\partial x}(\gamma(t))= 
\eps \delta_{1/\eps*} \sum_{i=1}^{k}u_{i}(t) \frac{\partial X_{i}\circ \delta_{\eps}}{\partial x}(\gamma(t))\\
&=\eps \delta_{1/\eps*} \sum_{i=1}^{k}u_{i}(t)\frac{\partial X_{i}}{\partial x} (\delta_{\eps}\gamma(t))\delta_{\eps*} \\
&=\eps \delta_{1/\eps*} \left( \sum_{i=1}^{k}\eps u_{i}(t)\frac{\partial X_{i}}{\partial x} (\delta_{\eps}\gamma(t)) \right)\frac{1}{\eps}\delta_{\eps*} 
=M A^{\eta}(t)M^{-1},
\end{align}
where we recall that $\eta(t)=\delta_{\eps}\gamma(t)$ is associated with the control $\eps u$ (in the original system). 
An induction step and the fact that $M$ does not depend on $t$ implies  
\bqn
B_{i}^{\gamma}(t)=M B_{i}^{\eta}(t),\qquad \all t, \all i\geq 1.
\eqn
Here $B_{i}^{\gamma}(t)$ (resp.\ $B_{i}^{\eta}(t)$) are the matrices defined in Eq.~\eqref{eq:BB0}, associated with the geodesic $\gamma$ of the $\eps$-approximating system (resp.\ $\eta$ of the original system).
Then the criterion of Section~\ref{s:crit} implies  
\begin{equation} 
\dim \DD^{i}_{\gamma}(t)=\rank\{B^{\gamma}_{1}(t),\ldots,B^{\gamma}_{i}(t)\}=\rank\{B^{\eta}_{1}(t),\ldots,B^{\eta}_{i}(t)\}=\dim \DD^{i}_{\eta}(t),\qquad \all t, \all i\geq 1. \qedhere
\end{equation}
\end{proof}

\finereview}

\chapter{Regularity of $C(t,s)$ for the Heisenberg group (Proposition~\ref{p:brute})}\label{a:proofbrute}

For the reader's convenience, we briefly recall the statement of Proposition~\ref{p:brute}. We refer to Section~\ref{s:Heis} for all the relevant definitions. 
\begin{proposition*}
The function $C(t,s)$ is $C^1$ in a neighbourhood of the origin, but not $C^2$. In particular, the function $\partial_{ss} C(t,0)$ is not continuous at the origin. However, the singularity at $t=0$ is removable, and the following expansion holds, for $t> 0$:
\begin{multline}
\frac{\partial^2 C}{\partial s^2}(t,0) = 1 + 3 \sin^2(\phi_2- \phi_1) + \frac{1}{2}[2h_{z,2}\sin(\phi_2-\phi_1)-h_{z,1}\sin(2\phi_2 - 2\phi_1)]t - \\ - \frac{2}{15}h_{z,1}^2\sin^2(\phi_2-\phi_1)t^2+O(t^3).
\end{multline}
If the geodesic $\g_{2}$ is chosen to be a straight line (i.e. $h_{z,2}=0$), then
\begin{equation}\label{eq:devapp}
\frac{\partial^2 C}{\partial s^2}(t,0) = 1 + 3 \sin^2(\phi_2- \phi_1) - \frac{h_{z,1}}{2}\sin(2\phi_2 - 2\phi_1)t -\frac{2}{15}h_{z,1}^2\sin^2(\phi_2-\phi_1)t^2+O(t^3).
\end{equation}
where $\lambda_j = (ie^{i\phi_j},h_{z,j}) = (-\sin\phi_j,\cos\phi_j,h_{z,h}) \in T_0^*M$ is the initial covector of the geodesic $\gamma_j$.
\end{proposition*}
\begin{proof}
The proof is essentially a brute force computation. In the following, we show the relevant calculation to obtain the zeroth order term in Eq.~\eqref{eq:devapp}, which is sufficient to prove the non-continuity of the function $t \mapsto \partial_{ss}C(t,0)$ at $t=0$. Indeed, since $C(0,s) = s^2/2$, we obtain $\partial_{ss} C(0,0)=1$, while from Eq.~\eqref{eq:devapp}, $\lim_{t\to0^+} \partial_{ss} C(0,s) =1 + 3 \sin^2(\phi_2- \phi_1)$. For $i=1,2$, let $\gamma_i(\tau) = (w_i(\tau),z_i(\tau))$. Then
\begin{gather}
w_i(\tau) = \frac{e^{i\phi_i}}{a_i}\left(e^{ia_i \tau} -1\right) = i e^{i\phi_i}\tau - \frac{1}{2} a_i e^{i\phi_i}\tau^2 + O(\tau^3), \\
z_i(\tau) = \frac{a_i \tau - \sin(a_i \tau)}{2a_i^2} = O(\tau^3).
\end{gather}
For $(t,s) \neq (0,0)$, dropping the subscripts from $R_{t,s}$ and $\xi_{t,s}$, we have
\begin{multline}\label{eq:dev2}
\partial_{tt} C (t,s) = \frac{1}{2}\partial_{tt}R^2 \frac{\theta^2(\xi)}{\sin^2\theta(\xi)} + 4\partial_t R^2 \theta(\xi)\partial_t\xi+2R^2\dot{\theta}(\xi) (\partial_t\xi)^2 + 2R^2\theta(\xi)\partial_{tt}\xi = \\ = A_1(t,s) + A_2(t,s) + A_3(t,s) + A_4(t,s),
\end{multline}
where $A_i$ are the four addends of the upper line of Eq.~\eqref{eq:dev2}. In order to compute Eq.~\eqref{eq:dev2}, we employ the following calculations
\begin{gather}
R^2_{t,s} = |w_2(s)-w_1(t)|^2, \\
\partial_t R^2_{t,s} = \dot{w}_1(t)[\overline{w}_1(t) - \overline{w}_2(s)] + [w_1(t) - w_2(s)]\dot{\overline{w}}_1(t), \\
\partial_{tt}R^2_{t,s} = \ddot{w}_1(t)[\overline{w}_1(t) - \overline{w}_2(s)] + 2|\dot{w}_1(t)|^2 + \ddot{\overline{w}}_1(t)[w_1(t)-w_2(s)],\\
Z_{t,s} = -z_1(t) + z_2(s) + \frac{1}{2}\Im(w_1(t)\overline{w}_2(s)),\\
\partial_t Z_{t,s} = -\dot{z}_1(t) + \frac{1}{2}\Im(\dot{w}_1(t)\overline{w}_2(s)),\\
\partial_{tt}Z_{t,s} = -\ddot{z}_1(t) + \frac{1}{2}\Im(\ddot{w}_1(t)\overline{w}_2(s)),\\
\xi_{t,s} = Z_{t,s} / R^2_{t,s},\\
\partial_t\xi_{t,s} = \frac{\partial_t Z}{R^2}- \frac{Z}{R^4}\partial_t R^2,\\
\partial_{tt}\xi_{t,s} = \frac{\partial_{tt}Z}{R^2}-2\frac{\partial_t Z}{R^4}\partial_t R^2-\frac{Z}{R^4}\partial_{tt}R^2 + 4\frac{Z}{R^6}(\partial_t R^2)^2,
\end{gather}
where $\Im$ is the imaginary part, the overline is the complex conjugate, and the dot is the derivative w.r.t. the argument. Moreover, the Taylor series for $\theta$ is
\begin{equation}
\theta(x) = 6 x + O(x^3).
\end{equation}
By computing everything at $t=0$, and then taking the limit $s\to 0$, we obtain
\begin{gather}
\lim_{s\to 0}A_1(0,s) = 1, \\
\lim_{s\to 0}A_2(0,s) = 0, \\
\lim_{s\to 0}A_3(0,s) = 3\sin^2(\phi_1-\phi_2), \\
\lim_{s\to 0}A_4(0,s) = 0,
\end{gather}
therefore $\lim_{s\to0}\partial_{tt}C(0,s) = 1+ 3\sin^2(\phi_1-\phi_2)$, which is the zeroth order term of Eq.~\eqref{eq:devapp}. The term arising from the addend $A_3(0,s)$ is responsible for the discontinuity of $\partial_{tt}C(0,s)$ at $s=0$. The remaining terms can be obtained by taking expansions up to the fourth order of $R^2, Z, \theta$, and replacing them in Eq.~\eqref{eq:dev2}.
\end{proof}

\chapter{Basics on curves in Grassmannians (Lemma~\ref{l:flag0} and~\ref{l:flag})}\label{a:flag}

{\review
Let $W(\cdot)$ be a smooth curve in the Grassmanian $G_{k}(E)$ where $E$ is a vector space of dimension $n$. In other words $W(\cdot)$ is a smooth family of $k$-dimensional subspaces of $E$. 
A smooth section of $W(\cdot)$ is a smooth curve $t\mapsto w(t)$ in $E$ such that $w(t)\in W(t)$ for all $t$.  

Without loss of generality (all our considerations are local in $t$) we assume also that the family of subspaces is generated by a moving frame, namely one can find smooth sections $e_{1}(\cdot),\ldots,e_{k}(\cdot)$ such that, for all $t$, we have
\begin{equation}
W(t)=\tx{span}\{e_{1}(t),\ldots,e_{k}(t)\}.
\end{equation}

\bl \label{l:tgr}
For every fixed $t$, the differentiation of sections defines a linear map
\begin{equation} \label{eq:tgr}
\delta:W(t) \to E/W(t), \qquad \bar w\mapsto \dot{w}(t) \mod  W(t).
\end{equation}
where $w(\cdot)$ is a smooth section of $W(\cdot)$ such that $w(t)=\bar w\in W(t)$.
\el
\begin{proof}
We have to prove that the map \eqref{eq:tgr} is a well defined linear map. Let us consider a moving frame $\{e_{1}(s),\ldots,e_{k}(s)\}$ in $E$ such that for every $s$ one has
\begin{equation}
W(s)=\tx{span}\{e_{1}(s),\ldots,e_{k}(s)\}.
\end{equation}
Consider now two different smooth sections $w_{1}(\cdot),w_{2}(\cdot)$ of $W(\cdot)$ satisfying $w_{1}(t)=w_{2}(t)=\bar w$.
Their difference can be written as a linear combination, with smooth coefficients, of the frame $\{e_{1}(s),\ldots,e_{k}(s)\}$
\begin{equation}
w_{2}(s)-w_{1}(s)=\sum_{i=1}^{k}\alpha_{i}(s) e_{i}(s),
\end{equation}
where $\alpha_{i}(s)$ are smooth functions such that $\alpha_{i}(t)=0$ for every $i=1,\ldots,k$. It follows that
\begin{align} \label{eq:tgr2}
\dot{w}_{2}(s)-\dot{w}_{1}(s)=\sum_{i=1}^{k}\dot{\alpha}_{i}(s) e_{i}(s)+\sum_{i=1}^{k}\alpha_{i}(s) \dot{e}_{i}(s),
\end{align}
and evaluating \eqref{eq:tgr2} at $s=t$ one has
\begin{align*}
\dot{w}_{2}(t)-\dot{w}_{1}(t)=\sum_{i=1}^{k}\dot{\alpha}_{i}(t) e_{i}(t)\in W(t).
\end{align*}
This shows that $\dot w_{2}(t) = \dot w_{1}(t) \mod W(t)$, hence the map \eqref{eq:tgr} is well defined.
Analogously, one can prove that the map does not depend on the moving frame defining $W(t)$. Finally, the linearity of the map \eqref{eq:tgr} is evident.
\end{proof}
\brem The proof of Lemma \ref{l:tgr} shows that actually the tangent space to the Grassmannian $G_{k}(E)$ at a point $W$ is isomorphic with the set $\tx{Hom}(W,E/W)$.
\erem

Let us now consider a smooth curve $V(\cdot)$ in the Grassmanian $G_{k}(E)$
and define the flag for $E$ at each time $t$ as follows:
\begin{equation}
V^{(i)}(t) := \spn\left\lbrace \frac{d^j}{dt^j} v(t) \, \Bigg| \, v(t) \in V(t), \, v(t) \text{ smooth},\,0\leq j \leq i\right\rbrace \subset E, \qquad i \geq 0.
\end{equation}
In particular this defines a filtration of subspaces for all $t$:
\begin{equation}
V(t) = V^{(0)}(t) \subset V^{(1)}(t) \subset V^{(2)}(t) \subset \ldots \subset E.
\end{equation}

\begin{remark} \label{r:tgr} 
Notice that, following the notation just introduced, the image of the linear map \eqref{eq:tgr} is $W^{(1)}(t)/W(t)$. This shows that Lemma \ref{l:tgr} can be restated by saying that there exists a well-defined \emph{surjective} linear map 
\begin{equation}
\delta:W(t) \to W^{(1)}(t)/W(t).
\end{equation}
\end{remark}

In what follows we assume that the curve $V(\cdot)$ is \emph{equiregular} for all $t$, namely the dimensions $h_i(\cdot):=\dim V^{(i)}(\cdot)$ are constant.
\begin{proposition}
Let $V(\cdot)$ be an equiregular curve in $G_{k}(E)$. For every $i\geq 0$ the derivation of sections induces surjective linear maps
\begin{equation}
\delta_{i} : V^{(i)}(t)/V^{(i-1)}(t) \to V^{(i+1)}(t)/V^{(i)}(t), \qquad \all t.
\end{equation}
In particular, the following inequalities for the dimensions $h_i = \dim V^{(i)}$ hold true:
\begin{equation}
h_{i+1}-h_i \leq h_i - h_{i-1}, \qquad \all i\geq 0.
\end{equation}
\end{proposition}
\begin{proof}
Since the curve $V(\cdot)$ is equiregular, we can apply Lemma \ref{l:tgr} with $W(\cdot)=V^{(i)}(\cdot)$ in the Grassmannian $G_{h_{i}}(E)$. Notice that, $W^{(1)}(t)=(V^{(i)}(\cdot))^{(1)}(t)=V^{(i+1)}(t)$, i.e., the $(i+1)$-th extension coincides with the space generated by derivatives of sections of the $i$-th extension (see also Remark \ref{r:tgr}).

Thus we have well defined surjective linear maps 
\begin{equation} \label{eq:vi}
\delta_{i} : V^{(i)}(t) \to V^{(i+1)}(t)/V^{(i)}(t).
\end{equation}
For the same reason $V^{(i-1)}(t)\subset \ker \delta_{i}$ for every $i$. Hence \eqref{eq:vi} descends to a surjective linear map
\begin{equation}
\delta_{i} : V^{(i)}(t)/V^{(i-1)}(t) \to V^{(i+1)}(t)/V^{(i)}(t).\qedhere
\end{equation}
\end{proof}
\finereview}

\chapter{Normal conditions for the canonical frame}\label{a:normal}
Here we rewrite the \emph{normal} condition for the matrix $R(t)$ mentioned in Definition \ref{d:normalmov} (and defined in \cite{lizel}) according to our notation.\index{canonical frame!normal conditions}

\begin{definition} The matrix $R_{ab,ij}$ is \emph{normal} if it satisfies:
\begin{itemize}
\item[(i)] global symmetry: for all $ai,bj\in D$
\[R_{ab,ij}=R_{ba,ji}.\]
\item[(ii)] partial skew-symmetry: for all $ai,bi\in D$ with $n_{a}=n_{b}$ and $i<n_{a}$ 
\[R_{ab,i(i+1)}=R_{ba,i(i+1)}.\]
\item[(iii)] vanishing conditions: the only possibly non vanishing entries $R_{ab,ij}$ satisfy
\begin{itemize}
\item[(iii.a)] $n_{a}=n_{b}$ and $|i-j|\leq 1$,
\item[(iii.b)] $n_{a}>n_{b}$ and $(i,j)$ belong to the last $2n_{b}$ elements of Table \ref{t:normaltable}.
\begin{table}[htdp]
\begin{center}
\caption{Vanishing conditions.}
\begin{tabular}{|c||c|c|c|c|c|c|c|c|c|c|c|c|c|}
\hline
$i$ & $1$ & $1$ & $2$ & $\cdots$ & $\ell$ & $\ell$ & $\ell+1$ & $\cdots$ & $n_{b}$ & $n_{b}+1$ & $\cdots$ & $n_a-1$ & $n_a$ \\
\hline
$j$ & $1$ & $2$ & $2$ & $\cdots$ & $\ell$ & $\ell+1$ & $\ell+1$ & $\cdots$ & $n_b$ & $n_b$ & $\cdots$ & $n_b$ & $n_b$ \\
\hline
\end{tabular}
\label{t:normaltable}
\end{center}
\end{table}
\end{itemize}
\end{itemize}
The sequence is obtained as follows: starting from $(i,j)=(1,1)$ (the first boxes of the rows $a$ and $b$), each next even pair is obtained from the previous one by increasing $j$ by one (keeping $i$ fixed). Each next odd pair is obtained from the previous one by increasing $i$ by one (keeping $j$ fixed). This stops when $j$ reaches its maximum, that is $(i,j) = (n_b,n_b)$. Then, each next pair is obtained from the previous one by increasing $i$ by one (keeping $j$ fixed), up to $(i,j) = (n_a,n_b)$. The total number of pairs appearing in the table is $n_b+n_a-1$.
\end{definition}

\chapter{Coordinate representation of flat, rank 1 Jacobi curves (Proposition~\ref{prop:rank1})}\label{prop:rank1_proof}
\begin{proposition*}[Special case of Theorem~\ref{t:Sasymptotic}]
Let $\Lambda(\cdot)$ a Jacobi curve of rank $1$, with vanishing $R(t)$. The matrix $S$, in terms of the canonical frame, is
\begin{equation}
S_{ij}(t) = \frac{(-1)^{i+j-1}}{(i-1)!(j-1)!}\frac{t^{i+j-1}}{(i+j-1)} = \wh{S}_{ij} t^{i+j-1},\qquad i,j=1,\dots,n .
\end{equation}
Its inverse is
\begin{equation}
S^{-1}(t)_{ij} = 
 \frac{ - 1}{i+j-1} \binom{n + i -1}{i-1}\binom{n+ j -1}{j-1}  \frac{(n!)^2}{(n-i)!(n-j)!} = \frac{\wh{S}^{-1}_{ij}}{t^{i+j-1}},\qquad i,j=1,\dots,n .
\end{equation}
\end{proposition*}
\begin{proof}
From Eqs.~\eqref{Brank1} and~\eqref{Ainvrank1}, we obtain
\begin{multline}
S_{ij}(t) = \sum_{k=1}^n A^{-1}_{ik}B_{kj} 
= \sum_{k=1}^i\frac{(-1)^{i-k}t^{i-k}}{(i-k)!}\frac{(-1)^{j}t^{k+j-1}}{(k+j-1)!} 
= (-1)^{j}t^{i+j-1}\sum_{k=1}^i \frac{(-1)^{i-k}}{(k+j-1)!(i-k)!} = \\
= (-1)^{j} t^{i+j-1}\sum_{\ell=0}^{i-1}\frac{(-1)^\ell}{(i+j-1-\ell)!\ell!}  
= \frac{(-1)^{j} t^{i+j-1}}{(i+j-1)!}\sum_{\ell=0}^{i-1} \binom{i+j-1}{\ell}(-1)^\ell = \\
= \frac{(-1)^{i+j-1}t^{i+j-1}}{(i+j-1)!}\binom{i+j-2}{j-1} = 
 \frac{(-1)^{i+j-1}}{(i-1)!(j-1)!}\frac{t^{i+j-1}}{(i+j-1)}.
\end{multline}
By Cramer's rule, the inverse of $S(t)$ is
\begin{equation}\label{Sinvtemp}
S^{-1}_{ij}(t) = \frac{(-1)^{i+j}\det\left[\dfrac{(-1)^{\ell+k-1}}{(\ell-1)!(k-1)!}\dfrac{t^{\ell+k-1}}{(\ell+k-1)}\right]_{\substack{\ell\neq j \\ k\neq i}}}{\det\left[\dfrac{(-1)^{\ell+k-1}}{(\ell-1)!(k-1)!}\dfrac{t^{\ell+k-1}}{(\ell+k-1)}\right]} = \frac{-(i-1)!(j-1)!\det\left[\dfrac{1}{\ell+k-1}\right]_{\substack{\ell\neq j \\ k\neq i}}}{t^{i+j-1}\det\left[\dfrac{1}{\ell+k-1}\right]}.
\end{equation}
Now we compute the ratio of determinants in the last factor of Eq.~\eqref{Sinvtemp}. Consider a generic matrix of the form $H_{\ell k} = \frac{1}{x_\ell + x_k}$, for $\ell,k=1,\dots,n$. For fixed $i,j \in \{1,\dots,n\}$, we can express the determinant of $H$ in terms of the the $i,j$-th minor, by rows and columns operations as follows. First, subtract the $i$-th column from each other column. We obtain a new matrix, $H^\prime$, whose $i$-th column is the same of $H$, while, for $k\neq i$
\begin{equation}
H^\prime_{\ell k} = \frac{1}{x_\ell+y_k}- \frac{1}{x_\ell+y_i}= \frac{y_i-y_k}{(x_\ell+y_i)(x_\ell+y_k)},\qquad \ell,k=1,\dots,n.
\end{equation}
Indeed $\det H^\prime = \det H$. Then, we collect the factor $\frac{1}{x_\ell + y_i}$ from each row, and the factor $(y_i - y_k)$ from each column but the $i$-th. We obtain
\begin{equation}
\det\left[\dfrac{1}{x_\ell+x_k}\right] = \prod_{\ell=1}^n \frac{1}{x_\ell + y_i}\prod_{\substack{k=1 \\ k\neq i}}^n (y_i-y_k)\det \begin{bmatrix}
\frac{1}{x_1 + y_1} & \frac{1}{x_1 + y_2} & \dots & 1 & \dots &\frac{1}{x_1 + y_n} \\
\frac{1}{x_2 + y_1} & \frac{1}{x_2 + y_2} & \dots & 1 & \dots  &\frac{1}{x_2 + y_n} \\
\vdots & \vdots & & \vdots &  & \vdots \\
\frac{1}{x_n + y_1} & \frac{1}{x_n + y_2} & \dots & 1 & \dots & \frac{1}{x_n + y_n}
\end{bmatrix},
\end{equation}
where the entries of the $i$-th column are equal to $1$. Now, subtract the $j$-th row from each other row, but the $j$-th itself. Collect again the common factors. We obtain
\begin{equation}\label{Hilbertmatrix}
\det\left[\dfrac{1}{x_\ell+x_k}\right] = (-1)^{i+j} \prod_{\ell=1}^n \frac{1}{x_\ell + y_i}\prod_{\substack{k=1 \\ k\neq i}}^n (y_i-y_k)\prod_{\substack{k=1 \\ k\neq i}}^n \frac{1}{x_j + y_k}\prod_{\substack{\ell=1 \\ \ell\neq j}}^n (x_j-x_\ell)\det\left[\dfrac{1}{x_\ell+x_k}\right]_{\substack{\ell \neq j \\ k \neq i}}.
\end{equation}
Now we apply the result of Eq.~\eqref{Hilbertmatrix} to our case, i.e. $x_\ell = y_\ell = \ell - \tfrac{1}{2}$. Therefore we obtain
\begin{multline}\label{detratio}
\frac{\det\left[\dfrac{1}{\ell+k-1}\right]_{\substack{\ell\neq j \\ k\neq i}}}{\det\left[\dfrac{1}{\ell+k-1}\right]} = 
 (-1)^{i+j} \prod_{\ell=1}^n (\ell + i -1)\prod_{\substack{k=1 \\ k\neq i}}^n\frac{1}{i-k}\prod_{\substack{k=1 \\ k\neq i}}^n (j + k-1)\prod_{\substack{\ell=1 \\ \ell\neq j}}^n \frac{1}{j-\ell} = \\ = \frac{1}{i+j-1}\frac{(n!)^2}{(i-1)!(j-1)!}\binom{i+n-1}{i-1}\binom{j+n-1}{j-1}.
 \end{multline}
Eq.~\eqref{Sinvtemp} and Eq.~\eqref{detratio}, together, give the desired formula.
\end{proof}

\chapter{A binomial identity (Lemma~\ref{l:lemmacoeff})}\label{s:lemmacoeff}
\begin{lemma*}
Let
\begin{equation}\label{original}
\Omega(n,m) = \frac{n m}{(n+1)(m+1)} \sum_{j=1}^{n}  \sum_{i=1}^{m} (-1)^{i+j}\binom{n+i-1}{i-1}\binom{n+1}{i+1}\binom{m+j-1}{j-1}\binom{m+1}{j+1} \frac{i+j+2}{i+j+1}.
\end{equation}
Then
\begin{equation}\label{final}
\Omega(n,m)= \begin{cases}
0 & |n-m|\geq 2,  \\
\frac{1}{4(n+m)} & |n-m| = 1, \\
\frac{n}{4n^2-1} & n = m.
\end{cases}
\end{equation}
\end{lemma*}
\begin{proof}
It is clear that $\Omega(n,m) = \Omega(m,n)$, then we can assume without loss of generality that $n \leq m$. The case $m=n=1$ can be easily proved by a direct computation. Then, we also assume $m \geq 2$. Let us write $\Omega(n,m)$ in a more compact form. In order to do that, let $M(n,m)$ be the $n\times m$ matrix of components
\begin{equation}
M(n,m)_{ij} \doteq  (-1)^{i+j} \frac{i+j+2}{i+j+1},\qquad i=1,\ldots,n,\quad j=1,\ldots,m.
\end{equation}
and let $v(m)$ be the $m$-dimensional column vector of components
\begin{equation}
v(m)_j = \frac{m}{m+1}\binom{m+1}{j+1}\binom{m+j-1}{j-1},\qquad j=1,\ldots,m.
\end{equation}
Then 
\begin{equation}
\Omega(n,m) = v(n)^* M(n,m) v(m).
\end{equation} 
Consider first the $i-th$ component of the $n$-dimensional vector $w(n,m)\doteq  M(n,m)v(m)$, namely
\begin{equation}
w(n,m)_i = \sum_{j=1}^m (-1)^{i+j} \frac{i+j+2}{i+j+1} \frac{m}{m+1}\binom{m+1}{j+1}\binom{m+j-1}{j-1} =\frac{(-1)^i}{(m-1)!} \sum_{j=0}^m (-1)^j \binom{m}{j}Q_i(j),
\end{equation}
where, for each $i=1,\dots,n$, $Q_i(j)$ is a rational function (in the variable $j$) defined by
\begin{equation}
Q_i(j) =  \frac{(m+j-1)!}{(j-1)!(j+1)} \frac{i+j+2}{i+j+1} = j(j+2)(j+3) \dots (j+m-1)	\frac{i+j+2}{i+j+1}.
\end{equation}
Notice that the factor $(j+1)$ does not appear (remember also that $m \geq 2$). The idea is to exploit the following beautiful identity.
\begin{lemma}\label{lem}
Let $m\geq 2$. Let $P(x)$ be any polynomial of degree smaller than $m$, then
\begin{equation}
\sum_{j=0}^m (-1)^j \binom{m}{j}P(j) = 0.
\end{equation}
\end{lemma}
\begin{proof}
It is sufficient to prove the statement for $P(x) = x^i$, with $0\leq i < m$, since any polynomial of degree smaller than $m$ is a linear combination of such monomials. By Newton's binomial formula, we have
\begin{equation}
(x-1)^m = (-1)^m \sum_{j=0}^m (-1)^{j} 	\binom{m}{j} x^{j}.
\end{equation}
The result easily follows observing that any derivative of order strictly smaller than $m$, evaluated at $x=1$ vanishes.
\end{proof}
We will see that, for many values of $i$, the denominator of $Q_i(j)$ factors the numerator, and then $Q_i(j)$ is actually a polynomial of degree $m-1$ in the variable $j$. Then we apply Lemma~\ref{lem} to show that $w(n,m)_i \neq 0$ only if $i =m-1,m$. In particular, since $w(n,m)$ is a $n$-dimensional vector, if $n \leq m-2$ then $w(n,m) = 0$ and $\Omega(n,m)$ vanishes too. Then we will explicitly compute the coefficient for $n=m-1$ and $n=m$.

Observe that, for each $i=1,\dots,n$, the numerator of $Q_i(j)$ is a polynomial of degree $m$ in the variable $j$. Therefore there exists a polynomial $P_i(j)$ (of degree strictly smaller than $m$) and a number $R_i$ such that
\begin{equation}
Q_i(j) = P_i(j) + \frac{R_i}{i+j+1}.
\end{equation}
It is easy to compute the remainder. Observe that
\begin{equation}
R_i = -(i+j+1)P_i(j) + Q_i(j) (i+j+1).
\end{equation}
Then, evaluating at $j = -i-1$, we obtain
\begin{equation}\label{R}
R_i = \begin{cases}
0 & i=1,2,\dots,m-2, \\
\displaystyle (-1)^{m-1}\frac{m!}{m-1} & i =m-1,\\
\displaystyle (-1)^{m-1}\frac{(m+1)!}{m} & i =m .
\end{cases}
\end{equation}
By Lemma~\ref{lem} we have
\begin{equation}
w(n,m)_i = \frac{(-1)^i}{(m-1)!} \sum_{j=0}^m (-1)^j \binom{m}{j}\frac{R_i}{i+j+1},
\end{equation}
which, by Eq.~\eqref{R}, is indeed zero if $i=1,2,\dots,m-2$. Then, since $\Omega(n,m) = v(n)^* w(n,m)$, we obtain after some straightforward computations the following formula:
\begin{equation}\label{cases}
\Omega(n,m) = \begin{cases}
0 & m-n > 2, \\
\displaystyle\binom{2m -3}{m-2} \sum_{j=0}^m  \binom{m}{j}\frac{(-1)^j}{j+m} & n = m-1,\\
\displaystyle\binom{2m-2}{m-2} (m+1) \sum_{j=0}^m  \binom{m}{j}\frac{(-1)^j}{j+m} - \binom{2m-1}{m-1} m \sum_{j=0}^m  \binom{m}{j}\frac{(-1)^j}{j+m+1} & n=m.
\end{cases}
\end{equation}
In order to obtain the result, it only remains to compute the sums appearing in Eq.~\eqref{cases}. Indeed these are of the form
\begin{equation}
S_k \doteq  \sum_{j=0}^m (-1)^j\binom{m}{j}\frac{1}{j+k},
\end{equation}
where $k$ is a positive integer. We have the following, remarkable identity.
\begin{equation}\label{identity}
\sum_{j=0}^m (-1)^j\binom{m}{j}\frac{1}{j+k} = \frac{m! (k-1)!}{(m+k)!}.
\end{equation}
By plugging Eq.~\eqref{identity} in Eq.~\eqref{cases} we obtain the result. Then we only need to prove Eq.~\eqref{identity}. Indeed, for $k$ a positive integer, let us define the following function
\begin{equation}
f_k(x)\doteq \sum_{j=0}^m (-1)^j \binom{m}{j}\frac{(-x)^{j+k}}{j+k}.
\end{equation}
Indeed $S_k = f_k(-1)$. Let us compute the derivative of $f_k$.
\begin{equation}
\frac{d f_k}{dx} = - \sum_{j=0}^m (-1)^j\binom{m}{j}(-x)^{j+k-1} = (-1)^k x^{k-1}(1+x)^m.
\end{equation}
where we used Newton's binomial formula. Then 
\begin{equation}
S_k = f_k(-1) = (-1)^k \int_0^{-1}  x^{k-1}(1+x)^m.
\end{equation}
By integrating by parts $k-1$ times, we obtain the result
\begin{equation}
S_k = f_k(-1) = \frac{m! (k-1)!}{(m+k)!}.\qedhere
\end{equation}
\end{proof}

\chapter{A geometrical interpretation \texorpdfstring{of $\dot{c}_t$}{}} \label{s:ctdot} \index{geodesic cost!geometrical interpretation}

{\review
In this appendix we provide a geometrical interpretation of the derivative $\dot{c}_t$ of the geodesic cost. 

In what follows, for simplicity, we restrict to the case of a geodesic cost induced by a Riemannian distance $\dist:M\times M \to \R$, namely
\begin{equation}
c_t(x)=-\frac{1}{2t} \dist^2(x,\gamma(t)),
\end{equation} 
where $\gamma(t) = \exp_{x_0}(tv)$ is a Riemannian geodesic starting at $x_0$ with initial vector $v\in T_{x_{0}}M$.

In the following, for any $x,y \in M$, the symbol $\Sigma_x \in M$ is the usual domain of smoothness of the function $y \mapsto \dist^2(x,y)$ (which, in the Riemannian setting, is precisely the complement of the cut locus). Thus, let us define $W_{x,y}^t \in T_y M$ as the tangent vector at time $t$ of the unique geodesic connecting $x$ with $y$ in time $t$. We have the identities
\begin{equation}\label{eq:utilitaria}
\frac{1}{2} \nabla_y \dist^2(x,y) = W_{x,y}^1 = t W_{x,y}^t,\qquad \dist^2(x,y)=\|W_{x,y}^1\|=t^{2}\|W_{x,y}^t\|.
\end{equation}
where $\nabla_y$ denotes the Riemannian gradient w.r.t. $y$. Next we compute, for every $t>0$
\begin{equation}
\begin{aligned}
\dot{c}_t(x) = \frac{d}{dt} c_t(x) & = \frac{1}{t^2}\dist^2(x,\gamma(t)) - \frac{1}{2t}\frac{d}{dt} \dist^2(x,\gamma(t))  \\
& = \frac{1}{2}\|W_{x,\gamma(t)}^t\|^2 - \langle\dot\gamma(t)| W^t_{x,\gamma(t)}\rangle  \\
& = \frac{1}{2}\|\dot{\gamma}(t) - W_{x,\gamma(t)}^t\|^2 -\frac{1}{2}\|\dot{\gamma}(t)\|^2.
\end{aligned}
\end{equation}
where we used \eqref{eq:utilitaria} and the Euclidean identity $\|v-w\|^{2}-\|v\|^{2}=\|w^{2}\|-2 \langle v|w\rangle$.

Let us rewrite the last expression. Since $\g$ is a geodesic, one has that $\|\dot{\gamma}(t)\|=\|v\|$ is constant. Moreover, by definition of $W_{x,y}^{t}$, we have $\dot \gamma(t) = W_{x_0,\gamma(t)}^t$. Thus, up to an additive constant (that does not change the fact that $\dot{c}_t$ has a critical point at $x_0$), we have
\begin{equation}\label{eq:ctdot1}
\dot{c}_t(x) = \frac{1}{2}\|W_{x_0,\gamma(t)}^t - W_{x,\gamma(t)}^t\|^2.
\end{equation}
\begin{remark}
There is no difference whatsoever in the sub-Riemannian case, replacing the initial vector $v$ of the geodesic by its initial covector $\lam$ and its squared norm $\|v\|^{2}$ by $2H(\lam)$. In this case, the Riemannian exponential map $\exp_{x_{0}}$ is naturally replaced by the sub-Riemannian exponential map $\EXP_{x_{0}}$. In Hamiltonian terms, if $H$ denotes the (sub)-Riemannian Hamiltonian and $\lambda_{x,y}^t$ is the covector at time $t$ of the unique minimizer connecting $x$ with $y$ in time $t$, we have (again, up to an additive constant):
\begin{equation}\label{eq:ctdot2}
\dot{c}_t(x) = H(\lambda_{x_0,\gamma(t)}^t - \lambda_{x,\gamma(t)}^t).
\end{equation}
\end{remark}

Formulae~\eqref{eq:ctdot1}-\eqref{eq:ctdot2} have a natural physical interpretation as follows. Suppose that two guys $A$ and $B$ live on a curved (sub)-Riemannian manifold, at points $x_A$ and $x_B$ respectively (see Figure~\ref{fig:interpr}). Then $A$ chooses a geodesic $\gamma(t)$, starting from $x_A$, and tells $B$ to meet at some point $\gamma(t)$ (at time $t$). The guy $B$ must choose carefully his initial velocity (or covector) in order to meet $A$ at the point $\gamma(t)$ starting from $x_B$, following a geodesic for time $t$. When they meet at $\gamma(t)$ at time $t$, they compare their velocities (or their covectors) by computing the length of their difference (or the energy of the difference of the covectors). This is the value of the function $\dot{c}_t$, up to a constant (see Figure~\ref{fig:interpr}).

If $A$ and $B$ live in a positively (resp. negatively) curved Riemannian manifold they experience that their vectors (when compared at the point of meeting $\gamma(t)$) are more (resp. less) divergent w.r.t. the flat case (see Figure~\ref{fig:interpr}).
The curvature hides in the behaviour of this function for small $t$ and $x$ close to $x_{0}$. 
\brem
Notice that we do not need any parallel transport (the guys meet at the point $\gamma(t)$ and make \emph{there} their comparison) and we only used the concept of ``optimal trajectory'' and ``difference of the cost''. This interpretation indeed works for a general optimal control system.
\erem
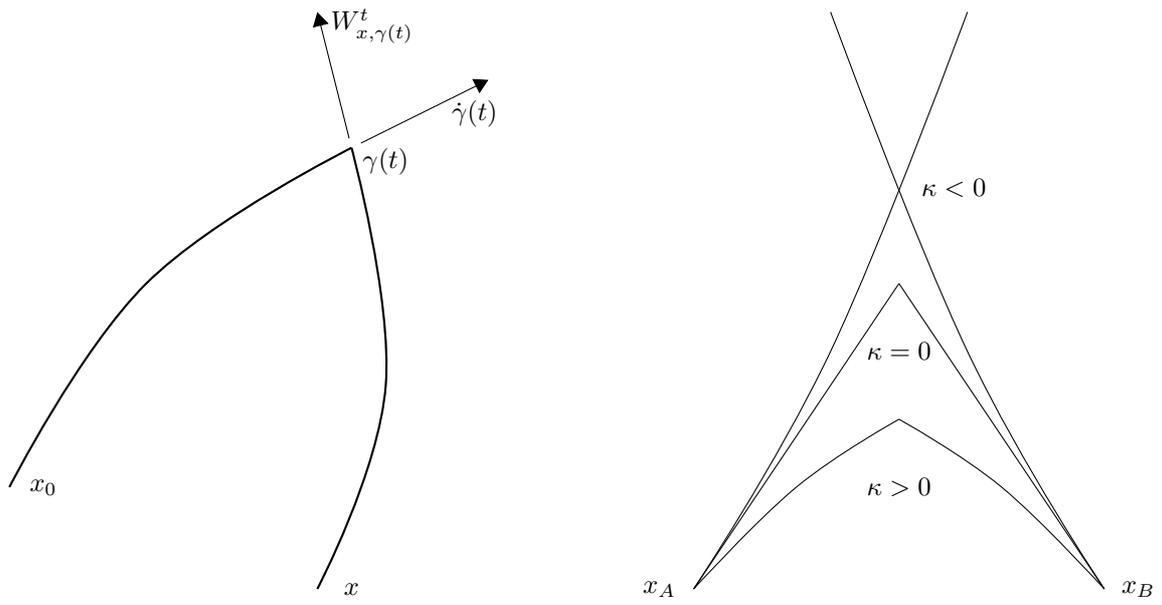
\begin{figure}
\centering
\begin{tikzpicture}[scale=0.9]
\draw[thick]  plot[smooth, tension=.7] coordinates {(-2,-2.5) (0,0.5) (3,2.5)};
\draw[thick]  plot[smooth, tension=.7] coordinates {(3,2.5) (3.5,-1) (2.5,-4)};
\draw (3,2.5) node (v1) {};
\draw [-triangle 60](v1) -- (2.5,4.5);
\draw [-triangle 60](v1) -- (5,3.5);
\node at (-1.5,-2.5) {$x_0$};
\node at (3,-4) {$x$};
\node at (3.5,2.3) {$\gamma(t)$};
\node at (4.8,3) {$\dot\gamma(t)$};
\node at (3.3,4.3) {$W_{x,\gamma(t)}^{t}$};
\draw (11,0.5) -- (8,-4) node (v2) {};
\draw (11,0.5) -- (14,-4) node (v3) {};
\draw  plot[smooth, tension=.7] coordinates {(v2)};
\draw  plot[smooth, tension=.7] coordinates {(v2) (9.5,-2.5) (11,-1.5)};
\draw  plot[smooth, tension=.7] coordinates {(v3) (12.5,-2.5) (11,-1.5)};
\draw  plot[smooth, tension=.7] coordinates {(8,-4) (10,-0.5) (12,4.5)};
\draw  plot[smooth, tension=.7] coordinates {(v3) (12,-0.5) (10,4.5) };
\node at (7.5,-4) {$x_A$};
\node at (14.5,-4) {$x_B$};
\node at (11,-2.5) {$\kappa>0$};
\node at (11,-0.5) {$\kappa=0$};
\node at (11.8,1.9) {$\kappa<0$};
\end{tikzpicture}
\caption{A geometrical interpretation for the function $\dot{c}_t$.}\label{fig:interpr}
\end{figure}

\finereview}

\backmatter

\bibliographystyle{amsalpha}
\bibliography{../biblio/Curv-Biblio}

\printindex

\end{document}